\documentclass[reqno]{amsart}

\usepackage{microtype}

\usepackage{amsmath,mathabx}
\usepackage{mathtools}
\usepackage{amsthm}
\usepackage{amssymb}
\usepackage{hyperref}
\usepackage[linesnumbered,ruled,lined]{algorithm2e}
\usepackage{enumerate,todonotes}
\usepackage{xcolor}

\usepackage{orcidlink}
\usepackage{tikz}
\usetikzlibrary{calc,positioning,shapes.geometric}
 
\colorlet{darkgreen}{green!50!black}

\newcommand{\abs}[1]{\lvert #1 \rvert}

\newcommand{\TC}{\mathsf{TC}}

\newcommand{\Z}{\mathbb{Z}}

\newcommand{\IRR}{\mathsf{IRR}}
\newcommand{\WP}{\mathsf{WP}}

\newcommand{\red}{\mathsf{red}}

\def\Prob{\mathop{\mathsf{Prob}}}

\newcommand{\Runs}{\mathrm{Runs}}
\newcommand{\free}{\mathop{\scalebox{1.5}{\raisebox{-0.2ex}{$\ast$}}}}
\newcommand{\MT}{\mathsf{MT}}
\newcommand{\MP}{\mathsf{MP}}

\newcommand{\FIM}{\mathsf{FIM}}
\newcommand{\PM}{\mathsf{PM}}
\newcommand{\GT}{\mathsf{GT}}
\newcommand{\DIS}{\mathsf{DIS}}
\newcommand{\EQ}{\mathsf{EQ}}
\newcommand{\IDX}{\mathsf{IDX}}
\newcommand{\AIDX}{\mathsf{augIDX}}
\newcommand{\Sch}{\mathsf{Sch}}

\newcommand{\rest}{\mathord\restriction}
\newcommand{\inv}{^{-1}}

\theoremstyle{definition}
\newtheorem{definition}{Definition}
\newtheorem{example}[definition]{Example}
\newtheorem{remark}[definition]{Remark}
\newtheorem{oproblem}[definition]{Open problem}
\theoremstyle{plain}
\newtheorem{theorem}[definition]{Theorem}

\newtheorem{lemma}[definition]{Lemma}
\newtheorem{corollary}[definition]{Corollary}

\newtheorem{claim}{Claim}

\title[Streaming algorithms for groups and semigroups]{Streaming algorithms for groups and semigroups}

\begin{document}

\author[M.~Lohrey]{Markus Lohrey\,\orcidlink{0000-0002-4680-7198}}
\email{lohrey@eti.uni-siegen.de}

\author[L.~L\"uck]{Lukas L\"uck}
\email{lukas.lueck@gmx.net}

\author[A.~Thumm]{Alexander Thumm\,\orcidlink{0009-0005-4240-2045}}
\email{alexander.thumm@uni-siegen.de}

\author[J.~Xochitemol]{Julio Xochitemol\,\orcidlink{0000-0002-5348-0402}}
\email{Julio.JXochitemol@uni-siegen.de}

\address[Markus~Lohrey, Lukas L\"uck, Alexander Thumm, Julio  Xochitemol]{Universit{\"a}t Siegen, Germany}

\begin{abstract}
We investigate deterministic and randomized streaming algorithms for word problems in finitely generated groups and semigroups. 
For this we introduce the notion of a distinguisher: a randomized streaming algorithm that processes two input words in parallel and, with high probability, reaches identical memory states if the words represent the same element, and distinct states otherwise.
We construct such distinguishers with low error probability using logarithmic~--~and in some cases doubly logarithmic~--~space.
For example, our results apply to linear semigroups and to semigroups obtained (under suitable restrictions) via standard constructions such as graph products, wreath products, and semilattice decompositions. 
In case of commutative semigroups and cancellative nilpotent semigroups, we achieve space complexity $\mathcal{O}(\log \log n)$.

We complement these upper bounds with lower bounds demonstrating that certain well-known semigroups do not admit sublinear-space distinguishers. This includes, for example, free inverse monoids of rank at least two and Thompson’s group $F$. Finally, we study randomized streaming algorithms for subgroup membership problems in free groups and their direct products.
\end{abstract}

\maketitle

 \section{Introduction}

 \subsection{Word problem for groups}
 The starting point of this work is the word problem for a finitely generated group $G$. It is defined as follows: 
 Fix a finite set of generators $\Sigma$ for $G$,
 where w.l.o.g.~$a^{-1} \in \Sigma$ for every $a \in \Sigma$. Hence, every element of $G$ can be written as a finite product of elements from $\Sigma$. The input for the word problem is a finite
 word $a_1 a_2 \cdots a_n$ over the alphabet $\Sigma$ and the question is whether this word evaluates to the group identity of $G$.
 The word problem was introduced by Dehn in 1911 \cite{Dehn11}. It is arguably the most important computational problem in group theory and has
 been studied by group theorists as well as computer scientists; see \cite{Mill92} for a survey. Since the 1970s, complexity theoretic investigations of word problems also moved into focus.
 For many important classes of groups it turned out that the word problem belongs to low-level complexity
 classes. An important result in this direction was proved by Lipton and Zalcstein \cite{LiZa77} and Simon \cite{Sim79}:
 if $G$ is a finitely generated linear group over an arbitrary field $F$ (i.e., a finitely generated group of invertible
 matrices over $F$), then the word problem for $G$ can be solved in deterministic logarithmic space.  
 Lipton and Zalcstein considered the case of a field $F$ of characteristic zero, and Simon covered the case of prime 
 characteristic. Related results can be found in \cite{KonigL18algo,Weiss16BS}.

 Language-theoretic aspects of the word problem have also been studied extensively. 
 In this context, the word problem of a group $G$ with a finite generating set $\Sigma$ is identified with the formal language $\WP(G,\Sigma)$ consisting of all
 words over the alphabet~$\Sigma$ that evaluate to the identity element of $G$. 
 For instance, Anissimov~\cite{Ani71} showed that $\WP(G,\Sigma)$ is regular if and only if $G$ is finite, and
 Muller and Schupp~\cite{MuSch83} showed that $\WP(G,\Sigma)$ is context-free  if and only if $G$ is virtually free.\footnote{For a class of groups $\mathcal{C}$, a group $G$ is called virtually $\mathcal{C}$ if it has a finite-index subgroup in $\mathcal{C}$.}
 For an overview of results in this direction, we refer to \cite{HRR2017}.
 
 \subsection{Streaming algorithms for word problems}
 In this paper we initiate the study of streaming algorithms for word problems. These are algorithms that do not have random access on the whole input. Instead, the $k$-th input symbol is only available at time $k$ \cite{AlonMS99}. Quite often, streaming algorithms are randomized and have a bounded error probability. Usually, one is interested in the space used by a streaming algorithm, but update times
 (that is, the worst case time spend to process a new input symbol) have also been studied. 
 Clearly, every regular language $L$ admits a deterministic streaming algorithm with constant
 space, namely one given by a deterministic finite automaton recognizing $L$.
 Randomized streaming algorithms for context-free languages have been studied in \cite{BabuLRV13,BathieS21,FrancoisMRS16,MagniezMN14}.
 
 Before we discuss our results, 
 we should say a word about our model of (deterministic or randomized) streaming algorithms. 
 It is non-uniform in the sense that for every input length
 $n$ we have a separate algorithm that only handles inputs of length at most $n$. Formally, this yields an infinite sequence 
 $(\mathcal{A}_n)_{n \geq 0}$ of (deterministic or probabilistic) finite
 automata, and we are only interested in the behavior of $\mathcal{A}_n$ on input words of length at most $n$.
 Non-uniformity means in this context that the $\mathcal{A}_n$ do not have to follow a common pattern. In principle, the 
 function $n \mapsto \mathcal{A}_n$ could be undecidable.
 The space complexity of the streaming algorithm $(\mathcal{A}_n)_{n \geq 0}$ is then the function $n \mapsto \lceil \log_2 (\text{number of states
 of $\mathcal{A}_n$}) \rceil$. This is the number of bits needed to encode the states of $\mathcal{A}_n$.
 In this model, every problem can be solved by a deterministic streaming algorithm of space complexity $\mathcal{O}(n)$.
 The streaming algorithm simply stores the whole input word of length $n$ in space $\mathcal{O}(n)$ (assuming the input alphabet
 has constant size).

 For a finitely generated group $G$ with finite generating set $\Sigma$,
the \emph{deterministic} (resp., \emph{randomized) streaming space complexity} of $\WP(G,\Sigma)$ is the 
space complexity of the best deterministic (resp., randomized) streaming algorithm for $\WP(G,\Sigma)$. 
If $\WP(G,\Sigma)$ has  deterministic/randomized streaming space complexity $s(n)$ and $\Sigma'$ is another
generating set for $G$ then there is a constant $c$ such that the  deterministic/randomized streaming space complexity 
of $\WP(G,\Sigma')$ is bounded by $s(c n)$;  see Lemma~\ref{lemma-gen-set}. For the space complexity 
functions $s(n)$ that appear in this paper we have $s(c n) = \Theta(s(n))$ for every constant $c$.
This allows us to suppress the generating set and say that the deterministic/randomized streaming space complexity
of $\WP(G)$ is in $\mathcal{O}(s(n))$, resp. $\Omega(s(n))$. 

The deterministic streaming space complexity of $\WP(G,\Sigma)$
is directly linked to the growth function $\gamma_{G,\Sigma}(n)$ of the group $G$.
The latter is the number of different group elements 
of $G$ that can be represented by words over the finite generating set $\Sigma$ of length at most $n$ (also here the 
generating set $\Sigma$ only has a minor influence). The deterministic streaming space complexity of $\WP(G,\Sigma)$
turns out to be $\log_2 \gamma_{G,\Sigma}(n/2)$ up to a small additive constant (Theorem~\ref{thm-det-growth}). 
The growth of finitely generated groups is a well-investigated
topic in geometric group theory. A famous theorem of Gromov says that a finitely generated group has polynomial growth if and only if it is virtually nilpotent; see \cite{Harpe00,Ceccherini-SilbersteinA21}
for  a discussion. Theorem~\ref{thm-det-growth} reduces all questions about the  deterministic streaming space complexity of word problems to questions about growth
functions. Due to this, we mainly study randomized streaming algorithms for word problems in this paper.

In the randomized setting, the growth of $G$ still yields a lower bound: 
The randomized streaming space complexity of $\WP(G,\Sigma)$ is lower bounded by
$\Omega(\log \log \gamma_{G,\Sigma}(n/2))$ (Theorem~\ref{thm-lower-bound-random}). 
A large class of groups, where this lower bound can be exactly matched by an upper bound, is the class of finitely generated linear groups.
Recall that Lipton and Zalcstein  \cite{LiZa77}  and Simon \cite{Sim79} showed that the word problem of a finitely generated linear group can be solved in logarithmic space. 
Their algorithm can be turned into a randomized streaming algorithm with logarithmic space complexity. 
From this we would like to get more groups with space-efficient randomized streaming algorithms using group constructions 
like graph products (that generalize free and direct products) or wreath products. It turns out that for this purpose a stronger notion of randomized streaming algorithms for word problems is useful. 

\subsection{Distinguishers}
A so-called \emph{$\epsilon$-distinguisher} with $0 \leq \epsilon < 1/2$ for a  finitely generated group $G$ with the finite generating set $\Sigma$
is a randomized streaming algorithm $\mathcal{R} = (\mathcal{A}_n)_{n \geq 0}$ (in the non-uniform sense explained above) such that 
for all words $u,v \in \Sigma^*$ of length at most $n$ the following hold: 
\begin{itemize}
\item If $u$ and $v$ evaluate to the same element of $G$
then with probability at least $1-\epsilon$, $u$ and $v$ lead to the same memory state of $\mathcal{A}_n$, and 
\item if $u$ and $v$ evaluate to different elements of $G$
then with probability at least $1-\epsilon$, $u$ and $v$ lead to different memory states of $\mathcal{A}_n$; see Section~\ref{sec-injective}.
\end{itemize}
The error probability $\epsilon$ may depend on the input length $n$ and will be in $o(1)$ for the all the $\epsilon$-distinguishers
that we construct.

It is easy to obtain from an $\epsilon$-distinguisher $\mathcal{R}$ for the group $G$ a randomized streaming algorithm 
$\mathcal{S}$ for the word problem of $G$ with error probability $\epsilon$. Moreover, the space complexity of $\mathcal{S}$ is only twice
the space complexity of $\mathcal{R}$; see Lemma~\ref{lemma-injective-0}. 

An advantage of the distinguisher concept is that its definition can be directly extended to monoids. 
In fact such a distinguisher serves as a natural notion of a streaming algorithm for the word
problem of a monoid. The word problem for a monoid $M$ generated by $\Sigma$ is the set of all pairs $(u,v) \in \Sigma^* \times \Sigma^*$
such that $u$ and $v$ represent the same element of $M$. In an $\epsilon$-distinguisher for $M$ the two input words $u$ and $v$ arrive in an
interleaved fashion. If $(u,v)$ belongs to the word problem of $M$ then with high probability the distinguisher arrives in the same memory state
after reading $u$ and $v$, otherwise the memory states are different with high probability.
We define an $\epsilon$-distinguisher for a semigroup $S$ that is not a monoid as an $\epsilon$-distinguisher for the monoid $S^1$ (the semigroup $S$ together
with an adjoined neutral element).

\subsection{Upper bounds on the space complexity of distinguishers}
We show that for every finitely generated linear monoid $M$ (i.e., a monoid $M$ of matrices over an arbitrary field)
there is an $\epsilon(n)$-distinguisher with 
space complexity $\mathcal{O}(\log n)$ and inverse polynomial 
error probability $\epsilon(n) = 1/n^c$ for every constant $c \geq 1$ (Theorem~\ref{thm-lin}).

For the following classes of (semi)groups we obtain $\epsilon(n)$-distinguishers with space complexity $\mathcal{O}(\log \log n)$ (which
is the best one can hope for a finitely generated infinite semigroup due to Theorem~\ref{thm-lower-bound-random}):
\begin{itemize}
\item finitely generated cancellative semigroups of polynomial growth (Theorem~\ref{thm-cancellative-poly-growth})
and hence, in particular, finitely generated virtually nilpotent groups (Corollary~\ref{thm-virt-nilpotent});
\item finitely generated commutative semigroups (Theorem~\ref{thm-comm}).
\end{itemize}
In both cases, an inverse polylogarithmic error probability $\epsilon(n) = (1/\log n)^c$ 
(for any constant $c \geq 1$) can be obtained.
Using a known gap theorem for the growth of linear groups \cite{Miln68,Wolf68}, it turns out that the randomized streaming space complexity of the word problem for a finitely generated
linear group $G$ is either $\Theta(\log \log n)$ (if $G$ is virtually nilpotent) or $\Theta(\log n)$ (if $G$ is not virtually nilpotent), 
see Corollary~\ref{thm--gap-linear}.

In order to obtain further (in general non-linear) (semi)groups with logspace distinguishers, we prove transfer theorems for the following operations:

\medskip
\paragraph{\em Finite extensions.}
 If $G$ is a finitely generated group having an $\epsilon(n)$-distinguisher with space complexity $s(n)$ then every
finite extension of $G$ has an $\epsilon(c n)$-distinguisher with space complexity $s(c n) + \mathcal{O}(1)$ for some constant $c$
(Theorem~\ref{thm-finite-ext}).

\medskip
\paragraph{\em Graph products.} Graph products of monoids interpolate between free products and direct products. 
If $M_1, \ldots, M_c$ are finitely generated left-cancellative monoids having $\epsilon(n)$-distinguishers with space 
complexities $s_1(n), \ldots, s_c(n)$, then there is an $\epsilon'(n)$-distinguisher of space complexity $\mathcal{O}(\sum_{i=1}^c s_i(n)+\log n)$
for any graph product of the $M_i$, where $\epsilon'(n) = \Theta(\epsilon(n) \cdot n^2)$ (Theorem~\ref{thm-GP}).
We do not know whether the factor $n^2$ in $\epsilon'(n)$ is needed but it still allows nontrivial applications (for which we clearly need
$\epsilon'(n) < 1$): if for instance all $M_i$ are linear then we can enforce $\epsilon(n) = 1/n^c$ for any $c \geq 1$ (Theorem~\ref{thm-lin}).
We also prove a variant of Theorem~\ref{thm-GP}, where every $M_i \setminus \{1\}$ is required to be an ideal of $M_i$ (Theorem~\ref{thm-GP-S}).

\medskip
\paragraph{\em Wreath products.}
 If $A$ is a finitely generated abelian group and $G$ is a finitely generated group with an $\epsilon(n)$-distinguisher 
of space complexity $s(n)$, then there is an $\epsilon'(n)$-distinguisher of space complexity $\mathcal{O}( s(n)+\log n)$
for the wreath product $A \wr G$, where  $\epsilon'(n) = \Theta(\epsilon(n) \cdot n^2)$
(Corollary~\ref{thm-wreath-by-abelian}). This result can be extended to the case where the abelian group $A$ is replaced
by a cancellative and commutative monoid (Corollary~\ref{thm-wreath-cancel-comm}).

\medskip
\paragraph{\em Semilattices of semigroups}
This is an important construction in semigroup theory, which is usually described as a decomposition: a semigroup $S$ is
decomposed as $S = \biguplus_{\alpha \in Y} S_\alpha$, where $Y$ is a semilattice (a commutative and idempotent semigroup),
every $S_\alpha$ is a subsemigroup, and the multiplication in $S$ must be compatible with the multiplication in $Y$ in a certain way.
The multiplication between elements from 
different $S_\alpha$ can be complicated in general, but if every $S_\alpha$ is a monoid, it turns out to be more structured (for groups
this was known before) and moreover the semilattice $Y$ must be finite if $S$ is finitely generated.
 We show that if $S$ is a finitely generated semilattice of monoids $M_\alpha$, each having an $\epsilon(n)$-distinguisher with 
 space complexity $s(n)$, then $S$ has a $(c \cdot \epsilon(d n))$-distinguisher with space complexity $c \cdot s(dn)$ (Theorem~\ref{thm-semilattice-deco}, which is in fact  a bit more general). Here, $c$ is the size of 
 the (necessarily finite) semilattice and $d$ is another constant depending on $S$. From this result we derive logspace distinguishers
 for free Clifford semigroups and regular nilpotent semigroups (Corollaries~\ref{thm-regular-nilpotent} and \ref{thm-clifford}).
 
\medskip
\noindent
Using the above transfer results, we obtain also non-linear groups with 
a logspace distinguisher, e.g., metabelian groups (Corollary~\ref{coro-metabel}) and free solvable groups (Corollary~\ref{thm-free-solvable}). 
Regarding the probability of error in the transfer results, we will carry out more refined calculations of the probability of false positives and false negatives, respectively.

\subsection{Lower bounds on the space complexity of distinguishers}

In Section~\ref{sec lower} we prove lower bounds for the randomized streaming space complexity of group word problems
and distinguishers for semigroups. All the lower bounds that we state below are shown using randomized communication 
complexity; see Section~\ref{sec-CC}.

We start with some sharp lower bounds for certain inverse monoids. Inverse monoids are an important class of monoids that
are somehow located between groups and monoids \cite{petrich:1984}. 
They are characterized by the condition that every element $x$ has a unique ``inverse'' 
$x^{-1}$ satisfying $x x^{-1} x = x$ and $x^{-1} x x^{-1} = x^{-1}$. The class of inverse monoids is a variety. Hence, for every set $\Sigma$ there
exists the free inverse monoid $\FIM(\Sigma)$ generated by $\Sigma$. We show that if $\Sigma$ is a finite set of size at least two, then every
$\epsilon$-distinguisher (for a constant $\epsilon < 1/2$) for $\FIM(\Sigma)$
 has space complexity $\Theta(n)$ (Theorem~\ref{thm-FIM-2}). If $\Sigma$ is a singleton set then there is a $0$-distinguisher for $\FIM(\Sigma)$ 
 with space complexity $\mathcal{O}(\log n)$ and this bound can be also matched with a lower bound of $\Omega(\log n)$
 for every constant error probability (Theorem~\ref{thm-FIM-1}).
 Similar results hold for polycyclic monoids (another important
 class of inverse monoids that is tightly related to context-free languages \cite{RenKam09}). In particular, let us mention that 
 the bicyclic monoid $B = \langle a,b \mid ab=1 \rangle$ has a $0$-distinguisher with space complexity $\mathcal{O}(\log n)$ 
 and that this upper bound can be matched with a lower bound for every constant error probability (Theorem~\ref{thm-bcm}).
 
For wreath products of the form $M \wr G$ such that $G$ is a finitely generated infinite group and $M$ is a non-commutative or non-cancellative monoid,
we show that every $\epsilon$-distinguisher ($0 < \epsilon < 1/2$) must have space complexity $\Omega(n)$
(Theorem~\ref{lower-bound-non-cancel-non-commute}), which nicely contrasts the above mentioned upper bounds for wreath products;
see Theorem~\ref{thm-wreath-cancel-comm}. 

For wreath products $H \wr G$ with $G$ and $H$ finitely generated groups, $H$ non-abelian, and $G$ infinite a lower bound of $\Omega(n)$
holds for the randomized streaming space complexity of the word problem 
(Theorem~ \ref{lower-wreath}). From this lower bounds we can derive lower bounds for some concrete (and quite famous)
groups. Our first example is Thompson's group $F$. It is a finitely presented group introduced by Richard Thompson in 1965.
Thompson's group $F$  belongs due to its unusual properties to the most intensively studied infinite groups; see e.g. \cite{CaFlPa96}. From a computational perspective it is interesting to note that $F$ is co-context-free (i.e., the set of all words over any set of generators that do not evaluate to the group identity is a context-free language) \cite{LehSchw07}. This implies that the word problem for Thompson's group $F$
is in {\sf DSPACE}$((\log n)^2)$. 
Since $F$ is non-abelian and contains a copy of $F \wr \mathbb{Z}$, we obtain the lower bound $\Omega(n)$ for the streaming
space complexity of the word problem of $F$  (Corollary~\ref{coro-thompson}). 

We also consider the famous Grigorchuk group $G$ \cite{Grigorchuk80}, which 
was the first example of a group with intermediate growth as well as the 
first example of a group that is amenable but not elementary amenable. We 
show that the deterministic streaming space complexity of the word problem for $G$ is $\mathcal{O}(n^{0.768})$, whereas the 
randomized streaming space complexity is $\Omega(n^{1/3})$ (Theorem~\ref{thm-grig}).

 Concerning our transfer result for graph products, we show that the conditions on the monoids $M_i$ in Theorems~\ref{thm-GP} and \ref{thm-GP-S} 
 cannot be completely avoided. This is already true for free product: We show that if $M$ is any monoid where $M \setminus \{1\}$ is 
 not an ideal and $B$ is the bicyclic monoid, then every $\epsilon$-distinguisher ($0 < \epsilon < 1/2$) for the free product $B \ast M$
 has space complexity $\Theta(n)$; see Theorem~\ref{thm-lower-bound-free-product}.

 \subsection{Streaming algorithms for membership problems}
In Section~\ref{sec member} we consider randomized streaming algorithms for subgroup and submonoid membership problems. In a subgroup
(resp., submonoid) membership problem one has a subgroup (resp., submonoid) $H$ of a finitely generated group $G$ and for a given input word $w \in \Sigma^*$ ($\Sigma$ is again
a finite set of generators for $G$) one has to determine whether $w$ represents an element of $H$. The word problem is the special
case where $H=1$. We present a randomized streaming algorithm with logarithmic space complexity for the case where $G$ is a finitely generated free group and $H$ is a finitely generated subgroup of $G$ (Theorem~\ref{thm-sbmp-free}). 
We also generalize this result by constructing for every $c \ge 1$
 a $1/n^c$-distinguisher for the Schreier coset graph of a finitely generated subgroup in a finitely generated free group
 (Theorem~\ref{thm-schreier-dist}).

Finally, we show that Theorem~\ref{thm-sbmp-free} cannot be extended 
to one of the following three situations:
\begin{itemize}
\item $G$ is a finitely generated free group of rank two but $H$ is not finitely generated (Theorem~\ref{thm-thompson2}).
\item $G$ is a direct product of two free groups of rank two and $H$ is a finitely generated subgroup of $G$ (Theorem~\ref{thm-mihailova}).
\item $G$ is the free group $F(a,b)$ generated by $a$ and $b$ and $H$ is the free monoid $\{a,b\}^* \subseteq F(a,b)$.
\end{itemize}
In all three cases we prove a lower bound of $\Omega(n)$ for a randomized streaming algorithm for the corresponding 
membership problem.

\subsection*{Related results.}
In this paper, we are only interested in streaming algorithms for a fixed infinite group. 
Implicitly, streaming algorithms for finite groups are studied in \cite{GowersV19}.
Obviously, every finite group $G$ has deterministic streaming space complexity $\mathcal{O}(\log |G|)$.\footnote{In our setting,
$|G|$ would be a constant, but for the moment let us make the dependence on the finite group $G$ explicit.}
In \cite{GowersV19}, it is shown that for the group $G = \mathsf{SL}(2,\mathbb{F}_p)$ this upper bound is matched
by a lower bound, which even holds for the randomized streaming space complexity. More precisely, Gowers and Viola
study the communication cost of the following problem: Alice receives a sequence of elements $a_1,\ldots, a_n \in G$,
Bob receives a sequence of elements $b_1,\ldots, b_n \in G$ and they are promised that the interleaved product
$a_1 b_1 \cdots a_n b_n$ is either $1$ or some fixed element $g \in G \setminus\{1\}$ and their task is to determine which
of these two cases holds. For $G = \mathsf{SL}(2,\mathbb{F}_p)$ it is shown that the randomized communication 
complexity of this problem is $\Theta(\log |G| \cdot n)$ (the upper bound is trivial). From this it follows easily that 
the randomized streaming space complexity of $\mathsf{SL}(2,\mathbb{F}_p)$ is $\Omega(\log |G|)$.

Our transfer theorems for graph products (Theorem~\ref{thm-GP}) and wreath products (Corollary~\ref{thm-wreath-by-abelian})
have similar counterparts in classical complexity theory:
For graph products, the following result is shown in \cite{DiekertK16}: If the word problem for every group $G_i$ ($1 \le i \le k$)
can be solved in deterministic logspace on a Turing machine then the same is true for every graph product of the groups $G_1, \ldots, G_k$. 
Kausch in his thesis \cite{kausch2017parallel} strengthened this result by showing that the word problem of the graph
product is $\mathsf{AC}_0$-Turing-reducible to the word problems of the $G_i$ and the free group of rank two.
A similar result holds for wreath products: the word problem for $G \wr H$ is 
$\mathsf{AC}_0$-Turing-reducible to the word problems for $G$ and $H$ \cite{MiasnikovVW19}.

Interestingly, the word problem for a free inverse monoid can be solved in logspace on a Turing machine; see \cite{LoOn07} for a more general
result. Recall that we prove a space lower bound of $\Omega(n)$ for any $\epsilon$-distinguisher for a free inverse monoid generated
by at least two elements. 

Some of the results in this paper were presented at the conferences MFCS 2022 and MFCS 2024; extended abstracts appeared in \cite{LohreyL22,LohreyX24}.

\section{Preliminaries}

For integers $a \le b$ let $[a,b]$ be the integer interval $\{a, a+1, \ldots,b\}$. 
We write $\exp(x)$ for $e^x$, where $e$ is Euler's number. For a number $\alpha$ we denote with $\mathbb{P}_\alpha$ the
set of all primes $p \leq \alpha$.

\subsection{Words}
Let $\Sigma$ be a finite alphabet.
As usual, we write $\Sigma^*$ for the set of all finite words over the alphabet $\Sigma$. 
The empty word is denoted with $\varepsilon$.
For  a word $w = a_1 a_2 \cdots a_n$ ($a_1, a_2, \ldots, a_n \in \Sigma$) let $|w|=n$
be its length and $w[i] = a_i$ (for $1 \leq i \leq n$) the symbol at position $i$.
A prefix of a word $w$ is a word $u$ such that $w = uv$ for some word $v$.
We denote with $\mathcal{P}(w)$ the set of all prefixes of $w$.
Let  $\Sigma^+ = \Sigma^* \setminus \{ \varepsilon \}$ be the set of non-empty words and 
$\Sigma^{\le n} = \{ w \in \Sigma^* \colon |w| \leq n\}$ be the set of all words of length at most $n$. 
For a subalphabet $\Theta \subseteq \Sigma$ we denote with $\pi_{\Theta} : \Sigma^* \to \Theta^*$ the
projection homomorphism that deletes all symbols from $\Sigma \setminus \Theta$ in a word:
$\pi_{\Theta}(a)=a$ for $a \in \Theta$ and $\pi_{\Theta}(a)=\varepsilon$ for $a \in \Sigma \setminus \Theta$.

\subsection{Probabilities}
We assume that the reader has some basic knowledge in probability theory. 
We will only deal with discrete probability distributions. 
We write $[0,1]_{\mathbb{R}}$ for the set $\{ p \in \mathbb{R} \colon 0 \leq p \leq 1\}$ of all probabilities.

If $\iota$ is a probability distribution on a finite set $A$ and  $\mathcal{E} : A \to \{0,1\}$ is a boolean predicate on $A$,
we define the probability
\[
\Prob_{a \sim \iota}[\mathcal{E}(a)]  \coloneqq  \sum_{a \in A} \mathcal{E}(a) \cdot \iota(a),
\]
which is of course the expected value of $\mathcal{E}$.
If $\iota$ is the uniform distribution on $A$, we just write $\Prob_{a \in A}[\mathcal{E}(a)]$ or simply $\Prob_{a}[\mathcal{E}(a)]$.
When writing 
\[
\Prob_{a \sim \iota_1, \, b \sim \iota_2}[\mathcal{E}(a,b)] 
\]
we assume that $a$ and $b$ are chosen independently.

Throughout we will make frequent use of the Chernoff bound, of which there are many variants. 
We use the following form, which can be found in \cite[eq.~(1)]{DiGo23}.
\begin{theorem}
Let $\delta > 0$, and
$X_1, X_2, \ldots X_k$ be independent identically distributed Bernoulli random variables with 
$\Prob[X_i = 1] = \epsilon$ and $\Prob[X_i = 0] = 1-\epsilon$. 
Then 
\begin{equation}\label{chernoff}
\Prob\left[\sum_{i=1}^k X_i > (1+\delta) \epsilon k\right] < \exp\left( - \frac{\delta^2 \epsilon k}{\delta+2}\right) 
\stackrel{(\text{if }\delta \geq 1)}{\leq} 
\exp\left( - \frac{\delta \epsilon k}{3} \right) .
\end{equation}
\end{theorem}

\subsection{Communication complexity} \label{sec-CC}

Our lower bounds for randomized streaming space complexity will be based on randomized communication
complexity. In this section, we 
present the necessary background; see \cite{KuNi96}
for a detailed introduction.  

Consider a function $f \colon X \times Y \to \{0,1\}$ for some finite sets $X$ and $Y$.
A \emph{randomized (communication) protocol} $P$ for $f$ consists of two parties called Alice and Bob.
The input for Alice is an element $x \in X$ and a random choice $r \in R$.
Similarly, Bob's input is an element $y \in Y$ and a random choice $s \in S$.
Here, $R$ and $S$ are finite sets with probability distributions $\iota$ and $\zeta$, respectively.
The goal of Alice and Bob is to compute $f(x,y)$.
For this, they communicate in a finite number of rounds, where in each round either Alice sends
a bit string to Bob or vice versa. 
The protocol determines which of the two communication directions is chosen.
At the end, Bob outputs a bit $P(x,r,y,s)$.  
In a \emph{one-way protocol}, there is only one round, where Alice sends a bit string to Bob.
The protocol $P$ \emph{computes} $f$ if 
for all $(x,y) \in X \times Y$ we have
\begin{equation} \label{eq-rcc}
\Prob_{r \sim \iota, \, s \sim \zeta}[P(x,r,y,s) \neq f(x,y)] \le \frac{1}{3}.
\end{equation}

The cost of the protocol is the maximum number of transmitted bits, where the maximum
is taken over all $(x,r,y,s) \in X \times R \times Y \times S$.
The \emph{randomized (one-way) communication complexity} of $f$ is the minimal cost
among all (one-way) randomized protocols that compute $f$. Here, the size of the finite sets $R$ and $S$ is not restricted.
The choice of the constant $1/3$ in \eqref{eq-rcc} is arbitrary in the sense that changing the constant
to any $\epsilon < 1/2$ only changes the communication complexity
by a constant (depending on $\epsilon$),
see \cite[p.~30]{KuNi96}. Also note that we only use the private version of randomized communication protocols, where Alice
and Bob make private and independent random choices from the sets $R$ and $S$, respectively, and their choices are not known to the other
party (in contrast to the public version).

A \emph{promise communication problem} is defined by a function $f$ from a subset $A \subseteq X \times Y$ to $\{0,1\}$.
In this case, \eqref{eq-rcc} must only hold for all $(x,y) \in A$. 

We will use the following communication problems.
\begin{itemize}
\item $\EQ_n$ (equality): 
  Alice's and Bob's inputs are numbers $x \in [1,n]$ and $y \in [1,n]$, respectively. 
  Bob should output $1$ if and only if $x=y$.
\item $\GT_n$ (greater than): 
  Alice's and Bob's inputs are numbers $x, y \in [1,n]$, respectively. 
  Bob should output $1$ if and only if $x < y$.
\item  $\DIS_n$ (disjointness):  
  Alice's and Bob's inputs are bit strings $x, y \in \{0,1\}^n$, respectively.
  Bob should output $1$ if and only if there is no position $i \in [1,n]$ such that $x[i] = y[i] = 1$.
\item  $\IDX_n$ (index):  
  Alice's input is a bit string $x \in \{0,1\}^n$, whereas Bob's input is a number $i \in [1,n]$.
Bob should output $1$ if and only if $x[i]=1$.
\item $\AIDX_n$ (augmented index):  
  Alice's input is a bit string $x \in \{0,1\}^n$, whereas Bob's input is a number $i \in [1,n]$ and a bit string $y \in \{0,1\}^{n-i}$ with the 
  restriction that $y$ is a suffix of $x$.
Bob should output $1$ if and only if $x[i]=1$. This is a promise problem as defined above.
\end{itemize}
We make use of the following results. The statements in the first two points hold for the general as well as the one-way communication model.

\begin{theorem} \label{theorem-randCC}
The following hold.
\begin{itemize}
\item The randomized (one-way)  communication complexity of $\EQ_n$ is $\Theta(\log \log n)$ \cite[Section~3.2]{KuNi96}.
\item The randomized (one-way)  communication complexity of $\DIS_n$ is $\Theta(n)$  \cite[Section~4.6]{KuNi96}.
\item The randomized one-way  communication complexity of $\GT_n$ is $\Theta(\log n)$ \cite[Theorem~3.8]{KremerNR99} and \cite[Theorem~19]{MiltersenNSW98}.
\item The randomized one-way  communication complexity of $\IDX_n$ is $\Theta(n)$ \cite[Theorem~3.7]{KremerNR99}.
\item The randomized one-way  communication complexity of $\AIDX_n$ is $\Theta(n)$ \cite[Theorem~5.1]{BaIPW10}.
\end{itemize} 
\end{theorem}

\subsection{Automata}

In this section we introduce deterministic and probabilistic automata. The latter will serve as our formalization of randomized streaming algorithms.

\subsubsection{Deterministic automata} \label{sec-det-auto}

A \emph{deterministic automaton} is a tuple 
\[\mathcal{A} = (Q, \Sigma, q_0, \delta,F),\] where 
$Q$ is a possibly infinite set of states, $\Sigma$ is a finite alphabet, $q_0 \in Q$ is the initial state,
$\delta : Q \times \Sigma \to Q$ is the transition function, and $F$ is the set of final states.
The transition function $\delta : Q \times \Sigma \to Q$ is extended to a mapping
$\delta : Q \times \Sigma^* \to Q$ in the usual way: $\delta(q,\varepsilon)=q$ and $\delta(q, aw) = \delta( \rho(q,a),w)$
for all $q \in Q$, $a \in \Sigma$ and $w \in \Sigma^*$. 
The language accepted by $\mathcal{A}$ is $L(\mathcal{A})= \{w \in \Sigma^* : \delta(q_0,w) \in F\}$.
The \emph{growth function} of $\mathcal{A}$ is $\gamma_{\mathcal{A}}(n) = | \{ \delta(q_0, w) \colon w \in \Sigma^{\le n} \}|$.
If $Q$ is finite, then $\mathcal{A}$ is called a \emph{deterministic finite automaton}, DFA for short.

\subsubsection{Probabilistic finite automata}

A \emph{probabilistic finite automaton} (PFA) \cite{Paz71,Rabin63}
\[\mathcal{A} = (Q,\Sigma,\iota,\rho,F)\]
consists of a finite set of states $Q$, a finite alphabet $\Sigma$,
an \emph{initial state distribution} $\iota \colon Q \to [0,1]_{\mathbb{R}}$,
a \emph{transition probability function} $\rho \colon Q \times \Sigma \times Q \to [0,1]_{\mathbb{R}}$
and a set of final states $F \subseteq Q$ such that
$\sum_{q \in Q} \iota(q) = 1$ and $\sum_{q \in Q} \rho(p,a,q) = 1$ for all $p \in Q$, $a \in \Sigma$.
Intuitively, for an input word $w \in \Sigma^*$, the PFA $\mathcal{A}$ first chooses randomly, according
to the distribution $\iota$, an initial state $q_0$. Then, it reads the word $w$ symbol by symbol. If $p \in Q$
is the current state (initially, it is $q_0$) and $a \in \Sigma$ is the next symbol in $w$, then the next state
is $q$ with probability $\rho(p,a,q)$.

If the transition probability function $\rho$ maps into $\{0,1\}$ then $\mathcal{A}$ is called a
\emph{semi-probabilitistic finite automaton} (semiPFA). 
This means that after choosing the initial state according to the distribution $\iota$, $\mathcal{A}$
proceeds deterministically. It is natural to identify the transition probability function $\rho \colon Q \times \Sigma \times Q \to [0,1]_{\mathbb{R}}$
of a semiPFA with the unique transition function $\delta \colon Q \times \Sigma \to Q$ such that
$\rho(q,a,\delta(q,a)) = 1$ for all $q \in Q$ and $a \in \Sigma$. We then extend $\delta$ to $\delta \colon Q \times \Sigma^* \to Q$ as in
Section~\ref{sec-det-auto}.  A semiPFA  $(Q, \Sigma, \iota, \delta, F)$ can be also seen as a probability distribution
on the set of DFA $(Q, \Sigma, q, \delta, F)$ for $q \in Q$, where $(Q, \Sigma, q, \delta, F)$ has probability $\iota(q)$.
If in addition also  $\iota$ maps into $\{0,1\}$, then $\mathcal{A} = (Q,\Sigma,\iota,\delta,F)$ can be identified with the DFA $(Q,\Sigma, q_0, \delta,F)$
where $q_0 \in Q$ is the necessarily unique state with $\iota(q_0)=1$.

Let  $\mathcal{A} = (Q,\Sigma,\iota,\rho,F)$ be a PFA. In the following, we consider discrete random variables that take
values in the set $Q$. For this, one should assume w.l.o.g.~that $Q$ is a subset of $\mathbb{N}$ (which 
will be also useful for other reasons in Sections~\ref{sec-graph-prod} and \ref{sec-wreath}).
For such a random variable $X$ and $a \in \Sigma$ we define the random variable $X \cdot a$ (which also takes values from $Q$) by
\[ 
\Prob[X \cdot a = q] = \sum_{p \in Q} \Prob[X = p] \cdot \rho(p,a,q).
\]
For a word $w \in \Sigma^*$ we define a random variable $\mathcal{A}(w)$ with values from $Q$ inductively
as follows: the random variable $\mathcal{A}(\varepsilon)$ is defined such that
$\Prob[\mathcal{A}(\varepsilon) = q] = \iota(q)$ for all $q \in Q$. Moreover, $\mathcal{A}(wa) = \mathcal{A}(w) \cdot a$ for all $w \in \Sigma^*$ and $a \in \Sigma$.
Thus, $\Prob[\mathcal{A}(w) = q]$ is the probability that $\mathcal{A}$ is in state $q$ after reading $w$.

We can define $\Prob[\mathcal{A}(w) = q]$ also via runs:
A \emph{run} on a word $a_1a_2 \cdots a_m \in \Sigma^*$ in the PFA $\mathcal{A}$ is a sequence $\pi = (q_0,a_1,q_1,a_2,\dots,a_m,q_m)$
where $q_0, \dots, q_m \in Q$. We say that $\pi$ ends in $q_m$.
Given a run $\pi$ in $\mathcal{A}$ we define $\iota\rho^*(\pi)= \iota(q_0) \cdot \prod_{i=1}^n \rho(q_{i-1},a_i,q_i)$.
For every $w \in \Sigma^*$ the function $\iota\rho^*$ is a probability distribution on
the set $\Runs(w)$ of all runs of $\mathcal{A}$ on $w$. Then, $\Prob[\mathcal{A}(w) = q]$ is the sum of 
all the probabilities $\iota\rho^*(\pi)$, where $\pi \in \Runs(w)$ ends in $q$.

If $\mathcal{A} = (Q, \Sigma, \iota, \delta, F)$ is a semiPFA, then we obtain
\[
\Prob[\mathcal{A}(w) = q] = \sum \{ \iota(p) \colon p \in Q, \delta(p,w) = q \}.
\]
For a language $L \subseteq \Sigma^*$, a PFA $\mathcal{A}$ and a word $w \in \Sigma^*$ we define
the \emph{error probability} of $\mathcal{A}$ on $w \in \Sigma^*$ for $L$ as
\[
\epsilon(\mathcal{A},w,L) = \begin{cases}
\Prob[ \mathcal{A}(w) \notin F] & \text{ if } w \in L, \\
\Prob[ \mathcal{A}(w) \in F] & \text{ if } w \notin L.
\end{cases}
\]

\subsubsection{Sequential transducer} \label{sec transducer}

In Section~\ref{sec-graph-prod} we make use of (left-)sequential transducers, see e.g.~\cite{Ber79} for more details.
A sequential transducer is a tuple 
\[\mathcal{T} = (Q, \Sigma, \Gamma, q_0, \delta),\] where 
$Q$ is a finite set of states, $\Sigma$ is the input alphabet, $\Gamma$ is the output alphabet,
$q_0 \in Q$ is the initial state, and $\delta : Q \times \Sigma \to Q \times \Gamma^*$ is the transition function.
If $\delta(q,a) = (p,u)$, then this should be read as follows: If the transducer is in state $q$ and the next input
symbol is $a$ then it moves to state $p$ and outputs the word $u$. This is often written as
\[
q \xrightarrow{a/u} p .
\]
We extend $\delta$ to a mapping $\delta : Q \times \Sigma^* \to Q \times \Gamma^*$ as follows, where
$q \in Q$, $a \in \Sigma$ and $w \in \Sigma^*$:
\begin{itemize}
\item $\delta(q,\varepsilon) = (q,\varepsilon)$ for all $q \in Q$, and
\item if $\delta(q,a) = (p, u)$ and $\delta(p, w) = (r, v)$ then $\delta(q, aw) = (r, uv)$.
\end{itemize}
Finally, we define the function $f_{\mathcal T} : \Sigma^* \to \Gamma^*$ computed by $\mathcal{T}$ as follows (where
$w \in \Sigma^*$ and $x \in \Gamma^*$):
$f_{\mathcal T}(w) = x$ if and only if $\delta(q_0, w) = (q,x)$ for some $q\in Q$. 
Intuitively, in order to compute $f_{\mathcal T}(w)$, $\mathcal{T}$ reads the word $w$
starting in the initial state $q_0$ and thereby concatenates all the outputs produced in the transitions.

\subsection{Semigroups, monoids and groups} \label{sec-semigroups}

We assume that the reader has some basic knowledge in group and semigroup theory.
In this section we recall some elementary definitions from this area.
For more details the reader can consult the monographs \cite{ClPrest61,Grillet95} (for semigroups)
and \cite{KaMe79,Rob96} (for groups).

A \emph{semigroup} is an algebraic structure $(S, \circ)$ with carrier set $S$ and an associative binary operation $\circ$ on $S$.
If the operation $\circ$ is clear from the context, we omit it and identify the semigroup with the carrier set $S$.
Moreover, we often write $ab$ instead of $a \circ b$ for $a,b \in S$.

For semigroups $(S, \circ)$ and $(T, \circ)$, a mapping $h\colon S \to T$ is called a \emph{(semigroup) homomorphism} if 
 $h(a \circ b) = h(a) \circ h(b)$ for all $a,b \in S$.

An important example of a semigroup is the \emph{free semigroup} $(\Sigma^+, \circ)$ consisting of all non-empty finite words 
over a non-empty set $\Sigma$ and the operation of concatenation. As usual, we write $uv$ instead of $u \circ v$ for words $u,v \in \Sigma^+$.
The important property of the free semigroup is that 
 for every semigroup  $S$ and mapping $h\colon \Sigma \to S$ there is a unique 
homomorphism $h'\colon \Sigma^+ \to S$ with $h'(a) = h(a)$ for every $a \in \Sigma$.
In particular, there is a homomorphism $\pi_S\colon S^+ \to S$ with $\pi_S(a) = a$ for $a \in S$, called
the \emph{evaluation homomorphism}. For a word $u \in S^+$ we say that $\pi_S(u)$ is the \emph{value} of $u$
or that $u$ \emph{represents} $\pi_S(u)$.
For two words $u,v \in S^+$ we write $u \equiv_S v$, or $u = v$~in~$S$, if $\pi_S(u) = \pi_S(v)$.
Note that writing $ab$ instead of $a \circ b$ for $a,b \in S$ means that one identifies a word $u \in S^+$ with $\pi_S(u) \in S$.
The notations $\pi_S(u)$ and $u \equiv_S v$ will be only used if it is important that $u$ and $v$ are words over $S$.

Fix a semigroup $S$.
A \emph{zero element} in $S$ is a necessarily unique element $0 \in S$ such that $0 a = a 0 = 0$ for all $a \in S$. 
If $S$ has a zero element and
for every $a \in S$ there is an $n \geq 1$ with $a^n = 0$, then $S$ is called a
\emph{nilsemigroup}. 
An element $a \in S$ is called \emph{idempotent}
if $a^2 = a$ in $S$. 
The semigroup $S$ is \emph{left-cancellative} if $ab = ac$ implies $b=c$, and \emph{right-cancellative} if 
$ba = ca$ implies $b=c$. It is \emph{cancellative} if it is both left- and right-cancellative.

\subsubsection{Subsemigroups, ideals and generating sets}

A \emph{subsemigroup} of $S$ is a subset $T \subseteq S$ such that $ab \in T$ for all $a,b \in T$.
An \emph{ideal} in $S$ is a subset $I \subseteq S$ such that $ab, ba \in I$ for all $a \in S$ and $b \in I$. 

For a subset $\Sigma \subseteq S$, the subsemigroup generated by $\Sigma$ is the intersection of all subsemigroups of $S$ that
contain $\Sigma$. It is equal to the image $\pi_S(\Sigma^+)$.
The semigroup $S$ is generated by $\Sigma$ if $S = \pi_S(\Sigma^+)$, i.e., $\pi_S : S^+ \to S$ restricts to a surjective homomorphism from $\Sigma^+$
to $S$. We also say that $\Sigma$ is a \emph{generating set} of $S$. 
If $S$ has a finite generating set then $S$ is called \emph{finitely generated} (f.g.~for short). In this paper we will only consider finitely generated semigroups. Clearly, the free semigroup $\Sigma^+$ is generated by $\Sigma$.

If $h : S \to T$ is a homomorphism, where $S$ is generated by $\Sigma$ and $T$ is generated by $\Gamma$ then we occasionally identify
$h$ with a homomorphism $\tilde{h} : \Sigma^+ \to \Gamma^+$, where $\tilde{h}(a) \in \Gamma^+$ for $a \in \Sigma$ is a word that represents
$h(a) \in T$.

\subsubsection{Congruences}
A \emph{congruence} on the semigroup $S$ is an equivalence relation $\rho$ on $S$ such that $a\, \rho\, a'$ and $b \,\rho\, b'$ imply
$ab \, \rho\, a'b'$ for all $a,a',b,b' \in S$. 
In this situation one can construct the quotient semigroup $S/\rho$. Its carrier set is the set of all equivalence classes 
$\{ [a]_\rho \colon a \in S \}$ of $\rho$ and the multiplication is defined by $[a]_\rho [b]_\rho = [ab]_\rho$. For every binary relation $R$ on $S$
there is a smallest (with respect to set inclusion) congruence $\rho_R$ containing $R$. It is the intersection of all congruences containing $R$.
If $S$ is finitely generated then also every quotient $S/\rho$ is finitely generated.
If $I \subseteq S$ is an ideal then the relation $\rho_I = \{ (a,b) \colon a,b \in I \} \cup \{ (a,a) \colon a \in S \setminus I \}$ is a 
congruence. The corresponding quotient $S/\rho_I$ is also called the \emph{Rees quotient} of $S$ with respect to $I$.

\subsubsection{Semigroup presentations}
Every semigroup $S$ is isomorphic to a quotient semigroup $\Sigma^+/\rho_R$ for a set $\Sigma$ and a relation $R \subseteq \Sigma^+ \times \Sigma^+$. If $\Sigma$ and $R$ are both finite then $S$ is called  \emph{finitely presented}.
 We then also write $\langle \Sigma \mid R \rangle^+$ (or more intuitively
$\langle \Sigma \mid u = v \text{ for } (u,v) \in R \rangle^+$) for the quotient $\Sigma^+/\rho_R$. The pair $(\Sigma, R)$ is also
called a \emph{semigroup presentation} for $\langle \Sigma \mid R \rangle^+$.

\subsubsection{Regular and inverse semigroups} \label{sec-regular-semigroup}

 Consider a semigroup $S$ and $a \in S$. An \emph{inverse} of $a$ is an element $b \in S$ such that
 $aba = a$ and $b = bab$ hold in $S$. A \emph{regular semigroup} is a semigroup, where every element has at least one inverse.
 An \emph{inverse semigroup} is a regular semigroup, where idempotents commute.
 Equivalently, it is a regular semigroup, where every element $a$  has a unique inverse,
 which is then denoted with $a^{-1}$  \cite[Theorem~II.1.2]{petrich:1984}.
  Note that $a a^{-1} = a^{-1}a=1$ does not necessarily hold (in fact, an inverse semigroup does not need to have
  a neutral element). The idempotents of an inverse semigroup are exactly the elements of the form $a a^{-1}$. 
  A famous result for inverse semigroups is the representation theorem of Vagner and Preston. It states that
  a semigroup is inverse if and only if it is isomorphic to a semigroup of partial injections on a set $A$ (the multiplication
  is the ordinary composition operation of partial mappings).
  
\subsubsection{Monoids}
A \emph{neutral element} in a semigroup  $S$ is a necessarily unique element $e \in S$ such that
$ea = ae = a$ for all $a \in S$. If there exists a neutral element, then $S$ is a \emph{monoid}.
Most of the times, we denote the neutral element with $1$.
A \emph{submonoid} of $M$ is a subsemigroup of $M$ that in addition contains the neutral element $1 \in M$.
A \emph{monoid morphism} $h : M \to P$ for monoids $M$ and $P$ is a 
semigroup morphism that in addition maps the neutral element of $M$ to the neutral element of $P$.

The concepts introduced for semigroups so far can be easily extended to monoids by replacing $( \cdot )^+$ by $( \cdot )^*$. 
The \emph{free monoid} generated by $\Sigma$ is of course the set of all finite words $\Sigma^*$ with the operation of concatenation
and the empty word $\varepsilon$ is its neutral element. The evaluation homomorphism $\pi_S : S^+ \to S$ for a semigroup
then becomes $\pi_M : M^* \to M$ for a monoid $M$ and we have $\pi_M(\varepsilon) = 1$. 
For a relation $R \subseteq \Sigma^* \times \Sigma^*$ one obtains 
the monoid $\Sigma^*/\rho_R$, which will be denoted by
$\langle \Sigma \mid R \rangle^*$. Note that in $\langle \Sigma \mid R \rangle^*$, the neutral element is the congruence class of $\varepsilon$. In contrast, if we have a semigroup $\langle \Sigma \mid R \rangle^+$, which has a neutral element (and hence is a monoid), then the neutral element is necessarily an equivalence class of non-empty words.

For a semigroup $S$, we denote by $S^1$ the smallest monoid containing $S$. More explicitly, we have $S^1 = S$ if $S$ is a monoid and, otherwise, $S^1$ is obtained from $S$ by adjoining the new element $1 \notin S$ to $S$ and defining $1 a = a 1 = a$ for all $a \in S$.
Note that if $S$ is generated (as a semigroup) by $\Sigma$ then $S^1$ is generated (as a monoid) by $\Sigma$ as well, where the empty word $\varepsilon \in \Sigma^\ast$ represents the neutral element $1 \in S^1$.

Let $M$ be an f.g.~monoid with finite generating set $\Sigma$; that is, $\pi_M : M^* \to M$ restricts to a surjective monoid homomorphism from the free monoid $\Sigma^*$ to $M$.
For $n \in \mathbb{N}$, let $B_{M,\Sigma}(n) =  \pi_M(\Sigma^{\le n})$ 
be the set of all monoid elements
that can be written as words of length at most $n$ over the generating set $\Sigma$.
The \emph{growth function} $\gamma_{M,\Sigma} : \mathbb{N} \to \mathbb{N} \setminus \{0\}$ is defined by
\begin{equation}  \label{eq-growth}
\gamma_{M,\Sigma}(n) = |B_{M,\Sigma}(n)|.
\end{equation}
For different finite generating sets $\Sigma_1, \Sigma_2$ of $M$, the functions $\gamma_{M,\Sigma_1}$ and $\gamma_{M,\Sigma_2}$ 
are in general different, but their growth rates are the same. More precisely, there exist constants $a,b \ge 1$ such that $\gamma_{S,\Sigma_1}(n) \le \gamma_{S,\Sigma_2}(a \cdot n)$ and
$\gamma_{M,\Sigma_2}(n) \le \gamma_{M,\Sigma_1}(b \cdot n)$
 \cite[Corollary~7.10]{Ceccherini-SilbersteinA21}. 
We can also define the growth function of a semigroup $S$ by $\gamma_{S,\Sigma} = \gamma_{S^1,\Sigma}$.\footnote{In the literature
one usually defines the growth function of a semigroup by counting the size of the sets $\{ \pi_S(w) \colon w \in \Sigma^+, |w| \leq n \}$.
Compared to our definition, this yields an additive difference of at most one.}

The \emph{Cayley graph} of $M$ with respect to $\Sigma$ is the deterministic automaton (without final states)
$C(M,\Sigma) = (M, \Sigma, 1, \delta)$, where $\delta(s,a) = sa$ for all $s \in M$ and $a \in \Sigma$. 
Note that $\gamma_{M,\Sigma}$ is the same as $\gamma_{C(M,\Sigma)}$.
The Cayley graph is usually viewed as a directed graph with the distinguished node $1$ and edges labelled with generators.

\subsubsection{Groups} \label{sec-groups}

A \emph{group} is a monoid $G$ such that for every $g \in G$ there exists a necessarily unique inverse $a^{-1} \in G$
with $aa^{-1} = a^{-1}a=1$. The neutral element of a group is usually called the \emph{identity element} (or just the \emph{identity})
of $G$. Clearly, every group is an inverse monoid. 

As for monoids, most semigroup concepts directly extend to groups.
A \emph{subgroup} of $G$ is a submonoid $H$ of $G$ such that in addition
$g^{-1} \in H$ for every $g \in H$; one writes $H \leq G$ in this situation.
The (right) \emph{cosets} $Hg$ for $g \in G$ then form a partition of~$G$. The \emph{index} of $H$ in $G$, denoted by $[G:H]$, is the number of 
cosets of $H$. If $[G:H] < \infty$ then $G$ is called a \emph{finite extension} of $H$.

 A {(group) homomorphism} $h : G \to H$ is a monoid homomorphism such that in addition
$h(g^{-1}) = h(g)^{-1}$ for every $g \in G$.

For a subset $\Sigma \subseteq G$, we denote with $\langle \Sigma \rangle$ the subgroup
of $G$ generated by $\Sigma$. 
It is the intersection of all subgroups of $G$ that contain $\Sigma$. 
This is the same as the submonoid generated by $\Sigma \cup \Sigma^{-1}$. If $\langle \Sigma \rangle = G$, then $\Sigma$ is a (group) \emph{generating
set} of~$G$.
A generating set $\Sigma$ for $G$ is called \emph{symmetric} if $\Sigma = \Sigma^{-1}$, in which case
$\Sigma$ generates $G$ as a monoid.
 In the rest of the paper, we assume
that all group generating sets are symmetric. It is well known that if $G$ is a finite extension of $H$ then 
$G$ is finitely generated if and only if $H$ is finitely generated. The group $G$ is called \emph{torsion-free} if for every $a \in G$, the cyclic group $\langle a \rangle$
is infinite, i.e., $\langle a \rangle \cong \mathbb{Z}$.

A \emph{normal subgroup} of $G$ is a subgroup $H \leq G$ such that $g^{-1} h g \in H$ for all $g \in G$, $h \in H$.
For a normal subgroup $H$ of $G$ one defines the quotient group $G/H$ consisting of all cosets $Hg$ ($g \in G$)
with the well-defined multiplication $Hg_1 Hg_2 = H(g_1g_2)$.
The \emph{normal closure} $N(\Sigma)$ of $\Sigma \subseteq G$  is the smallest normal subgroup of $G$ that contains $\Sigma$.
It can be also defined as the subgroup $\langle \{ g^{-1} a g \colon  a \in \Sigma, g \in G \}\rangle$.

The \emph{commutator} of $g,h \in G$ is the element
$[g,h] = ghg^{-1}h^{-1}$ and for subsets $A,B \subseteq G$ we write $[A,B]$ for the subgroup $\langle \{ [a,b] \colon  a \in A, b \in B \} \rangle$.
This is a normal subgroup of $G$.

For a monoid $M$ we denote with $U(M)$ the \emph{group of units} in $M$. It contains all elements $a \in M$ such that
there is $b \in M$ with $ab = ba = 1$. Clearly, this set together with the multiplication in $M$ is a group.

In this paper, we only consider finitely generated groups. 
Let $G = \langle \Sigma \rangle$ be a finitely generated group with $\Sigma$ a finite symmetric generating set for $G$.
Hence we have the surjective monoid morphism $\pi_G : \Sigma^* \to G$
satisfying $\pi_G(a)=a$ for $a \in \Sigma$.
As for semigroups we write $u \equiv_G v$ for $\pi_G(u) = \pi_G(v)$.
For a word $u = a_1 a_2 \cdots a_n \in \Sigma^*$ with $a_i \in \Sigma$ we define the word
$u^{-1} = a_n^{-1} \cdots a_2^{-1} a_1^{-1} \in \Sigma^*$. Clearly, we have $\pi_G(u^{-1}) = \pi_G(u)^{-1}$.

The definition of the growth function and Cayley graph for a finitely generated monoid can be directly used for groups.
Growth functions of groups have been intensively studied in the literature; see e.g.\ \cite{Mann11}.
Note that the Cayley graph of a group is symmetric in the sense that for every $a$-labelled edge from $g$ 
to $h$ there is an $a^{-1}$-labelled edge from $h$ to $g$.
The  \emph{word problem for $G$} with respect to the generating set
$\Sigma$ is the language 
\[ \WP(G,\Sigma) = \{ w \in \Sigma^* \colon \pi_G(w)=1 \}.\]
The group counterpart of the free monoid $\Sigma^*$ is the free goup.
Fix a finite alphabet $\Gamma$ and take a copy $\Gamma^{-1} = \{ a^{-1} \colon a \in \Gamma \}$
of formal inverses. Let $\Sigma = \Gamma \cup \Gamma^{-1}$.
We extend the mapping $a \mapsto a^{-1}$ ($a \in \Gamma$) to the whole alphabet $\Sigma$
by setting $(a^{-1})^{-1} = a$. For a word $u\in \Sigma^*$ 
the word $u^{-1}$ is defined as above.
 A word $u \in \Sigma^*$ is called \emph{irreducible} if it contains no factor of the form $aa^{-1}$
for $a \in \Sigma$. 
Let $\IRR(\Sigma) \subseteq \Sigma^*$ be the set of irreducible words.
The \emph{free group} $F(\Gamma)$ can be defined as the set $\IRR(\Sigma)$ of irreducible
words together with the following multiplication operation: Let $u, v \in \IRR(\Sigma)$. Then one can uniquely
write $u$ and $v$ as $u = xy$ and $v = y^{-1}z$ such that $xz \in \IRR(\Sigma)$ and define
the product of $u$ and $v$ in the free group $F(\Gamma)$ as $xz$. For every word $w \in \Sigma^*$ 
we can define an irreducible word $\red(w)$ as follows. If $w \in \IRR(\Sigma)$ then
$\red(w) = w$ and if 
 $w = u a a^{-1}v$ for $u,v \in \Sigma^*$ and $a \in \Sigma$ then $\red(w) = \red(uv)$. It is important that this definition does not 
 depend on which factor $aa^{-1}$ is deleted in $w$. More precisely, the reduction relation $u a a^{-1}v \to uv$ for all  $u,v \in \Sigma^*$ and $a \in \Sigma$ is confluent.
The reduction mapping $w \mapsto \red(w)$ then becomes the unique monoid morphism mapping
$w \in \Sigma^*$ to the element of the free group represented by $w$. 

Group presentations are a common way to describe groups. Let $\Gamma$ and $\Sigma$ be as in the previous paragraph
and let $R \subseteq F(\Gamma)$. Then the quotient group $F(\Gamma)/N(R)$ is also denoted by $\langle \Gamma \mid R \rangle$
and the pair $(\Gamma, R)$ is called a \emph{group presentation}. The group $\langle \Gamma \mid R \rangle$ is finitely generated (since we 
assume $\Gamma$ to be finite) and every f.g.~group can be written in this form. 
If $R$ is also finite then $\langle \Gamma \mid R \rangle$
is called \emph{finitely presented}.

\subsubsection{Actions} \label{sec-action}
A mapping $\alpha : X \times M \to X$ is a \emph{right action} of a monoid $M$ on a set~$X$ if $\alpha(x, 1) = x$ and $\alpha(\alpha(x, a), b) = \alpha(x, ab)$ hold for all $x \in X$ and $a,b \in M$. 
We will often simply denote $\alpha(x,a)$ as $xa$ or $x^a$ if there is no risk of confusion.

The right action $\alpha$ is called \emph{transitive} if for all $x,y \in X$ there is an $a \in M$ with $xa = y$.
The set $xM \coloneqq \{ xa \mid a \in M \}$ is called the \emph{orbit} of $x \in X$.
Thus, $\alpha$ is transitive if and only if $xM = X$ for all $x \in X$.
If $M$ is a group $G$, then the orbits $xG$ partition $X$; that is, being in the same orbit is an equivalence relation on $X$.

The \emph{stabilizer} of an element $x \in X$ is the set $\mathsf{Stab}_M(x)  \coloneqq  \{ a \in M \mid xa = x \}$, which is a submonoid of $M$ (and a subgroup if $M$ is a group).
If $x,y \in X$ and $a,b \in M$ satisfy $xa = y$ and $yb = x$, then $a\,\mathsf{Stab}_M(y)\, b \subseteq \mathsf{Stab}_M(x)$.
If $M$ is a group $G$, then
 this implies the equality $\mathsf{Stab}_G(xa) = a^{-1}\,\mathsf{Stab}_G(x)\,a$ for all $x \in X$ and $a \in G$; that is, the elements of a single orbit have conjugate stabilizers.

\section{Part A: Streaming algorithms for languages}

\subsection{Streaming algorithms for languages: definitions} \label{sec-streaming-def}

In this section we define our model of randomized streaming algorithms. It is a non-uniform model in the sense
that for every input length $n$ we have a separate algorithm that handles inputs of length at most $n$.
Formally, a (non-uniform) \emph{randomized streaming algorithm} is a sequence $\mathcal{R} = (\mathcal{A}_n)_{n \ge 0}$
of  PFA $\mathcal{A}_n$ over the same input alphabet $\Sigma$.
If every $\mathcal{A}_n$ is deterministic (resp., semi-probabilitistic), we speak of a \emph{deterministic} (resp., \emph{semi-randomized}) \emph{streaming algorithm}.

Let $\epsilon_0, \epsilon_1 : \mathbb{N} \to [0,1)_{\mathbb{R}}$ be functions.
A randomized streaming algorithm $\mathcal{R} = (\mathcal{A}_n)_{n \ge 0}$ is \emph{$(\epsilon_0,\epsilon_1)$-correct} for a language $L \subseteq \Sigma^*$
if for every large enough $n \geq 0$ and every word $w \in \Sigma^{\le n}$ we have the following:
\begin{itemize}
\item if $w \in L$ then $\epsilon(\mathcal{A}_n,w,L) \leq \epsilon_1(n)$, and
\item if $w \notin L$ then $\epsilon(\mathcal{A}_n,w,L) \leq \epsilon_0(n)$. 
\end{itemize}
If $\epsilon_0 = \epsilon_1 =: \epsilon$ then we also say that $\mathcal{R}$ is $\epsilon$-correct for $L$. 
Typically, the functions $\epsilon_0, \epsilon_1$ are monotonically decreasing.
We say that $\mathcal{R}$ is a 
\begin{itemize}
\item \emph{randomized streaming algorithm for $L$} if it is $1/3$-correct for $L$;
\item \emph{$0$-sided randomized streaming algorithm for $L$} if it is $(1/3,0)$-correct for $L$;
\item \emph{$1$-sided randomized streaming algorithm for $L$} if it is $(0,1/3)$-correct for $L$;
\item \emph{deterministic streaming algorithm for $L$} if it is deterministic and $0$-correct for $L$;
\item \emph{nondeterministic streaming algorithm for $L$} if it is $(0,\epsilon)$-correct for $L$ for a
 function $\epsilon$ with $0 \leq \epsilon(n) < 1$ for all $n$;
\item \emph{co-nondeterministic streaming algorithm for $L$} if it is $(\epsilon,0)$-correct for $L$ for a  function $\epsilon$ with $0 \leq \epsilon(n) < 1$ for all $n$.
\end{itemize}
The choice of $1/3$
for the error probability in the first three points is not important.  Using a standard application of the Chernoff bound, one can make
the error probability an arbitrarily small constant; see Theorem~\ref{thm-prob-amplify} below.

The \emph{space complexity} of the randomized streaming algorithm  $\mathcal{R} = (\mathcal{A}_n)_{n \ge 0}$
is the function $s(\mathcal{R},n) = \lceil \log_2 |Q_n| \rceil$, where $Q_n$ is the state set of $\mathcal{A}_n$.
The motivation for this definition is that states of $Q_n$ can be encoded by bit strings of length at most $\lceil \log_2 |Q_n| \rceil$.
The \emph{randomized streaming space complexity of the language $L$} is the smallest possible function $s(\mathcal{R},n)$, where 
$\mathcal{R}$ is a randomized streaming algorithm for $L$. In an analogous way we define the $0$-sided (resp., $1$-sided) randomized
streaming space complexity, the deterministic streaming space complexity, and the (co-)nondeterministic streaming space complexity of a language $L$. Note that all these functions are monotone since we quantify over all words of length at most $n$ in the above definition of 
an $(\epsilon_0,\epsilon_1)$-correct randomized streaming algorithm.

The (non)deterministic streaming space complexity of a language $L$ is directly linked to the \emph{automaticity of $L$}. 
The automaticity of $L \subseteq \Sigma^*$ is the function $A_L(n)$ that maps $n$ to the number of states of a smallest DFA $\mathcal{A}_n$
such that for all words $w \in \Sigma^{\le n}$ we have: $w \in L$ if and only if $w$ is accepted by $\mathcal{A}_n$.
If we allow the automata $\mathcal{A}_n$ to be nondeterministic then we obtain the \emph{nondeterministic automaticity $N_L(n)$ of $L$}.
Hence, the deterministic (resp., nondeterministic) streaming space complexity of $L$ is exactly $\lceil \log_2 A_L(n) \rceil$ (resp., $\lceil \log_2 N_L(n) \rceil$). The (nondeterministic)
automaticity of languages was studied in \cite{GlaisterS98,ShallitB96}.
Interesting in our context is the following result of Karp \cite{Karp67}: if $L$ is a non-regular language then
$A_L(n) \geq (n+3)/2$ for infinitely many $n$. Hence, for every non-regular language the 
deterministic streaming space complexity is lower bounded by $\log_2(n) - c$ for a constant $c$
and infinitely many $n$. 

As remarked before, our model of streaming algorithms is non-uniform in the sense that for every input length $n$ we have a separate 
streaming algorithm $\mathcal{A}_n$.\footnote{This is analogous to circuit complexity, where for every input length $n$ one
has a separate boolean circuit with $n$ input gates.} Clearly, lower bounds for non-uniform streaming algorithms are stronger
than lower bounds for uniform streaming algorithms.
Moreover, the streaming algorithms
that we construct for concrete problems will be mostly uniform in the sense that there is an efficient algorithm that constructs from a given $n$,
the automaton $\mathcal{A}_n$.  

\subsection{Streaming algorithms: general results}

Before we investigate streaming algorithms for word problems we prove a few general results that are of independent interest.
Let us first prove that (as stated above) the error probability of a randomized streaming algorithm can be pushed down to any constant
$\epsilon > 0$ at the cost of an additional constant factor in the space complexity.

\begin{theorem} \label{thm-prob-amplify}
Let  $r : \mathbb{N} \to \mathbb{N}$ be a monotonic function with $r(n)$ odd for all $n$. and 
$\mathcal{R}$ be a randomized streaming algorithm such that 
$\mathcal{R}$ is $\frac{1}{3}$-correct for the language $L$. Then there exists a randomized streaming 
algorithm $\mathcal{S}$ such that $s(\mathcal{S},n) = r(n) \cdot s(\mathcal{R},n)$ and 
$\mathcal{S}$ is $\exp(-r(n)/30)$-correct for the language $L$.
\end{theorem}

\begin{proof}
Let $\mathcal{R}=(\mathcal{A}_n)_{n \ge 0}$ with $\mathcal{A}_n = (Q_n,\Sigma,\iota_n,\rho_n,F_n)$.
We use the standard idea of running $r(n)$ copies of $\mathcal{A}_n$ in parallel and making a majority vote at the end.\footnote{We assume that
all $r(n)$ are odd in order to avoid ties; this assumption is only made for convenience.}
Formally, for an $n \ge 0$ and odd $k\ge 1$ we define the PFA $\mathcal{A}^{k}_n$ as follows:
\begin{itemize}
\item $\mathcal{A}^{k}_n = (Q^{k}_n,\Sigma,\iota^{k}_n,\rho^{k}_n, F_{n,k})$,
\item $\iota^{k}_n(q_1, \ldots, q_{k}) = \prod_{1 \leq i \leq k} \iota_n(q_i)$,
\item $\rho^{k}_n((p_1, \ldots, p_{k}), a, (q_1, \ldots, q_{k})) = \prod_{1 \leq i \leq k} \rho_n(p_i,a,q_i)$, and
\item $F_{n,k} = \{ (q_1, \ldots, q_{k}) \colon q_i \in F_n \text{ for more than $k/2$ many $i \in [1,k]$} \}$.
\end{itemize}
We then define the new randomized  streaming algorithm $\mathcal{R}^r =(\mathcal{A}^{r(n)}_n)_{n \ge 0}$.
In order to bound the error probability of $\mathcal{R}^r$ by $\exp(-r(n)/30)$ we have to show that 
$\epsilon(\mathcal{A}^{k}_n,w,L) \leq  \exp(-k/30)$ for every input word $w \in \Sigma^{\le n}$.
For this we introduce identically distributed independent Bernoulli 
random variables $X_1, \ldots, X_{k}$ with $\Prob[X_i=1] = \frac{1}{3}$. Then, for every  $w \in \Sigma^{\le n}$ we have:
\[
\epsilon(\mathcal{A}^{k}_n,w,L) \leq \Prob\left[\sum_{i=1}^{k} X_i > \frac{k}{2}\right].
\]
Let $\delta = \frac{1}{2}$. With $\epsilon = \frac{1}{3}$ we obtain with the Chernoff bound \eqref{chernoff}:
\[
 \Prob\left[\sum_{i=1}^{k} X_i > \frac{k}{2}\right] = \Prob\left[\sum_{i=1}^{k} X_i > (1+\delta) k \epsilon \right] <  \exp\left( - \frac{\delta^2 \epsilon k}{\delta+2}\right) =
  \exp\left(-\frac{k}{30}\right).
\]
The space complexity of $\mathcal{R}^r$ is clearly $r(n)$ times the space complexity of $\mathcal{R}$.
\end{proof}
Let  $\mathcal{A} = (Q,\Sigma,\iota,\rho,F)$ be a PFA and $0 < \delta < 1$, $\epsilon > 0$. 
We say that $\delta$ is an $\epsilon$-isolated cut-point for $\mathcal{A}$ if for all words
$w \in \Sigma^*$ we have
\begin{equation} \label{eq-cut-point}
|\Prob[\mathcal{A}(w) \in F] - \delta| \geq \epsilon .
\end{equation}
Assume that this holds.
The language $L(\mathcal{A}, \delta)$ accepted by $\mathcal{A}$ with cut-point $\delta$ is the set of all words $w$ with
$\Prob[\mathcal{A}(w) \in F] > \delta$, or, equivalently, $\Prob[\mathcal{A}(w) \in F] \geq \delta+\epsilon$.
Paz proved in \cite[Theorem~30']{Paz66} that there 
exists a DFA for $L(\mathcal{A}, \delta)$ with $(1+ 1/2\epsilon)^{|Q|-1}$ states.
A proof can be found in \cite[p.~160]{Paz71}; it uses 
the proof technique for a slightly weaker result of Rabin \cite[Theorem~3]{Rabin63}.
Paz's proof easily yields the following result: 

\begin{theorem}
Let $L$ be a language with randomized streaming space complexity $S(n)$. Then the 
deterministic streaming space complexity of $L$ is bounded by $2^{S(n)+1}$.
\end{theorem}

\begin{proof}
Let $\mathcal{R} = (\mathcal{A}_n)$ be a randomized streaming algorithm for $L$ such that 
$S(n) = \lceil \log_2 |Q_n| \rceil$. Fix an $n$ and
set $\delta = 1/2$ and $\epsilon = 1/6$, so that $\delta-\epsilon=1/3$ and 
$\delta+\epsilon=2/3$. We cannot directly apply the above mentioned result of Paz since $1/2$ is not 
necessarily a $1/6$-isolated cut-point for $\mathcal{A}_n$, for which \eqref{eq-cut-point} only has to hold for words $w$
of length at most $n$. But we can argue as follows:
Recall the automaticity function $A_L(n)$ of the language $L$ from Section~\ref{sec-streaming-def}. 
Then the deterministic streaming space complexity of $L$ is $\lceil \log_2 A_L(n) \rceil$.

It is shown in \cite{KanepsF91} (see also \cite{ShallitB96}) that $A_L(n)$ is the maximal number $k$ for which there exist words $v_1, \ldots, v_k \in \Sigma^{\le n}$
such that for all $i,j \in [1,k]$ with $i < j$ there is a word $w_{i,j} \in \Sigma^*$ with
$v_iw_{i,j}, v_jw_{i,j} \in \Sigma^{\le n}$ and $v_iw_{i,j} \in L$ if and only if  $v_jw_{i,j} \notin L$.

Assume that $k = A_L(n)$, fix the above words $v_i$ and $w_{i,j}$, and consider $i,j \in [1,k]$ with $i < j$. Since $v_iw_{i,j}, v_jw_{i,j} \in \Sigma^{\le n}$
and $v_iw_{i,j} \in L$ if and only if  $v_jw_{i,j} \notin L$,
we get
$$
|\Prob[\mathcal{A}_n(v_iw_{i,j}) \in F] - \Prob[\mathcal{A}_n(v_jw_{i,j}) \in F]| \geq 2 \epsilon = \frac{1}{3}
$$
whenever $i < j$.
In the proof of \cite[Theorem~30']{Paz66} (see \cite[p.~160]{Paz71}) it is shown that this implies
\begin{equation*}\label{eq-paz}
k \leq (1+ 1/2\epsilon)^{|Q_n|-1} = 4^{|Q_n|-1} .
\end{equation*}
We obtain $A_L(n) \leq 4^{|Q_n|} \leq 4^{2^{S(n)}}$ and hence  $\lceil \log_2 A_L(n) \rceil \leq 2^{S(n)+1}$.
\end{proof}
We now turn to the connection between randomized and semi-randomized streaming algorithms.
Our next result states that a randomized streaming algorithm can be transformed into an equivalent semi-randomized
streaming algorithm with a moderate blow-up in the space complexity.

\begin{theorem} \label{thm-semi-randomized}
Let $0 < \epsilon(n) < 1/4$ for all $n \geq 0$ 
and let $\mathcal{R}$ be a randomized streaming algorithm which is $\epsilon(n)$-correct for the language $L$. Then there is a semi-randomized streaming 
algorithm $\mathcal{S}$ with $s(\mathcal{S},n) = s(\mathcal{R},n) + \Theta(\log n + \log(1/\epsilon(n)))$ and 
$\mathcal{S}$ is $2\epsilon(n)$-correct for the language $L$.
\end{theorem}

\begin{proof}
Let $\mathcal{R}=(\mathcal{A}_n)_{n \ge 0}$.
Let us fix an $n$ and consider the PFA
$\mathcal{A}_n = (Q,\Sigma,\iota,\rho,F)$. We first transform $\mathcal{A}_n$ into an
acyclic PFA $\mathcal{A}'_n = (Q',\Sigma,\iota',\rho',F')$, where acyclic means that for every run
$\pi = (q_0,a_1,\dots,a_m,q_m)$ of $\mathcal{A}'_n$  such that $m \leq n$ and  $q_i = q_j$ for some $i<j$ we have $\iota\rho^*(\pi) =0$.
We define the components of  $\mathcal{A}'_n$ as follows:
\begin{itemize}
\item $Q' = Q \times [0,n]$
\item $F' = F \times [0,n]$
\item For all $p,q \in Q$, $i \in [0,n-1]$, $j \in [0,n]$ and $a \in \Sigma$ we set 
\[
\rho'((p,i),a,(q,j)) = \begin{cases}
 \rho(p,a,q) & \text{ if } j = i+1 \\
 0 & \text{ if } j \neq i+1. 
 \end{cases}
\]
\item For states $(p,n)$ we could define $\rho'$ arbitrarily. Let us set $\rho'((p,n),a,(p,n)) = 1$ for all $a \in \Sigma$
and $\rho'((p,n),a,(q,i)) = 0$ whenever $(q,i) \neq (p,n)$.
\item For all states $(q,0) \in Q'$ we set $\iota'(q,0) = \iota(q)$. Moreover, $\iota'(q,i) = 0$ if $i > 0$.
\end{itemize}
The randomized streaming algorithm $\mathcal{R}' = (\mathcal{A}'_n)_{n \ge 0}$ is also \emph{$\epsilon(n)$-correct} for the language $L$.
Moreover, the space complexity of $\mathcal{R}'$ is $s(\mathcal{R},n) + \lceil \log_2 (n+1)\rceil$.

We now define a random DFA $\mathsf{D}_n$. Only DFAs of the form 
$(Q', \Sigma, (q_0,0), \delta, F')$ for some transition function $\delta \colon Q' \times \Sigma \to Q'$ and initial state $(q_0, 0) \in Q'$
may have a non-zero probability. The components $\delta$ and $(q_0, 0)$ are randomly chosen as follows:
\begin{itemize}
\item For every state $(p,i)$ of $\mathcal{A}'_n$  
and every $a \in \Sigma$ we choose a state
$(q,j)$ with probability $\rho'((p,i),a,(q,j))$ and define $\delta( (p,i), a) = (q,j)$.
\item The initial state $(q_0,0) \in Q'$ is chosen with probability $\iota'(q_0,0)$.
\end{itemize}
The above choices are made independently. 
Let $\mathsf{sup}(\mathsf{D}_n)$ be the support of $\mathsf{D}_n$ (the set of DFAs that have a non-zero probability).

For every fixed word $w \in \Sigma^*$ with $|w|\leq n$ and $\mathcal{D} \in \mathsf{sup}(\mathsf{D}_n)$ define 
$Z[w,\mathcal{D}] \in \{0,1\}$
by $Z[w,\mathcal{D}] = 1$ if and only if $w \in L \setminus L(\mathcal{D}) \cup L(\mathcal{D}) \setminus L$. In other words:
$Z[w,\mathcal{D}] = 1$ if and only if
$\mathcal{D}$ makes an error (with respect to the language $L$) on the word $w$.
For the expected value of  $Z[w,\mathcal{D}]$ we obtain
\[
\mathsf{E}[w]  \coloneqq  \sum_{\mathcal{D} \in \mathsf{sup}(\mathsf{D}_n)} \Prob[\mathsf{D}_n = \mathcal{D}] \cdot Z[w,\mathcal{D}]  \leq \epsilon(n),
\]
because the left-hand side of the inequality is exactly the error probability of $\mathcal{A}'_n$ on $w$. 
For this, it is important that we construct the DFAs from the acyclic PFA
$\mathcal{A}'_n$: In our original PFA
$\mathcal{A}_n$ there could be a run of the form $\pi = (\ldots, p,a,q, \ldots, p,a,q',\ldots)$ with $q \neq q'$ 
and $\iota\rho^*(\pi) > 0$. But runs of this form cannot occur in a DFA.

The rest of  the proof follows the arguments from the proof of Newman's theorem from communication complexity, see e.g. \cite{KuNi96}.
Fix a number $t$ that will be suitably chosen later.  
For a $t$-tuple of DFAs $\overline{\mathcal{D}} = (\mathcal{D}_{n,1}, \ldots, \mathcal{D}_{n,t})$ with $ \mathcal{D}_{n,i} \in \mathsf{sup}(\mathsf{D}_n)$
we construct a semiPFA $\mathcal{B}(\overline{\mathcal{D}})$ by taking the disjoint 
union of the $\mathcal{D}_{n,i}$. To define the initial state distribution $\bar{\iota}$ of $\mathcal{B}(\overline{\mathcal{D}})$, let
$q_{0,i}$ be the initial state of $\mathcal{D}_{n,i}$. Then we set $\bar{\iota}(q_{0,i}) = 1/t$. Thus, the starting state
of a run in $\mathcal{B}(\overline{\mathcal{D}})$ is chosen uniformly among the initial states of the $\mathcal{D}_{n,i}$.

We show that there exists a $t$-tuple $\overline{\mathcal{D}}$ of the above form such that 
for every input word $w \in \Sigma^{\le n}$ 
the error probability of $\mathcal{B}_n  \coloneqq  \mathcal{B}(\overline{\mathcal{D}})$ on $w$ is at most $2\epsilon(n)$. 
Then $(\mathcal{B}_n)_{n \geq 0}$
is the desired semi-randomized streaming algorithm from the theorem.

Fix again an input word $w \in \Sigma^{\le n}$  and a $t$-tuple 
$\overline{\mathcal{D}} = (\mathcal{D}_{n,1}, \ldots, \mathcal{D}_{n,t})$.
Then the error probability
of $\mathcal{B}(\overline{\mathcal{D}})$ on $w$ is
\[
\epsilon(\mathcal{B}(\overline{\mathcal{D}}), w, L) = \frac{1}{t} \cdot \sum_{i=1}^t Z[w,\mathcal{D}_{n,i}].
\]
We now choose the tuple $\overline{\mathcal{D}} = (\mathcal{D}_{n,1}, \ldots, \mathcal{D}_{n,t})$ randomly by taking $t$ 
independent samples of the random DFA $\mathsf{D}_n$.
With the Chernoff bound \eqref{chernoff} and  $\mathsf{E}[w]  \leq \epsilon(n)$ (i.e., $\frac{2\epsilon(n)}{\mathsf{E}[w]} -1 \ge 1$)
we obtain
\begin{eqnarray*}
& & \Prob\left[\frac{1}{t} \cdot \sum_{i=1}^t Z[w,\mathcal{D}_{n,i}] > 2\epsilon(n) \right]  \\
& = & \Prob\left[\sum_{i=1}^t Z[w,\mathcal{D}_{n,i}] > \bigg(1+ \frac{2\epsilon(n)}{\mathsf{E}[w]} -1\bigg) \cdot \mathsf{E}[w] \cdot t  \right]  \\
& \leq & \exp \left( - \frac{2\epsilon(n)/\mathsf{E}[w] -1}{3}  \cdot \mathsf{E}[w] \cdot t\right) \\
& = & \exp \left( \frac{-2\epsilon(n) + \mathsf{E}[w]}{3}   \cdot t\right) \\
& \leq & \exp \left( -\frac{\epsilon(n)\cdot t}{3}\right).
\end{eqnarray*}
By the union bound, the probability that 
$\epsilon(\mathcal{B}(\overline{\mathcal{D}}), w, L) > 2\epsilon(n)$ for some word $w \in \Sigma^*$ of length
at most $n$ (where $\overline{\mathcal{D}}$ is randomly chosen using $t$ 
independent samples of the random DFA $\mathsf{D}_n$) is bounded by
\[ |\Sigma|^{n+1} \cdot \exp(-\epsilon(n) \cdot t/3) = \exp(\ln |\Sigma| \cdot (n+1) -\epsilon(n) \cdot t/3).\]
If we choose $t = (3(n+1) \ln |\Sigma| + 1)/\epsilon(n)$ then this probability is $\exp(-1/3) < 1$.
With such a $t$ the space complexity of $\mathcal{B}(\overline{\mathcal{D}})$ becomes
$s(\mathcal{R}',n) + \lceil \log_2 t \rceil = 
s(\mathcal{R},n) + \Theta(\log n + \log(1/\epsilon(n)))$.
\end{proof}
Note that if $s(\mathcal{R},n) \geq \Omega(\log n)$ and $\epsilon(n) \geq \Omega(1/n^c)$ for some constant $c \geq 1$
then $s(\mathcal{S},n) = \Theta(s(\mathcal{R},n))$ in Theorem~\ref{thm-semi-randomized}.
Also notice that the proof of Theorem~\ref{thm-semi-randomized} uses non-uniformity in a crucial way.

Theorem~\ref{thm-semi-randomized} might explain the fact that many streaming algorithms in the literature only make a random guess
at the very beginning and then proceed deterministically. This is for instance the case for the famous AMS-algorithm \cite{AlonMS99} for 
approximating $k^{th}$ moments in a data stream.

\begin{oproblem}
Is the additive $\log n$ term in Theorem~\ref{thm-semi-randomized} necessary? In Newman's theorem (which inspired our proof 
of Theorem~\ref{thm-semi-randomized}) this additive term cannot be avoided. Newman's theorem states that
a public coin randomized communication protocol with cost $C(n)$ and error probability $\epsilon$ can be transformed into an equivalent 
private coin randomized communication protocol with cost $C(n) + \Theta(\log n + \log(1/\epsilon(n)))$ and error probability $2\epsilon$. The equality problem has a
public coin randomized communication protocol with cost $\mathcal{O}(1)$, whereas every  
private coin protocol for equality has cost $\Omega(\log n)$.
\end{oproblem}

Our final result in this section is a trade-off between the space complexity and the error probability 
for semi-randomized streaming algorithms:

\begin{theorem} \label{thm-trade-off}
Let $s(n)$ be the deterministic streaming space complexity of the language $L \subseteq \Sigma^*$ and let 
$\mathcal{R}=(\mathcal{A}_n)_{n \ge 0}$ be a semi-randomized streaming algorithm that is
$\epsilon(n)$-correct for the language $L$. Then for every large enough $n \geq 0$ we have
\[
s(\mathcal{R},n) \geq \min \{ s(n), \log_2 (1/\epsilon(n)) \} .
\]
\end{theorem}

\begin{proof}
Fix an $n \geq 0$ large enough such that for every word $w \in \Sigma^{\le n}$ the error probability
$\epsilon(\mathcal{A}_n,w,L)$ is bounded by $\epsilon(n)$.
Let $\mathcal{A}_n = (Q_n,\Sigma,\iota_n,\delta_n,F_n)$. Hence, we have 
$s(\mathcal{R},n) = \lceil \log_2 |Q_n| \rceil$. If $s(\mathcal{R},n) \geq \log_2 (1/\epsilon(n))$ then we are done.
Therefore, assume that $1/\epsilon(n) > 2^{s(\mathcal{R},n)} \ge |Q_n|$, i.e., $\epsilon(n) < 1/|Q_n|$.
There must exist a state $q_n \in Q_n$ with $\iota_n(q_n) \geq 1/|Q_n| > \epsilon(n)$. Consider the DFA 
$\mathcal{B}_n = (Q_n,\Sigma,q_n,\delta_n,F_n)$. If there is a word $w \in \Sigma^{\leq n}$ such that
$w \in L(\mathcal{B}_n) \Leftrightarrow w \notin L$ then we would have $\epsilon(n) \geq \iota_n(q_n) > \epsilon(n)$, which 
yields a contradition. Therefore we have $L(\mathcal{B}_n) \cap \Sigma^{\le n} = L \cap \Sigma^{\le n}$.
Since $\mathcal{B}_n$ is a DFA with state set $Q_n$, we get $s(\mathcal{R},n) \ge s(n)$.
\end{proof}

\subsection{Streaming algorithms for group word problems}

Let us now specialize the considerations from the previous section to the case where the language $L$ is 
a group word problem $\WP(G,\Sigma)$.
The following lemma is simple but important:

\begin{lemma} \label{lemma-gen-set}
Let $\Sigma_1$ and $\Sigma_2$ be finite symmetric generating sets for the group $G$ and let
$s_i(n)$ be the deterministic/randomized streaming space complexity of $\WP(G,\Sigma_i)$. 
Then there is a constant $c$, depending on $G$, $\Sigma_1$ and $\Sigma_2$,
such that $s_1(n) \leq s_2(c \cdot n)$.
\end{lemma}

\begin{proof}
For every generator $a \in \Sigma_1$  there is a word $w_a \in \Sigma_2^*$ such that $\pi_G(a) = \pi_G(w_a)$.
Let $c = \max \{ |w_a| \colon a \in \Sigma_1\}$ and let $\phi : \Sigma_1^* \to \Sigma_2^*$
be the homomorphism with $\phi(a) = w_a$ for $a \in \Sigma_1$.
Let $\mathcal{R}_2 = (\mathcal{A}_{2,n})_{n \geq 0}$ be a deterministic/randomized
streaming algorithm for the language $\WP(G,\Sigma_2)$ with space complexity $s_2(n)$. We obtain a deterministic/randomized
streaming algorithm $\mathcal{R}_1 = (\mathcal{A}_{1,n})_{n \geq 0}$ for $\WP(G,\Sigma_1)$ as follows: On input $w \in \Sigma_1^{\le n}$,
the PFA $\mathcal{A}_{1,n}$ simulates the PFA $\mathcal{A}_{2, c \cdot n}$ on the input word $\phi(w) \in \Sigma_2^{\le c \cdot n}$.
This yields a deterministic/randomized streaming algorithm for $\WP(G,\Sigma_1)$ with space complexity $s_2(c \cdot n)$.
\end{proof}
As already mentioned in the introduction, Lemma~\ref{lemma-gen-set} allows us to speak of the 
 deterministic/randomized streaming space complexity
of the word problem of the group  $G$ without mentioning the generating set, as long as we use the Landau notation for 
the space complexity. In such situations we will write $\WP(G)$ instead of $\WP(G,\Sigma)$.

The (non)deterministic streaming space complexity of $\WP(G)$ is directly
linked to the growth of $G$ by the following theorem. Recall that the deterministic (resp., nondeterministic) streaming
space complexity of a language $L$ is $\lceil \log_2 A_L(n) \rceil$ (resp., $\lceil \log_2 N_L(n) \rceil$).
For the deterministic case, the following result has been shown independently in \cite[Lemma~26]{Remscrim21} in a slightly weaker form.

\begin{theorem} \label{thm-det-growth}
Let $G$ be a finitely generated infinite group and let $\Sigma$ be a finite symmetric generating set for $G$. 
Then, for the deterministic automaticity and nondeterministic automaticity of $W  \coloneqq  \WP(G,\Sigma)$ we have:
\begin{eqnarray*} 
A_W(n) & = & \begin{cases}
\gamma_{G,\Sigma}(n/2) & \text{ if $n$ is even,} \\
\gamma_{G,\Sigma}(\lfloor n/2 \rfloor) + 1 & \text{ if $n$ is odd,}
\end{cases} \\
N_W(n) & = & \gamma_{G,\Sigma}(\lfloor n/2 \rfloor) .
\end{eqnarray*}
\end{theorem}

\begin{proof}
We start with the upper bound for the deterministic automaticity in case $n$ is even.
In the following we identify the ball $B_{G,\Sigma}(n/2)$ with its induced subgraph of the Cayley graph $C(G,\Sigma)$.
We obtain a partial DFA $\mathcal{A}_n$ by taking the edge-labelled graph $B_{G,\Sigma}(n/2)$
with the initial and unique final state $1$. It is a partial DFA in the sense that for every $g \in B_{G,\Sigma}(n/2)$
and every $a \in \Sigma$, $g$ has at most one outgoing edge labelled with $a$ (that leads to $g a$ if $ga \in B_{G,\Sigma}(n/2)$).
In order to add the missing transitions we choose an element $g_f \in B_{G,\Sigma}(n/2) \setminus B_{G,\Sigma}(n/2 - 1)$
(here, we set $B_{G,\Sigma}(-1) = \emptyset$). Such an element exists because $G$ is infinite.
If $g \in B_{G,\Sigma}(n/2)$ has not outgoing $a$-labelled edge in $B_{G,\Sigma}(n/2)$ then we add an 
$a$-labelled edge from $g$ to $g_f$. We call those edges  \emph{spurious}.  The resulting DFA is $\mathcal{A}_n$.

We claim that for every word $w \in \Sigma^{\le n}$, $w$ is accepted by $\mathcal{A}_n$ 
if and only if $w \in \WP(G,\Sigma)$. This is clear, if no spurious edge is traversed while reading $w$ into $\mathcal{A}_n$.
In this case, after reading $w$, we end up in state $\pi_G(w)$. Now assume that a spurious edge is traversed while reading $w$ into $\mathcal{A}_n$
and let  $x$ be the shortest prefix of $w$ such that a spurious edge is traversed while reading the last symbol of $x$.
Let us write $w = xy$. We must
have $|x| > n/2$ and $\pi_G(x) \notin B_{G,\Sigma}(n/2)$. Moreover, $|y| < n- n/2  = n/2$.
Since  $\pi_G(x) \notin B_{G,\Sigma}(n/2)$, we have 
$w = xy \notin \WP(G,\Sigma)$. Moreover, $w$ is rejected by $\mathcal{A}_n$, because $x$ leads in $\mathcal{A}_n$ from
the initial state $1$ to state $g_f$ and there is no path of length at most $n/2  -1$ from $g_f$ back to the final state $1$.

For the case that $n$ is odd, we take the ball $B_{G,\Sigma}(\lfloor n/2 \rfloor)$. Instead of adding spurious edges we add a failure state $f$.
If $g \in B_{G,\Sigma}(\lfloor n/2 \rfloor)$ has no outgoing $a$-labelled edge in $B_{G,\Sigma}(\lfloor n/2 \rfloor)$, then we add an $a$-labelled
edge from $g$ to $f$. Moreover, for every $a \in \Sigma$ we add an $a$-labelled loop at state $f$. As for the case $n$ even, one can
show that the resulting DFA accepts a word $w  \in \Sigma^{\le n}$ if and only if  $w \in \WP(G,\Sigma)$. 

The upper bound for the nondeterministic automaticity follows with the same arguments. Notice that the failure state $f$ in case
$n$ is odd is not needed in a nondeterministic automaton.

For the lower bound we start with the nondeterministic automaticity.
Let $k = \gamma_{G,\Sigma}(\lfloor n/2 \rfloor)$ and choose words $w_1, \ldots, w_k$ such that
$|w_i| \leq \lfloor n/2 \rfloor$ for all $i \in [1,k]$ and $w_i \not\equiv_G w_j$ whenever $i \neq j$.
Then for every $i \in [1,k]$ we have $w_i w_i^{-1} \in \WP(G,\Sigma)$ and $w_j w_i^{-1} \notin \WP(G,\Sigma)$
for all $j \in [1,k] \setminus \{i\}$. Moreover, $|w_j w_i^{-1}| \leq n$ for all $i,j \in [1,k]$. In the language of \cite{GlaisterS98}, $\{w_1, \ldots, w_k\}$
is a set of uniformly $n$-dissimilar words.
By \cite[Lemma~3.1]{GlaisterS98} this implies that the nondeterministic automaticity of $\WP(G,\Sigma)$ satisfies
$N_{\WP(G,\Sigma)}(n) \geq \gamma_{G,\Sigma}(\lfloor n/2 \rfloor)$.
This shows the lower bound on the nondeterministic automaticity.

For the lower bound on the deterministic automaticity, let
$\mathcal{A} = (Q,\Sigma, q_0, \delta, F)$ be a smallest DFA such that 
$L(\mathcal{A}) \cap \Sigma^{\leq n} = \WP(G,\Sigma)\cap \Sigma^{\leq n}$.
We have to show that 
\[
|Q| \ge \begin{cases}
\gamma_{G,\Sigma}(n/2) & \text{ if $n$ is even,} \\
\gamma_{G,\Sigma}(\lfloor n/2 \rfloor) + 1 & \text{ if $n$ is odd}.
\end{cases}
\]
Let us consider two words $u,v \in \Sigma^*$ of length at most $\lfloor n/2\rfloor$ such that
$u \not\equiv_G v$  and $\delta(q_0,u) = \delta(q_0,v)$. We then have $u v^{-1} \not\in \WP(G,\Sigma)$ and $v v^{-1} \in \WP(G,\Sigma)$.
On the other hand, we have $\delta(q_0,u v^{-1}) = \delta(q_0,v v^{-1})$, which is a contradiction (note that $|u v^{-1}|, |v v^{-1}| \leq n$).
Hence, if $\delta(q_0,u) = \delta(q_0,v)$ for  two words $u,v \in \Sigma^*$ of length at most $\lfloor n/2\rfloor$, then $u \equiv_G v$.

Let $Q'  = \{\delta(q_0,w) \colon w \in \Sigma^*, |w| \leq \lfloor n/2\rfloor \} \subseteq Q$. The previous paragraph shows that 
$|Q'| \geq \gamma_{G,\Sigma}(\lfloor n/2 \rfloor)$. If $n$ is even then $\lfloor n/2 \rfloor = n/2$
and we are done. So, let us assume that $n$ is odd.

If $|Q| > \gamma_{G,\Sigma}(\lfloor n/2 \rfloor)$ then we are again done.
So, let us assume that $Q = Q'$ and $|Q| = \gamma_{G,\Sigma}(\lfloor n/2 \rfloor)$. Then, to every state $q \in Q$ we 
can assign a unique group element $g_q \in B_{G,\Sigma}(\lfloor n/2 \rfloor)$ such that 
for every word $w \in \Sigma^*$ with $|w| \leq \lfloor n/2\rfloor$ we have $\delta(q_0,w)=q$ if and only if $\pi_G(w)=g_q$.
The mapping $q \mapsto g_q$ is a bijection between $Q$ and  $B_{G,\Sigma}(\lfloor n/2 \rfloor)$.

Let us now take a state $q \in Q$ and a generator $a \in \Sigma$ such that
$g_q \cdot a \notin B_{G,\Sigma}(\lfloor n/2 \rfloor)$. Such a state and generator must exist since $G$ is infinite. 
Let $u,v \in \Sigma^*$ be words of length at most $\lfloor n/2\rfloor$ such that $\delta(q_0,u)=q$
and $\delta(q_0,v)=\delta(q,a) = \delta(q_0,ua)$. We obtain $\delta(q_0,vv^{-1}) = \delta(q_0,uav^{-1})$.
But $vv^{-1} \in \WP(G,\Sigma)$ and $uav^{-1} \notin \WP(G,\Sigma)$ since
$\pi_G(uav^{-1}) = g_q \cdot a \cdot \pi_G(v^{-1})$ and $g_q \cdot a \notin B_{G,\Sigma}(\lfloor n/2 \rfloor)$,
$\pi_G(v^{-1}) \in B_{G,\Sigma}(\lfloor n/2 \rfloor)$. This is a contradiction
since $vv^{-1}$ and $uav^{-1}$ both have length at most $n$. Hence, we must have $|Q| > \gamma_{G,\Sigma}(\lfloor n/2 \rfloor)$.
\end{proof}
The growth of f.g.~groups is well-studied and Theorem~\ref{thm-det-growth} basically closes the chapter
on (non)deterministic streaming algorithms for word problems. Hence, in the rest of the paper we focus on
randomized streaming algorithms. Here, we can still prove a lower bound (that will turn out to be sharp in some cases
but not always) using the randomized one-way communication complexity of the equality problem (Theorem~\ref{theorem-randCC}).

\begin{theorem} \label{thm-lower-bound-random}
Let $G$ be an f.g.\ group with the finite generating set $\Sigma$.
The randomized streaming space complexity of $\WP(G,\Sigma)$ is
$\Omega(\log \log \gamma_{G,\Sigma}(\lfloor n/2 \rfloor))$.
\end{theorem}

\begin{proof}
We use the fact that the randomized one-way communication complexity of the equality problem is $\Theta(\log \log n)$ 
(Theorem~\ref{theorem-randCC}).  
Suppose that $\mathcal{R} = (\mathcal{A}_n)_{n \ge 0}$ is a randomized streaming algorithm 
for $\WP(G,\Sigma)$ with space complexity $s(n)$. 
We will construct a randomized one-way communication protocol for equality on numbers from the interval $[1, \gamma_{G,\Sigma}(\lfloor n/2 \rfloor)]$.
Identifying this interval with $B_{G,\Sigma}(\lfloor n/2 \rfloor)$ in an arbitrary way, we can assume that the inputs for Alice and Bob are group elements
from $B_{G,\Sigma}(\lfloor n/2 \rfloor)$.

Let $g \in B_{G,\Sigma}(\lfloor n/2 \rfloor)$ be the input of Alice and $h \in B_{G,\Sigma}(\lfloor n/2 \rfloor)$ be the input of Bob.
Alice chooses an arbitrary word $u \in \Sigma^{\le \lfloor n/2 \rfloor}$ with $\pi_G(u) = g$ and runs 
 (using her random choices) the PFA $\mathcal{A}_n$ on input
$u$. The state $q$ reached at the end (which can be encoded by a bit string of length at most $s(n)$) is communicated to Bob. 
Bob then chooses an arbitrary word $v \in \Sigma^{\le \lfloor n/2 \rfloor}$ with $\pi_G(v) = h^{-1}$
and simulates (using his random choices) the PFA $\mathcal{A}_n$ on input $v$ starting from state $q$
and accepts if and only if a final state of $\mathcal{A}_n$ is reached.
We have $g=h$ if and only if $uv = 1$  in $G$. Moreover, $uv$ has length at most $n$.
It follows that in case $g=h$ Bob accepts with probability at least $3/4$. In case $g \neq h$ he rejects 
with probability at least $3/4$.
This shows that we obtain indeed a randomized one-way protocol for equality with cost $s(n)$.
Theorem~\ref{theorem-randCC} implies $s(n) \ge \Omega(\log \log \gamma_{G,\Sigma}(\lfloor n/2 \rfloor))$.
\end{proof}

\begin{remark} \label{remark-infinite-group}
Since every f.g.~infinite group $G$ has growth at least $n$, Theorem~\ref{thm-lower-bound-random} shows that the randomized streaming space complexity of $\WP(G)$ is  $\Omega(\log \log n)$.
\end{remark}

\begin{remark} \label{remark-growth}
Later in this paper, we will make use of the following two famous results on the growth of groups, see also \cite{Harpe00,Ceccherini-SilbersteinA21}:
\begin{itemize}
\item Gromov's theorem \cite{Gro81}: An f.g.~group $G$ has polynomial growth if and only if $G$ is virtually nilpotent (i.e., $G$ has a nilpotent subgroup of finite index).
\item Milnor-Wolf theorem \cite{Miln68,Wolf68}; see also \cite[p.~202]{Harpe00}: 
An f.g.~solvable group $G$ is either virtually nilpotent (and hence has polynomial growth) or there is a constant
$c > 1$ such that $G$ has growth $c^n$ (i.e., $G$ has exponential growth). 
It is well known that the same dichotomy also holds for f.g.~linear groups. This is a consequence of Tits alternative \cite{Tits72}:
An f.g.~linear group $G$ is either virtually solvable or contains a free group of rank at least two (in which case $G$ has exponential growth).
\end{itemize}
\end{remark}
The dichotomy theorem of Milnor and Wolf does not generalize to all f.g.~groups. Grigorchuk \cite{Grigorchuk80} constructed a group 
whose growth is lower bounded by $\exp(n^{0.515})$ \cite{Bar01} and upper bounded by $\exp(n^{0.768})$ \cite{Bartholdi98}.
The randomized streaming space complexity of the word problem for this remarkable group will be addressed in Theorem~\ref{thm-grig}.

\subsection{Comparison with sofic groups}

In this section we will briefly discuss a relationship between randomized streaming space complexity and sofic groups.
There are many equivalent definitions of sofic groups. The following is from \cite{ArzChe20,Cav16}. 

With $\mathsf{Sym}(k)$ we denote the \emph{symmetric group} on $[1,k]$ (the set of all permutations 
on $[1,k]$ together with the operation of function composition). For $\sigma \in \mathsf{Sym}(k)$
the \emph{normalized Hamming} weight $w_H(\sigma)$ is defined by
\[
w_H(\sigma) = \frac{1}{k} \cdot |\{ i \in [1,k] \colon \sigma(i) \neq i \}|.
\]
Let $G$ be an f.g.~group and $\Sigma$ be a finite symmetric generating set for $G$. Let $\pi_G : \Sigma^* \to G$
be the canonical morphism that evaluates words in the group $G$.
Then $G$ is called \emph{sofic} if for every $n \geq 0$ there exists a $k \geq 1$ and a monoid morphism
$\sigma : \Sigma^* \to \mathsf{Sym}(k)$ (with $\sigma(a^{-1}) = \sigma(a)^{-1}$) such that 
for every word $w \in \Sigma^{\leq n}$ the following holds:
\begin{itemize}
\item if $\pi_G(w) = 1$ then $w_H(\sigma(w)) \leq 1/n$, and
\item if $\pi_G(w) \neq 1$ then $w_H(\sigma(w)) \geq 1-1/n$.
\end{itemize}
In case $G$ is sofic, we define the \emph{sofic dimension growth} of $G$ (with respect to $\Sigma$) as 
the function $\kappa_{G,\Sigma} : \mathbb{N} \to \mathbb{N}$ such that $\kappa_{G,\Sigma}(n)$ is the minimal
value $k$ for which the above conditions hold.  For different finite generating sets $\Sigma_1, \Sigma_2$ of $G$ 
the functions $\kappa_{G,\Sigma_1}$ and $\kappa_{G,\Sigma_2}$ are in general different, but their asymptotic behavior 
is the same (analogously to the growth functions $\gamma_{G,\Sigma_1}$ and $\gamma_{G,\Sigma_2}$); see \cite[Proposition 3.3.2]{Cav16} for a precise statement. 

It is a famous open problem whether every f.g.group $G$ is sofic.\footnote{One can define the concept 
of sofic groups also for non-finitely generated groups, but here we only talk about finitely generated groups.}
The connection to randomized streaming space complexity can be seen as follows: Assume that
$G$ is sofic and consider its sofic dimension growth $\kappa_{G,\Sigma}$.
For every $n \geq 0$ let $k_n = \kappa_{G,\Sigma}(n)$ and let
$\sigma_n : \Sigma^* \to \mathsf{Sym}(k_n)$ be the monoid morphism satisfying the above conditions for soficity.
Write $\sigma_{n,a} \in \mathsf{Sym}(k_n)$ for $\sigma_n(a)$ ($a \in \Sigma$).
 Then we obtain a semi-randomized streaming algorithm $\mathcal{R} =  (\mathcal{A}_n)_{n \geq 0}$ 
that is $1/n$-correct for $\WP(G,\Sigma)$ as follows: define the semiPFA
$\mathcal{A}_n = (Q_n, \Sigma, \iota_n, \delta_n, F_n)$ with
\begin{itemize}
\item  $Q_n = [1,k_n] \times [1,k_n]$,
\item  $\iota_n(i,i) = 1/k_n$ for all $i \in [1,k_n]$ and $\iota_n(i,j) = 0$ for $i \neq j$,
\item $\delta_n((i,j),a) = (i,\sigma_{n,a}(j))$ for all $i,j \in [1,k_n]$ and $a \in \Sigma$, and
\item $F_n = \{ (i,i) \colon i \in [1,k_n]\}$.
\end{itemize}
The space complexity of this algorithm is $s(\mathcal{R},n) = \lceil 2 \log_2 k_n \rceil$.

The above semi-randomized streaming algorithm $\mathcal{R} =  (\mathcal{A}_n)_{n \geq 0}$  has
some particular properties:
\begin{itemize}
\item for every $a \in \Sigma$, the transition function $q \mapsto \rho_n(q,a)$ (for $q \in Q_n$) is 
a permutation on $Q_n$, and 
\item the initial state distribution $\iota_n$ is  a uniform distribution on
a subset of $Q_n$.
\end{itemize}
The second property is not a real restriction. With an additional constant factor in the space complexity
one can easily ensure that $\iota_n$ is the uniform distribution on
a subset of $Q_n$.
The first property is a severe restriction that makes the existence of non-sofic
groups possible.

\section{Part B: Distinguishers for (semi)groups} \label{sec part B}

In order to obtain more groups, whose word problem has a low streaming space complexity (e.g., $\mathcal{O}(\log n)$)
we will consider group theoretical constructions like for instance graph products and wreath products. 
At this point it turned out to be useful to consider the stronger model of distinguishers that was mentioned in the introduction.
It is easy to turn a distinguisher for an f.g.~group $G$ into a randomized streaming algorithm for $\WP(G,\Sigma)$; thereby
the space complexity only increases by a multiplicative factor 2 (see Lemma~\ref{lemma-injective-0}). 
An advantage of the distinguisher model is that it directly extends to semigroups (and in fact to arbitrary deterministic automata).

In Section~\ref{sec-injective} we formally introduce the distinguisher model for arbitrary (infinite) deterministic automata and prove
two general lower bounds. In Section~\ref{sec dist monoids} we then specialize to distinguishers for semigroups.
In Sections~\ref{sec-dist-linear}--\ref{sec-commutative} we will directly construct small-space distinguishers for 
certain (classes of) semigroups. Finally, in Part C of the paper we will prove transfer results for distinguishers with respect to various (semi)group constructions.

\subsection{Distinguishers for infinite deterministic automata} \label{sec-injective}

For the purpose of Part B and C of the paper, final states in automata are not important. 
We will therefore skip the set of final states in the description of a (deterministic or probabilistic) automaton.

Consider a (in general infinite) deterministic automaton $\mathcal{A} = (Q,\Sigma, q_0, \delta)$
and let $\epsilon_0, \epsilon_1 : \mathbb{N} \to [0,1]_{\mathbb{R}}$ be monotonically decreasing functions.
A collection of semiPFAs $\mathcal{R} = (\mathcal{A}_n)_{n \ge 0}$ with $\mathcal{A}_n = (Q_n,\Sigma,\iota_n,\delta_n)$ 
is called an \emph{$(\epsilon_0, \epsilon_1)$-distinguisher} for $\mathcal{A}$ if the following properties hold
for all large enough $n \geq 0$ and all words $u,v \in \Sigma^{\le n}$: 
\begin{itemize}
\item If $\delta(q_0,u) = \delta(q_0,v)$ then $\Prob_{q \sim \iota_n}[\delta_n(q,u) = \delta_n(q,v)] \ge 1-\epsilon_1(n)$.
In other words: for a randomly chosen initial state,
the semiPFA $\mathcal{A}_n$ arrives with probability at least $1-\epsilon_1(n)$ in the same state after reading $u$ and $v$.
\item If $\delta(q_0,u) \neq \delta(q_0,v)$ then $\Prob_{q \sim \iota_n}[\delta_n(q,u) \neq \delta_n(q,v)] \geq 1-\epsilon_0(n)$. 
In other words: for a randomly chosen initial state, the semiPFA $\mathcal{A}_n$ arrives with probability at least
$1-\epsilon_0(n)$ in different states after reading $u$ and $v$.
\end{itemize}
If $\epsilon_0 = \epsilon_1$, we speak of an $\epsilon_0$-distinguisher.

The \emph{space complexity} of an $(\epsilon_0, \epsilon_1)$-distinguisher $\mathcal{R} = (\mathcal{A}_n)_{n \ge 0}$
is the function $s(\mathcal{R},n) = \lceil \log_2 |Q_n| \rceil$, where $Q_n$ is the state set of $\mathcal{A}_n$.
The motivation for this definition is that states of $Q_n$ can be encoded by bit strings of length at most $\lceil \log_2 |Q_n| \rceil$.
We can always assume that $s(\mathcal{R},n)$ is monotone. If $s(\mathcal{R},n) > s(\mathcal{R},m)$ for $n < m$ then we could replace
$\mathcal{A}_n$ by $\mathcal{A}_m$.

We will describe distinguishers with pseudocode, consisting of an initialization phase and an update phase, where the next input letter
$a$ will be processed (this corresponds to the transition function $\delta_n$). The values of the variables comprise the state of the distinguisher. 
In the initialization phase, certain variables will be set to random values. This corresponds to the 
initial state distribution $\iota_n$.

A deterministic distinguisher is a $0$-distinguisher. In this case, we can turn every $\mathcal{A}_n$
into a DFA by replacing the distribution $\iota_n$ by a unique initial state $q_{0,n} \in Q_n$ such that $\iota_n(q_{0,n}) > 0$
(the concrete choice of $q_{0,n}$ is not important).
Deterministic distinguisher are not very interesting. A deterministic automaton $\mathcal{A} = (Q,\Sigma, q_0, \delta)$
has a deterministic distinguisher with space complexity $\lceil \gamma_{\mathcal{A}}(n)\rceil$, where 
$\gamma_{\mathcal{A}}(n)$ is the growth function of $\mathcal{A}$ 
from Section~\ref{sec-det-auto} (take the induced automaton on
all states in $\{ \delta(q_0,w) \colon w \in \Sigma^{\le n} \}$ and add the missing outgoing transitions in an arbitrary way) and this cannot be
improved, which can be shown by the usual fooling argument. Therefore, our focus will be on distinguishers with a non-zero 
error probability. 

The following two general lower bounds on the space complexity of distinguishers are analogous to the corresponding
results for randomized streaming algorithms for languages (Theorem~\ref{thm-lower-bound-random} and 
Theorem~\ref{thm-trade-off}).

\begin{theorem}  \label{thm-growth}
Let $0 < \epsilon < 1/2$ be a constant and $\mathcal{A}= (Q,\Sigma, q_0, \delta)$ be a deterministic automaton.
Then every  $\epsilon$-distinguisher for $\mathcal{A}$ must have space complexity
$\Omega(\log \log \gamma_{\mathcal{A}}(n))$.
\end{theorem}

\begin{proof}
We prove the result by a reduction from the randomized 
communication complexity of the equality problem (Theorem~\ref{theorem-randCC}).
Let $\mathcal{R} = (\mathcal{A}_n)_{n \ge 0}$  be an $\epsilon$-distinguisher
for $\mathcal{A}$ and let $\mathcal{A}_n = (Q_n,\Sigma,\iota_n,\delta_n)$.
We identify the states in $B_n  \coloneqq  \{ \delta(q_0, w) \colon w \in \Sigma^{\le n} \}$ 
 with the interval $[1, \gamma_{\mathcal{A}}(n)]$. 

A randomized communication protocol for the equality problem on numbers from the interval
$[1, \gamma_{\mathcal{A}}(n)]$ (i.e., $\EQ_{\gamma_{\mathcal{A}}(n)}$) works as follows:
On input 
 $q \in B_n$, Alice chooses a word $u \in \Sigma^{\le n}$ with $\delta(q_0,u) = q$, randomly chooses
 an initial state $q_{n,0} \in Q_n$ according to the distribution $\iota_n$ and sends the states $q_{n,0}$ and $\delta_n(q_{n,0}, u)$ to Bob.
 Let $r \in B_n$ be Bob's input. He then chooses a word $v \in \Sigma^{\le n}$ with $\delta(q_0,v)=r$
and outputs $1$ if and only if $\delta_n(q_{n,0}, u) = \delta_n(q_{n,0},v)$. This yields a randomized one-way protocol
for $\EQ_{\gamma_{\mathcal{A}}(n)}$ with error at most $\epsilon < 1/2$. By Theorem~\ref{theorem-randCC}, the cost of this protocol must be 
$\Omega(\log \log \gamma_{\mathcal{A}}(n))$. Since the protocol communicates two states from $Q_n$, 
the space complexity of $\mathcal{R}$
must be $\Omega(\log \log \gamma_{\mathcal{A}}(n))$.
\end{proof}

\begin{theorem} \label{thm-trade-off-dist}
Let $\mathcal{A}$ be a deterministic automaton. Then, every
$\epsilon(n)$-distinguisher $\mathcal{R}$ for $\mathcal{A}$ satisfies 
$s(\mathcal{R},n) \geq \min \{ \log_2 \gamma_{\mathcal{A}}(n), \log_2 (1/\epsilon(n)) \}$ for $n$ large enough.
\end{theorem}

\begin{proof}
The proof is the same as for Theorem~\ref{thm-trade-off}. Recall that there is a deterministic distinguisher for 
$\mathcal{A}$ with space complexity $\lceil \log_2 \gamma_{\mathcal{A}}(n) \rceil$.
\end{proof}
The theorem implies that an $\epsilon(n)$-distinguisher with $\epsilon(n) \leq 1/\gamma_{\mathcal{A}}(n)$
cannot yield a space improvement over a deterministic distinguisher.

\begin{oproblem}  \label{open-dist2}
An interesting open problem concerns error reduction for distinguishers. Recall that for a randomized streaming algorithm for a language $L$ one can 
reduce the error by running several independent copies of the algorithm; see Theorem~\ref{thm-prob-amplify}. It is not clear whether an analogous
result holds for distinguishers. Running $k$ independent copies of an $(\epsilon_0(n), \epsilon_1(n))$-distinguisher for a deterministic automaton $\mathcal{A}$ yields
an $(\epsilon_0(n)^k, 1-(1-\epsilon_1(n))^k)$-distinguisher for $\mathcal{A}$.
The problem is that $1-(1-\epsilon_1(n))^k$ decreases when $k$ increases.
Note that in case $\epsilon_1(n) = 0$, error reduction by running independent copies works.
\end{oproblem}

An $\epsilon$-distinguisher for a deterministic automaton $\mathcal{A}$ can be seen as a non-uniform probabilistic 
approximation of $\mathcal{A}$ by finite systems. Proving interesting results in this context for arbitrary deterministic infinite automata seems
to be difficult. Therefore, we will restrict for the main part of this paper to Cayley graphs of finitely generated monoids and try
to deduce results about the space complexity of distinguishers for Cayley graphs from the algebraic properties of the underlying
monoid.

\subsection{Distinguishers for monoids and semigroups: general results} \label{sec dist monoids}

Let $M$ be an f.g.~monoid with the finite generating set $\Sigma$. 
An $(\epsilon_0, \epsilon_1)$-distinguisher for $M$ (with respect to $\Sigma$)
is an $(\epsilon_0, \epsilon_1)$-distinguisher for the Cayley graph $C(M,\Sigma)$.
This means that for all large enough $n \geq 0$ and all words $u,v \in \Sigma^{\le n}$ we have: 
\begin{itemize}
\item If $u \equiv_M v$ then $\Prob_{q \sim \iota_n}[\delta_n(q,u) = \delta_n(q,v)] \ge 1-\epsilon_1(n)$.
In other words: for a randomly chosen initial state,
the semiPFA $\mathcal{A}_n$ arrives with probability at least $1-\epsilon_1(n)$ in the same state after reading $u$ and $v$.
\item If $u \not\equiv_M v$ then $\Prob_{q \sim \iota_n}[\delta_n(q,u) \neq \delta_n(q,v)] \geq 1-\epsilon_0(n)$. 
In other words: for a randomly chosen initial state, the semiPFA $\mathcal{A}_n$ arrives with probability at least
$1-\epsilon_0(n)$ in different states after reading $u$ and $v$.
\end{itemize}
An $(\epsilon_0, \epsilon_1)$-distinguisher for a semigroup $S$  with respect to the finite generating set $\Sigma$
is an $(\epsilon_0, \epsilon_1)$-distinguisher for the monoid $S^1$  with respect to $\Sigma$.
In the following, we are mainly interested in distinguishers for monoids. Nevertheless, some results will be formulated
for semigroups since in the proofs we use results from the literature that are formulated for semigroups.
For groups, we have the following easy lemma:

\begin{lemma} \label{lemma-injective-0} 
Let  $\mathcal{R}$ be an $(\epsilon_0, \epsilon_1)$-distinguisher for an f.g.~group $G$ with respect to $\Sigma$.
Then $\WP(G,\Sigma)$ has an $(\epsilon_0, \epsilon_1)$-correct semi-randomized streaming algorithm with space complexity $2 \cdot s(\mathcal{R},n)$.
\end{lemma}

\begin{proof}
Let  $\mathcal{R}=(\mathcal{A}_n)_{n \ge 0}$ with $\mathcal{A}_n = (Q_n,\Sigma,\iota_n,\delta_n)$.
Using the above definition of an $(\epsilon_0, \epsilon_1)$-distinguisher with the empty string $v = \varepsilon$ we get for every word
$u \in \Sigma^{\le n}$:
\begin{itemize}
\item If $u \in \WP(G,\Sigma)$ then $\Prob_{q \sim \iota_n}[\delta_n(q,u) = q] \ge 1-\epsilon_1(n)$.
\item If $u \notin \WP(G,\Sigma)$ then $\Prob_{q \sim \iota_n}[\delta_n(q,u) \neq q] \geq 1-\epsilon_0(n)$.
\end{itemize}
This allows to construct an $(\epsilon_0, \epsilon_1)$-correct 
semi-randomized streaming algorithm $(\mathcal{B}_n)_{n \ge 0}$ for $\WP(G,\Sigma)$. Thereby the space complexity of the algorithm only doubles:
We define $\mathcal{B}_n = (Q_n \times Q_n, \Sigma, \iota'_n, \delta'_n,F_n)$ where
\begin{itemize}
\item $\iota'_n(p,p) = \iota_n(p)$ for all $p \in Q_n$ and $\iota'_n(p,q) = 0$ if $p \neq q$,
\item $\delta'_n((p,q), a) = (p,\delta_n(q,a))$ for $p,q \in Q_n$ and $a \in \Sigma$, and
\item $F_n = \{ (p,p) \colon p \in Q_n \}$.
\end{itemize}
It is easy to check that this semi-randomized streaming algorithm is indeed $(\epsilon_0, \epsilon_1)$-correct for $\WP(G,\Sigma)$.
\end{proof}
Due to Lemma~\ref{lemma-injective-0}, our goal in the rest of the paper will be the construction of space efficient
$(\epsilon_0, \epsilon_1)$-distinguishers for groups and, more generally, for monoids.

\begin{oproblem} \label{open-dist1} \label{op-distinguisher-WP}
It is open, whether a converse of Lemma~\ref{lemma-injective-0} holds: Assume that $\WP(G,\Sigma)$ has an $(\epsilon_0, \epsilon_1)$-correct semi-randomized streaming algorithm with space complexity $s(n)$. Does this imply that there exists an
$(\mathcal{O}(\epsilon_0), \mathcal{O}(\epsilon_1))$-distinguisher for $G$ with space complexity $\mathcal{O}(s(n))$?
Also note that if this question has a positive answer, then error reduction for group distinguishers is possible (see open problem~\ref{open-dist2})
by first transforming the distinguisher to a randomized streaming algorithm for the word problem using Lemma~\ref{lemma-injective-0}.
Then, one can reduce the error using Theorem~\ref{thm-prob-amplify}, and finally transform the resulting randomized streaming algorithm
back to a distinguisher.
\end{oproblem}
Since every f.g.~infinite monoid $M$ has linear growth, we obtain the following corollary 
from Theorem~\ref{thm-growth}:

\begin{corollary} \label{coro general lower bound for monoids}
Every $\epsilon$-distinguisher
($0 < \epsilon < 1/2$) for an f.g.~infinite monoid has space complexity $\Omega(\log \log n)$.
\end{corollary}

We will use $(\epsilon_0, \epsilon_1)$-distinguishers for monoids in order
to get transfer results for various constructions from (semi)group theory.
For this, we need some further observations on $(\epsilon_{0}, \epsilon_{1})$-distinguishers that we discuss in the rest of the section.

For equivalence relations $\equiv_1$ and $\equiv_2$ on a set $A$ and a subset $S \subseteq A$ we say that:
\begin{itemize}
\item $\equiv_1$ refines $\equiv_2$ on $S$ if for all $a,b \in S$ we have: if $a \equiv_1 b$ then $a \equiv_2 b$,
\item $\equiv_1$ equals $\equiv_2$ on $S$ if for all $a,b \in S$ we have: $a \equiv_1 b$ if and only if $a \equiv_2 b$.
\end{itemize}
For a semiPFA $\mathcal{A} = (Q,\Sigma, \iota, \delta)$ and a state $q \in Q$ we define the equivalence
relation $\equiv_{\mathcal{A},q}$ on $\Sigma^*$ as follows: $u \equiv_{\mathcal{A},q} v$ if and only if $\delta(q,u) = \delta(q,v)$.
Whenever $\mathcal{A}$ is clear from the context, we just write $\equiv_{q}$ instead of $\equiv_{\mathcal{A},q}$.

\begin{lemma} \label{lemma-injective}
Let $(\mathcal{A}_n)_{n \ge 0}$ be an $(\epsilon_{0}, \epsilon_{1})$-distinguisher for the finitely generated monoid $M$ with respect
to the finite generating set $\Sigma$. Let $\mathcal{A}_n = (Q_n,\Sigma, \iota_n, \delta_n)$.
Consider a set $U \subseteq \Sigma^{\le n}$. 
Then,  the following statements hold, where $\equiv_{q}$ refers to $\mathcal{A}_n$:
\begin{itemize}
\item $\Prob_{q \sim \iota_n}[\equiv_{M} \text{equals} \equiv_{q} \text{on } U]\geq 1- \max\{\epsilon_0(n), \epsilon_1(n)\}\binom{|U|}{2}$,
\item $\Prob_{q \sim \iota_n}[\equiv_{M} \text{refines} \equiv_{q}\text{on } U]\geq 1-\epsilon_1(n) \binom{|U|}{2}$,
\item $\Prob_{q \sim \iota_n}[\equiv_{q} \text{refines} \equiv_{M}\text{on } U]\geq 1-\epsilon_0(n) \binom{|U|}{2}$.
\end{itemize}
\end{lemma}

\begin{proof}
All three statements follow from the union bound and the fact that there are $\binom{|U|}{2}$ unordered pairs 
of different elements from $U$. For the first statement note that 
$\Prob_{q \sim \iota_n}[(u \equiv_M v \text{ and } u \not\equiv_{q} v) \text{ or } (u \not\equiv_M v \text{ and } u \equiv_{q} v)] \leq \max\{\epsilon_0(n), \epsilon_1(n)\}$ for all $u,v \in U$.
For the second statement, note that $\Prob_{q \sim \iota_n}[u \equiv_M v \text{ and } u \not\equiv_{q} v] \leq \epsilon_1(n)$, and similarly for the third statement.
\end{proof}
Recall that for a word $w$ we write $\mathcal{P}(w)$ for the set of all prefixes of $w$.

\begin{lemma} \label{lemma-injective-reduce}
Let $M$ be a finitely generated monoid with the finite generating set $\Sigma$ and let $\mathcal{A} = (Q,\Sigma, \iota, \delta)$ be a semiPFA with $q \in Q$.
Consider words $u,v \in \Sigma^*$ such that $\equiv_{M}$ refines $\equiv_q$ on $\mathcal{P}(u) \cup \mathcal{P}(v)$ and
let $u = xyz$ with $y \equiv_M \varepsilon$. Then $\equiv_{M}$ refines $\equiv_q$ on $\mathcal{P}(xz) \cup \mathcal{P}(v)$.
\end{lemma}

\begin{proof}
Assume that $s,t \in \mathcal{P}(xz) \cup \mathcal{P}(v)$ are such that $s \equiv_{M} t$. We have to show that $\delta(q,s) = \delta(q,t)$.
If $s \notin \mathcal{P}(u) \cup \mathcal{P}(v)$ we must have $s = xz'$ for a prefix $z'$ of $z$.
Since $x \equiv_M xy$, $x,xy \in \mathcal{P}(u)$ and $\equiv_{M}$ refines $\equiv_q$ on $\mathcal{P}(u) \cup \mathcal{P}(v)$, we have $\delta(q,x) = \delta(q,xy)$.
This implies that $\delta(q,s) = \delta(q,xz') = \delta(q,xyz')$. In addition we have $s \equiv_M xyz'$ and $xyz' \in \mathcal{P}(u)$.
In this way we obtain from $s$ a word $\tilde{s} \in \mathcal{P}(u) \cup \mathcal{P}(v)$ such that $\delta(q,s) = \delta(q,\tilde{s})$ and
$s \equiv_M \tilde{s}$ (we might have $s = \tilde{s}$). In the same way, we can obtain from $t$ 
a word $\tilde{t} \in \mathcal{P}(u) \cup \mathcal{P}(v)$ such that $\delta(q,t) = \delta(q,\tilde{t})$ and
$t \equiv_M \tilde{t}$.
Since $s \equiv_{M} t$ we have $\tilde{s} \equiv_{M} \tilde{t}$. Since 
$\equiv_{M}$ refines $\equiv_q$ on $\mathcal{P}(u) \cup \mathcal{P}(v)$ and $\tilde{s}, \tilde{t} \in \mathcal{P}(u) \cup \mathcal{P}(v)$ we get
$\delta(q,s) = \delta(q,\tilde{s}) = \delta(q,\tilde{t}) = \delta(q,t)$.
\end{proof}

\begin{lemma} \label{lemma-injective-reduce2}
Let $M$, $\mathcal{A}$, and $q$ be as in Lemma~\ref{lemma-injective-reduce}.
Consider words $u,v \in \Sigma^*$ such that $\equiv_q$ refines $\equiv_M$ on $\mathcal{P}(u) \cup \mathcal{P}(v)$ and
let $u = xyz$ with $\delta(q,x) = \delta(q,xy)$. Then $\equiv_q$ refines $\equiv_M$ on $\mathcal{P}(xz) \cup \mathcal{P}(v)$.
\end{lemma}

\begin{proof}
Assume that $s,t \in \mathcal{P}(xz) \cup \mathcal{P}(v)$ are such that $\delta(q,s) = \delta(q,t)$. We have to show that $s \equiv_{M} t$.
If $s \notin \mathcal{P}(u) \cup \mathcal{P}(v)$ we must have $s = xz'$ for a prefix $z'$ of $z$. 
Since $\delta(q,x) = \delta(q,xy)$ and $x,xy \in \mathcal{P}(u)$, we have $x \equiv_M xy$.
This implies $xyz' \equiv_M x z' = s$,
$\delta(q,s) = \delta(q,xz') = \delta(q,xyz')$, and $xyz' \in \mathcal{P}(u)$.
The same argument works also for $t$ and yields words $\tilde{s}, \tilde{t} \in \mathcal{P}(u) \cup \mathcal{P}(v)$ 
such that $\delta(q,s) = \delta(q,\tilde{s})$, $\delta(q,t) = \delta(q,\tilde{t})$, $s \equiv_M \tilde{s}$, and
$t \equiv_M \tilde{t}$. We obtain $\delta(q,\tilde{s}) = \delta(q,s) = \delta(q,t) =  \delta(q,\tilde{t})$.
Since $\equiv_q$ refines $\equiv_M$ on $\mathcal{P}(u) \cup \mathcal{P}(v)$ we get $s \equiv_M \tilde{s} \equiv_M \tilde{t} \equiv_M t$.
\end{proof}

 \subsection{Distinguishers for linear monoids} \label{sec-dist-linear}

A monoid is \emph{linear} if it is isomorphic to a monoid of matrices over a field $K$.
The monoid of all $(r \times r)$-matrices with entries from $K$ is denoted with $M_r(K)$.
In this section, we show that with some care one can turn the logspace algorithms from \cite{LiZa77,Sim79} into $(\epsilon_0(n),0)$-distinguishers with  $\epsilon_0(n) = 1/n^c$ for a constant $c$ and space complexity $\mathcal{O}(\log n)$.

To this end, we will make use of the following well-known result of DeMillo, Lipton, Schwartz and Zippel \cite{DemilloL78,Schwartz80,Zippel79}. 
Therein, the \emph{degree} of a multivariate polynomial $p(x_1,\ldots,x_m) \in K[x_1,\ldots,x_m]$ with coefficients from the field $K$ is the maximal sum $k_1+ k_2 + \cdots + k_m$
 where $x_1^{k_1} x_2^{k_2} \cdots x_m^{k_m}$ is a monomial of $p$.

\begin{theorem} \label{schwartz-zippel} 
Let $p(x_1,\ldots,x_m) \in K[x_1,\ldots,x_m]$ be a non-zero multivariate polynomial of degree $d$, and let $S
\subseteq K$ be finite.
If $(s_1,\ldots,s_m) \in S^m$ is randomly chosen according to the uniform distribution, then 
$\Prob[p(s_1,\ldots,s_m)=0] \leq \frac{d}{|S|}$.
\end{theorem}
Recall that a prime field is either the field $\mathbb{Q}$ or a finite field $\mathbb{F}_p$ for a prime $p$.
For a field $F$ and variables $x_1, \ldots, x_m$, the transcendental extension $F(x_1, \ldots, x_m)$
is the field of all fractions $q_1/q_2$ for polynomials $q_1,q_2 \in F[x_1,\ldots,x_m]$ with $q_2 \neq 0$.

We also need the following lemma, which is shown in \cite{LiZa77} for a finitely generated linear group instead of a finitely generated linear monoid. The proof for the monoid case
is the same as for the group case. In order to keep the paper self-contained we provide the details. 

\begin{lemma} \label{lemma-LiZa77}
Let $M$ be an f.g.\ linear monoid. Then $M$ is isomorphic to a 
submonoid of $M_r (K)$, where the field $K$ is of the form 
$K = F(x_1,\ldots,x_m)$ for a prime field $F$.
\end{lemma}

\begin{proof} 
Assume that $M$ is finitely generated by the matrices $A_1, \ldots, A_k \in M_r (K)$.
Let $C \subseteq K$ be the finite set consisting of all field elements that appear in one of the matrices
$A_i$.
We can therefore replace the field $K$ by the subfield $L$ generated
by $C$. If we choose a maximal algebraically independent subset
$\{x_1, \ldots, x_m\} \subseteq C$, then $L$ is isomorphic to a finite algebraic extension
of the field of fractions $F(x_1,\ldots,x_m)$, where $F$ is the
prime field of $K$, see e.g.\ \cite[p.~156]{Jaco64}.
Let $[L:F(x_1,\ldots,x_m)]=d$ be the degree of this algebraic
extension.  The field $L$ can also be viewed as a $d$-dimensional associative algebra 
over the base field $F(x_1, \ldots, x_m)$ (that is, a vector space with a multiplication that together with the vector addition yields a ring structure). By the
regular representation for associative algebras \cite{Jac74}, $L$ is isomorphic 
to a subring of $M_d(F(x_1, \ldots, x_m))$.
Hence, by replacing in the generator
matrices~$A_i$ every matrix entry by the corresponding matrix from $M_d(F(x_1, \ldots, x_m))$,
it follows that $M$ is isomorphic to
a submonoid of $M_{rd}(F(x_1,\ldots,x_m))$. 
\end{proof}

Let us now turn to the main result of this section.

\begin{theorem} \label{thm-lin}
For every f.g.~linear monoid $M$ and every $c > 0$ there exists a $(1/n^c,0)$-distinguisher 
with space complexity $\mathcal{O}(\log n)$.
\end{theorem}

\begin{proof}
By Lemma~\ref{lemma-LiZa77}, we can assume that 
 $M$ is a finitely generated submonoid of $M_r (K)$ where $K = F(x_1,\ldots,x_m)$ for some prime field $F$. 

Let us first assume that $F = \mathbb{Q}$. Let $\Sigma$ be a generating set for $M$.  
Then every generator $A \in \Sigma$ is a matrix, whose entries are quotients of polynomials from 
$\mathbb{Z}[x_1,\ldots,x_m]$. Therefore there exists a fixed non-zero polynomial $t \in \mathbb{Z}[x_1,\ldots,x_m]$
such that every matrix $\hat A  \coloneqq  t \cdot A$ for $A \in \Sigma$ has entries from $\mathbb{Z}[x_1,\ldots,x_m]$.
Let $r$ be the dimension of the matrices.
Let $d$ be the maximal degree of $t$ and all polynomials that appear in matrices $\hat A$ with 
$A \in \Sigma$. The parameters $m$, $r$, and $d$ are constants in the further considerations.

Fix an input length $n \geq 1$.
Clearly, for all matrices $A_1, \ldots, A_k, B_1, \ldots, B_l \in \Sigma$ with $k,l \leq n$ we have
$\prod_{i=1}^k A_i = \prod_{i=1}^l B_i$ if and only if $t^{n-k} \prod_{i=1}^k \hat A_i = t^{n-l} \prod_{i=1}^l \hat B_i$.

Consider two input words $A_1 A_2 \cdots A_k, B_1 B_2 \cdots B_l \in \Sigma^*$ with $k,l \leq n$ and assume that
\[
t^{n-k} \prod_{i=1}^k \hat A_i \neq t^{n-l} \prod_{i=1}^l \hat B_i .
\]
Define the matrix
\begin{equation} \label{matrix N}
N  \coloneqq  t^{n-k} \prod_{i=1}^k \hat A_{i} - t^{n-l} \prod_{i=1}^l \hat B_{i} \in  \mathbb{Z}[x_1,\ldots,x_m]^{r \times r}.
\end{equation}
Note that all entries of the matrix $N$ are polynomials
of degree at most $dn$ and at least one of them is not the zero polynomial.

Let $\bar{S}= [1,4dn^{c+1}]^m \subseteq \mathbb{N}^m$, where $c$ is the value from the theorem. For a tuple $\bar{s} = (s_1,\ldots,s_m) \in \bar{S}$ and 
a matrix $C \in \mathbb{Z}[x_1,\ldots,x_m]^{r \times r}$
let $C(\bar{s})$ be the integer matrix obtained from $C$ by replacing every variable $x_i$ by $s_i$.
For a uniformly chosen tuple $\bar{s} \in \bar{S}$, Theorem~\ref{schwartz-zippel} implies that
\begin{equation*} \label{prob-A-neq-0}
\Prob_{\bar{s} \in \bar{S}}[N(\bar{s}) \neq 0] \geq 1-\frac{1}{4n^c}.
\end{equation*}
Let us now consider a tuple $\bar{s} \in \bar{S}$ such that $N(\bar{s}) \neq 0$.
Every entry in a matrix $\hat A(\bar{s})$ ($A \in \Sigma$) has an absolute value of size at most $\mathcal{O}( (4dn^{c+1})^d ) = \mathcal{O}(n^{d(c+1)})$ ($d$ is a constant)
and also $|t(\bar{s})| \le \mathcal{O}(n^{d(c+1)})$.
Therefore, all entries in the matrix  $t(\bar{s})^{n-k} \prod_{i=1}^k \hat A_{i}(\bar{s})$ have an absolute value 
of size at most $\mathcal{O}(r^n n^{d(c+1)n})$, and similarly for 
$t(\bar{s})^{n-l} \prod_{i=1}^l \hat B_{i}(\bar{s})$. Hence,
$N(\bar{s})$ is a non-zero matrix with all entries of absolute value $\mathcal{O}(r^n n^{d(c+1)n})$.

The number of different prime factors
of a positive number $D \le \mathcal{O}(r^n n^{d(c+1)n})$ is bounded by
\[
\frac{ \ln D}{\ln \ln D} \cdot (1+o(1)) \le \mathcal{O}\bigg( \frac{n  \log n}{\log n}\bigg) = \mathcal{O}(n);
\]
see \cite[Theorem~16]{Robin1983}.
By a weak form of the prime number theorem, the number of primes of size at most $n^{c+2} \ln n$ is $\Theta(n^{c+2})$.
Hence, by uniformly choosing a prime of size at most $\alpha  \coloneqq  n^{c+2} \ln n$ we can obtain the bound
\[
\Prob _{\bar{s} \in \bar{S},\, p \in \mathbb{P}_\alpha}[N(\bar{s}) \bmod p = 0 \mid N(\bar{s}) \neq 0]   \leq \mathcal{O}\bigg(\frac{1}{n^{c+1}}\bigg) \leq
\frac{1}{4n^c} 
\]
for $n$ large enough.
Hence, we obtain
\begin{eqnarray*}
& & \Prob_{\bar{s} \in \bar{S}, \, p \in \mathbb{P}_\alpha}\bigg[t(\bar{s})^{n-k} \prod_{i=1}^k \hat A_{i}(\bar{s}) \not\equiv t(\bar{s})^{n-l} \prod_{i=1}^l \hat B_{i}(\bar{s}) \bmod p \bigg] \\
& = & \Prob_{\bar{s} \in \bar{S}, \, p \in \mathbb{P}_\alpha}[N(\bar{s}) \bmod p \neq 0 \mid N(\bar{s}) \neq 0]  \cdot \Prob_{\bar{s} \in \bar{S}}[N(\bar{s}) \neq 0]  \\
& \geq & \bigg(1- \frac{1}{4n^c} \bigg)^2 \\
& \geq & 1 - \frac{1}{2n^c}
\end{eqnarray*}
for $n$ large enough. The same calculation shows that also
\[
\Prob_{\bar{s} \in \bar{S}, \, p \in \mathbb{P}_\alpha}[t(\bar{s})\bmod p \neq 0] \geq 1 - \frac{1}{4n^c}
\]
(this probability is in fact much closer to $1$) and hence
\[
\Prob_{\bar{s} \in \bar{S},\, p \in \mathbb{P}_\alpha}\!\bigg[t(\bar{s})\bmod p \neq 0 \;\wedge\;
t(\bar{s})^{n-k} \prod_{i=1}^k \hat A_{i}(\bar{s}) \not\equiv t(\bar{s})^{n-l} \prod_{i=1}^l \hat B_{i}(\bar{s}) \bmod p \bigg] \geq 1 - \frac{1}{n^c}.
\]
The streaming algorithm for inputs of length at most $n$ is now clear. Initially, the algorithm guesses  
$\bar{s} \in \bar{S}$ (for $\bar{S} = [1,2dn^{c+1}]^m$) and a prime $p \in \mathbb{P}_\alpha$ for $\alpha =  n^{c+2} \ln n$. All these numbers
need $\mathcal{O}(\log n)$ bits in total. If $t(\bar{s}) \bmod p  = 0$ then the algorithm ignores the input word.
Otherwise, the algorithm initializes a matrix $C  \coloneqq  t(\bar{s})^{n} \cdot \mathsf{Id}_r \bmod p$,
where $\mathsf{Id}_r$ is the $r$-dimensional identity matrix. Then, for every
new generator matrix $A \in \Sigma$ the algorithm updates $C$ by
\[
C  \coloneqq  t(\bar{s})^{-1} \cdot C \cdot \hat{A}(\bar{s}) \bmod p .
\]
All computation are carried out in the field $\mathbb{F}_p$. If $\prod_{i=1}^k A_i = \prod_{i=1}^l B_i$ ($k,l \leq n$) then 
after reading the input words $A_1 \cdots A_k$ and $B_1 \cdots B_l$, the algorithm arrives with probability one in 
the same state (this holds also in case $t(\bar{s}) \bmod p  = 0$).
On the other hand, if $\prod_{i=1}^k A_i \neq \prod_{i=1}^l B_i$ then the reached states differ with probability
at least $1-1/n^c$ by the above error analysis.

Let us now briefly discuss the case where the underlying prime field is $F = \mathbb{F}_p$ for a prime $p$.
Then we have to work in a finite extension $\mathbb{F}_{p^e}$ for some $e \ge 1$ such that
$p^e \geq dn^{c+1}$, which can be achieved by taking $e$ of size $\Theta(\log n)$.
By choosing a tuple $\bar{s} = (s_1,\ldots,s_m) \in \bar{S}  \coloneqq  \mathbb{F}_{p^e}^m$ randomly, we obtain for the matrix $N$ from \eqref{matrix N}
\[
\Prob_{\bar{s} \in \bar{S}}[N(\bar{s}) \neq 0] \geq 1-\frac{1}{4n^c}.
\]
Since an $r$-dimensional matrix over the field $\mathbb{F}_{p^e}$ can be stored in space
$\mathcal{O}(\log n)$ ($r$ and $p$ are constants and $e = \Theta(\log n)$), this yields the desired algorithm in the same way as for the case $F = \mathbb{Q}$.
\end{proof}

 \subsection{Distinguishers for nilpotent (semi)groups} \label{sec-nilpotent}

The \emph{lower central series} of a group $G$ is the series $G = G_0 \trianglerighteq G_1 \trianglerighteq G_2 \trianglerighteq \cdots$ 
where $G_{i+1} = [G,G_i]$. 
A group $G$ is \emph{$i$-step nilpotent} if $G_i = 1$ and it is \emph{nilpotent} if it is $i$-step nilpotent  for some $i \geq 0$.

Every nilpotent group is linear and the $1$-step nilpotent groups are exactly the abelian groups.
For nilpotent groups we can improve the algorithm from the proof of Theorem~\ref{thm-lin}, at least if we sacrifice  
the inverse polynomial error probability:

\begin{theorem} \label{thm-nilpotent}
For every f.g.~nilpotent group $G$ and every constant $c > 0$ there exists a $(1/(\log n)^c,0)$-distinguisher 
with space complexity $\mathcal{O}(\log \log n)$.\footnote{The inverse polylogarithmic error probability cannot be improved, see Remark~\ref{remark-inverse-polylog} below.} 

\end{theorem}

\begin{proof}
With $\mathsf{UT}_d(\mathbb{Z})$ we denote the set of all upper triangular $(d \times d)$-matrices over $\mathbb{Z}$ with all
diagonal entries equal to $1$ (so-called unitriangular matrices). These matrices form an f.g.~nilpotent group and every f.g.~torsion-free 
nilpotent group can be embedded into some $\mathsf{UT}_d(\mathbb{Z})$ \cite[Theorem 17.2.5]{KaMe79}.

We can assume that our f.g.~nilpotent group $G$ is infinite.
Then $G$ has an f.g.~torsion-free nilpotent subgroup
$H$ such that the index $[G:H]$ is finite \cite[Theorem 17.2.2]{KaMe79}.
Since $H$ must be also f.g., it can be embedded into 
$\mathsf{UT}_d(\mathbb{Z})$ for some $d \ge 1$.

By Theorems~\ref{thm-subgroup} and \ref{thm-finite-ext} below, it suffices to show that every $\mathsf{UT}_d(\mathbb{Z})$ has a
$(1/(\log n)^c,0)$-distinguisher 
with  space complexity $\mathcal{O}(\log \log n)$.

Fix a finite generating set $\Sigma$ for $\mathsf{UT}_d(\mathbb{Z})$ and an input length $n$.
Consider a product $M \coloneqq  \prod_{i=1}^m M_i$ with $m \le n$ and 
$M_i \in \Sigma$. From \cite[Proposition~4.18]{Loh14} it follows that the absolute 
value of every entry of the matrix $M$ has size at most $\mathcal{O}(m^{d-1}) \leq \mathcal{O}(n^{d-1})$.
The distinguisher for $\mathsf{UT}_d(\mathbb{Z})$ guesses a prime number $p \in \mathbb{P}_\alpha$ for 
$\alpha = \Theta((\log n)^{c+1} \cdot \log \log n)$ and computes the product matrix $M$ modulo $p$. For this,
$\mathcal{O}(\log\log n)$ bits are sufficient.
 
Consider two input words $u = M_1 M_2 \cdots M_l \in \Sigma^*$ and $v = N_1 N_2 \cdots N_m \in \Sigma^*$ with $l,m \leq n$.
If $\prod_{i=1}^l M_i = \prod_{i=1}^m N_i$ then our distinguisher will reach with probability one the same
state after reading the input words $u$ and $v$, respectively. On the other hand, if $\prod_{i=1}^l M_i \neq \prod_{i=1}^m N_i$,
then consider a non-zero matrix entry $a \in \mathbb{Z}$ of the matrix $\prod_{i=1}^l M_i - \prod_{i=1}^m N_i$.
We have $|a| \leq \mathcal{O}(n^{d-1})$. The number of different prime factors of $a$ is therefore bounded by
$\mathcal{O}(\log n/\log \log n)$ \cite[Theorem~16]{Robin1983}. By randomly choosing a prime number $p \in \mathbb{P}_\alpha$ 
for $\alpha = \Theta((\log n)^{c+1} \cdot \log \log n)$ we can therefore obtain a probability of at most $1/(\log n)^c$ 
for $a \bmod p = 0$. Hence, with probability $1- 1/(\log n)^c$ we reach different states after reading $u$ and $v$, respectively.
\end{proof}
Note that if $G$ is infinite, the space bound from Theorem~\ref{thm-nilpotent} is sharp up to constant factors even if we allow a constant error probability; see Remark~\ref{remark-infinite-group}.

A group $G$ is \emph{virtually nilpotent} if it is a finite extension of a nilpotent group. In other words: $G$ has a nilpotent subgroup $H$ with finitely many cosets.
Recall that Gromov \cite{Gro81} proved that a  finitely generated group has polynomial growth if and only if it is virtually nilpotent.

\begin{corollary} \label{thm-virt-nilpotent}
For every f.g.~virtually nilpotent group $G$ and every constant $c > 0$ there exists a $(1/(\log n)^c,0)$-distinguisher 
with space complexity $\mathcal{O}(\log \log n)$.
\end{corollary}

\begin{proof}
In Section~\ref{sec-simple-transfer} we will show that finite extensions blow up the space complexity of a distinguisher only
by a constant factor in the argument $n$; see Theorem~\ref{thm-finite-ext}. With Theorem~\ref{thm-nilpotent} this yields the corollary.
\end{proof}

\begin{corollary}\label{thm--gap-linear}
Let $G$ be an infinite finitely generated linear group.
\begin{itemize}
\item If $G$ is virtually nilpotent then the $(0$-sided) randomized streaming space complexity of  $\WP(G)$ is $\Theta(\log \log n)$.
\item If $G$ is not virtually nilpotent then $(0$-sided) the randomized streaming space complexity of $\WP(G)$ is $\Theta(\log n)$.
\end{itemize}
\end{corollary}

\begin{proof}
The upper bounds follow from Theorems~\ref{thm-lin} and \ref{thm-virt-nilpotent}. Since $G$ is infinite, the 
randomized streaming space complexity of $\WP(G)$ is $\Omega(\log \log n)$ (see Remark~\ref{remark-infinite-group}), which yields the lower
bound for the virtually nilpotent case. If $G$ is not virtually nilpotent, then $G$ has growth $c^n$ for some constant $c>1$ (see Remark~\ref{remark-growth}), which
yields the lower bound $\Theta(\log n)$ by Theorem~\ref{thm-lower-bound-random}.
\end{proof}
It is conjectured that for every f.g.~group $G$  that is not virtually nilpotent the growth is lower bounded by $\exp(n^{0.5})$.
This is known as the \emph{gap conjecture} \cite{Grig14}. It would imply that for every f.g.~group $G$ that is not virtually nilpotent,
the randomized streaming space complexity of $\WP(G)$ is lower bounded  by $\Omega(\log n)$.

Grigorchuk \cite{Grig88} proved that an f.g.~cancellative monoid has polynomial growth if and only
if it has a virtually nilpotent group of left quotients. Since a cancellative monoid embeds into its group of left quotients,
it follows that an f.g.~cancellative monoid of polynomial growth embedds into an f.g.~virtually nilpotent group.
With Theorems~\ref{thm-virt-nilpotent} and \ref{thm-subgroup} 
we obtain:

\begin{theorem} \label{thm-cancellative-poly-growth}
 Every finitely generated cancellative monoid of polynomial growth has a $((\log n)^{-c},0)$-distinguisher 
with space complexity $\mathcal{O}(\log \log n)$ for every constant $c > 0$.
 \end{theorem}
 We will later (see the comment after Theorem~\ref{thm-bcm}) see that there are finitely generated monoids of polynomial growth, for which every
 distinguisher has space complexity $\Omega(\log n)$.
   
 Nilpotent groups were generalized to nilpotent semigroups by 
 Maltsev \cite{malcev1953} as well as  Neumann and Taylor \cite{NT63}. 
 They both showed that the class of nilpotent 
groups can be defined by identities that only use the group multiplication (and not inverses). These identities
then serve as the defining identities of nilpotent semigroups. The identities used by Maltsev and Neumann and Taylor
are slightly different. Here we adopt the definition of Neumann and Taylor.
To be more specific, define words $X_i$ and $Y_i$ over
formal variables $x,y, z_1, \ldots, z_{i-1}$  for $i \ge 1$ inductively as follows: 
\[ X_1  = xy, \  \  Y_1 = yx,  \  \ X_i = X_{i-1} z_i Y_{i-1}, \ \ Y_i = Y_{i-1} z_i X_{i-1} \text{ for } i \geq 2 .
\]
Then, a semigroup $S$ is \emph{$i$-step nilpotent} if it satisfies the identity $X_i = Y_i$ (for all values from $S$ that are substituted
for the variables $x,y, z_1, \ldots, z_i$). Moreover, $S$ is nilpotent if it is $i$-step nilpotent for some $i \geq 1$.

Every f.g.~cancellative nilpotent semigroup embeds into an f.g.~nilpotent group \cite{NT63}. 
We therefore obtain the following result (that can be also obtained from Theorem~\ref{thm-cancellative-poly-growth}):
\begin{theorem} \label{thm-cancellative-nilpotent}
 Every f.g.~cancellative nilpotent semigroup has a $((\log n)^{-c},0)$-dis\-tinguisher 
with space complexity $\mathcal{O}(\log \log n)$ for every constant $c > 0$.
 \end{theorem}
On the other hand, it is known that there exists a nilpotent semigroup $S$ 
with exponential growth; see e.g.~\cite[page~512]{Shnee01}.
Hence, by Theorem~\ref{thm-growth}, every $\epsilon$-distinguisher (for $0 < \epsilon < 1/2$) for $S$
has  space complexity $\Omega(\log n)$. 

\begin{oproblem}
Does every nilpotent semigroup have
an $\epsilon$-distinguisher with space complexity $\mathcal{O}(\log n)$ for some $\epsilon < 1/2$.
\end{oproblem}

\begin{remark} \label{remark-inverse-polylog}
By Theorem~\ref{thm-trade-off-dist} the inverse polylogarithmic error in Corollary~\ref{thm-virt-nilpotent} cannot be improved if $G$ is infinite:
Consider an $\epsilon(n)$-distinguisher with space complexity $r(n) \le \mathcal{O}(\log \log n)$ for the infinite group $G$. 
The growth of $G$ is at least $\Omega(n)$.
Hence, if $n$ is large enough, we must have $r(n) \geq \log_2(1/\epsilon(n))$ by Theorem~\ref{thm-trade-off-dist}.
We get $\log_2(1/\epsilon(n)) \leq c \cdot \log_2 \log_2 n$ for some constant $c>0$, i.e.,  $\epsilon(n) \geq 1/\log_2^c n$.

Analogously, the inverse polylogarithmic error in Theorems~\ref{thm-cancellative-poly-growth}
and \ref{thm-cancellative-nilpotent} cannot be improved if the semigroup $S$ is infinite.
\end{remark}

\subsection{Distinguishers for commutative semigroups} \label{sec-commutative}

A semigroup $S$ is \emph{commutative} if $ab = ba$ for all $a,b \in S$.
Note that every f.g.~commutative semigroup $S$ has polynomial growth: if the generating set $\Sigma$ has size $k$, then $\gamma_{S,\Sigma}(n) \leq \mathcal{O}(n^k)$. 
Hence, Theorem~\ref{thm-cancellative-poly-growth} (or \cite[Proposition~3.2]{GrilletComm})
yields the following:

\begin{lemma} \label{lemma-cancel-comm}
For every f.g.~cancellative commutative semigroup $S$ and every constant $c > 0$ there exists a $((\log n)^{-c},0)$-distinguisher 
with space complexity $\mathcal{O}(\log \log n)$.
\end{lemma}

The goal of this section is to generalize Lemma~\ref{lemma-cancel-comm} to all finitely generated commutative semigroups as follows.

\begin{theorem} \label{thm-comm}
For every f.g.~commutative semigroup $S$ and every constant $c > 0$ there exists a $((\log n)^{-c},0)$-distinguisher 
with space complexity $\mathcal{O}(\log \log n)$.
\end{theorem}

To prove Theorem~\ref{thm-comm} we use a structure theorem for commutative semigroups, which we now recall. 
For this we need several definitions.
  
A \emph{subdirect product} of semigroups $S_i$ ($i \in I$) is a subsemigroup $S$ of the direct product 
$\Pi_{i \in I} S_i$ such that $\pi_j(S) = S_j$ for every $j \in I$. Here, $\pi_j : \Pi_{i \in I} S_i \to S_j$ is the 
projection homomorphism with $\pi_j( (s_i)_{i \in I} ) = s_j$. Since there exists a surjective homomorphism 
from $S$ to every $S_j$, it follows that if $S$ is finitely generated, then also every $S_j$ is finitely generated.

By \cite[Theorem~VI.2.2]{GrilletComm}, an f.g.~commutative semigroup is a subdirect product of 
\emph{finitely} many (necessarily f.g.) semigroups $S_1, \ldots, S_k$, where
every $S_i$ is of one of the following types:
\begin{itemize}
\item an f.g.~commutative cancellative  semigroup,
\item an f.g.~commutative nilsemigroup (which is necessarily finite),
\item an f.g.~subelementary semigroup.
\end{itemize}
Here, a \emph{subelementary} semigroup is a commutative semigroup $S = C \uplus N$ for disjoint subsemigroups $C$ and $N$
such that the following hold:
\begin{itemize}
  \item $C$ is cancellative in $S$, i.e.,  $ab = ac$ implies $b=c$ for all $a \in C$ and $b,c \in S$.
  \item $N$ is a nilsemigroup and an ideal in $S$.
\end{itemize}
In particular, $C$ is cancellative. 
Moreover, if $S$ is finitely generated, then $C$ is also finitely generated since $C \uplus \{0\}$ ($C$ together
with a zero element) is a quotient of $S$ (namely the Rees quotient $S/N$); see \cite[Lemma~VI.2.3]{GrilletComm}.
If $C$ is an abelian group, then $S = C \uplus N$ is called an \emph{elementary} semigroup. 

For the purpose of proving Theorem~\ref{thm-comm} it suffices by Lemma~\ref{lemma-cancel-comm} 
to consider the case of an f.g.~subelementary semigroup. 
Moreover, every subelementary semigroup $S$ can be embedded into an elementary semigroup $T$ \cite[Proposition~VI.3.1]{GrilletComm}, and if $S$ is finitely generated, then
so is $T$ \cite[Proposition~VI.3.3]{GrilletComm}.
 Hence, in order to prove Theorem~\ref{thm-comm}, it suffices to show the following lemma.

\begin{lemma} \label{thm-elementary}
For every f.g.~elementary semigroup $S$ and every constant $c > 0$ there exists a $((\log n)^{-c},0)$-distinguisher 
with space complexity $\mathcal{O}(\log \log n)$.
\end{lemma}

\begin{algorithm}[h]
  \SetKwComment{Comment}{(}{)}
  \SetKwInput{KwGlobal}{global variables}
  \SetKwInput{KwInit}{initialization}
  \SetKwInput{KwNext}{next input letter}
  \BlankLine
  \KwGlobal{$q_\alpha \in Q_{\alpha, n}$ ($\alpha \in X$) and $\chi \in X$} 
  \BlankLine
  \KwInit{}
  \For{$\alpha \in X$}{choose $q_\alpha$ according to the distribution $\iota_{\alpha, n}$\label{alg-com-elementary-init}}
  $\chi \coloneqq 1$\label{alg-com-elementary-chi1}
  \BlankLine
  \KwNext{$a \in \Sigma$ with $\pi(a) = \beta$}
  \For{$\alpha \in X$}
      {\uIf{$\alpha \leq \chi \cdot \beta$}{
        $q_\alpha \coloneqq \delta_{\alpha, n}(q_\alpha, r_\alpha(g_{\chi, a}))$\label{alg-com-elementary-mult}}
      \Else{$q_\alpha \coloneqq q^\ast_{\alpha}$\label{alg-com-elementary-dflt}}}
       $\chi \coloneqq \chi \cdot \beta$\label{alg-com-elementary-chi2}
  \caption{A $((\log n)^{-c},0)$-distinguisher for an elementary semigroup.} \label{alg-com-elementary}
\end{algorithm}

\begin{proof}
  Let $S = G \uplus N$ be an f.g.\ elementary semigroup where $G$ is an abelian subgroup of~$S$ and $N$ is an ideal of $S$ and a nilsemigroup.
  Further, let $\Sigma \subseteq S$ be a finite generating set.
  We write $\Sigma$ as a disjoint union $\Sigma = \Sigma_G \uplus \Sigma_N$ with $\Sigma_G = \Sigma \cap G$ and $\Sigma_N = \Sigma \cap N$.
  Then $\Sigma_G \subseteq G$ is a generating set of $G$, as the complement of $G$ in $S$ is an ideal.
  In particular, $G$ is finitely generated. Be aware, however, that $N$ need not be finitely generated.

  Let us now consider Green's $\mathcal{H}$-preorder on $S$, which we simply denote by $\leq$.
  Since $S$ is a commutative semigroup, it can be defined by $a \leq b$ if and only if $a = bc$ for some $c \in S$, and the corresponding equivalence relation $\mathcal{H}$ (with $a \mathrel{\mathcal{H}} b$ if and only if $a \leq b \leq a$) is a congruence relation on $S$.
  Let $\pi \colon S \to X$ with $X = S / \mathcal{H}$ be the corresponding quotient homomorphism.

  The monoid $X$ consists of the $\mathcal{H}$-classes of $S$. As a quotient of the commutative semigroup $S$, it is commutative as well.
  Moreover, it is a \emph{nilmonoid} \cite[Corollary~V.1.8]{GrilletComm}, i.e.,  a monoid obtained from a commutative nilsemigroup by adjoining a neutral element.
  The neutral element $1 \in X$ corresponds to $\pi^{-1}(1) = G$, which forms a single $\mathcal{H}$-class of~$S$.
  Likewise, the zero element $0 \in X$ corresponds to $\pi^{-1}(0) = \{0\} \subseteq N$.
  Moreover, Green's $\mathcal{H}$-preorder on $X$, which is also denoted by $\leq$, is a partial order with greatest element $1 \in X$ and least element $0 \in X$.
  The monoid $X$ is generated by $\pi(\Sigma)$ as a semigroup (and by $\pi(\Sigma_N)$ as a monoid).
  Since every f.g.\ commutative nilmonoid is clearly finite, so is the monoid $X$.

  Let us now consider the action of the group $G$ on $S$ by right multiplication.
  For every $a \in S$ and $g \in G$, we clearly have $a = a g g^{-1} \leq ag \leq a$ and, hence, $a \mathrel{\mathcal{H}} ag$.
  Thus, the action of $G$ preserves the $\mathcal{H}$-classes of $S$.
  In fact, the $\mathcal{H}$-classes of $S$ are precisely the orbits under the action of $G$; see \cite[Proposition~IV.5.1]{GrilletComm}.
  Therefore, we can also write $X$ as the quotient $X = S / G$.
  
  Recall from Section~\ref{sec-action} that the stabilizers $\mathsf{Stab}_G(a)$ and $\mathsf{Stab}_G(b)$ of elements $a,b \in S$ from the same orbit of the action under $G$
  are conjugated, i.e., $\mathsf{Stab}_G(a) = g\,\mathsf{Stab}_G(b) g^{-1}$ for some $g \in G$. Since  
  $G$ is commutative, we get $\mathsf{Stab}_G(a) = \mathsf{Stab}_G(b)$.
  Hence, the stabilizer $\mathsf{Stab}_G(a) = \{ g \in G \mid ag = a \}$ of any $a \in S$ depends only on the orbit $aG \subseteq S$ of $a$, that is, on the $\mathcal{H}$-class $\pi(a) \in X$ of $a$.
  For $\alpha \in X$, we thus obtain a well-defined quotient group $G_\alpha \coloneqq G / \mathsf{Stab}_G(a)$ where $a \in \pi^{-1}(\alpha)$.
  We denote the corresponding quotient homomorphism by $r_\alpha \colon G \to G_\alpha$.
  It is easy to see that these maps are compatible with one another in the following sense.

  \begin{claim}\label{pro-com-elementary-claim}
    For all $g,h \in G$ and $\alpha, \beta \in X$ with $\alpha \leq \beta$, $r_\beta(g) = r_\beta(h)$ implies $r_\alpha(g) = r_\alpha(h)$.
  \end{claim}
  \begin{proof}
    Let $a,b,c \in S$ with $a = bc$, $\pi(a) = \alpha$, and $\pi(b) = \beta$.
    By definition of the mapping $r_\beta$, we have $r_\beta(g) = r_\beta(h)$ if and only if $gh^{-1} \in \mathsf{Stab}_G(b)$ if and only if
    $bg = bh$.
    The latter implies $ag = bcg = bch = ah$ and, therefore, $r_\alpha(g) = r_\alpha(h)$.
  \end{proof}
  Let us now choose for each $\alpha \in X$ a canonical representative $b_\alpha \in \pi^{-1}(\alpha) \subseteq S$ where $b_1 = 1$.
  At this point, every $a \in S$ can be written as $a = b_{\pi(a)} g$ for some $g \in G$, since $a$ and $b_{\pi(a)}$ are contained in the same orbit under the action of $G$.
  The expression $b_{\pi(a)} g$ is not unique, but for $\alpha,\beta \in X$ and $g,h \in G$ we have $b_\alpha g = b_\beta h$  if and only if
  $\alpha=\beta$ and $r_\alpha(g) = r_\beta(h)$.
  Moreover, let us choose for each $\alpha \in X$ and $a \in \Sigma$ an element $g_{\alpha, a} \in G$ such that $b_\alpha a = b_{\alpha \cdot \pi(a)} g_{\alpha, a}$ holds in~$S$.
  For $a \in \Sigma_G$ we can choose $g_{\alpha,a} = a$  since $b_\alpha a = b_{\alpha \cdot \pi(a)} g_{\alpha, a} = b_{\alpha \cdot 1} g_{\alpha, a} = b_{\alpha} g_{\alpha, a}$.
  
  Thus, for $a_1, a_2, \dotsc, a_n \in \Sigma$, we have 
  \begin{equation} \label{eq repr of a1a2...a_n}
a_1a_2 \cdots a_n = b_{\chi_n} g_{\chi_0, a_1} g_{\chi_1, a_2} \cdots g_{\chi_{n-1}, a_n}
\end{equation}
  in~$S$ where $\chi_0 = 1 \in X$ and $\chi_i = \pi(a_1a_2 \cdots a_{i}) \in X$ for $1 \leq i \leq n$.
  Finally, let 
  \[\Gamma = \{ g_{\chi, a} \mid \chi \in X, a \in \Sigma \},\] 
  which finitely generates $G$ since 
  $a = g_{\alpha,a} \in \Gamma$ for every $a \in \Sigma_G$.

  With all prerequisites in place, let us now construct a distinguisher for the semigroup~$S$.
  By Theorem~\ref{thm-nilpotent}, we can fix for each $\alpha \in X$ a $((\log n)^{-c}, 0)$-distinguisher $\mathcal{R}_\alpha = (\mathcal{A}_{\alpha, n})_{n \geq 0}$ for the f.g.\ abelian group $G_\alpha$ with respect to the generating set $\Gamma_\alpha \coloneqq r_{\alpha}(\Gamma)$ such that the space complexity of $\mathcal{R}_\alpha$ is $\mathcal{O}(\log \log n)$.
  Now fix the input length $n$.
  For our construction, we use the semiPFAs $\mathcal{A}_{\alpha, n} = (Q_{\alpha, n}, \Gamma_{\alpha}, \iota_{\alpha, n}, \delta_{\alpha, n})$ and (arbitrary) default states $q_{\alpha}^\ast \in Q_{\alpha, n}$.
  The semiPFA $\mathcal{A}_n$ for our distinguisher $\mathcal{R} = (\mathcal{A}_n)_{n \geq 0}$ for $S$ is implicitly defined by Algorithm~\ref{alg-com-elementary}.
  From its global variables  $\chi \in X$, $q_\alpha$ ($\alpha \in X$) we see that the space complexity of the distinguisher is
  \[
    s(\mathcal{R}, n) = \log_2 \abs{X} + \sum_{\alpha \in X} s(\mathcal{R}_{\alpha}, n) \in \mathcal{O}(\log \log n)
  \]
  (note that $|X|$ is a constant).
  Let $w \in \Sigma^{\leq n}$ be an input word.
  Then Algorithm~\ref{alg-com-elementary} incrementally computes $\pi(w) \in X$, which is stored in the program variable $\chi$, in lines~\ref{alg-com-elementary-chi1} and~\ref{alg-com-elementary-chi2}.
  It also implicitly computes a word $\tilde{w} \in \Gamma^\ast$ as a concatenation of elements $g_{\chi, a} \in \Gamma$ in line~\ref{alg-com-elementary-mult}.
  This word satisfies $w = b_{\pi(w)} \tilde{w}$ in~$S$ by construction; see \eqref{eq repr of a1a2...a_n}.
  We then have the following equivalences for all words $u,v \in \Sigma^{\leq n}$:
  \begin{eqnarray}
u \equiv_S v   & \Longleftrightarrow & \pi(u) = \pi(v)  \eqqcolon \beta \text{ in } X \text{ and } r_\beta(\tilde{u}) = r_\beta(\tilde{v}) \text{ in } G_\beta \label{equiv1-comm} \\
& \Longleftrightarrow & \pi(u) = \pi(v)  \eqqcolon \beta \text{ in } X \text{ and } \forall \alpha \leq \beta \colon r_\alpha(\tilde{u}) = r_\alpha(\tilde{v}) \text{ in } 
G_\alpha \label{equiv2-comm}
\end{eqnarray}
where the second equivalence follows from Claim~\ref{pro-com-elementary-claim}.
  
  Note that Algorithm~\ref{alg-com-elementary} simulates on input $w$ the semiPFA $\mathcal{A}_{\alpha, n}$ on input $\tilde{w}_\alpha  \coloneqq r_\alpha(\tilde w)$ for each $\alpha \leq \pi(w)$ in lines~\ref{alg-com-elementary-init} and~\ref{alg-com-elementary-mult}.
  For $\alpha \in X$ with $\alpha \not\leq \pi(w)$ such a simulation is performed initially, but the state $q_\alpha$ is then eventually set to the default state $q_\alpha^\ast$ in line~\ref{alg-com-elementary-dflt}.
  It is now straightforward to analyze the error probability of $\mathcal{A}_n = (Q_n, \Sigma, \iota_n, \delta_n)$ for two input words $u, v \in 
  \Sigma^{\leq n}$. 
  
  \medskip\noindent
  \emph{Case 1.} $u \not\equiv_S v$. 
  If $\pi(u) \not\equiv_X \pi(v)$, then the error probability is zero, since the variable $\chi$
  will have different values after reading $u$~and~$v$.
  If $\pi(u)$ and $\pi(v)$ have the same value  $\beta$ in $X$, then we must have $r_\beta(\tilde{u}) \not\equiv_{G_\beta} r_\beta(\tilde{v})$
  by \eqref{equiv1-comm},
   where $\tilde{u}, \tilde{v} \in \Gamma^\ast$ are the implicitly computed words on inputs $u$ and $v$, respectively. 
  The algorithm inputs $r_\beta(\tilde{u})$ and $r_\beta(\tilde{v})$ into the semiPFA $\mathcal{A}_{\beta, n}$ and both words have length at most $n$; hence, the resulting states of $\mathcal{A}_{\beta, n}$ will differ with probability at least $1-(\log n)^{-c}$.
  
  \medskip\noindent
  \emph{Case 2.}  $u \equiv_S v$. Then the error probability is equal to zero: 
  The variable $\chi$ holds the value $\beta \in X$ of $\pi(u) \equiv_X \pi(v)$ after reading $u$ and $v$, respectively. 
  Moreover, since $r_\alpha(\tilde{u}) \equiv_{G_\alpha} r_\beta(\tilde{v})$ for all $\alpha \leq \beta$ by \eqref{equiv2-comm},
  all semiPFAs $\mathcal{A}_{\alpha, n}$ for $\alpha \leq \beta$ 
   will indeed be in the same state with probability one after reading $u$ and $v$.
  Morover, all variables $q_\alpha$ for $\alpha \not\le \beta$ have the value $q_\alpha^\ast$ after reading $u$ as well as after reading $v$.
\end{proof}

\section{Part C: Transfer results for distinguishers} \label{sec-C-transfer}

In this section, we will investigate the question to which extent semigroup constructions preserve
the space complexity of distinguishers.

 \subsection{Finitely generated subgroups, finite extensions, direct products}
 \label{sec-simple-transfer}
 
 For many algorithmic problems in group theory, the complexity is preserved when (i) going down to a finitely generated
 subgroup or (ii) going up to a finite extension.
 This is also true for the space complexity of distinguishers:

 \begin{theorem} \label{thm-subgroup}
 Let $M$ be an f.g.~monoid and $N$ an f.g.~submonoid of $M$. If $\mathcal{R}$ is an $(\epsilon_0(n), \epsilon_1(n))$-distinguisher for $M$
 then $N$ has an $(\epsilon_0(c \cdot n), \epsilon_1(c \cdot n))$-distinguisher with space complexity $s(\mathcal{R}, c \cdot n)$
for some constant $c$. 
 \end{theorem}
  
 \begin{proof}
 Fix the generating sets $\Sigma$ and $\Gamma$ of $M$ and $N$, respectively. Then for every generator $a \in \Gamma$ 
 there is a word $w_a \in \Sigma^*$ such that $a$ and $w_a$ represent the same element of $N$.
 We can then argue as in the proof of Lemma~\ref{lemma-gen-set}.
 \end{proof}

 \begin{theorem} \label{thm-finite-ext}
Assume that $H$ is an f.g.~group and $G$ is a subgroup of $H$ of finite index (hence, also $G$ must be f.g.).
Assume that $\mathcal{R}$ is an $(\epsilon_0(n), \epsilon_1(n))$-distinguisher for $G$.
Then $H$ has an $(\epsilon_0(c \cdot n), \epsilon_1(c \cdot n))$-distinguisher with space complexity $s(\mathcal{R}, c \cdot n) + \mathcal{O}(1)$
for some constant $c$. 
\end{theorem}

\begin{proof}
We can assume that $G$ is a normal subgroup of $H$:
It is well known that there exists
a normal subgroup $N$ of $H$ (the so-called normal core of $G$) such that $N \leq G$ and $N$ has finite index in $H$, see e.g.~\cite[Excercise 1.6.9]{Rob96}.
Since $N$ has finite index in $H$, also $N$ must be finitely generated.  Since $N \leq G$, Theorem~\ref{thm-subgroup} implies that
$N$ has an $(\epsilon_0(dn), \epsilon_1(dn))$-distinguisher with space complexity $s(\mathcal{R},dn)$ for some constant $d$.
This shows that we can replace $G$ by $N$.

For the rest of the proof we assume that $G$ is normal in $H$. Fix a generating set $\Sigma$ for $G$.
Let $h_1, \ldots, h_k \in H$ be a set of coset representatives for $G$ where $h_1 = 1$. Then
$\Gamma = \Sigma \cup \{h_2, \ldots, h_k\}$ generates $H$. Since $G$ is normal,
for every $a \in \Sigma$ and every $i \in [1,k]$ there exists an element $g(i,a) \in G$ such that
$h_i a = g(i,a) h_i$ in $H$ (with $g(1,a) = a$). Moreover, for all $i,j \in [1,k]$ there are $g(i,j) \in G$ and
$1 \le \alpha(i,j) \le k$ such that $h_i h_j = g(i,j) h_{\alpha(i,j)}$ in $H$ (with $\alpha(1,i) = \alpha(i,1)=i$ and
$g(i,1) = g(1,i) = 1$). Below, we identify the group elements $g(i,a), g(i,j)$ with words over the alphabet $\Sigma$.
Let $c$ be the maximal length of these words.

\begin{algorithm}[t]
 \SetKwComment{Comment}{(}{)}
\SetKwInput{KwGlobal}{global variables}
\SetKwInput{KwInit}{initialization}
\SetKwInput{KwNext}{next input letter}
\BlankLine
\KwGlobal{$i \in [1,k]$, $q \in Q_{c n}$} 
\BlankLine
\KwInit{}
guess $q$ according to the distribution $\iota_{cn}$ \label{line-init-q-r1} \\
$i  \coloneqq 1$  \\[0.5em]
\KwNext{$a \in \Gamma$}
\eIf{$a \in \Sigma$}{$q  \coloneqq \delta_{cn}(q, g(i,a))$}
{let $a = h_j$ with $j \ge 2$ \\
$q  \coloneqq \delta_{cn}(q, g(i,j))$ \\
$i  \coloneqq \alpha(i,j)$ \\
\BlankLine
\caption{An $(\epsilon_0(cn), \epsilon_1(cn))$-distinguisher for the finite extension $H$ of $G$.} 
\label{alg dist finite index}}
\end{algorithm}

Let $\mathcal{R} = (\mathcal{A}_n)_{n \ge 0}$ with $\mathcal{A}_{n} = (Q_{n},\Sigma,\iota_{n},\delta_{n})$
be the $(\epsilon_0 (n),\epsilon_1 (n))$-distinguisher for $G$ with respect to $\Sigma$.
An $(\epsilon_0 (cn),\epsilon_1 (cn))$-distinguisher for $H$ with respect to $\Gamma$ 
is shown in Algorithm~\ref{alg dist finite index}.

Fix an input length $n \geq 0$. Algorithm~\ref{alg dist finite index} stores a coset representative $h \in \{h_1,\ldots,h_k\}$ 
in the form of an index $i \in [1,k]$
and a state $q \in Q_{c n}$ of the semiPFA $\mathcal{A}_{c n}$. Initially, we set $i  \coloneqq1$ (since $h_1$ is the group identity) 
and $q$ is guessed according to
the initial state distribution $\iota_{c n}$.

Consider an input word $w \in \Sigma^{\le n}$ and 
assume that $q_0$ is the initially guessed value for the variable $q$. Then, after reading an input word $w \in\Gamma^{\le n}$
representing the group element $g h_j \in H$ ($g \in G$, $j \in [1,k]$), the program variable $i$ has the value $j$ (thus, the right coset will be computed) and the state
variable $q$ satisfies $q = \delta_{cn}(q_0, w_g)$ for a word $w_g \in \Sigma^*$ of length at most $c n$ representing the group 
element $g$.
From this it follows easily that the algorithm is an $(\epsilon_0(cn), \epsilon_1(cn))$-distinguisher for $H$.
Since $i$ needs space $\mathcal{O}(1)$ and $q$ needs space $s(\mathcal{R}, c \cdot n)$, the space complexity of the algorithm is $s(\mathcal{R}, c \cdot n) + \mathcal{O}(1)$.
\end{proof}
 Also direct products preserve the space complexity of distinguishers (simply run the distinguishers
 for the two factor groups in parallel):
 
 \begin{lemma} \label{thm-direct-prod}
Let $M$ (resp., $N$) be a finitely generated monoid for which there exists an $(\epsilon_0, \epsilon_1)$-distinguisher $\mathcal{R}$
(resp., $(\zeta_0, \zeta_1)$-distinguisher $\mathcal{S}$). 
Then there exists a $(\max\{\epsilon_0,\zeta_0\}, \epsilon_1+\zeta_1)$-distinguisher
for $M \times N$ with space complexity $s(\mathcal{R},n) + s(\mathcal{S},n)$.\footnote{Here, $\max\{\epsilon_0,\zeta_0\}$ denotes the pointwise
maximum of the two functions $\epsilon_0$ and $\zeta_0$.}
\end{lemma}
Recall that a group $G$ is metabelian if it has an abelian normal subgroup $A \le G$ such that the 
quotient $G/A$ is abelian as well. Every finitely generated metabelian group can be embedded into 
a direct product of finitely generated linear groups (over fields of different characteristics) \cite{Wehr80}. Hence, 
with Lemma~\ref{thm-direct-prod} and Theorem~\ref{thm-lin} we obtain:

\begin{corollary} \label{coro-metabel}
For every finitely generated metabelian group and every $c > 0$ there exists a $(1/n^c,0)$-distinguisher 
with space complexity $\mathcal{O}(\log n)$.
\end{corollary}

  \subsection{Distinguishers for graph products} \label{sec-graph-prod}
  
 In this section we investigate graph products of monoids. This construction is a common generalization of free and direct products.
 Originally, graph products were introduced for groups \cite{Gre90}, but the construction easily extends to monoids (and even semigroups; see
 the end of this section) \cite{VelCos01}.
 
  Fix monoids $M_1, \ldots, M_c$ and a symmetric and irreflexive edge relation $I \subseteq [1,c] \times [1,c]$.
  Here, we only consider the case where all $M_i$ are finitely generated.
 Write $M_i$ as $M_i = \langle \Sigma_i \mid R_i\rangle^* = \Sigma_i^*/\rho_{R_i}$ with $\Sigma_i$ finite, whereas
 $R_i \subseteq \Sigma_i^* \times \Sigma_i^*$ is potentially infinite.
 Assume w.l.o.g.~that the generating sets $\Sigma_i$ are pairwise 
 disjoint. Then the graph product $M = \mathsf{GP}(M_1, \ldots, M_c, I)$   
is defined as $M = \langle \Sigma \mid R \rangle^* = \Sigma^*/\rho_R$ where
\begin{eqnarray}
\Sigma & = & \bigcup_{i=1}^c \Sigma_i \text{ and } \label{def-Sigma-GP} \\
R & = & \bigcup_{i=1}^c R_i   \cup \bigcup_{(i,j) \in I} \{ (ab, ba) \colon a \in \Sigma_i, b \in \Sigma_j\} . \label{relation-R}
\end{eqnarray}
Up to isomorphism, this definition does not depend
on the presentations $(\Sigma_i, R_i)$ for the monoids $G_i$. Note that $M$ is finitely generated by $\Sigma$.

In case $I = \emptyset$, one obtains the free product $\Asterisk_{i \in [1,c]} M_i$
of the monoids $M_1, \ldots, M_c$. On the other hand, if
$([1,c], I)$ is a complete graph, then $M$ is the direct product $\prod_{i \in [1,c]} M_i$.
The graph product $\mathsf{GP}(M_1, \ldots, M_c, I)$ 
is obtained from the free product 
$\free_{i \in [1,c]} M_i$ by allowing elements from monoids $M_i$ and $M_j$ with $(i,j) \in I$ to commute.
 
Graph products  $\mathsf{GP}(G_1, \ldots, G_c, I)$, where every $G_i$ is isomorphic to the infinite cyclic group $\mathbb{Z}$,
are also known as \emph{graph groups} (or \emph{right-angled Artin groups}). We will make use of the fact that
every graph group is linear \cite{Hum94}.

Consider the graph product $M = \mathsf{GP}(M_1, \ldots, M_c, I)$ and, as above, let $\Sigma_i$ 
be a generating set for $M_i$ and $\Sigma  =  \bigcup_{i=1}^c \Sigma_i$.
 For a word $u \in \Sigma^*$, the \emph{block factorization} of $u$ is the unique factorization $u = u_1 u_2 \cdots u_l$
 such that $l \geq 0$,  $u_1, \ldots, u_l \in \bigcup_{i \in [1,c]} \Sigma_i^+$ and $u_j u_{j+1} \not\in \bigcup_{i \in [1,c]} \Sigma_i^+$
 for all $j \in [1,l-1]$. The factors
 $u_1, u_2, \ldots, u_l$ are also called the \emph{blocks} of $u$.
 
 We define several rewrite relations on words from $\Sigma^*$ as follows. Take $u,v \in \Sigma^*$ and let
 $u = u_1 u_2 \cdots u_l$ be the block factorization of $u$. 
 \begin{itemize}
 \item We write $u \to_s v$ ({\it s} for swap) if there is $i \in [1,l-1]$ and $(j,k) \in I$
 such that $u_i \in \Sigma_j^+$, $u_{i+1} \in \Sigma_k^+$ and 
 $v = u_1 u_2 \cdots u_{i-1} u_{i+1} u_i u_{i+2} \cdots u_l$. In other words, we 
  swap consecutive commuting blocks. 
 \item We write $u \to_d v$ ({\it d} for delete) if there is $i \in [1,l]$  and $j \in [1,c]$ 
 such that $u_i \in \Sigma_j^+$, $u_i = 1$ in $M_j$ and $v = u_1 u_2 \cdots u_{i-1} u_{i+1} u_{i+2} \cdots u_l$.
In other words, we delete a block that is trivial in its monoid.
\item We write $u \leftrightarrow_r v$ ({\it r} for replace) if there is $i \in [1,l]$  and $j \in [1,c]$ 
 such that $u_i, u'_i \in \Sigma_j^+$, $u_i \equiv_{M_j} u'_i$ and $v = u_1 u_2 \cdots u_{i-1} u'_i u_{i+1} u_{i+2} \cdots u_l$.
In other words, we replace a block by an equivalent non-empty word. 
\end{itemize}
Clearly, in all three cases we have $u \equiv_M v$. If  $u \rightarrow_{d} v$, then the number of blocks of $v$ is smaller than
the number of blocks of $u$ and if $u \to_s v$ then the number of blocks of $v$ can be smaller than
the number of blocks of $u$ (since two blocks can be merged into a single block).
Note that $\leftrightarrow_r$ is a symmetric relation but neither $\to_d$ nor $\to_s$ are symmetric. For $\to_d$ this is clear. 
For $\to_s$ note that for $a \in \Sigma_j$, $b \in \Sigma_k$ with $(j,k) \in I$ we have $aba \to_s aab$ but $aab \to_s aba$ does
not hold since $aa$ is a single block.
We write 
$u \to_{\mathit{sd}} v$ if $u \to_s v$ or $u \to_{d} v$. 

Let us say that a word $u \in \Sigma^*$ with $l$ blocks
is \emph{reduced} if there is no $v \in \Sigma^*$ such that $u \to_{\mathit{sd}}^* v$ and $v$ has at most $l-1$ blocks.
Note that $\to_s$ restricted to reduced words is symmetric.
Clearly, for every word $u \in \Sigma^*$ there is a reduced word $u' \in \Sigma^*$ such that
$u \to_{\mathit{sd}}^* u'$. The following result can be found in \cite[Proposition~3.8]{DanGou23} in a slightly different notation.\footnote{The difference is that
in  \cite{DanGou23} elements of $M_i \setminus \{1\}$ are considered as single symbols that replace the blocks in our setting. This makes our
replace relation $\leftrightarrow_r$ obsolete. On the other hand, in \cite{DanGou23} merge steps $ab \to c$ for $a,b,c \in M_i \setminus \{1\}$ with $ab=c$ in $M_i$ are needed. In our setting, these steps are implicit when two blocks are merged.} 

\begin{lemma} \label{lemma-rewrite-GP-2}
Let $M$ be a graph product as above and $u,v \in \Sigma^*$.
Then the following are equivalent:
 \begin{itemize}
\item $u \equiv_M v$.
\item There are reduced words $u',v'$ such that $u \to_{\mathit{sd}}^* u'$, $v \to_{\mathit{sd}}^* v'$, and $u'  \leftrightarrow_r^* v'$.
\end{itemize}
\end{lemma}
Consider a word $u \in \Sigma^*$ and its block factorization $u = u_1 u_2 \ldots u_l$.
 A \emph{pure prefix} of $u$ is a word $u_{k_1} u_{k_2} \cdots u_{k_m}$ such that for some $i \in [1,c]$ we have
 \begin{itemize}
 \item $1 \leq k_1 < k_2 < \cdots < k_m \leq l$, 
\item $u_{k_1}, u_{k_2}, \ldots, u_{k_m} \in \Sigma_i^+$, and
\item if $k_j <  p < k_{j+1}$ for some $j \in [1,m-1]$ or $1 \le p < k_1$ then $u_p \notin \Sigma^+_i$.
\end{itemize}
Our first main result on graph products is the following:

\begin{theorem}  \label{thm-GP}
Let $M = \mathsf{GP}(M_1, \ldots, M_c, I)$ be a graph product as above, where every $M_i$ is a left-cancellative monoid. Let
$\mathcal{R}_i =(\mathcal{A}_{i,n})_{n \ge 0}$ be an $(\epsilon_0,\epsilon_1)$-distinguisher for $M_i$. Fix a constant $d \geq 1$ and define 
\begin{eqnarray*}
\zeta_0(n) &=& 2 \epsilon_0(n) (n+c)^2 + 1/n^d, \\
\zeta_1(n) &=& 2 \epsilon_1(n) (n+c)^2.
\end{eqnarray*} 
Then, there is a $(\zeta_0,\zeta_1)$-distinguisher for the graph product $M$
with space complexity $\mathcal{O}(\sum_{i=1}^c s(\mathcal{R}_i,n) + \log n)$.\footnote{Note that Theorem~\ref{thm-GP} only makes sense if $\epsilon_i(n) < 1/2(n+c)^2$ for $i \in \{1,2\}$.}
\end{theorem}

\begin{proof}
Let us fix an input length $n$ and let
$\mathcal{A}_{i,n} = (Q_{i,n}, \Sigma_i, \iota_{i,n}, \rho_{i,n})$, where w.l.o.g.~$Q_{i,n} = [1, |Q_{i,n}|]$ consists of the first $|Q_{i,n}|$ natural numbers.
To simplify the notation, we will omit the second subscript $n$ in the following, i.e., we write
$\mathcal{A}_{i} = (Q_{i}, \Sigma_i, \iota_{i}, \rho_{i})$ with $Q_{i} = [1, |Q_{i}|]$ for the semiPFA $\mathcal{A}_{i,n}$.  
For a state $q \in Q_i$, we write $\equiv_{q}$ for the 
 equivalence relation $\equiv_{\mathcal{A}_i,q}$
defined in Section~\ref{sec-injective}.
For a word $w \in \Sigma^*$, we write $\pi_i(w)$ for the projection $\pi_{\Sigma_i}(w)$.

For every $i \in [1,c]$ we choose a new symbol $a_i$ and consider the infinite cyclic group $\langle a_i \rangle \cong \mathbb{Z}$.
Let $\Delta = \{a_1, a_1^{-1}, \ldots, a_c, a_c^{-1}\}$ and consider the graph group $H = \mathsf{GP}(\langle a_1\rangle, \ldots, \langle a_c\rangle, I)$.
Since every graph group is linear, there is a $(1/m^d,0)$-distinguisher $(\mathcal{B}_m)_{m \ge 0}$ 
with space complexity $\mathcal{O}(\log m)$ for $H$ by Theorem~\ref{thm-lin}.
Let $\mathcal{B}_m = (R_m, \Delta, \lambda_m, \sigma_m)$.

We build from the semiPFA $\mathcal{A}_{i}$ and a state $q \in Q_i$ a sequential transducer (see Section~\ref{sec transducer})
$\mathcal{T}_{i,q} = (Q_{i}, \Sigma_i, \{a_i,a_i^{-1}\}, q, \delta_{i})$, where 
$$
\delta_i(p, a) = (\rho_i(p,a),  a_i^{-p} a_i^{\rho_i(p,a)})
$$
for all $a \in \Sigma_i$ and $p \in Q_i$. 
Thus, $\mathcal{T}_{i,q}$ is obtained from $\mathcal{A}_{i}$ by replacing every transition $p \xrightarrow{a} q$ of 
$\mathcal{A}_{i}$ by
\[
p \xrightarrow{a/a_i^{-p} a_i^q} q 
\]
and making $q$ the intial state of $\mathcal{T}_{i,q}$.
Let $f_{i,q} = f_{\mathcal{T}_{i,q}} : \Sigma_i^* \to \{a_i,a_i^{-1}\}^*$ be the 
function computed by the sequential transducer $\mathcal{T}_{i,q}$.

For a tuple $\bar{q} = (q_1, \ldots, q_c) \in \prod_{i \in [1,c]} Q_{i}$
of states from the semiPFAs $\mathcal{A}_{i}$ we define the sequential transducer $\mathcal{T}_{\bar{q}}$ by taking the direct product
of the sequential transducers $\mathcal{T}_{i,q_i}$ ($i \in [1,c]$). Formally, it is defined as follows:
$$
\mathcal{T}_{\bar{q}} = \bigg(\prod_{i \in [1,c]} Q_{i}, \Sigma, \Delta, \bar{q}, \delta\bigg)
$$
where for every $i \in [1,c]$, $a \in \Sigma_i$, and $(p_1, \ldots, p_c) \in \prod_{i \in [1,c]} Q_{i}$ we have
$$
\delta( (p_1, \ldots, p_c), a) = \big( (p_1, \ldots, p_{i-1}, \rho_i(p_i,a), p_{i+1},\ldots,p_c), a_i^{-p_i} a_i^{\rho_i(p_i,a)}\big).
$$
Let $f_{\bar{q}} = f_{\mathcal{T}_{\bar{q}}} : \Sigma^* \to \Delta^*$ be the function computed by $\mathcal{T}_{\bar{q}}$
(note that it is not a homomorphism).
Let 
\begin{equation} \label{def-m-graph-prod}
m = 2 \cdot n \cdot \max\{ |Q_{i}| \colon i \in [1,c]\} \leq n \cdot 2^{1+ \max\{  s(\mathcal{R}_i,n) \colon i \in [1,c]\} }.
\end{equation}
Note that $|f_{\bar{q}}(w)| \leq m$ if $|w| \leq n$ (we consider $a_i^{-p_i} a_i^{\rho_i(p_i,a)}$ as a word of length at most $2 |Q_i|$ 
and do not replace it by $a_i^{-p_i + \rho_i(p_i,a)}$).

Our distinguisher for $M$ and input length $n$ will use the semiPFA
$\mathcal{B}_m$ for the graph group $H$. States of $\mathcal{B}_m$ can be stored with 
$\mathcal{O}(\log m) \leq \mathcal{O}(\log n + \max\{  s(\mathcal{R}_i,n) \colon i \in [1,c]\})$ bits.
Basically, for an input word $w \in \Sigma^{\le n}$ the algorithm  simulates the semiPFA $\mathcal{A}_{i}$ ($i \in [1,c]$)
on the projections $w_i = \pi_{i}(w)$ and feeds the word $f_{\bar{q}}(w)$ into the semiPFA $\mathcal{B}_m$. Here, the
state tuple $\bar{q}$ is randomly guessed in the beginning according to the distributions $\iota_{i}$. 
Algorithm~\ref{algo-graph-product} shows the pseudocode of the distinguisher. It 
stores at most $\sum_{i=1}^c s(\mathcal{R}_i,n) + \mathcal{O}(\max\{  s(\mathcal{R}_i,n) \colon i \in [1,c]\} + \log n)$ bits.

\begin{algorithm}[t]
\SetKwComment{Comment}{(}{)}
\SetKwInput{KwGlobal}{global variables}
\SetKwInput{KwInit}{initialization}
\SetKwInput{KwNext}{next input letter}
\BlankLine
\KwGlobal{$q_i  \in Q_{i}$ for all $i \in [1,c]$,  $r \in R_m$} 
\BlankLine
\KwInit{}
guess $q_i \in Q_{i} = [1,|Q_{i}|]$ for all $i \in [1,c]$ according to the input distribution $\iota_{i}$ of $\mathcal{A}_{i}$ \\
guess $r \in R_m$ according to the input distribution $\lambda_m$ of $\mathcal{B}_m$ \\
\BlankLine
\KwNext{$a \in \Sigma$}
\If{$a \in \Sigma_i$}
  {$r  \coloneqq \sigma_m(r, a_i^{-q_i} a_i^{\rho_{i}(q_i,a)})$  \\
   $q_i  \coloneqq \rho_{i}(q_i,a)$}
\caption{A $(\zeta_0,\zeta_1)$-distinguisher for a graph product of left-cancellative monoids.  \label{algo-graph-product}}
\end{algorithm}

Before we analyze the error probability of the algorithm we need some preparations. For $i \in [1,c]$ 
and a word $w \in \Sigma^*$, let $\mathcal{P}_i(w)  \coloneqq \mathcal{P}(\pi_i(w))$ be the set of all prefixes of the projection $\pi_i(w)$.
Assume that  $y \in \Sigma_i^+$ is a block of $w$ and write $w = xyz$. We then have
$f_{\bar{q}}(w) = f_{\bar{q}}(x) f_{\bar{r}}(y) f_{\bar{s}}(z)$, where $\delta(\bar{q}, x) = (\bar{r}, f_{\bar{q}}(x))$ and
$\delta(\bar{r}, y) = (\bar{s}, f_{\bar{r}}(y))$. The word $f_{\bar{r}}(y)$ is also a block of 
$f_{\bar{q}}$ (for this it is important that every $\mathcal{T}_{j,q}$ translates non-empty words into non-empty
words). Since $y \in \Sigma_i^+$ we have
$r_j = s_j$ for all $j \in [1,c] \setminus \{i\}$ and $f_{\bar{r}}(y) = f_{i,r_i}(y)$.
In addition, the definition of the sequential transducer $\mathcal{T}_{i,r_i}$ implies that  
$f_{\bar{r}}(y) \equiv_{\langle a_i \rangle} a_i^{-r_i+s_i}$.

Consider now two input words $u,v \in \Sigma^{\le n}$.
Let $\mathcal{S}_i = \mathcal{P}_i(u)  \cup \mathcal{P}_i(v)$ and $n_i = |\mathcal{S}_i|$ for $i \in [1,c]$.
By Lemma~\ref{lemma-injective} we have for all $i \in [1,c]$:
\begin{eqnarray}
\Prob_{q \sim \iota_i}[\equiv_{q} \text{refines} \equiv_{M_i} \text{on } \mathcal{S}_i] & \geq & 1-\epsilon_0(n) \binom{n_i}{2} \ge
1- \epsilon_0(n) n_i^2/2,  \label{graph-prod-refine-1} \\
\Prob_{q \sim \iota_i}[\equiv_{M_i} \text{refines} \equiv_{q}\text{on } \mathcal{S}_i] & \geq & 1-\epsilon_1(n) \binom{n_i}{2} \ge
1-  \epsilon_1(n) n_i^2/2 . \label{graph-prod-refine-2}
\end{eqnarray}

\begin{claim} \label{claim-GP1}
Assume that $\bar{q} = (q_1, \ldots, q_c)$ is such that 
$\equiv_{M_i}$ refines $\equiv_{q_i}$ on $\mathcal{S}_i$  for every $i \in [1,c]$.
If $u \to_{\mathit{sd}}^* u'$ and $v \to_{\mathit{sd}}^* v'$, then $f_{\bar{q}}(u) \to_{\mathit{sd}}^*  f_{\bar{q}}(u')$,
$f_{\bar{q}}(v) \to_{\mathit{sd}}^*  f_{\bar{q}}(v')$
 and $\equiv_{M_i}$ refines $\equiv_{q_i}$ on $\mathcal{P}_i(u') \cup \mathcal{P}_i(v')$ for every $i \in [1,c]$.
 \end{claim}

\medskip
\noindent
\emph{Proof of Claim~\ref{claim-GP1}.} It suffices to show the following: If $u \to_{\mathit{sd}} u'$ holds, 
then $f_{\bar{q}}(u) \to_{\mathit{sd}}  f_{\bar{q}}(u')$
 and $\equiv_{M_i}$ refines $\equiv_{q_i}$ on $\mathcal{P}_i(u') \cup \mathcal{P}_i(v)$ for every $i \in [1,c]$.
 From this (and the symmetric statement where $v \to_{\mathit{sd}}  v'$ and $u=u'$)
 we obtain the general statement by induction on the number of $\to_{\mathit{sd}}$-steps.
 We distinguish two cases.
 
 \medskip
 \noindent
\emph{Case 1.} $u \to_s u'$. We must have $u = x y_1 y_2 z$ and $u' = x y_2 y_1 z$ for 
blocks $y_1, y_2$ such that $y_1 \in \Sigma_i^+$, $y_2 \in \Sigma_j^+$ and $(i,j) \in I$ (in particular $i \neq j$).
We obtain 
\begin{eqnarray*}
f_{\bar{q}}(u) & = & f_{\bar{q}}(x) f_{\bar{p}}(y_1) f_{\bar{r}}(y_2) f_{\bar{s}}(z) \text{ and } \\
f_{\bar{q}}(u') & = & f_{\bar{q}}(x) f_{\bar{p}}(y_2) f_{\bar{r}'}(y_1) f_{\bar{s}}(z),
\end{eqnarray*}
where $\delta(\bar{q}, x) = (\bar{p}, f_{\bar{q}}(x))$,
$\delta(\bar{p}, y_1) = (\bar{r}, f_{\bar{p}}(y_1))$, 
$\delta(\bar{r}, y_2) = (\bar{s}, f_{\bar{r}}(y_2))$, 
$\delta(\bar{p}, y_2) = (\bar{r}', f_{\bar{p}}(y_2))$, and
$\delta(\bar{r}', y_1) = (\bar{s}, f_{\bar{r}'}(y_1))$.
If we write $\bar{p} = (p_1, \ldots, p_c)$, then there are states $r_i \in Q_i$ and 
$r_j \in Q_j$ such that 
\begin{eqnarray}
\bar{r} &=& (p_1, \ldots, p_{i-1}, r_i, p_{i+1}, \ldots, p_c), \label{tuple-r} \\
\bar{r}' &=& (p_1, \ldots, p_{j-1}, r_j, p_{j+1}, \ldots, p_c), \text{ and }  \label{tuple-r'} \\
\bar{s} &=& (p_1, \ldots, p_{i-1}, r_i, p_{i+1}, \ldots, p_{j-1}, r_j, p_{j+1}, \ldots, p_c) \label{tuple-s}
\end{eqnarray}
(we assume w.l.o.g.~that $i < j$).
Moreover, $f_{\bar{p}}(y_1) = f_{i,p_i}(y_1) = f_{\bar{r}'}(y_1) \in \{a_i,a_i^{-1}\}^+$ and
$f_{\bar{r}}(y_2) = f_{j,p_j}(y_2) = f_{\bar{p}}(y_2) \in \{a_j,a_j^{-1}\}^+$.
Thus, we have 
\begin{eqnarray*}
f_{\bar{q}}(u) & = & f_{\bar{q}}(x) f_{\bar{p}}(y_1) f_{\bar{r}}(y_2) f_{\bar{s}}(z) \\
& = &  f_{\bar{q}}(x) f_{i,p_i}(y_1) f_{j,p_j}(y_2) f_{\bar{s}}(z) \\
& \to_s  & f_{\bar{q}}(x) f_{j,p_j}(y_2) f_{i,p_i}(y_1) f_{\bar{s}}(z) \\
& = & f_{\bar{q}}(x) f_{\bar{p}}(y_2) f_{\bar{r}'}(y_1) f_{\bar{s}}(z) \\
& = & f_{\bar{q}}(u') .
\end{eqnarray*}
Moreover, since $\mathcal{P}_k(u') = \mathcal{P}_k(u)$ and
$\equiv_{M_k}$ refines $\equiv_{q_k}$ on $\mathcal{S}_k$ for all $k \in [1,c]$, it follows that
$\equiv_{M_k}$ refines $\equiv_{q_k}$ on $\mathcal{P}_k(u') \cup \mathcal{P}_k(v)$ for all $k \in [1,c]$.

 \medskip
 \noindent
\emph{Case 2.} $u \to_{d} u'$. Then we obtain a factorization $u = xyz$, where $y \in \Sigma_i^+$ is a block, $y \equiv_{M_i} \varepsilon$, and
$u' = xz$. We obtain a factorization 
$$f_{\bar{q}}(u) = f_{\bar{q}}(x) f_{\bar{r}}(y) f_{\bar{s}}(z),$$ 
where $\delta(\bar{q}, x) = (\bar{r}, f_{\bar{q}}(x))$ and
$\delta(\bar{r}, y) = (\bar{s}, f_{\bar{r}}(y))$.  The word $f_{\bar{r}}(y)$ is a block of $f_{\bar{q}}(u)$.
For the projection $\pi_{i}(u)$ we have $\pi_{i}(u) = \pi_{i}(x) y \pi_{i}(z)$.
Since $\equiv_{M_i}$ refines $\equiv_{q_i}$ on $\mathcal{S}_i$ and $\pi_{i}(x) \equiv_{M_i} \pi_{i}(x) y$, 
we obtain $\pi_{i}(x) \equiv_{q_i} \pi_{i}(x) y$. 
Since $r_i$ (resp., $s_i$) is the state reached from $q_i$ by the automaton $\mathcal{A}_i$ after reading 
$\pi_{i}(x)$ (resp., $\pi_{i}(x) y$), we obtain $r_i = s_i$ and hence $\bar{r} = \bar{s}$.
This implies 
\[
f_{\bar{r}}(y) = f_{i,r_i}(y) \equiv_{\langle a_i \rangle} a_i^{-r_i} a_i^{s_i} \equiv_{\langle a_i \rangle} \varepsilon.
\]
Moreover, we have 
\[
f_{\bar{q}}(u) =  f_{\bar{q}}(x) f_{\bar{r}}(y) f_{\bar{s}}(z)  \to_d f_{\bar{q}}(x)f_{\bar{s}}(z) =   
f_{\bar{q}}(x) f_{\bar{r}}(z) = f_{\bar{q}}(xz)  = f_{\bar{q}}(u').
\]
It remains to show that $\equiv_{M_j}$ refines $\equiv_{q_j}$ on $\mathcal{P}_j(u') \cup \mathcal{P}_j(v)$ for every $j \in [1,c]$.
For $j \neq i$ this is clear since $\mathcal{P}_j(u') \cup \mathcal{P}_j(v) = \mathcal{S}_j$. For $j = i$ we can use
Lemma~\ref{lemma-injective-reduce} for the words $\pi_i(u) = \pi_{i}(x) y \pi_{i}(z)$ and $\pi_i(v)$.
This concludes the proof of Claim~\ref{claim-GP1}.

\begin{claim} \label{claim-GP2}
Assume that $\bar{q} = (q_1, \ldots, q_c)$ is such that 
$\equiv_{q_i}$ refines $\equiv_{M_i}$ on $\mathcal{S}_i$  for every $i \in [1,c]$.
If $f_{\bar{q}}(u) \to_{\mathit{sd}}^*  \tilde{u}$ and
$f_{\bar{q}}(v) \to_{\mathit{sd}}^*  \tilde{v}$, then there are $u', v' \in \Sigma^*$ such that 
$u \to_{\mathit{sd}}^* u'$,  $v \to_{\mathit{sd}}^* v'$, $f_{\bar{q}}(u') = \tilde{u}$, $f_{\bar{q}}(v') = \tilde{v}$
and $\equiv_{q_i}$ refines $\equiv_{M_i}$ on $\mathcal{P}_i(u') \cup \mathcal{P}_i(v')$ for every $i \in [1,c]$.
\end{claim}

\medskip
\noindent
\emph{Proof of Claim~\ref{claim-GP2}.} The proof is very similar to the proof of Claim~\ref{claim-GP1}.
As in the proof of Claim~\ref{claim-GP1}, it suffices to consider the case where $f_{\bar{q}}(u) \to_{\mathit{sd}}  \tilde{u}$ and
$\tilde{v} = f_{\bar{q}}(v)$. 

 \medskip
 \noindent
\emph{Case 1.} $f_{\bar{q}}(u) \to_s \tilde{u}$. Since the blocks of $u$ are translated into the blocks
of $f_{\bar{q}}(u)$ by the sequential transducer $\mathcal{T}_{\bar{q}}$, we obtain a factorization
$u = x y_1 y_2 z$ for 
blocks $y_1 \in \Sigma_i^+, y_2 \in \Sigma_j^+$ of $u$ such that $(i,j) \in I$
(in particular $i \neq j$) and
\begin{eqnarray*}
f_{\bar{q}}(u) & = & f_{\bar{q}}(x) f_{\bar{p}}(y_1) f_{\bar{r}}(y_2) f_{\bar{s}}(z), \\
\tilde{u} & = & f_{\bar{q}}(x) f_{\bar{r}}(y_2) f_{\bar{p}}(y_1) f_{\bar{s}}(z).
\end{eqnarray*}
Here, the state tuples $\bar{p} = (p_1,\ldots,p_c)$, $\bar{r}$, and $\bar{s}$ are as in the proof of Claim~\ref{claim-GP1}, see in particular
\eqref{tuple-r} and \eqref{tuple-s}.
We can then define the tuple $\bar{r}'$ as in \eqref{tuple-r'} and get
\begin{alignat*}{3}
f_{\bar{p}}(y_1) & = f_{i,p_i}(y_1) & \ = \ & f_{\bar{r}'}(y_1) & \ \in \ & \{a_i,a_i^{-1}\}^+ \text{ and } \\
f_{\bar{r}}(y_2)  & = f_{j,p_j}(y_2) & \ = \ & f_{\bar{p}}(y_2) & \ \in \ & \{a_j,a_j^{-1}\}^+.
\end{alignat*}
We thus have 
$$
\tilde{u} = f_{\bar{q}}(x) f_{\bar{r}}(y_2) f_{\bar{p}}(y_1) f_{\bar{s}}(z) =
f_{\bar{q}}(x) f_{\bar{p}}(y_2) f_{\bar{r}'}(y_1) f_{\bar{s}}(z) = f_{\bar{q}}(xy_2y_1z).
$$
Clearly, we also have $u = x y_1 y_2 z \to_s xy_2y_1z$. So, we can set $u' = xy_2y_1z$.
Since $\mathcal{P}_k(u') = \mathcal{P}_k(u)$ for all $k \in [1,c]$, it follows that
$\equiv_{q_k}$ refines $\equiv_{M_k}$ on $\mathcal{P}_k(u') \cup \mathcal{P}_k(v)$ for all $k \in [1,c]$.

\medskip
 \noindent
\emph{Case 2.} $f_{\bar{q}}(u) \to_d \tilde{u}$.
Then we obtain a factorization $u= xyz$, where $y \in \Sigma_i^+$ is a block of $u$, 
\begin{eqnarray*}
f_{\bar{q}}(u)  & = & f_{\bar{q}}(x) f_{\bar{r}}(y) f_{\bar{s}}(z), \text{ and }\\
\tilde{u}  & = &  f_{\bar{q}}(x)  f_{\bar{s}}(z).  
\end{eqnarray*}
The state tuples $\bar{r}$ and $\bar{s}$ are such that
$\delta(\bar{q}, x) = (\bar{r}, f_{\bar{q}}(x))$ and
$\delta(\bar{r}, y) = (\bar{s}, f_{\bar{r}}(y))$.  Moreover, the word $f_{\bar{r}}(y)$ is a block of $f_{\bar{q}}(u)$ with
$$
a_i^{-r_i} a_i^{s_i} \equiv_{\langle a_i \rangle} f_{\bar{r}}(y) \equiv_{\langle a_i \rangle} \varepsilon.
$$
This implies that $r_i = s_i$ and hence $\bar{r} = \bar{s}$.
We therefore have 
$$\rho_i(q_i, \pi_{i}(x)) = r_i = s_i = \rho_i(q_i, \pi_{i}(x)y).$$ 
Since  $\equiv_{q_i}$ refines $\equiv_{M_i}$ on $\mathcal{S}_i$ and $\pi_{i}(x), \pi_{i}(x)y \in \mathcal{S}_i$,
 we get $\pi_{i}(x) \equiv_{M_i} \pi_{i}(x) y$, i.e., $y \equiv_{M_i} \varepsilon$ (here we use the assumption that $M_i$ is left-cancellative).
If we set $u' = xz$ we get $u \to_d u'$ and 
$$
\tilde{u} = f_{\bar{q}}(x) f_{\bar{s}}(z) = f_{\bar{q}}(x) f_{\bar{r}}(z) = f_{\bar{q}}(xz) = f_{\bar{q}}(u').
$$ 
It remains to show that $\equiv_{q_j}$ refines $\equiv_{M_j}$ on $\mathcal{P}_j(u') \cup \mathcal{P}_j(v)$ for every $j \in [1,c]$.
For $j \neq i$ this is clear since $\mathcal{P}_j(u') \cup \mathcal{P}_j(v) = \mathcal{S}_j$. For $j = i$ we can use
Lemma~\ref{lemma-injective-reduce2} for the words $\pi_i(u) = \pi_{i}(x) y \pi_{i}(z)$ and $\pi_i(v)$.
This concludes the proof of Claim~\ref{claim-GP2}.

\medskip
 \noindent
We now estimate the error probability of Algorithm~\ref{algo-graph-product} for the input words $u$ and $v$. There are two cases to consider:

\medskip
\noindent
\emph{Case 1.} $u \equiv_M v$. 
We will show that
Algorithm~\ref{algo-graph-product} reaches with probability at least $1- 2 \epsilon_1(n)(n+c)^2$
the same state when running on $u$ and $v$, respectively. 
For this,
assume that the randomly selected initial states $q_i \in Q_{i}$ are such that 
$\equiv_{M_i}$ refines $\equiv_{q_i}$ on $\mathcal{S}_i$ for all $i \in [1,c]$. 
By \eqref{graph-prod-refine-2} this happens with probability at least 
\[
1 - \epsilon_1(n) \sum_{i \in [1,c]} n_i^2/2 \; \geq \; 1 - \frac{\epsilon_1(n)}{2} \bigg(\sum_{i \in [1,c]} n_i\bigg)^2 \geq  1 - 2 \epsilon_1(n) (n+c)^2
\]
(note that $\sum_{i \in [1,c]} n_i$ is at most $2n+c$, where the additive term $c$ arises, since the empty prefix has to be counted $c$
times).

First note that $u \equiv_M v$ implies $\pi_{i}(u) \equiv_{M_i} \pi_{i}(v)$ for all $i \in [1,c]$.
Since $\equiv_{M_i}$ refines $\equiv_{q_i}$ on $\mathcal{S}_i$, we obtain 
$\rho_i(q_i,\pi_{i}(u)) = \rho_i(q_i,\pi_{i}(v))$. It remains to show that after reading $u$ and $v$,
also the states of $\mathcal{B}_m$ are the same. For this we show that
$f_{\bar{q}}(u) \equiv_H f_{\bar{q}}(v)$ in the graph group $H$.

From Lemma~\ref{lemma-rewrite-GP-2} it follows that there are reduced words $u', v' \in \Sigma^{\le n}$ such that
$u \to_{\mathit{sd}}^* u'$, $v \to_{\mathit{sd}}^* v'$, and $u' \leftrightarrow_{\mathit{r}}^* v'$.
Claim~\ref{claim-GP1} implies $f_{\bar{q}}(u) \to^*_{\mathit{sd}}  f_{\bar{q}}(u')$,
$f_{\bar{q}}(v) \to^*_{\mathit{sd}}  f_{\bar{q}}(v')$,
and $\equiv_{M_i}$ refines $\equiv_{q_i}$ on $\mathcal{P}_i(u') \cup \mathcal{P}_i(v')$ for every $i \in [1,c]$.
Since $u' \leftrightarrow_{\mathit{r}}^* v'$ we can write the block factorizations of $u'$ and $v'$ as
$u' = u_1 u_2 \cdots u_l$ and $v' = v_1 v_2 \cdots v_l$ with $u_i, v_i \in \Sigma_{j_i}^+$ for some 
$j_i \in [1,c]$ and $u_i \equiv_{M_{j_i}} v_i$ for all $i \in [1,l]$. 
The block factorizations of $f_{\bar{q}}(u')$ and $f_{\bar{q}}(v')$ can be written as
$f_{\bar{q}}(u') = \tilde{u}_1 \tilde{u}_2 \cdots \tilde{u}_l$ and $f_{\bar{q}}(v')  = \tilde{v}_1 \tilde{v}_2 \cdots \tilde{v}_l$
with $\tilde{u}_i, \tilde{v}_i \in \{a_{j_i}, a^{-1}_{j_i}\}^+$. 

We claim that $\tilde{u}_i \equiv_{\langle a_{j_i} \rangle} \tilde{v}_i$ for all $i \in [1,l]$, which implies
$f_{\bar{q}}(u')  \equiv_H f_{\bar{q}}(v')$. Since $u_i \equiv_{M_{j_i}} v_i$ for all $i \in [1,l]$,
we get the following for all $j \in [1,c]$ and $1 \leq k_1 < k_2 < \cdots < k_e \leq l$: 
$u''  \coloneqq u_{k_1} u_{k_2} \cdots u_{k_e} \in \Sigma_j^*$ is a pure prefix of $u'$ if and only if
$v''  \coloneqq v_{k_1} v_{k_2} \cdots v_{k_e} \in \Sigma_j^*$ is a pure prefix of $v'$. Moreover, if these two equivalent conditions hold, then
$u'' \equiv_{M_j} v''$.
Since $\equiv_{M_j}$ refines $\equiv_{q_j}$ on $\mathcal{P}_j(u') \cup \mathcal{P}_j(v')$ and $u'', v'' \in \mathcal{P}_j(u') \cup \mathcal{P}_j(v')$,
we obtain $\rho_j(q_j, u'') = \rho_j(q_j, v'')$. This implies $\tilde{u}_i \equiv_{\langle a_{j_i} \rangle} \tilde{v}_i$ for all $i \in [1,l]$
and hence $f_{\bar{q}}(u') \equiv_H f_{\bar{q}}(v')$. 
From this, we finally get $f_{\bar{q}}(u) \equiv_H f_{\bar{q}}(u') \equiv_H f_{\bar{q}}(v') \equiv_H f_{\bar{q}}(v)$.

Recall that $f_{\bar{q}}(u)$ (resp., $f_{\bar{q}}(v)$) is the word fed into the semiPFA $\mathcal{B}_m$ on input $u$ (resp., $v$).
Since $(\mathcal{B}_n)_{n \ge 0}$ is a $(1/n^d,0)$-distinguisher for the graph group $H$, it follows that $f_{\bar{q}}(u)$ and $f_{\bar{q}}(v)$ lead in
$\mathcal{B}_m$ with probability one to the same state. Hence,
Algorithm~\ref{algo-graph-product} reaches with probability at least $1- 2\epsilon_1(n)(n+c)^2$
the same state when running on $u$ and $v$, respectively.

\medskip
\noindent
\emph{Case 2.} $u \not\equiv_M v$. We will show that
Algorithm~\ref{algo-graph-product} reaches with probability at least $1 - (2 \epsilon_0(n) (n+c)^2 + 1/n^d)$
different states when running on $u$ and $v$, respectively.
To show this, assume that the randomly selected initial states $q_i \in Q_{i}$ are such that 
$\equiv_{q_i}$ refines $\equiv_{M_i}$ on $\mathcal{S}_i$ for all $i \in [1,c]$. 
By \eqref{graph-prod-refine-1} this happens with probability at least 
\[
1 - \epsilon_0(n) \sum_{i \in [1,c]} n_i^2/2 \; \geq \; 1 - \frac{\epsilon_0(n)}{2} \bigg(\sum_{i \in [1,c]} n_i\bigg)^2 \geq  1 - 2 \epsilon_0(n) (n+c)^2 .
\]
We claim that $f_{\bar{q}}(u) \not\equiv_H f_{\bar{q}}(v)$ holds, where $\bar{q}=(q_1,\ldots,q_c)$.
 In order to get a contradiction, assume that
$f_{\bar{q}}(u) \equiv_H f_{\bar{q}}(v)$. From Lemma~\ref{lemma-rewrite-GP-2} it follows that there are reduced words
 $\tilde{u}, \tilde{v} \in \Delta^*$ such that $f_{\bar{q}}(u) \to_{\mathit{sd}}^*  \tilde{u}$,
$f_{\bar{q}}(v) \to_{\mathit{sd}}  \tilde{v}$ and $\tilde{u} \leftrightarrow_r^* \tilde{v}$.
Claim~\ref{claim-GP2} implies that  there exist $u', v' \in \Sigma^*$ such that 
$u \to_{\mathit{sd}}^* u'$,  $v \to_{\mathit{sd}}^* v'$, $f_{\bar{q}}(u') = \tilde{u}$, $f_{\bar{q}}(v') = \tilde{v}$
and $\equiv_{q_i}$ refines $\equiv_{M_i}$ on $\mathcal{P}_i(u') \cup \mathcal{P}_i(v')$ for every $i \in [1,c]$.

Since $f_{\bar{q}}(u') = \tilde{u}\leftrightarrow_{\mathit{r}}^* \tilde{v} = f_{\bar{q}}(v')$ we can write the block factorizations  of
$f_{\bar{q}}(u')$ and $f_{\bar{q}}(v')$ as
$f_{\bar{q}}(u') = \tilde{u}_1 \tilde{u}_2 \cdots \tilde{u}_l$ and $f_{\bar{q}}(v') = \tilde{v}_1 \tilde{v}_2 \cdots \tilde{v}_l$ with
$\tilde{u}_i, \tilde{v}_i \in \{a_{j_i}, a_{j_i}^{-1}\}^+$ for some 
$j_i \in [1,c]$ and $\tilde{u}_i \equiv_{\langle a_{j_i}\rangle} \tilde{v}_i$ for all $i \in [1,l]$. 
Clearly, the block factorizations of $u'$ and $v'$ can then be written as
$u' = u_1 u_2 \cdots u_l$ and $v' = v_1 v_2 \cdots v_l$, where the block $u_i \in \Sigma^+_{j_i}$ (resp., $v_i\in \Sigma^+_{j_i}$) is translated
into the block $\tilde{u}_i$ (resp., $\tilde{v}_i$) by the sequential transducer $\mathcal{T}_{\overline{q}}$.

We claim that $u_i \equiv_{M_{j_i}} v_i$ for all $i \in [1,l]$. Since $\tilde{u}_i \equiv_{\langle a_{j_i}\rangle} \tilde{v}_i$ for all $i \in [1,l]$ we have the following
for all $j \in [1,c]$ and $1 \leq k_1 < k_2 < \cdots < k_e \leq l$:
$\tilde{u}''  \coloneqq \tilde{u}_{k_1} \tilde{u}_{k_2} \cdots \tilde{u}_{k_e} \in \{a_j,a_j^{-1}\}^*$ is a pure prefix of $f_{\bar{q}}(u')$ if and only if
$\tilde{v}''  \coloneqq \tilde{v}_{k_1} \tilde{v}_{k_2} \cdots \tilde{v}_{k_e} \in \{a_j,a_j^{-1}\}^*$ is a pure prefix of $f_{\bar{q}}(v')$.
Moreover, if these two equivalent conditions hold, then
 $\tilde{u}'' \equiv_{\langle a_j\rangle} \tilde{v}''$. Let $p_j = \rho_j(q_j, u_{k_1} u_{k_2} \cdots u_{k_e})$
and $r_j = \rho_j(q_j, v_{k_1} v_{k_2} \cdots v_{k_e})$. We then have 
$$
a_j^{-q_j} a_j^{p_j} \equiv_{\langle a_j\rangle} \tilde{u}'' \equiv_{\langle a_j\rangle} \tilde{v}'' \equiv_{\langle a_j\rangle}
a_j^{-q_j} a_j^{r_j},
$$
i.e., $p_j = r_j$.
Since $\equiv_{q_j}$ refines $\equiv_{M_j}$ on $\mathcal{P}_j(u') \cup \mathcal{P}_j(v')$ and $u_{k_1} u_{k_2} \cdots u_{k_e}$ as well as 
$v_{k_1} v_{k_2} \cdots v_{k_e}$ belong to $\mathcal{P}_j(u') \cup \mathcal{P}_j(v')$,
we obtain $u_{k_1} u_{k_2} \cdots u_{k_e} \equiv_{M_j} v_{k_1} v_{k_2} \cdots v_{k_e}$.
For the same reason, we also get $u_{k_1} u_{k_2} \cdots u_{k_{e-1}} \equiv_{M_j} v_{k_1} v_{k_2} \cdots v_{k_{e-1}}$.
Due to the left-cancellativity of the monoid $M_j$ we obtain $u_{k_{e}} \equiv_{M_j} v_{k_{e}}$.
Hence, we get
$u_i \equiv_{M_{j_i}} v_i$ for all $i \in [1,l]$, which implies $u'  \equiv_M v'$.
Finally, we get $u \equiv_M u' \equiv_M v' \equiv_M v$, which is a contradiction.
Hence, we must have $f_{\bar{q}}(u) \not\equiv_H f_{\bar{q}}(v)$.

Since the algorithm feeds $f_{\bar{q}}(u)$ (resp., $f_{\bar{q}}(v)$) into the semiPFA $\mathcal{B}_m$, the latter reaches different
states with probability at least $1 - 1/m^d \geq 1 - 1/n^d$
 (under the assumption that $\equiv_{q_i}$ refines $\equiv_{M_i}$ on $\mathcal{S}_i$ for all $i \in [1,c]$).
 Hence, the probability that Algorithm~\ref{algo-graph-product} reaches different states when running on $u$ and $v$
 is at least $(1 - 2 \epsilon_0(n) (n+c)^2) \cdot (1 - 1/n^d) \geq 1 - (2 \epsilon_0(n) (n+c)^2 + 1/n^d)$. 
\end{proof}

\begin{oproblem}
Theorem~\ref{thm-GP} only makes a non-trivial statement if $\zeta_0(n)$ and $\zeta_1(n)$ are strictly
smaller than $1$. Note that the error probabilities $\epsilon_0(n)$ and $\epsilon_1(n)$ are multiplied with $2 (n+c)^2$.
It would be interesting to see whether this multiplicative factor can be avoided or replaced by a smaller factor. But also
in the present form, Theorem~\ref{thm-GP} leads to useful applications. For instance, if all monoids $M_1, \ldots, M_c$ 
are linear then, by Theorem~\ref{thm-lin}, we can set $\epsilon_1(n) = 0$ and $\epsilon_0(n) = 1/n^c$ for any constant $c>0$ and still obtain
an $\mathcal{O}(\log n)$-space distinguisher for the graph product $M$. Similar remarks apply to the next theorem and the
transfer theorems for wreath products
that we show in the next section.
\end{oproblem}

 \begin{theorem}  \label{thm-GP-S}
 Let $M = \mathsf{GP}(M_1, \ldots, M_c, I)$ be a graph product of f.g.~monoids $M_1, \ldots, M_c$
such that $M_i \setminus \{1\}$ is an ideal of $M_i$ for every $i \in [1,c]$
and let 
$\mathcal{R}_i =(\mathcal{A}_{i,n})_{n \ge 0}$ be an $(\epsilon_0,\epsilon_1)$-distinguisher for $M_i$. Let $d \geq 1$ and define 
\begin{eqnarray*}
\zeta_0(n) &=& 2 \epsilon_0(n) n^2 + 1/n^d, \\
\zeta_1(n) &=& \epsilon_1(n) n.
\end{eqnarray*} 
Then, there is a $(\zeta_0,\zeta_1)$-distinguisher for the graph product $M$
with space complexity $\mathcal{O}(\sum_{i=1}^c s(\mathcal{R}_i,n) + \log n)$.
\end{theorem}

\begin{proof}
We can assume that the monoid generating set $\Sigma_i$ of $M_i$ does not contain the neutral element of $M_i$. Hence, a word $u \in \Sigma_i^+$ does not evaluate to the neutral element, which means
that the relation $\to_d$ is empty. Hence, Lemma~\ref{lemma-rewrite-GP-2} says that $u \equiv_M v$ if and only if 
there are reduced words $u', v'$ such that $u \to_s^* u'$, $v \to_s^* v'$ and $u' \leftrightarrow_r^* v'$. 

Fix an input length $n$ and let
$\mathcal{A}_{i,n} = \mathcal{A}_{i} = (Q_{i}, \Sigma_i, \iota_{i}, \rho_{i})$, where w.l.o.g.~$Q_{i} = [1, |Q_{i}|]$.
First, we modify every $\mathcal{A}_{i}$ such that
every state $q \in Q_{i}$ with $\iota_{i}(q) > 0$ has no incoming transitions. To achieve this we introduce for every
$q \in Q_{i}$ with $\iota_{i}(q) > 0$ a copy $q'$ and set $\iota_{i}(q')  \coloneqq \iota_{i}(q)$, $\rho_{i}(q',a)  \coloneqq \rho_{i}(q,a)$ for every $a \in \Sigma_i$,
and  $\iota_{i}(q)  \coloneqq 0$. Since $w \not\equiv_{M_i} \varepsilon$ for every word $w \in \Sigma^+_i$, this modification
does not increase the error probabilities. Moreover the number of states only doubles.
For a state $q \in Q_i$, we use the equivalence relation $\equiv_{q} \ \,  \coloneqq \ \, \equiv_{\mathcal{A}_i,q}$.

As in the proof of Theorem~\ref{thm-GP},
we fix for every $i \in [1,c]$ a new symbol $a_i$ that generates an infinite cyclic group $\langle a_i \rangle \cong \mathbb{Z}$.
Let $\Delta = \{a_1, a_1^{-1}, \ldots, a_c, a_c^{-1}\}$ and consider the graph group $H = \mathsf{GP}(\langle a_1\rangle, \ldots, \langle a_c\rangle, I)$.
We take a $(1/m^d,0)$-distinguisher $(\mathcal{B}_m)_{m \ge 0}$ 
with space complexity $\mathcal{O}(\log m)$ for the graph group $H$. Let $\mathcal{B}_m = (R_m, \Delta, \lambda_m, \sigma_m)$.
Define $m$ as in \eqref{def-m-graph-prod}.

For a word $u \in \Sigma^*$  we define $\max(u) \subseteq [1,c]$ as follows.
For $i \in [1,c]$ we have $i \in \max(u)$ if and only if $u \to_s v a$ for some $v \in \Sigma^*$ and $a \in \Sigma_i$.
In other words, we can move in $u$ a block from $\Sigma^+_i$ to the right end of $u$ using swap operations.
If $u = u_1 u_2 \cdots u_k$ is the block factorization of $u$ then
we have $i \in \max(u)$ if and only if there is an $s \in [1,k]$ such that that $u_s \in \Sigma_i^+$ and for every $t \in [s+1,k]$
there is $j \in [1,c]$ with $(i,j) \in I$ and $u_t \in \Sigma_j^+$.
Note that $u \to_s v$ or $u \to_r v$ implies that $\max(u)=\max(v)$.

\begin{algorithm}[t]
\SetKwComment{Comment}{(}{)}
\SetKwInput{KwGlobal}{global variables}
\SetKwInput{KwInit}{initialization}
\SetKwInput{KwNext}{next input letter}
\BlankLine
\KwGlobal{$p_i, q_i  \in Q_{i}$ for all $i \in [1,c]$,  $r \in R_m$, $\max \subseteq [1,c]$} 
\BlankLine
\KwInit{}
\BlankLine
$\max  \coloneqq \emptyset$ \\
guess $p_i \in Q_{i} = [1,|Q_{i}|]$ for all $i \in [1,c]$ according to $\iota_{i}$ \\
$q_i  \coloneqq p_i$ for all $i \in [1,c]$ \\
guess $r \in R_m$ according to $\lambda_m$  \\
\BlankLine
\KwNext{$a \in \Sigma$}
\BlankLine
let $i \in [1,c]$ such that $a \in \Sigma_i$ \\
$\max  \coloneqq (\max \setminus \{ j \in \max \colon (i,j) \notin I \}) \cup \{i\}$ \\
 $r  \coloneqq \sigma_m(r, a_i^{-q_i} a_i^{\rho_{i}(q_i,a)})$  \\
 $q_i  \coloneqq \rho_{i}(q_i,a)$ \\
 \label{reset} $q_j  \coloneqq p_j$ for all $j \in [1,c] \setminus \max$ \\
\caption{A $(\zeta_0,\zeta_1)$-distinguisher for a graph product of monoids $M_i$ with $M_i \setminus \{1\}$ an ideal.
 \label{algo-graph-product-2}}
\end{algorithm}

Consider Algorithm~\ref{algo-graph-product-2} and an input word $u \in \Sigma^{\le n}$.
In the variable $\max$ the algorithm stores $\max(u)$ (lines 1 and 6) and for every $j \notin \max(u)$ the algorithm sets $q_j$ to the initially guessed
state $p_j$ of $\mathcal{A}_j$ (lines 3 and 9). The algorithm stores $\mathcal{O}(\sum_{i=1}^c s(\mathcal{R}_i,n) + \log n)$ many bits.

\begin{claim} \label{claim-GP}
Whenever $u \to_s u'$ for $u, u' \in \Sigma^{\le n}$ then after reading $u$ and $u'$ (starting with the same guesses in lines 2 and 4),
the program variables $\max$, $q_i$ $(i \in [1,c])$ and $r$
have the same values. Moreover the words $\tilde{u}, \tilde{u}' \in \Delta^*$ that are input into the semiPFA $\mathcal{B}_m$ (line 7)
when reading $u$ and $u'$, respectively, are in relation $\equiv_H$.
\end{claim}

\noindent
\emph{Proof of Claim~\ref{claim-GP}.} 
It suffices to consider the case, where $u = v u_1 u_2$ and $u' = v u_2 u_1$ for blocks $u_1 \in \Sigma_{i_1}^+$,
$u_2 \in \Sigma_{i_2}^+$ with $(i_1, i_2) \in I$. Clearly, $\max(u) = \max(u')$. 
Moreover, every $q_j$ for $j \in [1,c] \setminus \{i_1, i_2\}$ has the same value after reading $u$ and $u'$, respectively.
To see that this holds also for $q_{i_1}$ and $q_{i_2}$, it suffices by symmetry to consider only $q_{i_1}$.
Observe  that $i_1 \in \max(v u_2)$ if and only if $i_1 \in \max(v)$. This means that
while reading in $v u_2 u_1$ the symbols from the $u_2$, there will be no
new reset of $q_{i_1}$ in line~\ref{reset}. In other words, if reading the symbols
of $u_2$ causes a reset of $q_{i_1}$, then this has no effect since $q_{i_1}$ has already the value $p_{i_1}$
after reading $v$. This shows that $q_{i_1}$ has the same value after reading 
$v u_2 u_1$ and $v u_1 u_2$.

From the above consideration it follows directly that 
the words $\tilde{u}, \tilde{u}' \in \Delta^*$ are in relation $\equiv_H$.
Since $(\mathcal{B}_m)_{m \ge 0}$ is a $(1/m^d,0)$-distinguisher for $H$ and $|\tilde{u}|, |\tilde{u}'| \le m$,
it follows that after reading $u$ and $u'$, respectively, also the program variable $r$ (the current state of $\mathcal{B}_m$) has
the same value.
\qed

\medskip
\noindent
We now analyze the error probabilities of Algorithm~\ref{algo-graph-product-2}. 
Consider input words $u,v \in \Sigma^{\le n}$ and let $p_i$ ($i \in [1,c]$) be the randomly guessed states from line 2 and $r_0$ the randomly
guessed state from line 4.

\medskip
\noindent
\emph{Case 1.} $u \equiv_M v$: Then there exist reduced words $u', v' \in \Sigma^*$ such that $u \to^*_s u'$, $v \to_s^* v'$ and 
$u' \leftrightarrow_r^* v'$. We have $u', v' \in \Sigma^{\le n}$.
By Claim~\ref{claim-GP}, it suffices to show that $u'$ and $v'$ lead with probability at least $1 - \epsilon_1(n) n$
to the same program state in Algorithm~\ref{algo-graph-product-2}.
We clearly have $\max(u') = \max(v')$, hence the program variable $\max$ has (with probability one) the same value after reading $u'$ and $v'$, respectively. Moreover, for every $i \in [1,c] \setminus \max(u')$ we have $q_i = p_i$.

Let $u' = u_1 u_2 \cdots u_k$ and $v' = v_1 v_2 \cdots v_k$ be the block factorizations of $u'$ and $v'$, respectively. 
We have $k \leq n$. Let
$u_j, v_j \in \Sigma_{i_j}^+$. We then have $u_j \equiv_{M_{i_j}} v_j$ for all $j \in [1,k]$. 
Observe that since $u'$ is reduced, we have $i_j \notin \max(u_1 u_2 \cdots u_{j-1})$, which
means that the program variable $q_{i_j}$ is in the initial state $p_{i_j}$ after reading $u_1 u_2 \cdots u_{j-1}$.
The same is true for $v'$.

Since $u_j \equiv_{M_{i_j}} v_j$ for all $j \in [1,k]$, we have
$\rho_{i_j}(p_{i_j}, u_j) = \rho_{i_j}(p_{i_j}, v_j)$ for all $j \in [1,k]$
with probability at least $1 -  \epsilon_1(n) n$.
Assume that this holds and let $r_j  \coloneqq  \rho_{i_j}(p_{i_j}, u_j) = \rho_{i_j}(p_{i_j}, v_j)$ for all $j \in [1,k]$.
Then for all $i \in \max(u') = \max(v')$ the program variable $q_i$ has the same value after reading $u'$ and $v'$ respectively
(namely $r_j$ with $j$ maximal such that $i_j = i$). In addition,
the words $\tilde{u}, \tilde{v} \in \Delta^*$ that Algorithm~\ref{algo-graph-product-2} inputs into the semiPFA $\mathcal{B}_m$ 
 while reading $u'$ and $v'$, respectively, are in relation $\equiv_H$. More precisely, we have
 $\tilde{u} \equiv_H a_{i_1}^{-p_{i_1}+r_1}  a_{i_2}^{-p_{i_2}+r_2} \cdots a_{i_k}^{-p_{i_k}+r_k} \equiv_H \tilde{v}$).
 Since $\tilde{u}, \tilde{v} \in \Delta^{\le m}$, we
 obtain $\sigma_m(r_0, \tilde{u}) = \sigma_m(r_0, \tilde{v})$ with probability one. Hence, also the program variable $r$ has the same value after
 reading $u'$ and $v'$, respectively (under the assumption that $\rho_{i_j}(p_{i_j}, u_j) = \rho_{i_j}(p_{i_j}, v_j)$ for all $j \in [1,k]$).
 
 \medskip
\noindent
\emph{Case 2.} $u \not\equiv_M v$: We will show that with probability at least $1-(2\varepsilon_0(n)n^2 + 1/n^d)$
the programm variable $r$ reaches different values after reading $u$ and $v$, respectively.
By Claim~\ref{claim-GP} we can assume that $u$ and $v$ are reduced.
Let $u = u_1 u_2 \cdots u_k$ and $v = v_1 v_2 \cdots v_l$ be the block factorizations of $u$ and $v$, respectively. 
We have $k,l \leq n$. 
For $i \in [1,c]$ let 
\[ B_i = \{ u_j \colon j \in [1,k], u_j \in \Sigma_i^+ \} \cup \{ v_j \colon j \in [1,l], v_j \in \Sigma^+_i \} \]
be the set of all blocks in $u$ and $v$ from $\Sigma_i^+$. Let $n_i = |B_i|$ and note that $\sum_{i=1}^c n_i \le 2n$.
Assume that $\equiv_{p_i}$ refines $\equiv_{M_i}$ on $B_i$ for every $i \in [1,c]$.  
This happens with probability at least 
\[ 
1 - \sum_{i=1}^c \epsilon_0(n) \binom{n_i}{2} \geq 1 - \frac{\epsilon_0(n)}{2}  \sum_{i=1}^c n_i^2 \geq 1 - 2 \epsilon_0(n) n^2.
\] 
The block factorizations of the words $\tilde{u}, \tilde{v} \in \Delta^*$ that Algorithm~\ref{algo-graph-product-2} 
inputs into the semiPFA $\mathcal{B}_m$ on input $u$ and $v$, respectively, can be written as
$\tilde{u} = \tilde{u}_1 \tilde{u}_2 \cdots \tilde{u}_k$, resp., $\tilde{v} = \tilde{v}_1 \tilde{v}_2 \cdots \tilde{v}_l$.
Every block $\tilde{x}$ with $x \in B_i$ ($i \in [1,c]$) satisfies 
\begin{equation}\label{eq-block-tilde-x}
\tilde{x} \equiv_{\langle a_{i} \rangle} a_{i}^{-p_{i} + \rho_i(p_i, x)}  .
\end{equation}
Note that $-p_{i} + \rho_i(p_i, x) \neq 0$. This follows from our preprocessing, ensuring that there are no non-trivial loops at
states with non-zero probability with respect to $\iota_i$.

We claim that $\tilde{u} \not\equiv_H \tilde{v}$.
In order to get a contradiction, assume that $\tilde{u} \equiv_H \tilde{v}$. Hence, we obtain words $\tilde{y}, \tilde{z} \in \Delta^*$ such that
$\tilde{u} \to_s \tilde{y}$, $\tilde{u} \to_s \tilde{z}$, $\tilde{y} \leftrightarrow_r \tilde{z}$. Since also $\tilde{u}$ and $\tilde{v}$
are reduced, we have $k = l$ and the block factorizations of $\tilde{y}$ and $\tilde{z}$ have the form
$\tilde{y} = \tilde{u}_{s_1} \tilde{u}_{s_2} \cdots \tilde{u}_{s_k}$ and 
$\tilde{z} = \tilde{v}_{t_1} \tilde{v}_{t_2} \cdots \tilde{v}_{t_k}$ for permutations $j \mapsto s_j$ and $j \mapsto t_j$.
Moreover, for every $j \in [1,k]$, $\tilde{u}_{s_j}$ and $\tilde{v}_{t_j}$ are from the same set $\{ a_i, a_i^{-1} \}^+$ ($i \in [1,c]$)
and by \eqref{eq-block-tilde-x} satisfy $\rho_i(p_i, u_{s_j}) = \rho_i(p_i, v_{s_j})$.  
Since $\equiv_{p_i}$ refines $\equiv_{M_i}$ on $B_i$ we obtain $u_{s_j} \equiv_{M_i} v_{t_j}$. This finally
implies $u \to_s y$, $v \to_s z$ and $y \leftrightarrow_r z$ and hence $u \equiv_M v$, which is a contradiction.

Hence, we have $\tilde{u} \not\equiv_H \tilde{v}$. This implies that $\sigma_m(r_0, \tilde{u}) \neq \sigma_m(r_0, \tilde{v})$
 with probability at least $1 - 1/m^d \ge 1 - 1/n^d$.

Taken together, it follows that Algorithm~\ref{algo-graph-product-2} reaches different program states on inputs $u$ and $v$ with
probability at least $(1 - 2 \epsilon_0(n) n^2)(1 - 1/n^d) \geq 1 - (2 \epsilon_0(n) n^2 + 1/n^d)$.
\end{proof}
The graph product construction can be extended to semigroups
$S_i = \langle \Sigma_i \mid R_i\rangle^+$ ($i \in [1,c]$). It is defined as 
$\mathsf{GP}^+(S_1, \ldots, S_c, I) =   \Sigma^+/\rho_R$ where $\Sigma$ and $R$ are defined as in \eqref{def-Sigma-GP} and \eqref{relation-R}; see also \cite[Section~7]{DanGou23}.
Note that if every $S_i$ is a monoid, then
$\mathsf{GP}^+(S_1, \ldots, S_c, I)$ and $\mathsf{GP}^*(S_1, \ldots, S_c, I)$
are non-isomorphic semigroups in general. Take for instance two monoids $M_1$ and $M_2$ with neutral elements $e_1$ and $e_2$, respectively.
In their free product according to the semigroup version, we have for instance $e_1 e_2 \neq e_2 e_1$, On the other hand,
if we take the free product using the monoid version, we have $e_1 e_2 = e_2 e_1$, which represents the neutral element of $M_1 \ast M_2$
(the equivalence class of the empty word in  $\Sigma^*/\rho_R$). 

It is easy to see that $\mathsf{GP}^+(S_1, \ldots, S_c, I)$ embeds into $\mathsf{GP}^*(M_1, \ldots, M_c, I)$ where each monoid $M_i = S_i \cup \{1\}$ is obtained from the corresponding semigroup $S_i$ by adjoining a new neutral element $1 \not\in S_i$ (regardless of whether $S_i$ already had a neutral element; in particular
$M_i$ and $S_i^1$ are different if $S_i$ is already a monoid) \cite[Proposition~7.3]{DanGou23}. 
Moreover, the monoids $M_i$ satisfy the assumption from Theorem~\ref{thm-GP-S}.
Hence, Theorem~\ref{thm-GP-S} can be directly applied to $\mathsf{GP}^+(S_1, \ldots, S_c, I)$.

 \subsection{Distinguishers for semidirect products} \label{sec-semi-direct}

Let $M$ and $P$ be monoids. We denote the neutral elements of both monoids with $1$, which should not lead
to missinterpretations.
The multiplication in $M$ is written as $x \ast y$ (for $x,y \in M$) for better readability, whereas the
multiplication in $P$ is written as $p \cdot q$ (for $p,q \in P$).
A \emph{right action} of $P$ on $M$ is a mapping $\alpha : M \times P \to M$
with the following properties, where we write $x^p$ for $\alpha(x,p)$:
\begin{itemize}
\item $x^{p \cdot q} = (x^p)^q$,
\item $(x \ast y)^p= x^p \ast  y^p$,
\item $x^1 = x$, and $1^p = 1$.
\end{itemize}
In other words, $\alpha$ is a right action of $P$ on the  \emph{set} $M$ as defined in Section~\ref{sec-action}, where the latter is considered as a set, which is also compatible with the monoid structure~of~$M$.

Dually, a \emph{left action} of $P$ on $M$ is a mapping $\beta : P \times M \to M$
with the following properties, where we write ${}^{p}{\kern-1pt}x$ for $\beta(p,x)$:
\begin{itemize}
\item ${}^{p \cdot  q}{\kern-1pt}x = {}^p{\kern-1pt}({}^{q}{\kern-1pt}x)$,
\item ${}^{p}{\kern-1pt}(x \ast y) = {}^{p}{\kern-1pt}x \ast  {}^{p}{\kern-1pt}y$,
\item ${}^{1}{\kern-1pt}x = x$, and ${}^{p}{\kern-.5pt}1 = 1$.
\end{itemize}

For a right action $\alpha$ of $P$ on $M$ one
defines the \emph{semidirect product} $P \ltimes_\alpha M$ as follows:
its carrier set is $P \times M$ and the multiplication is defined by $(p,x) (q,y) = (p \cdot q, x^q \ast y)$. 
Similarly, for a left action $\beta$ of $P$ on $M$ one
defines the \emph{semidirect product} $M \rtimes_\beta P$ as the monoid with 
the carrier set  $M \times P$ and the multiplication $(x,p) (y,q) = (x \ast {}^{p}{\kern-1pt}y, p \cdot q)$.
The neutral element is $(1,1)$ in both cases.
 
Note that if $P$ is generated by $\Gamma$ and $M$ is generated by $\Sigma$, then
$P \ltimes_\alpha M$ is generated by the set $(\{1\} \times \Sigma) \cup (\Gamma \times \{1\})$ and similarly
for $M \rtimes_\beta P$. In particular, if $M$ and $P$ are both finitely generated then so are $P \ltimes_\alpha M$ and
$M \rtimes_\beta P$.

 \begin{theorem} \label{thm-semidirect-prod}
 Let $P$ be a finite monoid with $c = |P|$ and let $M$ be a finitely generated monoid.
 Let $\mathcal{R}=(\mathcal{A}_{n})_{n \ge 0}$ be an $(\epsilon_0(n),\epsilon_1(n))$-distinguisher for $M$ with space complexity $s(n)$.
 \begin{itemize}
 \item If $\alpha : M \times P \to M$ is a right action of $P$ on $M$ then there is a constant $d > 0$ and an $(\epsilon_0(dn),c \cdot\epsilon_1(dn))$-distinguisher for $P \ltimes_\alpha M$ with space complexity $c \cdot s(dn)) + \log_2 c$. 
 \item If $\beta : P \times M \to M$ is a left action of $P$ on $M$ then there is a constants $d > 0$ and an $(\epsilon_0(dn),\epsilon_1(dn))$-distinguisher for $M \rtimes_\beta P$ with space complexity $s(dn) +  \log_2 c$.
 \end{itemize}
 \end{theorem}
 
 \begin{algorithm}[t]
 \SetKwComment{Comment}{(}{)}
\SetKwInput{KwGlobal}{global variables}
\SetKwInput{KwInit}{initialization}
\SetKwInput{KwNext}{next input letter}
\BlankLine
\KwGlobal{$q_r \in Q_{dn}$ for every $r \in P$, $p \in P$} 
\BlankLine
\KwInit{}
guess according to the distribution $\iota_{dn}$ a state $\tilde{q}_0$ and set every $q_r$ to $\tilde{q}_0$. \label{line-init-q-r2} \\
$p  \coloneqq 1$ (the neutral element of $M$) 
\BlankLine
\KwNext{$a \in \Sigma$}
\If{$a \in P$}{$p  \coloneqq p\cdot a$ ; \label{line-new-x} \\ $q_r  \coloneqq q_{a \cdot r}$ for all $r \in P$ \label{line-new-q_r}}
\If{$a \in \Sigma$}{$q_r  \coloneqq \delta_{dn}(q_r, a^r)$ for all $r \in P$ \label{line-new-q_r-2}}
\BlankLine
\caption{An $(\epsilon_0(dn), c \cdot \epsilon_1(dn))$-distinguisher for a semidirect product  $P \ltimes_\alpha M$.
\label{algo-semidirect}}
\end{algorithm}

 \begin{proof}
 Let us start with the first statement. Fix a finite generating set $\Sigma$ for $M$. Let $d$ be such that every element $a^r$ for $a \in \Sigma$ and 
 $r \in P$ can be represented by a word over $\Sigma$ of length at most $d$. Below, we identify $a^r$ with such a word.  
 We consider the finite generating set $\Sigma \uplus P$ for $P \ltimes_\alpha M$, where $a \in \Sigma$ is identified with
 $(1,a)$ and $p \in P$ is identified with $(p,1)$. 
 Fix an input length $n$ and consider the semiPFA $\mathcal{A}_{dn} = (Q_{dn}, \Sigma, \iota_{dn}, \delta_{dn})$ 
 from the $(\epsilon_0(n),\epsilon_1(n))$-distinguisher $\mathcal{R}$ for $M$.
 
 Our distinguisher for $P \ltimes_\alpha M$ is shown in Algorithm~\ref{algo-semidirect}.
 From the list of global variables, it is clear the the algorithm stores $c \cdot s(dn) + \log c$ bits.
 
 Consider now an input word $w \in (\Sigma \uplus P)^{\le n}$ 
 that represents the element $(s ,x) \in P \ltimes_\alpha M$. 
Then after reading $w$, the program variable $p$ stores the element $s \in P$.
Moreover, every program variable $q_r$ (for $r \in P$) stores a state of $\mathcal{A}_{dn}$ that is reached from
the guessed initial state $\tilde{q}_0$ by reading a word $v_r \in \Sigma^{\le dn}$ that represents in $P$ the element $x^r$. 
This follows easily by induction on $n$: The case $n=0$ is clear with $v_r = \varepsilon$.
Assume now that these invariants hold after reading $w \in (\Sigma \uplus P)^{\le n-1}$ with
the words $v_r \in \Sigma^{\leq d(n-1)}$
and let $w$ represent $(s,x) \in P \ltimes_\alpha M$. If the next input letter is $a \in P$, then $wa$ represents 
$(s,x) (a,1) = (s\cdot a, x^a)$. The variable $p$ is correctly set to $s \cdot a$ in line \ref{line-new-x}. Moreover,
the state variable $q_r$ is set  to $q_{a \cdot r} = \delta_{dn}(\tilde{q}_0, v_{a\cdot r})$ in line \eqref{line-new-q_r}, where $v_{a\cdot r} \in \Sigma^{\leq d(n-1)}$ is a word that represents the element
$x^{a\cdot r} = (x^a)^r$. Hence, every $q_r$ is updated correctly. 

If the next input letter is $a \in \Sigma$ then $wa$ represents
$(s, x) (1,a) = (s, x \ast a)$. Therefore, $p$ is correctly not updated. Moreover, the state variable $q_r$ is set in  \ref{line-new-q_r-2} to 
$\delta_{dn}(q_r, a^r) = \delta_{dn}(\tilde{q}_0, v_r a^r)$, where the word $v_r a^r \in \Sigma^{\le dn}$ represents the element $x^r \ast a^r = (x \ast a)^r$.
 Hence, every $q_r$ is updated correctly. 
 
 From the above invariant it follows directly that for two input words $u,v \in (\Sigma \uplus P)^{\le n}$ we have the following:
 \begin{itemize}
\item If $u \neq v$ in  $P \ltimes_\alpha M$ then on input $u$ and $v$, respectively, the algorithm reaches different memory states with probability at least $1-\epsilon_0(dn)$.
More precisely, the values of the variable $p$ are different with probability $1$ or the values of the variable $q_1$ are different with probability
at least $1- \epsilon_0(dn)$.
\item If $u = v$ in $P \ltimes_\alpha M$ then after reading $u$ and $v$, respectively, the values of the variable $p$ are equal with probability $1$
and for every $r \in P$, the values of the variable $q_r$ are equal with probability at least $1-\epsilon_1(dn)$. Hence, the memory states are the same with probability at least $1- |P| \cdot \epsilon_1(dn)$.
\end{itemize}
This concludes the proof of the first statement from the theorem. The second statement is straightforward due to the multiplication rule 
$(x,p) (y,q) = (x \ast {}^{p}{\kern-1pt}y, p \cdot q)$ for $M \rtimes_\beta P$. An 
$(\epsilon_0(dn),\epsilon_1(dn))$-distinguisher for $M \rtimes_\beta P$ can store as above the current element from $P$ in a variable $p$.
When reading a new letter $a \in P$ it only has to update $p$. When reading a new letter $a \in \Sigma$ the distinguisher has to read
${}^{p}{\kern-1pt}a$ into the semiPFA $\mathcal{A}_{dn}$, where $d$ is maximal length of a word over $\Sigma$ representing an element 
${}^{r}{\kern-1pt}a \in M$ for $a \in \Sigma$ and $r \in P$.
 \end{proof}

 \subsection{Distinguishers for wreath products} \label{sec-wreath}

In this section we will investigate distinguishers for wreath products. We start with the definition
of the wreath product of two monoids.

Let $M$ and $P$ be monoids. As in the previous section we denote the multiplication in $P$ with $\cdot$ and the multiplication in $M$ with $\ast$.
 The wreath product $M \mathop{wr} P$ of $M$ and $P$ is a semidirect product of $P$ and $M^P$ where $M^P$ is the set of all
 functions $f : P \to M$ with the multiplication defined by $p(fg) = pf \ast pg$ for $f,g \in M^P$ and $p \in P$. 
 Note that we write $pf$ instead of $f(p)$. The neutral element of $M^P$ is the mapping $i$ with $p i = 1$ for all
 $p \in P$. The monoid $P$ acts on $M^P$ on the left: 
 $p(^q  \! f) = (p \cdot q)f$. If $\beta$ denotes this left action then
 the \emph{wreath product} $M \mathop{wr} P$ is defined as $M^P \rtimes_\beta P$.
 
 If $M$ and $P$ are finitely generated, then the wreath product $M \mathop{wr} P$ is in general not finitely generated, simply because
 it is not countable. Therefore, one defines the \emph{restricted wreath product}. Start with the set $M^{(P)}$ of all finitely supported
 mappings $f : P \to M$. These are mappings $f$ such that $pf \neq 1$ for only finitely many $p \in P$. The set $M^{(P)} \times P$ is 
 in general not a submonoid of $M \mathop{wr} P$. The restricted wreath product is the submonoid of $M \mathop{wr} P$ generated by $M^{(P)} \times P$.
 In the following, we only consider the restricted wreath product. We use the standard notation $M \wr P$ for it and omit the adjective ``restricted''.
 Note that $M \mathop{wr} P = M \wr P$ if $P$ is finite.
 In general, also the restricted wreath product of two finitely generated monoids is not finitely generated. The following result was
 shown in \cite[Theorem~1.1]{RobRus03}: Let $M$ and $P$ be monoids. Then
$M \wr P$ is finitely generated if and only if (i) $M$ and $P$ are both finitely
generated and (ii) $M$ is trivial or $P = T \cdot U(P)$ for some finite set $T \subseteq P$
(recall that $U(P)$ is group of units of $P$).

The following result is a direct consequence of Theorem~\ref{thm-semidirect-prod}:
 \begin{theorem} \label{thm-wreath-right-finite}
 Let $P$ be a finite monoid with $c = |P|$ and let $M$ be a finitely generated monoid.
Let $\mathcal{R}=(\mathcal{A}_{n})_{n \ge 0}$ be an $(\epsilon_0(n),\epsilon_1(n))$-distinguisher for $M$ with space complexity $s(n)$.
Then there is a constants $d > 0$ and an $(\epsilon_0(dn),c \cdot \epsilon_1(dn))$-distinguisher for $M \wr P$ with space complexity $c \cdot s(dn) +  \log_2 c$.
 \end{theorem}

\begin{proof}
The monoid $M^P$ is a direct product of $c$ many copies of $M$ and therefore has an $(\epsilon_0(n), c \cdot \epsilon_1(n))$-distinguisher with space complexity $c \cdot s(n)$ by Lemma~\ref{thm-direct-prod}. The theorem follows from the second statement in Theorem~\ref{thm-semidirect-prod}.
\end{proof}
 In the following, we will only consider the case where $P$ is a finitely generated group $G$. 
 Then, the above mentioned result from \cite{RobRus03} implies that $M \wr G$ is finitely generated if and only if $M$ and $G$ are finitely generated.
 Moreover, for every element $(f,g) \in M \wr G$ the mapping $f : G \to M$ is finitely supported 
 \cite[Section~2]{RobRus03}.

 For the case $P =G$ it is common to define the left action of $G$ on $M^G$ slightly
 different by $g'(^g  \! f) = (g^{-1} \cdot g')f$ for $g,g' \in G$ and $f : G \to M$. The two versions of the wreath product that correspond to the two
 left actions of $G$ on $M^G$ are then isomorphic; an isomorphism is given by $(f,g) \mapsto (f',g)$ with $gf' = g^{-1}f$.
The reason for using the left action $g'(^g  \! f) = (g^{-1} \cdot g')f$ is that it yields the lamplighter interpretation for the wreath product, which is useful for the upcoming proofs.
First note that if $G$ is generated by $\Sigma$ and $M$ is generated by $\Gamma$ then
$M \wr G$ is generated by the set $\{ (i, a) \colon a \in \Sigma \} \cup \{ (f_b, 1) \colon b \in \Gamma \}$, where $i : G \to M$
is defined by $g i = 1$ for all $g \in G$ and $f_b : G \to M$ is defined by $1 f_b = b$ and $g f_b = 1$ for all $g \in G \setminus \{1\}$.
Moreover, $\{ (i, a) \colon a \in \Sigma \}$ generates a copy of $G$ in $M \wr G$ and $\{ (f_b, 1) \colon b \in \Gamma \}$ generates a copy
of $M$ in $M \wr G$. In the following we will identify $\{ (i, a) \colon a \in \Sigma \}$ with $\Sigma$ and 
$\{ (f_b, 1) \colon b \in \Gamma \}$ with $\Gamma$, so that $M \wr G$ is generated by $\Sigma \cup \Gamma$ (assuming $\Sigma \cap \Gamma = \emptyset$).
Note that
\begin{itemize}
\item $(f,g) (i, a) = (f, g \cdot a)$ and
\item $(f,g) (f_b, 1) = (f', g)$ where $g' f' = g' f$ for all $g' \neq g$ and $g f' = (g f) \ast b$. 
\end{itemize}
Hence, an element $(f,g)$ should seen as (i) an assignment of elements from $M$ to the elements of $G$ (this is the mapping $f$) such that
only finitely many elements of $G$ get an element from $M \setminus \{1\}$, and (ii) a distinguished lamplighter position $g \in G$.
When multiplying with the generator $(i, a)$ the lamplighter moves from $g$ to $g \cdot a$
and when multiplying with the generator $(f_b, 1)$
the element assigned to the current  lamplighter position is multiplied on the right with $b$ (and the lamplighter does not move).
The classical lamplighter group is $\mathbb{Z}_2 \wr \mathbb{Z}$ where the two elements in $\mathbb{Z}_2$ correspond to ``light off'' and ``light on''.
Note that in our first version of the wreath product multiplication with the generator $(f_b,1)$ multiplies the element at position $g^{-1}$ on the right
with $b$.

In the rest of this section, we construct distinguishers for wreath products $H \wr G$ of f.g.~groups $H$ and $G$. 
The case of a wreath product $H \wr G$ with $G$ finite  is covered by Theorem~\ref{thm-wreath-right-finite}.
The case of a wreath product $H \wr G$ with $G$ infinite turns out to be more interesting. 
We will start with wreath products $Z \wr G$, where $G$ is infinite and $Z$ is a cyclic group and split this case
into two subcases:
\begin{itemize}
\item $Z = \mathbb{Z}$; see Theorem~\ref{thm-wreath-main},
\item $Z = \mathbb{Z}_{p^k}$ for a prime $p$ and $k \geq 1$; see Theorem~\ref{thm-wreath-Zpk}.
\end{itemize}
The case $Z = \mathbb{Z}_{m}$ with $m$ not a prime power (and more generally, the case of a wreath product $A \wr G$ with $A$ finitely generated abelian)
will follow easily from those cases; see Corollary~\ref{thm-wreath-by-abelian}.
 In the following, $\Sigma$ denotes a finite generating set for the group $G$ and $a$ denotes a generator for the 
cyclic group $Z$. Then, $\Gamma = \Sigma \cup \{a,a^{-1}\}$ is a finite symmetric generating set for the wreath product $Z \wr G$.

\begin{theorem} \label{thm-wreath-main}
Let $G$ be an f.g.~infinite group  and $\mathcal{R} = (\mathcal{A}_n)_{n \ge 0}$
an $(\epsilon_0,\epsilon_1)$-distinguisher for $G$.
Let $d$ be a fixed constant and define 
\begin{eqnarray*}
\zeta_0(n) &=& 2\max\{\epsilon_0(n),\epsilon_1(n)\}n^2+ 1/n^d, \\
\zeta_1(n) &=& (2n^2+1) \epsilon_1(n) . 
\end{eqnarray*}
Then there is a $(\zeta_0,\zeta_1)$-distinguisher $\mathcal{R}'$ for $\mathbb{Z} \wr G$
with space complexity 
\[
  s(\mathcal{R}', n) = 3 \cdot s(\mathcal{R},n) + \Theta(\log n).
\]
\end{theorem}
 
 \begin{proof}
Fix an input length $n$ and 
let $\mathcal{A}_n = (Q_n, \Sigma, \iota_n, \delta_n)$. W.l.o.g. we can assume that $Q_n = [0,|Q_n|-1]$
consists of the first $|Q_n|$ non-negative integers.

For a word $u = v a^{\beta}$ with $v \in \Gamma^*$
and $\beta \in \{-1,1\}$ we define $\sigma(u) = \beta$. 

Consider a potential input word $w \in \Gamma^{\le n}$ and a state $q \in Q_n$ (later, it will
be randomly guessed according to the initial state distribution $\iota_n$). 
Define the mapping $\delta_q : \Sigma^* \to Q$ by
$\delta_q(u) = \delta_n(q,u)$ for $u \in \Sigma^*$.
With  $w$ and $q$  we associate a polynomial $P_{q,w}(x) \in \mathbb{Z}[x]$ as follows: 
Let $R_w$ be the set of all 
prefixes of $w$ that end with $a$ or  $a^{-1}$. We have $|R_w| \leq |w| \le n$. For every $v \in R_w$ 
consider the $\mathcal{A}_n$-state $q_v = \delta_q(\pi_\Sigma(v)) \in [0,|Q_n|-1]$
(where $\pi_\Sigma : \Gamma^* \to \Sigma^*$ is the canonical projection morphism).
 We then define the polynomial
\begin{equation}
\label{eq-poly-P_qw}
P_{q,w}(x)  \coloneqq \sum_{v \in R_w} \sigma(v) \cdot x^{q_v} .
\end{equation}
Note that this polynomial has degree at most $|Q_n|-1$, all its coefficients have absolute value at most $n$, and there are at most $n$ monomials.
Actually, the sum of the absolute values of the coefficients of $P_{q,w}(x)$ is bounded by $n$.

Let us write $(f_w, g_w) \in \mathbb{Z} \wr G$ for the group element represented by the word $w$. The element $g_w \in G$
is obtained by evaluating the projection  $\pi_\Sigma(w)$ in the group $G$.
For this, the distinguisher for $\mathbb{Z} \wr G$ that we aim for (and that we describe in detail later) will simply use the semiPFA $\mathcal{A}_n$.
The difficult part is the mapping $f_w$. For this we make use of the polynomial $P_{q,w}(x)$.

\begin{claim} \label{claim-WP1}
 Let $u,v \in \Gamma^{\le n}$ be two input words such that $f_u = f_v$.
Then we have
\[ \Prob_{q \sim \iota_n}[P_{q,u}(x) = P_{q,v}(x)] \geq 1-2\epsilon_1(n) n^2.\] 
\end{claim}
\emph{Proof of Claim~\ref{claim-WP1}.} Define the set of words $S = \pi_\Sigma(R_u \cup R_v) \subseteq \Sigma^*$
and let $G_S = \pi_G(S) \subseteq G$ be the finite set of group elements represented by the words in $S$. Clearly, $|S| \le 2n$.
We will use the equivalence relations $\equiv_q$ ($q \in Q_n$) from Lemma~\ref{lemma-injective}.
It suffices to show for every state $q \in Q_n$ the following: if $\equiv_G$ refines  $\equiv_q$ on $S$, then $P_{q,u}(x) = P_{q,v}(x)$.
If this is shown then the second statement of Lemma~\ref{lemma-injective} implies
\begin{eqnarray*}
\Prob_{q \sim \iota_n}[P_{q,u}(x)  = P_{q,v}(x)] & \geq & \Prob_{q \sim \iota_n}[\equiv_G  \text{refines} \equiv_q \text{on } S] \\ 
& \geq & 1-  \epsilon_1(n) \binom{|S|}{2} \\
& \geq & 1-2\epsilon_1(n) n^2 .
\end{eqnarray*}
So, consider a state $q \in Q_n$ and
assume that $\equiv_G$ refines $\equiv_q$ on $S$. 
Let $Q_S = \delta_q(S)$ be the image of $S$ under the mapping $\delta_q$.
It is the set of states
of $\mathcal{A}_n$ that can be reached from $q$ via words from $S$.
Let $[S]_{\equiv_G}$ and $[S]_{\equiv_q}$ be the set of equivalence
classes of $\equiv_G$ and  $\equiv_q$, respectively, on $S$.
With every state $r \in Q_S$ we can associate the equivalence class
$A_r = \delta^{-1}_q(r) \cap S \in [S]_{\equiv_q}$.  Similarly, 
with every group element $g \in G_S$ we associate the equivalence
class $B_g = \pi^{-1}_G(g) \cap S \in [S]_{\equiv_g}$ (where $\pi_G : \Sigma^* \to G$ is 
the evalutation morphism).
Moreover, the $A_r$ ($r \in Q_S$) and
$B_g$ ($g \in G_S$) are all equivalence classes of  $\equiv_q$ and  $\equiv_G$ on $S$, respectively.
Since $\equiv_G$ refines $\equiv_q$ on $S$, there is a surjective mapping
$h : G_S \to Q_S$ such that 
\[ A_r = \biguplus_{g \in h^{-1}\!(r)} B_g \]
for every $r \in Q_S$.
The definition of the mappings $f_u$ and $f_v$ yields
for every $g \in G_S$:
\begin{equation} \label{eq-wreath-Z-def_f_u-f_v}
\sum_{y \in \pi^{-1}_\Sigma\!(B_g) \cap R_u} \!\!\!\!\!\!\!\!\! \sigma(y) 
\ = \ f_u(g) \ = \ f_v(g) \ = \ \sum_{y \in \pi^{-1}_\Sigma\!(B_g) \cap R_v} \!\!\!\!\!\!\!\!\!  \sigma(y) .
\end{equation}
Note that for every $r \in Q_S$ we have
\[
 \pi^{-1}_\Sigma(A_r) \cap R_u = \pi^{-1}_\Sigma\left( \biguplus_{g \in h^{-1}\!(r)} B_g \right) \cap R_u = 
 \biguplus_{g \in h^{-1}\!(r)} (\pi^{-1}_\Sigma(B_g) \cap R_u)
 \]
 and similarly for $R_v$.
If we write $P_{q,u} \ast x^r$ for the coefficient of the monomial $x^r$ (where $r \in Q_S$) in the polynomial $P_{q,u}(x)$ and analogously
for $P_{q,v}(x)$, then we obtain for every $r \in Q_S$:
\begin{eqnarray*}
P_{q,u} \ast x^r  = \sum_{y \in \pi^{-1}_\Sigma\!(A_r) \cap R_u} \!\!\!\!\!\!\!\!\!  \sigma(y) 
&=& \sum_{g \in h^{-1}(r)} \sum_{y \in \pi^{-1}_\Sigma\!(B_g) \cap R_u } \!\!\!\!\!\!\!\!\!  \sigma(y)  \\
&\stackrel{\text{\eqref{eq-wreath-Z-def_f_u-f_v}}}{=} & \sum_{g \in h^{-1}(r)} \sum_{y \in \pi^{-1}_\Sigma\!(B_g) \cap R_v } \!\!\!\!\!\!\!\!\!  \sigma(y)  \\
& = & \sum_{y \in \pi^{-1}_\Sigma\!(A_r) \cap R_v} \!\!\!\!\!\!\!\!\!  \sigma(y) 
\ = \ P_{q,v} \ast x^r .
\end{eqnarray*}
Hence, we finally get $P_{q,u}(x)  = P_{q,v}(x)$, which proves Claim~\ref{claim-WP1}.

\begin{claim} \label{claim-WP2}
Let $u,v \in \Gamma^{\le n}$ be two input words such that $f_u \neq f_v$.
Then we have
\[ 
\Prob_{q \sim \iota_n}[P_{q,u}(x) \neq P_{q,v}(x)] \geq 1-2\max\{\epsilon_0(n),\epsilon_1(n)\} n^2.
\] 
\end{claim}
\emph{Proof of Claim~\ref{claim-WP2}.} The proof is very similar to the proof of Claim~\ref{claim-WP1}. Assume that $f_u \neq f_v$.
We use the same notations as in the proof of Claim~\ref{claim-WP1}.
By the first statement of Lemma~\ref{lemma-injective}
it suffices to show for every state $q \in Q_n$: if $\equiv_G$ equals $\equiv_q$ on $S$ then $P_{q,u}(x) \neq P_{q,v}(x)$.

Assume that $\equiv_G$ equals $\equiv_q$ on $S$ for a state $q$.
Then there is a bijection $h : G_S \to Q_S$ such that $B_{g} =  A_{h(g)}$
for every $g \in G_S$. Since $f_u \neq f_v$ there is a $g \in G_S$ with 
\[
\sum_{y \in \pi^{-1}_\Sigma\!(B_g) \cap R_u} \!\!\!\!\!\!\!\!\! \sigma(y) 
\ = \ f_u(g) \ \neq \ f_v(g) \ = \ \sum_{y \in \pi^{-1}_\Sigma\!(B_g) \cap R_v} \!\!\!\!\!\!\!\!\!  \sigma(y) .
\]
For $r = h(g)$ we obtain
\begin{eqnarray*}
P_{q,u} \ast x^r 
&=& \sum_{y \in \pi^{-1}_\Sigma\!(A_r) \cap R_u} \!\!\!\!\!\!\!\!\!  \sigma(y) \
= \sum_{y \in \pi^{-1}_\Sigma\!(B_g) \cap R_u } \!\!\!\!\!\!\!\!\!  \sigma(y)  \\
& \neq & \sum_{y \in \pi^{-1}_\Sigma\!(B_g) \cap R_v } \!\!\!\!\!\!\!\!\!  \sigma(y)  \
= \sum_{y \in \pi^{-1}_\Sigma\!(A_r) \cap R_v} \!\!\!\!\!\!\!\!\!  \sigma(y) 
\ = \ P_{q,v} \ast x^r .
\end{eqnarray*}
Thus, we have $P_{q,u}(x)  \neq P_{q,v}(x)$, which proves Claim~\ref{claim-WP2}.

\medskip
\noindent
In order to verify $f_u = f_v$, a randomized distinguisher could compute the polynomial $P_{q,w}(x)$ for 
an input word $w \in \Gamma^{\leq n}$ and a random initial state $q \sim \iota_n$. The problem is that the polynomial $P_{q,w}(x)$ does not fit
into the space bound we are aiming for. Therefore, we only can afford to compute a fingerprint of 
$P_{q,w}(x)$. This fingerprint is obtained in two steps. 

First, observe that the Cauchy bound\footnote{The Cauchy bound says that for a polynomial $p(x) = a_n x^n + a_{n-1} x^{n-1} + \cdots a_1x + a_0 \in \mathbb{R}[x]$
with $a_n \neq 0$ every root $\alpha$ of $p(x)$ satisfies $|\alpha| < 1 + \max \{ |a_0|, |a_1|, \ldots |a_{n-1}|\}/|a_n|$.} implies for all words $u,v \in \Gamma^{\leq n}$ that $P_{q,u}(x) = P_{q,v}(x)$ if and only if 
$P_{q,u}(2n+1) = P_{q,v}(2n+1)$. To see this, note that all coefficients of the polynomial $P_{q,u}(x) - P_{q,v}(x) \in \mathbb{Z}[x]$ are bounded in absolute value
by $2n$. Hence, $2n+1$ is not a root of the polynomial $P_{q,u}(x) - P_{q,v}(x)$.

The value $P_{q,w}(2n+1)$ still needs too many bits. Therefore, our distinguisher will compute it only modulo a sufficiently large 
prime number $p$. Note that for every word $w \in \Gamma^{\leq n}$ we have
\[
|P_{q,w}(2n+1)| \leq n \cdot (2n+1)^{|Q_n|-1}.
\]
Hence, for two input words $u,v \in \Gamma^{\leq n}$ we have
\[
D_{u,v}  \coloneqq |P_{q,u}(2n+1)-P_{q,v}(2n+1)|  \leq 2n(2n+1)^{|Q_n|-1} \leq (2n+1)^{|Q_n|}.
\]
Recall that $\mathbb{P}_{\alpha}$ is the set of prime numbers $p \leq \alpha$. We want to choose $\alpha = \alpha(n)$ large enough such that for a prime $p$ uniformly
chosen from $\mathbb{P}_{\alpha}$ we obtain, under the assumption $D_{u,v} \neq 0$,
\begin{equation} \label{prob-p}
 \Prob_{p \in \mathbb{P}_{\alpha}}[p \text{ divides } D_{u,v}] \leq 1/n^d ,
\end{equation}
where $d$ is the parameter from the theorem.

By \cite{Robin1983},
the number of different prime factors of $D_{u,v}$ is bounded by 
\begin{eqnarray*}
\frac{ \ln D_{u,v}}{\ln \ln D_{u,v}} \cdot (1+o(1)) & \le &
\frac{\ln \big( (2n+1)^{|Q_n|}\big)}{\ln \ln  \big( (2n+1)^{|Q_n|} \big)} \cdot (1+o(1)) \\[2mm]
& \leq & \mathcal{O}\bigg(\frac{|Q_n| \cdot \log n}{\log |Q_n| + \log\log n}\bigg)  \leq \mathcal{O}\bigg(\frac{|Q_n| \cdot \log n}{\log \log n}\bigg)
\end{eqnarray*}
(recall that $\log |Q_n| \geq s(\mathcal{R},n)-1 \ge \Omega(\log \log n)$ since $G$ is infinite, see Remark~\ref{remark-infinite-group}).
Moreover, there are $\Theta(x)$ many primes of size at most $x \log x$. 
Hence, by fixing the number $\alpha = \alpha(n)$ such that
\begin{eqnarray*}
\alpha(n) & = & \Theta\bigg( \frac{|Q_n| \cdot \log n \cdot n^d}{\log \log n}  \cdot \log\bigg(\frac{|Q_n| \cdot \log n \cdot n^d}{\log \log n}\bigg) \bigg) \\[2mm]
& = &
\Theta\bigg( \frac{|Q_n| \cdot \log n \cdot n^d}{\log \log n}  \cdot (\log |Q_n| + d \log n)\bigg)  \leq \Theta( |Q_n| \cdot n^{d+2}),
\end{eqnarray*}
we obtain \eqref{prob-p}.

\begin{algorithm}[t]
\SetKwComment{Comment}{(}{)}
\SetKwInput{KwGlobal}{global variables}
\SetKwInput{KwInit}{initialization}
\SetKwInput{KwNext}{next input letter}
\BlankLine
\KwGlobal{prime number $p \in \mathbb{P}_{\alpha}$, integer $z \in [0,p-1]$, state $q \in Q_n$} 
\BlankLine
\KwInit{}
guess $p \in \mathbb{P}_{\alpha(n)}$ according to the uniform distribution \label{guess p} \\
guess $q \in Q_n$ according to the input distribution $\iota_n$ of $\mathcal{A}_n$ \label{guess q} \\ 
$z  \coloneqq 0$ 
\BlankLine
\KwNext{$b \in \Gamma$}
\If{$b \in \Sigma$}
   {$q  \coloneqq \delta_n(q,a)$ \label{update q}}
\If{$b = a^{\beta}$ for $\beta \in \{-1,1\}$}
   {$z  \coloneqq (z + \beta \cdot (2n+1)^{q}) \bmod p$}
\caption{A $(\zeta_0,\zeta_1)$-distinguisher for a wreath product $\mathbb{Z} \wr G$.  \label{algo-Z-wr-G}}
\end{algorithm}

We can now finally explain our $(\zeta_0,\zeta_1)$-distinguisher for the wreath product $\mathbb{Z} \wr G$; see Algorithm~\ref{algo-Z-wr-G}.
For an input word $w \in \Gamma^{\le n}$ we simulate 
in lines \ref{guess q} and \ref{update q} the semiPFA $\mathcal{A}_n$ on the projection $\pi_{\Sigma}(w)$.
In the integer variable $z$ we compute the number $P_{q,w}(2n+1) \bmod p$, where $p$ is the prime guessed in line \ref{guess p} and $q$ is the state guessed
in line \ref{guess q}. If we want to describe the algorithm by a semiPFA, then the state of the algorithm would consist of the prime number $p$ and the current values
of $q$ and $z$. The prime number $p$ is not changed when a new symbol $b$ arrives.

The algorithm stores $3 \cdot  s(\mathcal{R},n) + \Theta(\log n)$ bits:
\begin{itemize}
\item $s(\mathcal{R},n) + \Theta(\log n) = \lceil\log |Q_n|\rceil + \Theta(\log n)$ bits for the prime $p \leq \alpha(n)$,
\item $s(\mathcal{R},n) + \Theta(\log n)$ bits for the number $z < p$.
\item $s(\mathcal{R},n)$ bits for the state $q$.
\end{itemize}
It remains to compute the error probabilities of the distinguisher. For this, consider  two words $u, v \in  \Gamma^{\le n}$. 
Let $(f_u, g_u) \in \mathbb{Z} \wr G$ (resp., $(f_v, g_v) \in \mathbb{Z} \wr G$) be the group element represented by the word $u$ (resp., $v$).
Let $z(p,q,u)$ (resp., $z(p,q,v)$) be the value of the variable $z$ that Algorithm~\ref{algo-Z-wr-G} computes on input $u$ (resp., $v$) when
$p$ and $q$ are the random choices in lines \ref{guess p} and \ref{guess q}, respectively.
Recall the definitions of $\zeta_0(n)$ and $\zeta_1(n)$ from the theorem.

\begin{claim} \label{claim-WP3}
 If $u \equiv_{\mathbb{Z} \wr G}  v$ then 
\[ \Prob_{p \in \mathbb{P}_{\alpha},q \sim \iota_n}[\delta_n(q, \pi_\Sigma(u)) = \delta_n(q, \pi_\Sigma(v)) \wedge z(p,q,u) = z(p,q,v)] \geq 1-\zeta_1(n) . \]
\end{claim}

\emph{Proof of Claim~\ref{claim-WP3}.}
Assume that $u \equiv_{\mathbb{Z} \wr G} v$, i.e., $f_u = f_v$ and $g_u = g_v$.
Since we run the algorithm $\mathcal{A}_n$ on $\pi_\Sigma(u)$ and $\pi_\Sigma(v)$, respectively, and 
$\pi_\Sigma(u) \equiv_G \pi_\Sigma(v)$, we get 
\[
\Prob_{q \sim \iota_n}[\delta_n(q, \pi_\Sigma(u)) = \delta_n(q, \pi_\Sigma(v))]\ge 1-\epsilon_1(n).
\]
Moreover, $f_u = f_v$ implies 
\[ \Prob_{p \in \mathbb{P}_{\alpha},q \sim \iota_n}[z(p,q,u) = z(p,q,v)] \geq
\Prob_{q \sim \iota_n}[P_{q,u}(x) = P_{q,v}(x)] \geq 1-2\epsilon_1(n) n^2, \] where the second inequality follows from Claim~\ref{claim-WP1}.
In total, we obtain
\begin{eqnarray*}
& & \Prob_{p \in \mathbb{P}_{\alpha}, \, q\sim \iota_n}[\delta_n(q, \pi_\Sigma(u)) \neq \delta_n(q, \pi_\Sigma(v)) \vee z(p,q,u) \neq z(p,q,v)] 
\\ & \leq & \Prob_{p \in \mathbb{P}_{\alpha}, \, q\sim \iota_n}[\delta_n(q, \pi_\Sigma(u)) \neq \delta_n(q, \pi_\Sigma(v))] + \Prob_{p \in \mathbb{P}_{\alpha}, \, q\sim\iota_n}[ z(p,q,u) \neq z(p,q,v)] \\
& \leq &  \epsilon_1(n) + 2\epsilon_1(n) n^2 = \zeta_1(n) .
\end{eqnarray*}
\begin{claim} \label{claim-WP4}
If $u \not\equiv_{\mathbb{Z} \wr G}  v$ then  
\[ \Prob_{p \in \mathbb{P}_{\alpha}, \, q \sim \iota_n}[\delta_n(q, \pi_\Sigma(u)) \neq \delta_n(q, \pi_\Sigma(v)) \vee z(p,q,u) \neq z(p,q,v)] \geq 1-\zeta_0(n) .\]
\end{claim}
\emph{Proof of Claim~\ref{claim-WP4}.}
Assume that $u \not\equiv_{\mathbb{Z} \wr G} v$. If $g_u \neq g_v$, i.e., 
$\pi_\Sigma(u) \not\equiv_G \pi_\Sigma(v)$, then we get
\[
\Prob_{q \sim \iota_n}[\delta_n(q, \pi_\Sigma(u)) \neq \delta_n(q, \pi_\Sigma(v))]\ge 1-\epsilon_0(n) \ge 1-\zeta_0(n).
\]
On the other hand, if the mappings $f_u$ and $f_v$ are different, we obtain
\[ \Prob_{q \sim \iota_n}[P_{q,u}(x) \neq P_{q,v}(x)] \geq 1-2\max\{\epsilon_0(n),\epsilon_1(n)\} n^2
\] 
from Claim~\ref{claim-WP2}.
Consider a state $q$ with  $P_{q,u}(x) \neq P_{q,v}(x)$. With inequality \eqref{prob-p} for 
$D_{u,v}  \coloneqq |P_{q,u}(2n+1) - P_{q,v}(2n+1)| > 0$
we obtain the bound
\[
\Prob_{p \in \mathbb{P}_{\alpha}}[z(p,q,u) = z(p,q,v)]  = \Prob_{p \in \mathbb{P}_{\alpha}}[p \text{ divides } D_{u,v}] \leq 1/n^d .
\]
Therefore we have
\[
 \Prob_{p \in \mathbb{P}_{\alpha}, \, q\sim\iota_n}[z(p,q,u) \neq z(p,q,v) \mid P_{q,u}(x) \neq P_{q,v}(x)] \geq 1-1/n^d .
\]
Finally, we get
\begin{eqnarray*}
& & \Prob_{p \in \mathbb{P}_{\alpha}, \, q\sim \iota_n}[z(p,q,u) \neq z(p,q,v)]  \\
&=& \Prob_{p \in \mathbb{P}_{\alpha}, \, q\sim\iota_n}[z(p,q,u) \neq z(p,q,v) \mid P_{q,u}(x) \neq P_{q,v}(x)] \cdot \Prob_{q \sim \iota_n}[P_{q,u}(x) \neq P_{q,v}(x)] \\
& \geq & (1- 1/n^d) \cdot (1-2\max\{\epsilon_0(n),\epsilon_1(n)\} n^2) \\
& \ge & 1 - (2\max\{\epsilon_0(n),\epsilon_1(n)\} n^2 +1/n^d) = 1-\zeta_0(n)
\end{eqnarray*}
This proves the theorem.
\end{proof}
We now deal with wreath products $\mathbb{Z}_{p^k} \wr G$ with $p$ prime and $k \geq 1$.
Note that the ring $\mathbb{Z}_{p^k}$ of integers modulo $p^k$ is not a ring unless $k = 1$.

\begin{theorem} \label{thm-wreath-Zpk}
Let $G$ be an f.g.~infinite group and $\mathcal{R} = (\mathcal{A}_n)_{n \ge 0}$ an $(\epsilon_0,\epsilon_1)$-distinguisher for $G$.
Let $d$ be a fixed constant and define 
\begin{eqnarray*}
\zeta_0(n) & = & 2\max\{\epsilon_0(n),\epsilon_1(n)\}n^2+ 1/n^d, \\
\zeta_1(n) & = & (2n^2+1) \epsilon_1(n) .
\end{eqnarray*}
Then there is
a $(\zeta_0,\zeta_1)$-distinguisher $\mathcal{R}'$ for $\mathbb{Z}_{p^k} \wr G$
with space complexity 
\[
  s(\mathcal{R}', n) = \mathcal{O}(s(\mathcal{R},n) + \log n).
\]
\end{theorem}

\begin{proof}
The proof is mostly identical to the proof of Theorem~\ref{thm-wreath-main} and we will reuse most of the notation. 
We consider the polynomial $P_{q,w}(x)$ 
from \eqref{eq-poly-P_qw} as a polynomial in~$\mathbb{Z}_{p^k}[x]$. 
Instead of evaluating it at a randomly chosen prime, we will compute its remainder with respect to a randomly chosen monic polynomial $\phi(x) \in \mathbb{Z}_{p^k}[x]$ of sufficiently low degree.
Note that if $\phi(x) = x - a$ has degree one, then computing this remainder is the same as evaluating $P_{q,w}(x)$ at the unique root $a$ of $\phi(x)$.

For a positive integer $m$, which we specify later, let us denote by $\Phi_m$ the set of all monic polynomials $\phi(x)$ of degree $m$ in $\mathbb{Z}_{p^k}[x]$ such that, firstly, $\phi(x)$ is irreducible as a polynomial in $\mathbb{Z}_p[x]$ and, secondly, all coefficients of $\phi(x)$ are in $[0,p-1]$.
Thus, the canonical homomorphism $\mathbb{Z}_{p^k}[x] \to \mathbb{Z}_{p}[x]$ (reducing all coefficients modulo~$p$) induces a bijection from $\Phi_m$ to the set of all monic irreducible polynomials of degree~$m$ in $\mathbb{Z}_p[x]$.
In particular,
\begin{equation}
  \lvert \Phi_m \rvert = \frac 1 m \sum_{m' | m} \mu(m') p^{m / m'} = \frac {p^m} m \cdot \big(1 + \mathcal{O}(m \cdot p^{-m/2})\big)
  = \frac {p^m} m (1+o(1))
  \label{eqn-wreath-Zpk-size}
\end{equation}
where $\mu$ denotes the Möbius function, meaning that $\mu(m') = (-1)^i$ if $m'$ is the product of $i$ distinct primes and $\mu(m') = 0$ otherwise; see e.g.\ \cite[Exercise~V.Ex.22]{Lang02}.

We claim that if $P(x) \in \mathbb{Z}_{p^k}[x]$ is a nonzero polynomial of degree at most~$D$, then a polynomial $\phi(x) \in \Phi_m$ chosen uniformly at random divides $P(x)$ with probability
\begin{equation}
  \Prob_{\phi(x) \in \Phi_m} \big[ \phi(x) \text{ divides } P(x) \text{ in } \mathbb{Z}_{p^k}[x] \big] \leq \frac D m \cdot \frac 1 {\lvert \Phi_m \rvert} 
  \stackrel{\eqref{eqn-wreath-Zpk-size}}{=} \frac D {p^m} (1+o(1)).
  \label{eqn-wreath-Zpk-prob}
\end{equation}
Indeed, for $k = 1$ this follows from the observation that at most $D/m$ distinct monic irreducible polynomials of degree $m$ can divide $P(x)$.
For $k > 1$ this follows from an argument by Agrawal and Biswas~\cite[Lemma~4.7]{AgrBis03}:
As $P(x) \neq 0$ in~$\mathbb{Z}_{p^k}[x]$, there is an exponent $e$ with $0 \leq e < k$ such that, firstly, all coefficients of $P(x)$ are divisible by $p^e$ and, secondly, $\tilde P(x) = p^{-e} P(x)$ satisfies $\tilde P(x) \neq 0$ in $\mathbb{Z}_p[x]$.
Further, if $\phi(x)$ divides $P(x)$ in $\mathbb{Z}_{p^k}[x]$, then $\phi(x)$ also divides $\tilde P(x)$ in $\mathbb{Z}_{p}[x]$.
Therefore, the estimate in \eqref{eqn-wreath-Zpk-prob} for $k > 1$ follows from the estimate for $k = 1$.

For input length $n$ the algorithm now proceeds as follows.
In the initialization phase we choose $\phi(x) \in \Phi_m$ uniformly at random, where $m = \log_p {\lvert Q_n \rvert} + d \cdot \log_p(n) + 1$
and $Q_n$ is the set of states of the semiPFA $\mathcal{A}_n$ from the distinguisher $\mathcal{R} = (\mathcal{A}_n)_{n \ge 0}$ for $G$.
We then keep track of the remainder $P_{q,w}(x) \bmod \phi(x)$, which is a polynomial of degree at most $m-1$.
As $P_{q,w}(x)$ has degree less than $D = \lvert Q_n \rvert$, we obtain the following bound as a direct consequence of \eqref{eqn-wreath-Zpk-prob}: for all words $u,v \in \Gamma^{\leq n}$, 
\begin{multline*}
  \smash{\Prob_{\substack{\phi(x) \in \Phi_m \\ q \sim \iota_n}}} \big[
    P_{q,u}(x) \bmod \phi(x) = P_{q,v}(x) \bmod \phi(x) \mid P_{q, u}(x) \neq P_{q, v}(x) 
  \big] \\
  \leq  \  \frac{|Q_n|}{p^m} (1+o(1)) = \frac{1}{n^{d} \cdot p} (1+o(1))
  \le \frac 1 {n^d}
\end{multline*}
for $n$ large enough.
The error bounds of the algorithm can then be computed as in the proof of Theorem~\ref{thm-wreath-main}. 
The algorithm stores the values $\phi(x)$, $q$ and $P_{q,w}(x) \bmod \phi(x)$ which requires 
$s(\mathcal{R},n) + \log_2 \lvert \Phi_m \rvert + m \cdot \log_2 p^k = \mathcal{O}(s(\mathcal{R},n) + \log n)$ bits in total.
\end{proof}

Note that in Theorems~\ref{thm-wreath-main} and \ref{thm-wreath-Zpk}, $\epsilon_1 = 0$ implies 
$\zeta_1 = 0$, that is, a one-sided error is preserved.

\begin{corollary} \label{thm-wreath-by-abelian}
Let $G$ be a finitely generated group for which there exists an $(\epsilon_0,\epsilon_1)$-distinguisher $\mathcal{R}$.
Let $A$ be a finitely generated abelian group and $d \ge 1$ a fixed constant.
Then there is a constant $\alpha$ and a $(\zeta_0,\zeta_1)$-distinguisher $\mathcal{S}$
 for $A \wr G$ with 
 \begin{eqnarray}
\zeta_0(n) &=& 2\max\{\epsilon_0(n),\epsilon_1(n)\} n^2 + 1/n^d, \label{eq-epsilon'}\\
\zeta_1(n) &=& \alpha(2 n^2+1)\epsilon_1(n), \nonumber \\
 s(\mathcal{S},n) & \leq & \mathcal{O}(s(\mathcal{R},n) + \log n). \nonumber
\end{eqnarray}
\end{corollary}

\begin{proof}
The wreath product $(H_1 \times H_2) \wr G$ is a subgroup of $(H_1 \wr G) \times (H_2 \wr G)$ \cite[Lemma~6.2]{KonigL18}. Since $A$ is a direct product
of copies of $\mathbb{Z}$ and finite cyclic groups $\mathbb{Z}_{p^k}$ for $p$ a prime and $k \ge 1$, the statement of the theorem follows from 
Lemma~\ref{thm-direct-prod} and Theorems~\ref{thm-wreath-main} and \ref{thm-wreath-Zpk}. The constant $\alpha$ is the number
of factors $\mathbb{Z}$ and $\mathbb{Z}_{p^k}$ in the direct product decomposition of $A$. 
\end{proof}
Corollary~\ref{thm-wreath-by-abelian} can be slightly extended to the case, where the abelian group $A$ is replaced 
by a finitely generated commutative and cancellative monoid:

\begin{corollary} \label{thm-wreath-cancel-comm}
Let $G$ be a finitely generated group for which there exists an $(\epsilon_0,\epsilon_1)$-distinguisher $\mathcal{R}$.
Let $M$ be a finitely generated commutative and cancellative monoid and $1/n^d$ a fixed constant.
Then there is a constant $\alpha$ and a $(\zeta_0,\zeta_1)$-distinguisher $\mathcal{S}$
 for $M \wr G$ with 
 \begin{eqnarray*}
\zeta_0(n) &=& 2\max\{\epsilon_0(n),\epsilon_1(n)\} n^2 + 1/n^d, \\
\zeta_1(n) &=& \alpha(2 n^2+1)\epsilon_1(n), \nonumber \\
 s(\mathcal{S},n) & \leq & \mathcal{O}(s(\mathcal{R},n) + \log n). \nonumber
\end{eqnarray*}
\end{corollary}

\begin{proof}
The monoid $M$ embeds into a finitely generated abelian group $A$ \cite[Proposition3.2]{GrilletComm}. Hence, $M \wr G$ embeds into $A \wr G$. The result 
follows from Corollary~\ref{thm-wreath-by-abelian}.
\end{proof}
We can also apply Corollary~\ref{thm-wreath-by-abelian} to free solvable groups (the free objects in the variety of solvable groups).

\begin{corollary}  \label{thm-free-solvable}
Every free solvable group has a $(1/n^d,0)$-distinguisher with space complexity $\Theta(\log n)$ for every $d \geq 1$.
\end{corollary}

\begin{proof}
Magnus' embedding theorem \cite{Mag39} says that every free solvable group can be embedded into an iterated wreath product 
$\mathbb{Z}^m \wr ( \mathbb{Z}^m \wr (\mathbb{Z}^m \wr \cdots ))$. Since $\mathbb{Z}^m$ is linear, we can, using Theorem~\ref{thm-lin},
obtain an $(\epsilon_0(n), 0)$-distinguisher for $\mathbb{Z}^m$ with space complexity $\mathcal{O}(\log n)$ for 
every inverse polynomial $\epsilon_0(n)$. We then apply Corollary~\ref{thm-wreath-by-abelian} a constant number
of times and obtain a $(1/n^d,0)$-distinguisher with space complexity $\Theta(\log n)$
\end{proof}

\begin{remark}
In \cite{Ush14} it is shown that the word problem of a free solvable group can be solved with
a randomized algorithm running in time $\mathcal{O}(n \cdot (\log n)^k)$ for some constant $k$. 
Our algorithm achieves the same running time (because for every new input symbol, only numbers 
of bit length $\mathcal{O}(\log n)$ have to be manipulated).
In contrast to our algorithm, the algorithm from \cite{Ush14} is non-streaming and does not work in logarithmic space.
\end{remark}

\subsection{Distinguishers for fundamental groups of graphs of groups}
\label{sec-fundamental}

Graph of groups (not to confuse with graph groups) are another important construction in group theory.
We present the definition of a graph of groups and its fundamental groups, although it is not needed for
our considerations.

In this section, by a graph $Y$, we mean a graph in the sense of Serre~\cite{Serre03}, where loops and multiple edges between vertices 
are allowed. Formally, it is a tuple $Y = (V,E,\alpha,\tau,{}\inv)$ such that:
\begin{itemize}
\item $V$ is a set vertices,
\item $E$ is a set of edge,
\item $\alpha\colon E\to V$ and $\omega\colon E\to V$ map an edge to its initial vertex and terminal vertex, respectively, and
\item ${}\inv : E \to E$ is a function such that for all edges $e \in E$ we have:
$(e^{-1})^{-1} = e$, $e \neq e^{-1}$ (thus, ${}\inv$ is a fixpoint-free involution) and
$\alpha(e) = \omega(e^{-1})$. 
\end{itemize}
 A \emph{path} in $Y$ is a word $\pi = e_1 e_2 \cdots e_k \in E^+$ such that
$\omega(e_i) = \alpha(e_{i+1})$ for all $1 \leq i < k$.  We say that $\pi$ is a path from 
$\alpha(e_1)$ to $\omega(e_k)$. 
If there is no $i \in [1,k-1]$ with $e_i = e_{i+1}^{-1}$ (i.e., if $\pi \in \IRR(E)$) then
$\pi$ is \emph{reduced}.
The graph $Y$ is \emph{connected} if for all $v,v' \in V$ with $v \neq v'$ there is a path 
from $v$ to $v'$. 
A \emph{spanning tree} of $Y$ is a subset $T \subseteq E$ such that the graph
$(V,T, \alpha\rest_T, \tau\rest_T, {}\inv\rest_T)$ has a unique reduced path from $v$ to $v'$ for all 
$v, v' \in V$, $v \neq v'$.
If $Y$ is connected then it has a spanning tree.

A \emph{graph of groups} $(G,Y)$  consists of a connected graph $Y = (V,E)$ and
\begin{itemize}
\item for every vertex $v \in V$, a group $G_v$,
\item for every edge $e \in E$, a group $G_e$ such that $G_e = G_{e\inv}$,
\item for every edge $e \in E$, injective homomorphisms $\alpha_e\colon G_e \to G_{\alpha(e)}$ and
$\omega_e \colon G_e \to G_{\omega(e)}$ such that $\alpha_e = \omega_{e\inv}$ for
all $e \in E$.
\end{itemize}
Consider a graph of groups $(G,Y)$.
We will only consider the case where the graph $Y = (V,E)$ is finite, i.e., $V$ and $E$ are both finite.
Choose a subset $\Delta \subseteq E$ that contains from every set $\{e, e^{-1}\}$ ($e \in E$) exactly one edge.
Let $T \subseteq E$ be a spanning tree of the graph $Y$, which exists since $Y$ is connected.
The \emph{fundamental  group} $\pi_1(G,Y,T)$ is obtained from the free product 
$H = F(\Delta) \ast \Asterisk_{v \in V} G_v$ by imposing the relations $e=1$ for all $e \in T$ and 
$e^{-1}\alpha_e(g) \, e = \omega_e(g)$ for all $e \in E$ and $g \in G_e$. More formally, it is the quotient
group $H/N$ where  $N$ is the normal closure of the set 
$T \cup  \{  e^{-1}\alpha_e(g) \, e \, \omega_e(g)^{-1} \colon e \in E, g \in G_e\}$.
It can be shown that if $T$ and $T'$ are two different spanning trees of $Y$ then 
 $\pi_1(G,Y,T)$ and  $\pi_1(G,Y,T')$ are isomorphic. Therefore, one just writes 
 $\pi_1(G,Y)$, which is  the fundamental group of $(G,Y)$. All vertex groups and edge groups canonically 
 embed into $\pi_1(G,Y)$.
 
 The graph of groups construction is of fundamental importance in the theory of groups acting on trees \cite{Serre03}.
 Important special cases of this construction are HNN-extension (where $Y$ consists of a single vertex $v$ and a single 
 edge $e$ with $\alpha(e) = \omega(e)=v$) 
 and amalgamated free products (where $Y$ consists of two vertices $v_0$, $v_1$ and a single 
 edge $e$ with $\alpha(e) = v_0$, $\omega(e)=v_1$); see \cite{Serre03} for more details.
 
Another interesting special case of graph of groups arises when all edge groups $G_e$ are finite.
In this case many positive algorithmic properties (e.g., decidability of the word problem) 
transfer from the vertex groups to the fundamental group. 
We do not know whether a corresponding transfer result also holds for groups having
a logspace distinguisher, but with an additional restriction we obtain such a transfer result directly from our previous results.
A group $G$ is called \emph{residually finite} if for every element $g \in G \setminus \{1\}$ there is a homomorphism $\phi : G \to H$
 from $G$ to a finite group $H$ such that $\phi(g) \neq 1$.

\begin{theorem}  \label{thm-fundamental-group}
Let $G$ be the fundamental group of a graph of groups with finite edge groups and f.g.~residually finite vertex groups $G_v$ ($v \in V$).
Let $\mathcal{R}_v$ be an $(\epsilon_0,\epsilon_1)$-distinguisher for $G_v$. Let $c = |V|$, $d \geq 1$ be a constant, and define 
\begin{eqnarray*}
\zeta_0(n) &=& 2\epsilon_0(n) (n+c)^2 + 1/n^d, \\
\zeta_1(n) &=& 2\epsilon_1(n) (n+c)^2.
\end{eqnarray*} 
Then, there is a constant $\gamma$ and a $(\zeta_0(\gamma n),\zeta_1(\gamma n))$-distinguisher for the fundamental group $G$
with space complexity $\mathcal{O}(\sum_{v \in V} s(\mathcal{R}_v, \gamma n) + \log n)$.
\end{theorem}

\begin{proof}
Let $(G,Y)$ be a graph of groups with finite edge groups and  f.g.~residually finite vertex groups $G_1, \ldots, G_c$.
From the construction of the fundamental group it follows that $\pi_1(G,Y)$ is f.g.~too. Moreover, $\pi_1(G,Y)$ is a finite extension
of a free product of (i) an f.g.~free group and (ii) finitely many groups $g^{-1} H g$ for $g \in \pi_1(G,Y)$ and $H$ an f.g.~subgroup of a vertex group $G_v$; see e.g. \cite[proof of Lemma~8]{LoSt08}.
 The result follows from Theorems~\ref{thm-subgroup}, \ref{thm-finite-ext} and \ref{thm-GP} (the latter is only used for free products).
\end{proof}
All the groups for which we can construct logspace distinguishers with our results from Section~\ref{sec-C-transfer}
are  residually finite. This is due to the following results:
\begin{itemize}
\item Every f.g.~linear group is residually finite \cite{Mal40}.
\item Every subgroup of a residually finite group is residually finite (trivially true).
\item Every finite extension of a residually finite groups is residually finite; 
this seems to be folklore; see e.g.~\cite[Proposition 2.2.12]{Ceccherini-SilbersteinC09}. 
\item A graph product of f.g.~residually finite groups is f.g.~residually finite \cite{HsWi99}.
\item A wreath product $H \wr G$ of f.g.~groups $G,H$ is 
residually finite if and only if (i) either $G$ is finite and $H$ is residually finite 
or (ii) $H$ is abelian and $G$ is residually finite \cite{Grue57}.
\end{itemize}
Also the fundamental group of a graph of groups with finite edge groups and 
residually finite vertex groups is residually finite \cite[II.2.6 Proposition 12]{Serre03}.

\begin{oproblem}
Is Theorem~\ref{thm-fundamental-group} also true without the restriction to residually finite vertex groups?
Decidability of the word problem (as well as many other algorithmic properties)  is preserved by graph of groups with finite edge groups.
\end{oproblem}

\subsection{Distinguishers for semilattice decompositions} \label{sec-transfer}

Our final transfer theorem for distinguishers concerns an important decomposition technique for semigroups.

A \emph{semilattice} is a commutative semigroup $Y$ such that $\alpha^2 = \alpha$ for all $\alpha \in Y$.
Note that every finitely generated semilattice is finite.
For a semilattice $Y$ there is a partial order $\leq$ on $Y$ defined by $\alpha \beta \leq \alpha$ for all $\alpha, \beta \in Y$.
Note that $Y$ is a monoid if and only if $Y$ has a largest element with respect to this partial order $\leq$.

A \emph{semilattice decomposition} of a semigroup $S$ consists of a partition $S = \biguplus_{\alpha \in Y} S_\alpha$ with the following properties:
\begin{itemize}
\item $Y$ is a semilattice.
\item For all $\alpha, \beta \in Y$ we have $S_\alpha S_\beta \subseteq S_{\alpha\beta}$ (where $\alpha\beta$ is the product of $\alpha$ and $\beta$ in the semilattice
$Y$).
\end{itemize}
Note that the second property implies that every $S_\alpha$ is a subsemigroup of $S$: $S_\alpha S_\alpha \subseteq S_{\alpha\alpha} = S_\alpha$.
We speak of a \emph{strong semilattice decomposition} if in addition 
 there exist semigroup homomorphisms $h_{\alpha, \beta} : S_\alpha \to S_\beta$ for all $\beta \leq \alpha$ with the following properties:
\begin{enumerate}[(i)]
\item \label{strong1} $h_{\alpha,\alpha}$ is the identity mapping on $S_\alpha$,
\item \label{strong2} $h_{\alpha,\beta} \circ h_{\beta, \gamma} = h_{\alpha,\gamma}$ (in $h_{\alpha,\beta} \circ h_{\beta, \gamma}$, we first apply $h_{\alpha,\beta}$ followed by $h_{\beta, \gamma}$),
\item \label{strong3} for all $\alpha, \beta \in Y$, $a \in S_{\alpha}$ and $b \in S_\beta$ we have $ab = h_{\alpha, \alpha\beta}(a) h_{\beta,\alpha\beta}(b)$ in $S$.
 \end{enumerate}
 Assume that $S = \biguplus_{\alpha \in Y} S_\alpha$ is a semilattice decomposition.
For $\alpha \in Y$ we define 
$S_{\alpha}^{\uparrow} = \bigcup_{\beta \geq \alpha} S_\beta$. Note that $S_{\alpha}^{\uparrow}$ is a subsemigroup of $S$ as well
(since $\beta_1 \geq \alpha$ and $\beta_2 \ge \alpha$ imply $\beta_1\beta_2 \ge \alpha$).
If the semilattice decomposition $S = \biguplus_{\alpha \in Y} S_\alpha$ is 
moreover strong, one can define a homomorphism $r_\alpha : S_{\alpha}^{\uparrow} \to S_\alpha$ by taking the union of all $h_{\beta,\alpha}$
for $\beta \ge \alpha$.
This mapping is indeed a homomorphism, since for all $a_1 \in S_{\beta_1}$ and $a_2 \in S_{\beta_2}$ with $\beta_1  \geq \alpha \le \beta_2$
we have 
\begin{eqnarray*}
r_\alpha(a_1) r_\alpha(a_2) & = & h_{\beta_1,\alpha}(a_1) h_{\beta_2,\alpha}(a_2) \\
&  \stackrel{\text{\eqref{strong2}}}{=} &  h_{\beta_1\beta_2, \alpha}( h_{\beta_1, \beta_1\beta_2}(a_1)) h_{\beta_1\beta_2, \alpha}( h_{\beta_2, \beta_1\beta_2}(a_2))\\
&  = &  h_{\beta_1\beta_2, \alpha}( h_{\beta_1, \beta_1\beta_2}(a_1)  h_{\beta_2, \beta_1\beta_2}(a_2)) \\
& \stackrel{\text{\eqref{strong3}}}{=} &  h_{\beta_1\beta_2,\alpha}(a_1 a_2)  =  r_\alpha(a_1a_2).
\end{eqnarray*}
Note that by \eqref{strong1}, every homomorphism $r_\alpha : S_{\alpha}^{\uparrow} \to S_\alpha$  is the identity mapping 
on $S_\alpha$. Such a homomomorphism is called a \emph{retraction}.\footnote{In general, for a semigroup $S$ and a subsemigroup
$T \subseteq S$, a homomorphism $h : S \to T$ is a retraction
if $h(x) = x$ for all $x \in T$.}

We define a \emph{retractive} semilattice decomposition as a semilattice decomposition $S = \biguplus_{\alpha \in Y} S_\alpha$
such that for every $\alpha \in Y$ there is a retraction $r_\alpha : S_{\alpha}^{\uparrow} \to S_\alpha$. We have just seen that
every strong semilattice decomposition is retractive. Vice versa, every retractive semilattice decomposition satisfies the properties
\eqref{strong1} and \eqref{strong3} for the homomorphisms $h_{\beta,\alpha} = r_\alpha\rest_{S_\beta}$ ($\beta \ge \alpha$), where
$r_\alpha\rest_{S_\beta}$ is the restriction of $r_\alpha$ to the subsemigroup $S_\beta \subseteq S_{\alpha}^{\uparrow}$.
Property \eqref{strong1} is clear, and \eqref{strong3} is obtained as follows for $a \in S_\alpha$, $b \in S_\beta$ (and hence $ab \in S_{\alpha\beta}$): 
\[
h_{\alpha, \alpha\beta}(a) h_{\beta,\alpha\beta}(b) = r_{\alpha\beta}(a) r_{\alpha\beta}(b) = r_{\alpha\beta}(ab) =  ab . 
\]

In a moment we will see an example for a retractive semilattice decomposition that does not satisfy \eqref{strong2}.
First, let us state the following well-known result.

 \begin{lemma}[\mbox{\cite[Theorem~II.2.6]{petrich:1984}}] \label{lemma-strong}
A semilattice decomposition $S = \biguplus_{\alpha \in Y} G_\alpha$, where every $G_\alpha$ is a group, is a strong semilattice decomposition.
 \end{lemma}

 Be aware that Lemma~\ref{lemma-strong} does not generalize to the case where the $G_\alpha$ are monoids, as the following example shows.
 
 \begin{example}\label{counterexample-strong-for-monoids}
 Consider the semilattice $Y = \{\alpha,\beta,\gamma\}$ with $\alpha > \beta > \gamma$ and the monoids $M_\alpha = \{ e_\alpha \}$, $M_\beta = \{ e_\beta, 0_\beta\}$, and
 $M_\gamma = \{ e_\gamma, 0_\gamma \}$, where $e_x$ is the neutral element of $M_x$ and $0_x$ is the zero element
 of $M_x$ (for $x \in \{\beta,\gamma\}$). Moreover, we set
 \begin{itemize}
  \item  $e_\alpha x  =  x e_\alpha = 0_\beta$ for all $x \in M_\beta$,
   \item   $e_\alpha x  =  x e_\alpha = x$ for all  $x \in M_\gamma$,
  \item   $x y  =  y x =  0_\gamma$ for all $x \in M_\beta, y \in M_\gamma$.
 \end{itemize}
 This defines a semigroup $S = \bigcup_{x \in \{\alpha,\beta,\gamma\}} M_x$ with a corresponding semilattice decomposition. 
 However, it is not a strong semilattice decomposition.
 Indeed, let us assume that there were homomorphisms
 $h_{\alpha,\beta}, h_{\alpha,\gamma}, h_{\beta,\gamma}$ with the properties \eqref{strong1}--\eqref{strong3}.
 We obtain with \eqref{strong1} and \eqref{strong3} that
 \begin{itemize}
 \item $0_\beta = e_\alpha e_\beta = h_{\alpha,\beta}(e_\alpha) e_\beta =  h_{\alpha,\beta}(e_\alpha)$,
\item $e_\gamma = e_\alpha e_\gamma = h_{\alpha,\gamma}(e_\alpha) e_\gamma =  h_{\alpha,\gamma}(e_\alpha)$, and
\item $0_\gamma = 0_\beta e_\gamma = h_{\beta,\gamma}(0_\beta) e_\gamma = h_{\beta,\gamma}(0_\beta)$.
\end{itemize}
We therefore get with \eqref{strong2} that
 \[
e_\gamma =  h_{\alpha,\gamma}(e_\alpha) = h_{\beta,\gamma}(h_{\alpha,\beta}(e_\alpha)) = h_{\beta,\gamma}(0_\beta) = 0_\gamma,
 \]
 which is a contradiction.
 \end{example}

On the other hand, we have the following result.

\begin{lemma} \label{lemma-monoids-retractive}
A semilattice decomposition $S = \biguplus_{\alpha \in Y} M_\alpha$, where every $M_\alpha$ is a monoid, is a retractive semilattice decomposition.
\end{lemma}

\begin{proof}
Let $e_\alpha$ be the neutral element of $M_\alpha$.
For $\alpha \in Y$ define the mapping $r_\alpha : S_\alpha^\uparrow \to S_\alpha$ by
$r_\alpha(a) = a e_\alpha$ for $a \in S_\alpha^\uparrow$.
Hence, we have $r_\alpha(a)=a$ for all $a \in S_\alpha$.
Moreover, for all $a_1, a_2 \in S_{\alpha}^\uparrow$ we obtain
\begin{equation*}
r_{\alpha}(a_1 a_2) = (a_1 a_2) e_\alpha = a_1 (a_2 e_\alpha) = (a_1 e_\alpha) (a_2 e_\alpha) =
r_{\alpha}(a_1) r_{\alpha}(a_2),
\end{equation*} 
where the third equality follows from $a_2 e_\alpha \in S_\alpha$ and hence $a_2 e_\alpha = e_\alpha a_2 e_\alpha$.
\end{proof}
By Lemma~\ref{lemma-monoids-retractive}, the semilattice decomposition from Example~\ref{counterexample-strong-for-monoids}
is an example of a retractive semilattice decomposition which is not strong.

For our transfer result for retractive semilattice decomposition, we need the following lemma.
 \begin{lemma} \label{lemma-f.g-semilattice-monoids}
 If $S = \biguplus_{\alpha \in Y} S_\alpha$ is a retractive semilattice decomposition 
 and $S$ is finitely generated, 
 then $Y$ is finite and every $S_\alpha$ is finitely generated.
 \end{lemma}
 
 \begin{proof}
 Note that the equivalence relation $\rho$ on $S$ with $(a,b) \in  \rho$ if and only if there is an $\alpha\in Y$ with $a,b \in S_\alpha$
 is a congruence on $S$ such that $Y = S/\rho$. Hence, $Y$ is a finitely generated semilattice and therefore finite (this argument holds for 
 arbitrary semilattice decompositions).
 
 Assume now that $\Sigma$ is a finite generating set of $S$. 
 We claim that for every $\alpha \in Y$, $S_\alpha$ is generated by the finite set 
 \[
 \Sigma_\alpha = r_\alpha(\Sigma \cap  S_\alpha^\uparrow) \subseteq S_\alpha.
 \]
To show this, consider an element $a \in S_\alpha$
 and write it as $a = a_1 a_2 \cdots a_n$ for $n \geq1$ and $a_i \in \Sigma$. 
 First assume that some $a_i$ does not belong to $S_\alpha^\uparrow$, i.e., 
 $a_i \in S_\beta$ with $\beta \not\geq \alpha$. 
 Then $a$ must belong to some $S_\gamma$ with $\gamma \leq \beta$. In particular, we have $\gamma \neq \alpha$
 which is a contradiction. 
 It follows that $a_i \in \Sigma \cap S_\alpha^\uparrow$ and hence $r_\alpha(a_i) \in \Sigma_\alpha$ for all $i$.
Finally, since $r_\alpha$ is a retraction and $a \in S_\alpha$, we have
 \begin{equation*}
 a = r_\alpha(a) = r_\alpha(a_1 a_2 \cdots a_n) = r_\alpha(a_1)  r_{\alpha}(a_2) \cdots r_{\alpha}(a_n),
 \end{equation*}
 which proves the lemma.
 \end{proof}

 \begin{theorem} \label{thm-semilattice-deco}
 Let $S = \biguplus_{\alpha \in Y} S_\alpha$ be a retractive semilattice decomposition with $S$ finitely generated. 
 Let $c = |Y| < \infty$ and assume that every 
  $S_\alpha$ has an  $(\epsilon_0(n), \epsilon_1(n))$-distinguisher
 with space complexity bounded by $s(n)$. 
 Then, there is a constant $d$ such that
 $S$ has an $(\epsilon_0(dn), c \cdot \epsilon_1(dn))$-distinguisher
 with space complexity bounded by $c \cdot s(dn)$.
 \end{theorem}
 
 \begin{proof}
 The distinguisher for $S$ works similarly to the one for a commutative semigroup; see the proof of  Theorem~\ref{thm-comm}.
By Lemma~\ref{lemma-f.g-semilattice-monoids}, $Y$ is finite and for every 
$\alpha \in Y$ we can fix a finite generating set $\Sigma_\alpha$ for the semigroup  $S_\alpha$. Let
$\mathcal{R}_\alpha = (\mathcal{A}_{\alpha,n})_{n \ge 0},$ be  an  $(\epsilon_0(n), \epsilon_1(n))$-distinguisher  for $S_\alpha$ such that $\mathcal{A}_{\alpha,n} = (Q_{\alpha,n}, \Sigma_\alpha, \iota_{\alpha,n}, \delta_{\alpha,n})$. Its space complexity is $\lceil \log_2 |Q_{\alpha,n}| \rceil \leq s(n)$.  Then $\Sigma = \bigcup_{\alpha \in Y} \Sigma_\alpha$ is a finite generating set for $S$. 
Since we have a retractive semilattice decomposition, there exist retractions 
$r_{\alpha} : S_\alpha^\uparrow \to S_\alpha$. 
In the following, we identify for every generator $a \in S_\beta$ and every $\alpha \leq \beta$ the
element $r_{\alpha}(a) \in S_{\alpha}$ with a word over the generating set $\Sigma_{\alpha}$. Let $d$
be the maximal length of these words for all $a \in S_\beta$ and all $\alpha \leq \beta$.
Note that $d$ is a constant.

 \begin{algorithm}[t]
 \SetKwComment{Comment}{(}{)}
\SetKwInput{KwGlobal}{global variables}
\SetKwInput{KwInit}{initialization}
\SetKwInput{KwNext}{next input letter}
\BlankLine
\KwGlobal{$q_\alpha \in Q_{\alpha,dn}$ for every $\alpha \in Y$, $\chi \in Y^1$} 
\BlankLine
\KwInit{} 
\For{$\alpha \in Y$}{
guess $q_\alpha$ according to the distribution $\iota_{\alpha,dn}$ \label{line-init-q-alpha}}
$\chi  \coloneqq 1$ \Comment*[r]{the neutral element of the monoid $Y^1$} \label{line-chi-init}
\BlankLine
\KwNext{$a \in \Sigma$}
let $\alpha \in Y$ such that $a \in \Sigma_\alpha$ \\
$\chi  \coloneqq\chi \alpha$ \label{line-chi}  \\
\For{$\beta \in Y$}{
    \uIf{$\beta \leq \chi$}    
          {$q_\beta  \coloneqq \delta_{\beta, dn}(q_\beta, r_{\beta}(a))$ \label{line-q-beta-new}}
    \Else{$q_\beta  \coloneqq q_\beta^*$ \label{line-q-beta-default}}}
   
\BlankLine

\caption{An $(\epsilon_0(dn), c \cdot \epsilon_1(dn))$-distinguisher for a retractive semilattice of semigroups $S = \biguplus_{\alpha \in Y} S_\alpha$.
\label{algo-semilattice} }
\end{algorithm}
 
Fix the input length $n$. 
Our distinguisher for $S$ is $\mathcal{R}=(\mathcal{A}_n )_{n\geq 0}$ where $\mathcal{A}_n$ is the  semiPFA implicitly defined by Algorithm \ref{algo-semilattice}.  For every $\alpha \in Y$ we fix an arbitrary default state $q_\alpha^{*}\in Q_{\alpha,dn}$.  
The variable $\chi \in Y^1$ can be stored with a constant number of bits and 
the total space needed for the variables $q_\alpha \in Q_{\alpha,dn}$ ($\alpha \in Y$) is
bounded by $c \cdot s(dn)$.
Note that in case $Y$ is not a monoid, the adjoined $1$ in $Y^1$ is only needed for the initialization of $\chi$ in line~\ref{line-chi-init}.

Let $w\in \Sigma^+$ be a non-empty input word of length at most $n$ and let $q_{\alpha,0}$ ($\alpha \in Y$) be the initial  states guessed in line~\ref{line-init-q-alpha} from Algorithm \ref{algo-semilattice}. Let $s \in S$ be the value of the word $w$ in the semigroup $S$.
We claim that after reading $w$, the values of the program variables $\chi$ and $q_\beta$ satisfy the following:
\begin{itemize}
\item $\chi \in Y$ is such that $s \in S_\chi$,
\item if $\beta \not\le \chi$ then $q_\beta = q_\beta^*$, and
\item if $\beta \le \chi$ then $q_\beta = \delta_{\beta,dn}(q_{\beta,0}, w_\beta) \in Q_{\beta, dn}$, where
$w_\beta \in \Sigma_\beta^+$ is a word of length at most $d |w| \le dn$, whose value is $r_{\beta}(s)$ in the
 semigroup $S_\beta$. 
 \end{itemize}
 These invariants hold after reading the first input letter. Now
 assume that they hold after reading $w \in \Sigma^+$ with $|w| < n$ and let 
 $a \in \Sigma_\alpha$ be the next input letter. The value of $wa$ in $S$ is $sa$. Moreover, 
 we have $sa \in S_{\chi \alpha}$ and $\chi\alpha$ will be the new value of the variable $\chi$ (line~\ref{line-chi}).
 For every $\beta \not\le \chi \alpha$, $q_\beta$ will be set to $q_\beta^*$ (line~\ref{line-q-beta-default}).
 Finally, if $\beta \le \chi\alpha$, then the new value of  $q_\beta$ will be
  $\delta_{\beta, dn}(q_\beta, r_{\beta}(a))$ (line~\ref{line-q-beta-new}) and 
  $w_\beta r_{\beta}(a) \in \Sigma_\beta^+$ is a word of length at most $d(|w|+1)\le dn$ representing the semigroup element 
  $r_{\beta}(s) r_{\beta}(a) = r_{\beta}(sa)$.
  
In order  to analyze the error probability of $\mathcal{R}$, we consider two input words $w, w'\in \Sigma$ of length at most $n$.
As before, let  $q_{\alpha,0}$ ($\alpha\in Y$) be the randomly guessed states from line~\ref{line-init-q-alpha}.
Let $p_\alpha \in Q_{\alpha,dn}$ ($\alpha \in Y$) and $\zeta \in Y$ be the values of  the program variables $q_\alpha$ and $\chi$, respectively,  after reading $w$ and define $p'_\alpha$ and $\zeta'$ with respect to $w'$ analogously. Let $s$ (resp., $s'$) be the value in $S$ of the word
$w$ (resp., $w'$). We then have $s \in S_{\zeta}$ and $s' \in S_{\zeta'}$.
By the above discussion, for every $\alpha \in Y$ the following hold:
\begin{itemize}
\item if $\alpha \le \zeta$,  then $p_\alpha = \delta_{\alpha,dn}(q_{\alpha,0}, w_\alpha)$
for a word $w_\alpha \in \Sigma_\alpha^+$ of length at most $dn$ representing the semigroup element
$r_{\alpha}(s) \in S_\alpha$, and
\item if $\alpha \not\le \zeta$ then $p_\alpha = q_\alpha^{*}$.
\end{itemize}
Analogous statements hold for words $w'_\alpha \in \Sigma_\alpha^+$ ($\alpha \le \zeta'$) with respect to the states $p'_\alpha$
and the semigroup element $s'$.

\smallskip
\noindent
\emph{Case 1:} $w \equiv_S w'$, i.e., $s = s'$.
We then have $\zeta = \zeta'$ and $p_\alpha = q_\alpha^{*} = p'_\alpha$ for all 
$\alpha \not\le \zeta$ (with probability one). Moreover, for every 
$\alpha \le \zeta$ we have $r_{\alpha}(s) = r_{\alpha}(s')$ and hence 
$w_\alpha \equiv_{S_\alpha} w'_\alpha$. Since $\mathcal{R}_{\alpha}$ is an $(\epsilon_0(n), \epsilon_1(n))$-distinguisher for $S_\alpha$,
we obtain
\[
\Prob_{q_{\alpha,0} \sim \iota_{\alpha,dn}}[p_\alpha \neq p'_\alpha] = \Prob_{q_{\alpha,0} \sim \iota_{\alpha,dn}}[ \delta_{\alpha,dn}(q_{\alpha,0}, w_\alpha) \neq  \delta_{\alpha,dn}(q_{\alpha,0}, w'_\alpha)] \leq \epsilon_1(dn).
\]
 The union bound shows that with probability at least $1- c \cdot \epsilon_1(dn)$ we have
$p_\alpha = p'_\alpha$ for all $\alpha \le \zeta$.

\smallskip
\noindent

\smallskip
\noindent
\emph{Case 2:} $w \not\equiv_S w'$, i.e., $s \neq s'$.
If $s$ and $s'$ belong to different sets $S_\alpha$ then we obtain $\zeta \neq \zeta'$ with probability one.
Assume now that $\zeta = \zeta'$. 
Since $w_\zeta \in \Sigma_\zeta^+$ is a word of length at most $dn$ representing the semigroup element
$r_{\zeta}(s) = s$ and similarly for $w'_\zeta$, we get
$w_\zeta \equiv_S w  \not\equiv_S w' \equiv_S w'_\zeta$, i.e., $w_\zeta \not\equiv_{S_\zeta} w'_\zeta$.
Since $\mathcal{R}_{\zeta}$ is an $(\epsilon_0(n), \epsilon_1(n))$-distinguisher for $S_{\zeta}$
we obtain $p_\zeta \neq p'_\zeta$ with probability at least $1- \epsilon_0(dn)$.
 \end{proof}
 Let us give two applications of Theorem~\ref{thm-semilattice-deco}.
 Recall the notion of a nilpotent semigroup from Section~\ref{sec-nilpotent} 
 and the notion of a regular semigroup from Section~\ref{sec-regular-semigroup}.

 \begin{corollary}  \label{thm-regular-nilpotent}
 Every f.g.~regular nilpotent semigroup has a $((\log n)^{-c},0)$-distinguisher 
with space complexity $\mathcal{O}(\log \log n)$ for every constant $c > 0$.
 \end{corollary}

\begin{proof}
By \cite{Lall72}, an f.g.~$i$-step nilpotent semigroup is regular if and only if it is a semi\-lattice 
of $i$-step nilpotent groups.
The result follows from Theorems~\ref{thm-nilpotent}~and~\ref{thm-semilattice-deco}.
\end{proof}
  A \emph{Clifford semigroup} $S$ is a semilattice of groups. Equivalently, it is an inverse semigroup $S$ where $ea = ae$ holds for 
 all $a \in S$ and idempotents $e \in S$ \cite[Thm.~II.2.6]{petrich:1984}. 
 For every set $\Sigma$ there is a free Clifford semigroup $\mathsf{FCS}(\Sigma)$
 generated by $\Sigma$,\footnote{Free semigroups can be defined for every variety of semigroups using the classical universal property.
The details of this are not needed here and the interested reader may consult e.g.\ \cite{Sapir14}.} 
which is a semilattice of f.g.~free groups \cite[p.~369]{petrich:1984}. 
Hence, by combining Theorem~\ref{thm-semilattice-deco} with Theorem~\ref{thm-lin} (and using the fact that f.g.~free groups are linear), we obtain the following.
 
 \begin{corollary} \label{thm-clifford}
Every f.g.~free Clifford semigroup has a $(1/n^c,0)$-distinguisher 
with space complexity $\mathcal{O}(\log n)$ for every constant $c$.
 \end{corollary}

\section{Part D: Lower bounds} \label{sec lower}

In this section, we will contrast the upper bounds from Section~\ref{sec part B} with lower bounds.
All lower bounds will derived from lower bounds in communication complexity; see 
Section~ \ref{sec-CC}.

\subsection{Inverse monoids}

We start with some lower bounds on the space complexity for distinguishers for inverse monoids.
Recall the definition of inverse semigroups from Section~\ref{sec-regular-semigroup}.

\subsubsection{Free inverse monoids}
  
 For every set $\Gamma$ there exists the \emph{free inverse monoid} $\FIM(\Gamma)$ generated by $\Gamma$.
 It can be constructed as follows. Consider the free group $F(\Gamma)$ generated by $\Gamma$
 and let $\Sigma = \Gamma \cup \Gamma^{-1}$. Then $\FIM(\Gamma) = \langle \Sigma \mid V\rangle^*$, where
 $V$ contains all pairs of the following form (the so-called Vagner equations):
 \begin{itemize}
 \item $(u u^{-1} u, u)$ for all $u \in \Sigma^*$,
 \item $(u u^{-1} v v^{-1}, v v^{-1} u u^{-1})$ for all $u,v \in \Sigma^*$.
 \end{itemize}
 Here, we make use of the convolution $u \mapsto u^{-1}$ ($u \in \Sigma^*$) that we defined in Section~\ref{sec-groups}.
 
 The above description of $\FIM(\Gamma)$ as $\langle \Sigma \mid V\rangle^*$ is not very useful for algorithmic purposes (e.g. for 
 solving the word problem). A more useful representation is based on so-called \emph{Munn trees}:
The Munn tree $\MT(u)$ of $u \in \Sigma^*$ is the following finite and prefix-closed
subset of $\IRR(\Sigma)$ (recall that $\mathcal{P}(u)$ is the set of all prefixes of $u$, including $\varepsilon$ and $u$):
\[  \MT(u) = \{  \red(v) \colon v \in \mathcal{P}(u)  \}. \]
Since the elements of the free group $F(\Gamma)$ can be identified with the elements from $\IRR(\Sigma)$,
$\MT(u)$ is the set of all nodes along the unique path in the Cayley graph of $F(\Gamma)$ 
that starts in $1$ and that is labeled with the word $u$. 
The subgraph of the Cayley-graph of $F(\Gamma)$, which is induced 
by $\MT(u)$ is connected. Hence it is a finite tree and we can identify
$\MT(u)$ with this tree. Figure~\ref{fig Munn tree} shows an example of a Munn tree.
Munn \cite{munn:1974} proved the following theorem:

\begin{figure}[t]
 \tikzset{alpha/.style={inner sep = 1pt, fill=white}}
 \tikzset{round/.style={inner sep = 1.5pt, circle, fill=black}}
\centering{
\begin{tikzpicture}

\draw     (0,0) node[round,red] (1) {} 
          ++(1,0)  node[round] (2) {} 
          ++(1,0)  node[round] (3) {}
          ++(1,0)  node[round] (4) {}
          ++(1,0)  node[round] (5) {} 
          ++(1,0)  node[round] (6) {} 
          ++(1,0)  node[round,blue] (7) {} 
          ++(-1,1)  node[round] (6') {} 
          ++(-1,0)  node[round] (5') {}
          ++(-2,0)  node[round] (3') {}
          ++(-2,0)  node[round] (1') {} ;
                    
\draw [->]
(1) edge node[below = -.65mm]{$a$} (2)
(2) edge node[below = -.65mm]{$a$} (3)
(3) edge node[below = -.65mm]{$a$} (4)
(4) edge node[below = -.65mm]{$a$} (5)
(5) edge node[below = -.65mm]{$a$} (6)
(6) edge node[below = -.65mm]{$a$} (7)
(1) edge node[right = -.65mm]{$b$} (1')
(3) edge node[right = -.65mm]{$b$} (3')
(5) edge node[right = -.65mm]{$b$} (5')
(6) edge node[right = -.65mm]{$b$} (6');
          \end{tikzpicture}}
\caption{The Munn tree of $w_0 = bb^{-1} aabb^{-1}aabb^{-1}abb^{-1}a$. Every edge labelled with $x \in \{a,b\}$
has an inverse edge labelled with $x^{-1}$, which is not shown. The red vertex is the identity element of the free group and the blue
vertex is $\red(w_0) = a^6$.
}  \label{fig Munn tree}
\end{figure}          

\begin{theorem}[\protect{\cite{munn:1974}}] \label{munn}
For all $u,v \in \Sigma^*$, we have:
$u \equiv_{\FIM(\Sigma)} v$ if and only if 
$\red(u) = \red(v)$ and $\MT(u) = \MT(v)$.
\end{theorem}
Thus, the element of $\FIM(\Sigma)$ that is represented by a word $u \in  \Sigma^*$
 can be uniquely represented by the pair
$(\MT(u), \red(u))$.
In fact, if we define 
on the set of all pairs $(U,v) \in 2^{\IRR(\Sigma)} \times \IRR(\Sigma)$
(with $v \in U$ and $U$ finite and prefix-closed)
a multiplication by
$(U,v)(V,w) = (\red(U \cup vV), \red(vw))$, 
then the resulting monoid is isomorphic to $\FIM(\Sigma)$. The neutral element is $(\{\varepsilon\}, \varepsilon)$.
If we exclude this element then we obtain the free inverse semigroup $\mathsf{FIS}(\Sigma)$ generated by $\Sigma$.
Note that $\mathsf{FIM}(\Sigma) = \mathsf{FIS}(\Sigma)^I$.

 If $\Sigma$ is finite and $k = |\Sigma|$ then we also write $\FIM_k$ for $\FIM(\Sigma)$.
 \begin{theorem} \label{thm-FIM-2}
 Let $k \ge 2$ and $0 < \epsilon< 1/2$ be a fixed constant. Then, every $\epsilon$-distinguisher
 or  $\FIM_k$ has space complexity $\Omega(n)$.
 \end{theorem}
  
 \begin{proof}
 We only have to prove the lower bound for $k=2$. For this, we make a reduction from the randomized one-way communication complexity of $\IDX_n$; see
 Section~\ref{sec-CC}.
 Let $\mathcal{R} = (\mathcal{A}_n)_{n \ge 0}$ 
be an $\epsilon$-distinguisher for $\FIM_2 = \FIM(\{a,b\})$ with $\mathcal{A}_n = (Q_n,\{a,a^{-1},b^{-1}\},\iota_n,\delta_n)$ and $0 < \epsilon< 1/2$.
We will construct a  randomized one-way communication protocol for $\IDX_n$. 
Let $x_1 x_2 \cdots x_n \in \{0,1\}^n$ be the input of Alice and $i \in [1,n]$ be the input of Bob.
The protocol works as follows, where $m = 5n$. 
\begin{itemize}
\item Alice chooses an initial state $q_0$ of $\mathcal{A}_n$ according to the distribution $\iota_m$,
computes the state $q_1 = \delta_m(q_0, w_0)$ where 
\[w_0 = b^{x_1}b^{-x_1} a b^{x_2}b^{-x_2} a \cdots b^{x_{n-1}}b^{-x_{n-1}} a b^{x_n}b^{-x_n}\]
and sends it to Bob.  
\item Bob then computes the state 
\[
q_2 = \delta_m(q_1, a^{i-n} bb^{-1} a^{n-i}) = \delta_m(q_0, w_0 a^{i-n} bb^{-1} a^{n-i})
\]
and outputs $1$ if $q_1 = q_2$ and $0$ otherwise.
\end{itemize}
 Figure~\ref{fig Munn tree} shows the Munn tree of the word $w_0$ that corresponds to Alice's input $1010110$.
Note that
\[
w_0 a^{i-n} bb^{-1} a^{n-i} \equiv_{\FIM_2} 
b^{x_1}b^{-x_1} a \cdots  b^{x_{i-1}}b^{-x_{i-1}} a  bb^{-1} a  b^{x_{i+1}}b^{-x_{i+1}} \cdots a b^{x_n}b^{-x_n}
\] 
(to see this, consider the Munn trees of both words),
which is $w_0$ in $\FIM_2$ if and only if $x_i = 1$.

Note that the length of $w_0 a^{i-n} bb^{-1} a^{n-i}$ is bounded by $5n$. Hence, if $x_i = 1$ then $q_1 = q_2$ with probability at least $1-\epsilon$
and if $x_i = 0$ then $q_1 \neq q_2$ with probability at least $1-\epsilon$. Thus, the protocol is correct. 
By Theorem~\ref{theorem-randCC} the cost of the protocol
must be $\Omega(n)$. Since Alice sends at most $s(\mathcal{R},5n)$ bits to Bob, we obtain
$s(\mathcal{R},5n) \geq \Omega(n)$, i.e., $s(\mathcal{R},n) \geq \Omega(n)$ (since 
$s(\mathcal{R},n)$ is monotone by our assumptions from Section~\ref{sec-injective}).
 \end{proof}
 For $k=1$, i.e., for the monoid $\FIM_1$ (also known as the \emph{free monogenic inverse monoid}), Munn trees are finite intervals on $\mathbb{Z}$. 
 Hence, elements of $\FIM_1$ can be represented by triples $[a, i, b]$ with $a, i, b \in \mathbb{Z}$, $a \leq b$, and $0, i \in [a,b]$.
 The multiplication becomes 
 \[
 [a, i, b] \cdot [c, j, d] = [\min\{a, i+c\}, i+j, \max\{b,i+d\}];
 \] 
 see also \cite[Proposition~IX.1.1]{petrich:1984}. Moreover, we have $[a,i,b]^{-1} = [a-i, -i, b-i]$ \cite[Corollary~IX.1.2]{petrich:1984}.
 The generator of $\FIM_1$ is $s = [0,1,1]$ with $s^{-1} = [-1,-1,0]$.
  
\begin{theorem} \label{thm-FIM-1}
$\FIM_1$  has a deterministic distinguisher with space complexity $\mathcal{O}(\log n)$.
Moreover, for every fixed constant $0 < \epsilon< 1/2$, every $\varepsilon$-distinguisher for $\FIM_1$ has space complexity $\Omega(\log n)$.
\end{theorem}
 
 \begin{proof}
 The upper bound is obtained by storing the integers in the triples $[a,i,b]$ in binary notation.
 The lower bound is shown by a reduction from the randomized one-way communication complexity of $\GT_n$. 
Assume that $\mathcal{R} = (\mathcal{A}_n)_{n \ge 0}$ 
is an $\epsilon$-distinguisher for $\FIM_1$ with $\mathcal{A}_n = (Q_n,\{s,s^{-1}\},\iota_n,\delta_n)$ and $0 < \epsilon< 1/2$.
We will construct a one-way randomized communication protocol for $\GT_n$. 
Let $x \in [1,n]$ be the input of Alice and $y \in [1,n]$ be the input of Bob.
Let $m = 4n$. Alice chooses an initial state $q_0$ of $\mathcal{A}_n$ according to the distribution $\iota_m$,
computes the state $q_1 = \delta_m(q_0, s^x s^{-x})$ and sends it to Bob.
Bob then computes the state $q_2 = \delta_m(q_1, s^y s^{-y}) = \delta_m(q_0, s^x s^{-x}s^y s^{-y})$ and outputs $1$ if $q_1 \neq q_2$ and $0$ otherwise.

Note that $s^x s^{-x}s^y s^{-y}$ has length at most $4n = m$ and that 
\begin{equation*}
s^x s^{-x}s^y s^{-y} \equiv_{\FIM_1} s^{\max\{x,y\}} s^{-\max\{x,y\}},
\end{equation*}
which is $s^x s^{-x}$ if and only if $x \geq y$.

Hence, if $x < y$ then $q_1 \neq q_2$ holds with probability at least $1 - \epsilon$, i.e., Bob outputs $1$ with 
probability at least $1 - \epsilon$. On the other hand, if $x \geq y$ then Bob outputs $0$ with 
probability at least $1 - \epsilon$. Hence, the protocol is correct. By Theorem~\ref{theorem-randCC} the cost of the protocol
must be $\Omega(\log n)$. Since Alice sends at most $s(\mathcal{R},4n)$ bits to Bob, we obtain
$s(\mathcal{R},4n) \geq \Omega(\log n)$, i.e., $s(\mathcal{R},n) \geq \Omega( \log n)$.
 \end{proof}
 
 \subsubsection{Polycyclic monoids}
 
  Another important class of inverse monoids are the polycyclic monoids. 
 For a set $\Sigma$, the \emph{polycyclic monoid} $\PM(\Sigma)$ is the monoid generated by all partial bijections on $\Sigma^*$ of 
 the following form, where $a \in \Sigma$:
 \begin{alignat*}{5}
 p_a &  : \ & \Sigma^*  & \to \ \Sigma^* a & & \ \text{ with } \ & p_a(w)  & = wa \\
 p^{-1}_a & : \ & \Sigma^* a & \to \ \Sigma^* & & \ \text{ with } \ & p^{-1}_a(wa) & = w  
 \end{alignat*}
 The neutral element is of course the identity mapping $p_a p_a^{-1}$ on $\Sigma^*$.
 Note that for two different letters $a,b \in \Sigma$, $p_a p^{-1}_b$ is the totally undefined mapping and hence the zero element
 of $\PM(\Sigma)$. If $\Sigma$ is finite and $k = |\Sigma|$ then we write $\PM_k$ for $\PM(\Sigma)$; it is a finitely 
 generated monoid. In fact, $\PM_k$ is finitely presented:
 For $k \ge 2$ the monoid is $\PM_k$ is isomorphic to
  \begin{align*}
 \langle a_0, b_0, \ldots, a_{k-1}, b_{k-1}, 0 \mid \ & a_i b_i = \varepsilon, a_i b_j = 0, \\
 & a_i 0 = 0a_i = b_i 0 = 0 b_i = 00 = 0  \  (i,j \in [0,k-1], i \neq j) \rangle^*
 \end{align*}
 and $\PM_1$ is isomorphic to $ \langle a,b \mid ab=\varepsilon \rangle^*$. The monoid $\PM_1$
 is known as the \emph{bicyclic monoid} and denoted with $B$.
 All monoids $\PM(\Sigma)$ are inverse since they are monoids of partial injections.
 Polycyclic monoids were introduced in \cite{NiPe70}.
 
 \begin{theorem}
 Let $k \ge 2$ and $0 < \epsilon< 1/2$ be a constant. Then, every $\epsilon$-distinguisher
 for  $\PM_k$ has space complexity $\Omega(n)$.
 \end{theorem}
 
 \begin{proof}
 We prove the lower bound for $k=2$ by a reduction from the randomized one-way communication complexity of $\AIDX_n$ (augmented index). 
 Assume that $\mathcal{R} = (\mathcal{A}_n)_{n \ge 0}$ 
is an $\epsilon$-distinguisher for $\PM_2$ with $\mathcal{A}_n = (Q_n,\{a_0, b_0, a_1, b_1, 0\},\iota_n,\delta_n)$ and $0 < \epsilon< 1/2$.
We will construct a randomized one-way  communication protocol for $\AIDX_n$. 
Let $x_1 x_2 \cdots x_n \in \{0,1\}^n$ be the input of Alice. The input of Bob is $i \in [1,n]$ and the promised suffix $x_{i+1} \cdots x_n$ of Alice's input.
Let $m = 2n+1$. Alice chooses an initial state $q_0$ of $\mathcal{A}_n$ according to the distribution $\iota_m$,
computes the state $q_1 = \delta_m(q_0, w_0)$ where $w_0 = a_{x_1} a_{x_2} \cdots a_{x_n}$
and sends it to Bob.
Bob then computes the states
\begin{alignat*}{2}
q_2 & \ = \ \delta_m(q_1, b_{x_n}\cdots b_{x_{i+1}}) & & \ = \ \delta_m(q_0, w_0 b_{x_n} \cdots b_{x_{i+1}}), \\
q_3 & \ = \ \delta_m(q_2, b_1 a_1) & & \ = \ \delta_m(q_0, w_0 b_{x_n} \cdots b_{x_{i+1}} b_1 a_1)
\end{alignat*}
and outputs $1$ if $q_2 = q_3$ and $0$ otherwise.
We have 
\[w_0 b_{x_n} \cdots b_{x_{i+1}} \equiv_{\PM_2} w_0 b_{x_n} \cdots b_{x_{i+1}} b_1 a_1\]
 if and only if $x_i = 1$.

The length of $w_0 b_{x_n} \cdots b_{x_{i+1}} b_1 a_1$ is bounded by $2n+1$. Hence, if $x_i = 1$ then $q_2 = q_3$ with probability at least $1-\epsilon$
and if $x_i = 0$ then $q_2 \neq q_3$ with probability at least $1-\epsilon$. Hence, the protocol is correct. 
By Theorem~\ref{theorem-randCC} the cost of the protocol
must be $\Omega(n)$. Since Alice sends at most $s(\mathcal{R},2n+1)$ bits to Bob, we obtain
$s(\mathcal{R},2n+1) \geq \Omega(n)$ and hence $s(\mathcal{R},n) \geq \Omega(n)$.
 \end{proof}
 Every element of the bicyclic monoid $B$ can be uniquely written as $b^m a^n$ for $m,n \geq 0$ and the multiplication is defined by
 the following rule: 
 \begin{equation} \label{eq-bicyclic}
 (b^k a^\ell) (b^m a^n) = \begin{cases}
 b^k a^{\ell-m+n} \text{ if } \ell \geq m, \\
 b^{k + m - \ell} a^n \text{ if } \ell < m .
 \end{cases}
 \end{equation}
 
\begin{theorem} \label{thm-bcm}
The bicyclic monoid $B$ has a deterministic distinguisher with space complexity $\mathcal{O}(\log n)$.
Moreover, for every fixed constant $0 < \epsilon< 1/2$,
 every $\varepsilon$-distinguisher for the bicyclic monoid has space complexity $\Omega(\log n)$.
\end{theorem} 
 
 \begin{proof}
The upper bound follows from the above description of $B$ in \eqref{eq-bicyclic} by storing the exponents of the generators $a$ and $b$
in binary notation. For the lower bound we make a reduction from the randomized one-way communication complexity of $\GT_n$.
Assume that $\mathcal{R} = (\mathcal{A}_n)_{n \ge 0}$ 
is an $\epsilon$-distinguisher for $B$ with $\mathcal{A}_n = (Q_n,\{a,b\},\iota_n,\delta_n)$ and $0 < \epsilon< 1/2$.
We will construct a randomized one-way communication protocol for $\GT_n$. 
Let $x \in [1,n]$ be the input of Alice and $y \in [1,n]$ be the input of Bob.
Let $m = 3n$. Alice chooses an initial state $q_0$ of $\mathcal{A}_n$ according to the distribution $\iota_m$,
computes the state $q_1 = \delta_m(q_0, a^x)$ and sends it to Bob.
Bob then computes the state $q_2 = \delta_m(q_1, b^y a^y) = \delta_m(q_0, a^x b^y a^y)$ and outputs $1$ if $q_1 \neq q_2$ and $0$ otherwise.

Note that the word $a^x b^y a^y$ has length at most $3n = m$ and that 
\begin{equation*}
a^x b^y a^y \equiv_B 
\begin{cases}
a^{x-y} a^y \equiv_B a^x & \text{ if } x \geq y, \\
b^{y-x} a^y \not\equiv_B a^x & \text{ if }  x < y .
\end{cases}
\end{equation*}
Hence, if $x < y$ then $q_1 \neq q_2$ holds with probability at least $1 - \epsilon$, i.e., Bob outputs $1$ with 
probability at least $1 - \epsilon$. On the other hand, if $x \geq y$ then Bob outputs $0$ with 
probability at least $1 - \epsilon$. Hence, the protocol is correct. By Theorem~\ref{theorem-randCC} the cost of the protocol
must be $\Omega(\log n)$. Since Alice sends at most $s(\mathcal{R},3n)$ bits to Bob, we obtain
$s(\mathcal{R},3n) \geq \Omega(\log n)$, i.e., $s(\mathcal{R},n) \geq \Omega(\log n)$.
 \end{proof}
Note that the bicyclic monoid as well as the free monogenic inverse monoid have polynomial growth. 
This shows that Theorem~\ref{thm-cancellative-poly-growth} does not extend to all monoids of polynomial growth.

\subsection{Lower bounds for free products}

We will show that the restrictions in Theorems~\ref{thm-GP} (left-cancellativity) and~\ref{thm-GP-S} (every $M_i \setminus \{1\}$ is an ideal) cannot be completely 
avoided. This is already true for the free product of two monoids.

The bicyclic monoid is not cancellative and has a deterministic distinguisher with space complexity 
$\mathcal{O}(\log n)$ (Theorem~\ref{thm-bcm}). Let $\mathcal{B}$ be the class of monoids $M$ for which there exist $x,y \in M$ such that 
 $xy=1$ and $yx \neq 1$. This is exactly the class of all monoids that contain the bicyclic monoid  $B$ as a submonoid \cite[Lemma~1.31]{ClPrest61}
or, equivalently, the class of monoids $M$ such that $M \setminus U(M)$ is not an ideal.
Let $\mathcal{C}$ be the class of all monoids $M$ for which there exist
 $x,y \in M \setminus \{1\}$ with $xy=1$. In other words: $M \setminus \{1\}$ is not an ideal.
 We have $\mathcal{B} \subseteq \mathcal{C}$.
 
 \begin{theorem} \label{thm-lower-bound-free-product}
 Let $M_1 \in \mathcal{B}$ and $M_2 \in \mathcal{C}$. 
  For every fixed constant $0 < \epsilon< 1/2$, every $\epsilon$-distinguisher
 for the free product $M_1 \ast M_2$   has space complexity $\Omega(n)$.
 \end{theorem}

 \begin{proof}
Let $a,b\in M_1$ such that $a\neq 1 \neq b$, $ab=1$ and $ba\neq1$,
and let $x,y\in M_2 \setminus \{1\}$ such that $xy=1$.
Assume that $\mathcal{R} = (\mathcal{A}_n)_{n \ge 0}$ 
is an $\epsilon$-distinguisher for $M_1*M_2$ with $\mathcal{A}_n = (Q_n,\Sigma,\iota_n,\delta_n)$ where $0 < \epsilon< 1/2$ and 
w.l.o.g.~$a,b,x,y\in \Sigma$ . We will construct a randomized one-way communication protocol for $\AIDX_n$. 

 Let $s=x_1 x_2 \cdots x_n$ with $x_i\in \{0,1\}$ be the input word for Alice, whereas Bob has as input $i\in [1,n]$ and the promised suffix $x_{i+1}\cdots x_n$  of $s$. Let $m=6n$. The protocol works as follows:
 \begin{itemize}
\item Alice chooses an initial state $q_0$ according to the distribution $\iota_m$, computes the state $q_1=\delta_m(q_0,u)$ for the word $u=a^{1+x_1}xa^{1+x_2}x\cdots xa^{1+x_n}$ and sends $q_1$ to Bob. 
\item Bob now computes $q_2=\delta_m(q_1,v)$ for 
$v=b^{1+x_n}yb^{1+x_{n-1}}y\cdots b^{1+x_{i+1}}y$ and $q_3=\delta_m(q_2,b^2a^2)$ and outputs $1$ if $q_2=q_3$ and $0$ otherwise. 
\end{itemize}
Notice that the word $uva^2b^2$ has length at most $6n=m$. 	
 
 Let $w=uv$ and hence $q_2=\delta_m(q_0,w)$ and $q_3=(q_0,wb^2a^2)$. Moreover, $w = a^{1+x_1}x\cdots a^{1+x_{i-1}}xa^{1+x_i}$ in $M_1 * M_2$.  If $x_{i}=1$, then 
 \[ w=a^{1+x_1}x \cdots a^{1+x_{i-1}}xa^2 = wb^2 a^2\] in $M_1*M_2$. Therefore, $q_2=q_3$ with probability at least $1-\epsilon$, i.e.,  Bob outputs $1$ with probability at least $1-\epsilon$. 
 
 On the other hand, if $x_i=0$ then 
 \[ w = a^{1+x_1}x \cdots a^{1+x_{i-1}}xa \neq   a^{1+x_1}x \cdots a^{1+x_{i-1}}x ba^2 = wb^2 a^2\] in $M_1*M_2$. Therefore, $q_2 \neq q_3$ with probability at 
 least $1-\epsilon$, i.e., Bob outputs $0$ with probability at least $1-\epsilon$. Hence the protocol is correct. 
 Since Alice sends at most $s(\mathcal{R},6n)$ bits to Bob, we obtain $s(\mathcal{R}, 6n) \geq \Omega(n)$ by 
 Theorem~\ref{theorem-randCC}.
 \end{proof}
 Theorem~\ref{thm-lower-bound-free-product} implies for instance, that every $\epsilon$-distinguisher (for $0 < \epsilon< 1/2$) for 
 the free product $B \ast \mathbb{Z}$ has space complexity $\Omega(n)$, although $B$ and $\mathbb{Z}$ both have a deterministic
 distinguisher with space complexity $\Theta(\log n)$.

 \begin{oproblem}
Theorems~\ref{thm-GP}, \ref{thm-GP-S}, and \ref{thm-lower-bound-free-product} still leave space for improvements.
In particular, it is open whether Theorem~\ref{thm-GP-S} can be generalized to the case
where  $M_i \setminus U(M_i)$ is an ideal for every $i \in [1,c]$.
\end{oproblem}

\subsection{Lower bounds for wreath products}

By the following result, the restriction to an abelian group $A$ in Corollary~\ref{thm-wreath-by-abelian} cannot be relaxed.
 We state the result for the randomized streaming space complexity of the word problem, which might be stronger than stating
 the lower bound for randomized distinguishers (see Lemma~\ref{lemma-injective-0} and Problem~\ref{op-distinguisher-WP}).
 
\begin{theorem} \label{lower-wreath}
Let $H$ be an f.g.~non-abelian group and $G$ be an f.g.~infinite group. Then the randomized streaming space
complexity of $\WP(H \wr G)$ is $\Theta(n)$.
\end{theorem}

\begin{proof}
Let $\mathcal{R} = (\mathcal{A}_n)_{n \ge 0}$  be a randomized streaming algorithm for  $\WP(H \wr G)$, where $\mathcal{A}_n = (Q_n,\Sigma,\iota_n,\delta_n, F_n)$.
We show that we obtain a randomized communication protocol for the disjointness problem with communication cost $3 \cdot s(\mathcal{R},12n-8)$.

Fix $n \geq 1$ and two elements $g, h \in H$ with $[g,h] \neq 1$. We can w.l.o.g.~assume that $g$ and $h$ belong to the finite generating set of $H$.
We also fix a finite generating set for $G$.
Let $s  \coloneqq t_1 t_2 \cdots t_{n-1}$ be a word over the generators of $G$ such that
 $t_1 t_2 \cdots t_i \not\equiv_G t_1 t_2 \cdots t_j$ whenever $i,j \in [0,n-1]$ with $i \neq j$. Such a word exists for every $n$ since the Cayley graph of $G$
is an infinite locally finite graph and hence contains an infinite ray by K\"onig's lemma.
For a word $x = a_1 a_2 \cdots a_{n}  \in \{0,1\}^n$ and an element $c \in \{g,h,g^{-1},h^{-1}\}$ we define the word
\begin{equation} \label{eq def x[c]}
x[c] = c^{a_1} t_1 c^{a_2} t_2 \cdots c^{a_{n-1}} t_{n-1} c^{a_{n}} s^{-1},
\end{equation}
where $c^0 = \varepsilon$.
It represents the element $(f_{x,c}, 1) \in H \wr G$ with 
\[
f_{x,c}(g) = \begin{cases}
 c & \text{ if $g = t_1 \cdots t_{i-1}$ in $G$ and $a_i=1$ for some $i \in [1,n]$,} \\
 1 & \text{ otherwise.}
\end{cases}
\]
Therefore, for two words $x,y \in \{0,1\}^n$ we have $x[g] y[h] x[g^{-1}] y[h^{-1}] = 1$ in $H \wr G$ if and only if 
there is no position $i \in [1,n]$ with $x[i] = y[i] = 1$. Note that the length of the word
$x[g] y[h] x[g^{-1}] y[h^{-1}]$ is $4(3n-2) = 12n-8$.
Let $m = 12n-8$.

Our randomized communication protocol for the disjointness problem works as follows, where $x \in \{0,1\}^n$ is the input for Alice
and $y \in \{0,1\}^n$ is the input for Bob.
\begin{itemize} 
\item Alice guesses an initial state $q_0$ with respect to the distribution $\iota_m$ and
sends the state $q_1 = \delta_m(q_0, x[g])$ to Bob.
\item Bob sends the state $q_2 = \delta_m(q_1, y[h])$ to Alice.
\item Alice sends the state $q_3 = \delta_m(q_2, x[g^{-1}])$ to Bob. 
\item Bob computes the state $q_4 = \delta_m(q_3, y[h^{-1}])$
 and finally accepts if and only if $q_4 \in F_m$.
\end{itemize}
Clearly, the protocol is correct and its communication cost is $3 \cdot s(\mathcal{R},12n-8)$. Hence, by Theorem~\ref{theorem-randCC} we 
have $3 \cdot s(\mathcal{R},12n-8) \ge \Omega(n)$ which implies $s(\mathcal{R}, n) \ge \Omega(n)$.
\end{proof}

\begin{remark}
It follows from Theorem~\ref{lower-wreath} and Corollary~\ref{thm-wreath-by-abelian} that the randomized streaming space complexity of
the word problem for an f.g.~group
is not a quasi-isometric invariant. To explain this, we need some definitions; see \cite{Loeh17} for more details: Two metric spaces
$(X,d_X)$ and $(Y,d_Y)$ are called quasi-isometric if there exists a function 
$f : X \to Y$ (a so-called \emph{quasi-isometry}) and real constants $a \geq 1$, $b,c \geq 0$ 
with the following properties:
\begin{itemize}
\item $\forall x,x' \in X :  a^{-1} \cdot d_X(x,x') - b \le d_Y(f(x),f(x')) \le a \cdot d_X(x,x') + b$,
\item $\forall y \in Y \exists x \in X : d_Y(y, f(x)) \le c$.
\end{itemize}
One can show that quasi-isometry between metric spaces is an equivalence relation. 

Consider now an f.g.~group $G$ with a finite generating set $\Sigma$.
We identify the Cayley graph $C(G,\Sigma)$ with the metric space obtained by replacing every edge by a line segment of length $1$; this
is called the geometric realization of $C(G,\Sigma)$. If $\Sigma'$ is another finite generating set of $G$ then 
$C(G,\Sigma)$ and $C(G,\Sigma')$ are quasi-isometric. Finally, two f.g.~groups $G$ and $H$ are quasi-isometric if a Cayley graph
$C(G,\Sigma)$ is quasi-isometric to a Cayley graph $C(H,\Gamma)$ (for some, or equivalently, all finite generating sets $\Sigma$ and 
$\Gamma$ of $G$ and $H$, respectively).

A \emph{quasi-isometric invariant} is a mapping $\mathcal{I}$, whose domain is the class of all f.g.~groups and such that $\mathcal{I}(G) = \mathcal{I}(H)$ for all f.g.~quasi-isometric groups $G$ and $H$.
A well-known quasi-isometric invariant is the growth type: say that two functions
$f,g : \mathbb{N} \to \mathbb{R}_{\ge 0}$ are quasi-equivalent, $f \sim g$ for short, 
if there are real constants $a,c>0$, $b,d \geq 0$ such that
$f(n) \leq a \cdot g(a \cdot n+b)+b$ and $g(n) \leq c \cdot f(c \cdot n+d)+d$ for all $n \in \mathbb{N}$. This yields an equivalence relation on
functions $f : \mathbb{N} \to \mathbb{R}_{\ge 0}$. For an f.g.~group $G$ with a finite generating set $\Sigma$, we define
the growth type of $G$ as the equivalence class $[\gamma_{G,\Sigma}(n)]_\sim$. This definition
does not depend on the chosen finite generating set $\Sigma$. As mentioned above, the growth type of an f.g.~group is 
a quasi-isometric invariant. By Theorem~\ref{thm-det-growth}, also the $\sim$-equivalence class of the deterministic streaming space complexity 
of the word problem of an f.g.~group is a quasi-isometric invariant.
For the randomized streaming space complexity this is not true. Take for instance the symmetric group $S_3$ of order 6 (the smallest non-abelian finite group)
and consider the groups $G = S_3 \wr \mathbb{Z}$ and $H = \mathbb{Z}_6 \wr \mathbb{Z}$.
By Corollary~\ref{thm-wreath-by-abelian}, the randomized streaming space complexity of $\WP(H)$ is $\Theta(\log n)$
whereas by  Theorem~\ref{lower-wreath}, the randomized streaming space complexity of $\WP(G)$ is $\Theta(n)$.
On the other hand, $G$ and $H$ are quasi-isometric by \cite[Theorem~1.2]{EsFiWr13}.
\end{remark}
The following result is a direct corollary of Theorem~\ref{lower-wreath}.
\begin{corollary} \label{coro-pre-thompson}
Let $H$ be a  finitely generated group and assume that $G$ is an infinite group 
such that $H \wr G$ embeds into $H$. Then 
the randomized streaming space complexity of $\WP(H)$ is $\Theta(n)$.
\end{corollary}

\begin{proof}
Since $H \wr G$ is non-abelian, also $H$ must be non-abelian. Hence, the corollary follows directly from 
Theorem~\ref{lower-wreath}.
\end{proof}
In 1965 Richard Thompson introduced three finitely presented groups $F < T < V$ acting  on the
unit-interval, the unit-circle and the Cantor set, respectively.
Of these three groups, $F$ received most attention (the reader should not confuse $F$ with a free group). 
This is mainly due to the still open conjecture that
$F$ is not amenable, which would imply that $F$ is another counterexample to a famous conjecture 
of von Neumann (a counterexample was found by  Ol'shanskii). 
The group $F$ consists of all homeomorphisms of the unit interval that
are piecewise affine, with slopes a power of $2$ and dyadic
breakpoints. It is a finitely presented group:
\begin{eqnarray}
F & = &  \langle a, b \mid [a b^{-1} \!\!\:,\,  a^{-1} b a],  [a b^{-1} \!\!\:,\, a^{-2} b a^2]   \rangle .
 \label{presentation-F}
\end{eqnarray}
An infinite natural presentation is $F = \langle a_0, a_1, a_2, \ldots \mid  a_i^{-1} a_k a_i  a^{-1}_{k+1} (i<k)  \rangle$.
The group $F$ is orderable (so in particular torsion-free), its derived subgroup $[F,F]$ is simple and the center of $F$ is trivial (in particular, $F$ is non-abelian);
see \cite{CaFlPa96} for more details. Important for us is the fact that 
$F$ contains a copy of $F\wr\Z$ \cite[Lemma~20]{GubaSapir99}. Hence, Corollary~\ref{coro-pre-thompson}
implies:

\begin{corollary} \label{coro-thompson}
The randomized streaming space complexity of $\WP(F)$ (for Thompson's group $F$) is $\Theta(n)$. 
\end{corollary}
For the case that $G$ is finite, we can prove the following variant of Corollary~\ref{coro-pre-thompson}.

\begin{theorem} \label{thm-wreath-finite}
Let $H$ be a finitely generated group and assume there is a non-trivial finite group $G$ such that
$H \wr G$ embeds into $H$. Then there is a constant $0 < c < 1$ such that
the randomized streaming space
complexity of $\WP(H)$ is in $\Omega(n^c)$.
\end{theorem}

\begin{proof}
We can assume that $G = \Z_k$ for some $k \geq 2$.
We fix an embedding $\phi : H \wr \Z_k \to H$. Let $\tau$ be the generator of $\Z_k$ and let $\Sigma$ be a finite generating set
for $H$.
We use a construction from \cite{BartholdiFLW20} that yields an embedding 
$\phi_m :  H \wr \Z_{k^m} \to H$ for every $m \geq 1$. More precisely, it is shown in \cite[Lemma~9.5]{BartholdiFLW20} that the following mapping $\phi_m$
(where $\tau_m$ is the generator of $\Z_{k^m}$ in $H \wr \Z_{k^m}$) defines an embedding of 
$H \wr \Z_{k^m}$ into $H$:
\begin{eqnarray*}
\phi_m(\tau_m) &=& \phi^{m}(\tau) \phi^{m-1}(\tau) \cdots \phi^{2}(\tau) \phi(\tau), \\
\phi_m(a) &=& \phi^m(a) \text{ for } a \in \Sigma . 
\end{eqnarray*}
Fix $\lambda \ge 2$ such that every group element
$\phi(\tau)$ and $\phi(a)$ for $a \in \Sigma$ can be represented by a word over the alphabet 
$\Sigma$ of length at most $\lambda$.
Then each of the group elements
$\phi_m(\tau_m)$ and $\phi_m(a)$ ($a \in \Sigma$) can be 
represented by a word of length at most 
$\sum_{i=1}^m \lambda^i \leq \lambda^{m+1}$.

Let us now fix an $n$. We want to get a randomized communication protocol for the disjointness problem on inputs of length $n$.
For this we choose $m = \lceil \log_k n \rceil$ so that $k^m \geq n$.  Now observe that the protocol from the proof of Theorem~\ref{lower-wreath}
also works if the group $G$ (the right factor of the wreath product)  is $\mathbb{Z}_\ell$ for some $\ell \geq n$.
In particular, we can take the copy of 
$H \wr \Z_{k^m}$ in $H$.  Note that $H$ must be non-abelian since the wreath product $H \wr \Z_k$ is non-abelian.

In our situation, the length of the word $u[g] v[h] u[g^{-1}] v[h^{-1}]$ from the proof of
Theorem~\ref{lower-wreath} blows up to $\mathcal{O}(n  \lambda^m) = \mathcal{O}(n^{1+1/\log_\lambda k})$, 
since every generator of $H \wr \Z_{k^m}$ becomes a word of length at most $\mathcal{O}(\lambda^{m}) = \mathcal{O}(n^{1/\log_\lambda k})$ in the group $H$. If $s_H(n)$ is the randomized streaming space complexity of $\WP(H)$, we obtain
\[ 
3 \cdot s_H( \mathcal{O}(n^{1+1/\log_\lambda k})) \geq \Omega(n) ,
\]
which yields $s_H(n) \ge \Omega(n^c)$, where one can take $c = \log_\lambda(k) / (\log_\lambda(k)+1)$.
 \end{proof}
 Theorem~\ref{thm-wreath-finite} can be applied to a large class of self-similar groups acting on regularly branching infinite trees.
 In particular, it is shown in \cite[Lemma~9.8]{BartholdiFLW20} that if $G$ is a weakly branched group whose branching subgroup $K$ contains elements of finite order (see \cite{BartholdiFLW20} for definitions), then $K$ contains a copy of $K \wr \mathbb{Z}_k$ for some $k \ge 2$. Hence we get:

 \begin{corollary} \label{coro-weakly-branched}
 Let $G$ is a weakly branched group whose branching subgroup $K$ contains elements of finite order. Then there is a constant $0 < c < 1$ such that
the randomized streaming space complexity of $\WP(G)$ is in $\Omega(n^c)$.
 \end{corollary}
 An important example of a group covered by Corollary~\ref{coro-weakly-branched} is the  Grigorchuk group;
 introduced by Grigorchuk in \cite{Grigorchuk80}. It is defined as an f.g.~group of automorphisms
of the infinite binary tree; the generators are usually denoted $a,b,c,d$ and satisfy the identities $a^2 = b^2 = c^2 = d^2 = 1$
and $bc = cb = d, bd = db = c, dc = cd = b$ (we do not need the precise definition).
The Grigorchuk group is an f.g. infinite torsion group and was the first example of a group with intermediate growth as well as the 
first example of a group that is amenable but not elementary amenable.
For the Grigorchuk group we can provide the following lower and upper bound on the constant $c$ from 
Corollary~\ref{coro-weakly-branched}:

\begin{theorem}  \label{thm-grig}
Let $G$ be the Grigorchuk group. Then the following hold:
\begin{itemize}
\item The deterministic streaming space complexity of $\WP(G)$ is $\mathcal{O}(n^{0.768})$.
\item The randomized streaming space complexity of $\WP(G)$ is $\Omega(n^{1/3})$.
\end{itemize}
\end{theorem}

\begin{proof}
The first statement follows from Theorem~\ref{thm-det-growth} and the fact that the growth of the Grigorchuk group is upper bounded by 
$\exp(n^{0.768})$ \cite{Bartholdi98}. For the second statement we use the subgroup $K \le G$ generated  by $t= (ab)^2$, $v=(bada)^2$, and $w= (abad)^2$.
We show that the randomized streaming space complexity of $K$ is $\Omega(n^{1/3})$.
This subgroup $K$ has the following properties that can be all found in \cite{Harpe00}:
\begin{itemize}
\item $K$ is not abelian; for instance $t v \neq vt$.
\item $K$ contains a copy of $K \times K$. More precisely, the mapping $\phi$ with 
\begin{alignat*}{2}
\phi(t,1) & = v  \qquad & \phi(1,t) & = w \\
\phi(v,1) & = v^{-1} t^{-1} v t \qquad & \phi(1,v) & = w^{-1} t w t^{-1} \\
\phi(w,1) & = v t v^{-1} t^{-1} \qquad & \phi(1,w) & = w t^{-1} w^{-1} t
\end{alignat*}
defines an injective homomorphism $\phi : K \times K \to K$ \cite[p.~262]{Harpe00}. 
\end{itemize}
The embedding $\phi$ can be used to define for every $k \geq 1$ an embedding 
 $\phi_k : K^{2^k} \to K$  inductively by $\phi_1 = \phi$ and $\phi_{k+1}( \overline{x}, \overline{y} ) = \phi( \phi_k(\overline{x}), \phi_k(\overline{y}))$
for all $\overline{x}, \overline{y} \in K^{2^k}$. Note that for a $2^k$-tuple $\overline{x} \in \{1, t,v,w,t^{-1},v^{-1},w^{-1}\}^{2^k}$ 
we have $|\phi_k(\overline{x})| \le 8^k$ when $\phi_k(\overline{x})$ is viewed as a word over $\{t,v,w,t^{-1},v^{-1},w^{-1}\}$.

We can now prove the second statement of the theorem using arguments similar to those from
the proof of Theorem~\ref{lower-wreath}. Let $\mathcal{R} = (\mathcal{A}_n)_{n \ge 0}$ be a 
randomized streaming algorithm for $\WP(K, \{t,v,w,t^{-1},v^{-1},w^{-1}\})$. 
We show that we obtain a randomized communication protocol for the disjointness problem with communication cost $3 \cdot s(\mathcal{R}, 4n^3)$.
Fix $n \geq 1$ and assume that $n=2^k$ is a power of two.
For a word $x = a_1 a_2 \cdots a_{n}  \in \{0,1\}^n$ and an element $s \in \{t,v,t^{-1},v^{-1}\}$ define the word
$$
x[s] = \phi_k(s^{a_1}, s^{a_2}, \ldots, s^{a_{n-1}}, s^{a_{n}}).
$$
For two words $x,y \in \{0,1\}^n$ we have $x[t] y[v] x[t^{-1}] y[v^{-1}] = 1$ in $K$ if and only if 
there is no position $i \in [1,n]$ with $x[i] = y[i] = 1$ (recall that $t$ and $v$ do not commute). 
Note that the length of the word $x[t] y[v] x[t^{-1}] y[v^{-1}]$ is $4 \cdot 8^{k} = 4 n^3$. 

Using this word and the semiPFA $\mathcal{A}_{4 n^3}$,
the randomized communication protocol for the disjointness problem works in the same way as in the proof of Theorem~\ref{lower-wreath}. Three states
of $\mathcal{A}_{4 n^3}$ will be exchanged between Alice and Bob. Therefore, 
the communication cost of the protocol is $3 \cdot s(\mathcal{R}, 4n^3)$.
Hence, we get $3 \cdot s(\mathcal{R},4n^3) \ge \Omega(n)$ which implies $s(\mathcal{R}, n) \ge \Omega(n^{1/3})$.
\end{proof}

\medskip

\begin{oproblem}
For the randomized streaming space complexity of $\WP(G)$ (for $G$ the Grigorchuk group) we proved the lower bound 
$\Omega(n^{1/3})$ and the upper bound $\mathcal{O}(n^{0.768})$ (the upper bound even holds for the 
deterministic streaming space complexity); see Theorem~\ref{thm-grig}. This leaves a gap that we would
like to close.
\end{oproblem}
We conclude this section with a lower bound on distinguishers for a wreath product $M \wr G$, where $M$ is a monoid
that this not cancellative or not commutative. This lower bound nicely contrasts Corollary~\ref{thm-wreath-cancel-comm}.

\begin{theorem} \label{lower-bound-non-cancel-non-commute}
Let $G$ be a finitely generated infinite group and $M$ be a finitely generated monoid which is not cancellative or not commutative. Then, 
for every fixed $0 < \epsilon < 1/2$, every $\epsilon$-distinguisher for $M \wr G$  has space complexity $\Omega(n)$.
\end{theorem}

\begin{proof}
Fix a finite generating set $\Sigma$ for the group $G$ and a generating set $\Gamma$ for $M$ so that
$M \wr G$ is generated by $\Sigma \uplus \Gamma$.
Let $\mathcal{R}=(\mathcal{A}_n)_{n\geq 0}$ be an $\epsilon$-distinguisher for $M \wr G$
 We will derive  a randomized  communication protocol for the disjointness problem in the non-commutative case with communication cost $3 \cdot s(\mathcal{R},6n-4)$, and a randomized one-way communication protocol for the index problem in the non-cancellative case with communication cost $s(\mathcal{R},6n-4)$.

\medskip
\noindent
\emph{Case 1:} $M$ is not commutative. The proof is very similar to the proof of Theorem~\ref{lower-wreath}. 
We reduce from the communication problem $\DIS_n$. Fix $n\geq 1$ and two elements $a,b \in M$ such that $ab \neq ba$ in $M$. W.l.o.g.~we can assume that $a, b \in \Gamma$. 
Define the word $s \coloneqq t_1 t_2\cdots t_{n-1} \in \Sigma^{n-1}$ as in the proof of Theorem~\ref{lower-wreath}, so that
the prefixes of $s$ represent pairwise different elements of $G$.
For a bit string $x = a_1 a_2 \cdots a_{n}  \in \{0,1\}^n$ and $c \in \{a,b\}$ we define the word $x[c]$ as in  \eqref{eq def x[c]}.
Then for two bit strings $x, y \in \{0,1\}^n$ we have 
$x[a] y[b] = y[b] x[a]$ in $M \wr G$ if and only if 
there is no position $i \in [1,n]$ with $x[i] = y[i] = 1$. Note that the length of the words $x[a] y[b]$, $y[b] x[a]$ is at most $6n =: m$.

In the randomized communication protocol for the disjointness problem, Alice first guesses an initial state $q_0$ of $\mathcal{A}_m$ according
to its initial state distribution. By exchanging four states of  $\mathcal{A}_m$, Bob can obtain the states
$q_1 = \delta_m(q_0, x[a] y[b])$ and $q_2 = \delta_m(q_0, y[b]x[a])$, where $x \in \{0,1\}^n$  is the input of Alice and 
$y \in \{0,1\}^n$  is the input of Bob. At the end, Bob accepts if and only if $q_1 = q_2$. The protocol solves 
 $\DIS_n$ with error probability at most $\epsilon$ and its communication cost is $4 \cdot s(\mathcal{R},6n)$. Hence, by Theorem~\ref{theorem-randCC} we must have $4 \cdot s(\mathcal{R},6n) \ge \Omega(n)$ which implies $s(\mathcal{R}, n) \ge \Omega(n)$.

\medskip
\noindent
\emph{Case 2:} $M$ is not cancellative. 
We can assume that $M$ is commutative, otherwise we are done by  Case 1. Let $a,b,c,d \in M$ such that $a \neq b$ and
$$
ac = bc =: d
$$
in $M$. W.l.o.g.~we can assume that $a,b,c,d \in \Gamma$.
Let us choose the word $s = t_1 t_2 \cdots t_{n-1} \in \Sigma^*$ as in Case 1. 
For a bit string $x  \in \{0,1\}^n$ we define the word
\begin{equation*} 
x[a,b] = y_1 t_1 y_2 t_2 \cdots y_{n-1} t_{n-1} y_n s^{-1} \in (\Sigma \uplus \Gamma)^*.
\end{equation*}
where $y_i = a$ if $x[i] =0$ and $y_i = b$ if $x[i] = 1$.
It represents the element $(f, 1) \in M \wr G$ with 
\[
f(g) = \begin{cases} 
a & \text{if $g \equiv_G t_1 t_2 \cdots t_{i-1}$ and $x[i] = 0$ for some $i \in [1,n]$,} \\
b & \text{if $g \equiv_G t_1 t_2 \cdots t_{i-1}$ and $x[i] = 1$ for some $i \in [1,n]$,} \\
1 & \text{in all other cases.} 
\end{cases}
\]
In addition, for a position $i \in [1,n]$ we define the word
\[ v_i = c \, t_1  \, c  \, t_2 \cdots  c  \, t_{i-1}   \, t_{i}  \, c \cdots t_{n-1}  \, c ,\]  
which multiplies the elements sitting at ray positions $t_1 \cdots t_{j-1}$ for $j \in [1,n] \setminus \{i\}$ with $c$,  and similary, 
\[ w_i = d  \, t_1  \, d  \, t_2 \cdots d  \, t_{i-1}  \, b  \, t_{i}  \, d \cdots t_{n-1}  \, d  
\]
(recall that $d=ac = bc$ in $M$).
We then have $x[a,b] v_i = w_i$ in $M \wr G$ if and only if  $x[i] = 1$.
Note that the lengths of the words $x[a,b] v_i$ and $w_i$ is at most $6n =: m$.

We can now construct a randomized one-way communication protocol for the index problem $\IDX_n$.  
The input of Alice is  $x \in \{0,1\}^n$, whereas the input of Bob is a position $i\in [1,n]$ and he has to verify whether $x[i]=1$. 
After Alice has guessed the initial state $q_0$ of $\mathcal{A}_m$, she sends the states $q_0$ and $q_1 = \delta_m(q_0, x[a,b])$ to Bob,
who can then compute the states 
$q_2 = \delta_m(q_1, v_i) = \delta_m(q_0, x[a,b] v_i)$ and $q_3 = \delta_m(q_0, w_i)$ and accept if and only if $q_2 = q_3$.
This one-way protocol solves $\IDX_n$ with error probability at most $\epsilon$ and its communication cost is $2 \cdot s(\mathcal{R},6n)$. Hence,  by Theorem~\ref{theorem-randCC} we must have 
$2 \cdot s(\mathcal{R},6n) \ge \Omega(n)$ which implies $s(\mathcal{R}, n) \ge \Omega(n)$.
\end{proof}

\section{Part E: Randomized streaming algorithms for membership problems} \label{sec member}

Let $G$ be an f.g. group with a finite symmetric generating set $\Sigma$ and let $A \subseteq G$ be a 
subset of $G$. As before, $\pi_G : \Sigma^* \to G$ is the morphism that maps a word $w \in \Sigma^*$ to 
the group element represented by $w$. We define the language 
\[ \MP(G, A, \Sigma) = \{ w\in \Sigma^* \colon \pi_G(w) \in A \} .\]
$\MP$ stands for membership problem.
Note that $\MP(G, \{1\}, \Sigma) = \mathsf{WP}(G, \Sigma)$.
One can easily show  a statement for the streaming space complexity of $\MP(G, A, \Sigma)$ 
analogously to Lemma~\ref{lemma-gen-set}, which allows us to skip the generating set $\Sigma$ 
and just write $\MP(G, A)$ in the following.

\subsection{Subgroup membership problem for free groups} \label{sec-MPs}

In the following we are mainly interested in randomized streaming algorithms for $\MP(G, H)$ when $H$ is a subgroup of $G$.
The main result of this section states that for every finitely generated free group $F(\Gamma)$ and every
finitely generated subgroup $G \leq F(\Gamma)$ 
 there exists a randomized streaming algorithm for 
$\MP(F(\Gamma), G)$ with space complexity $\mathcal{O}(\log n)$. For this we first need a few 
more definitions concerning finite automata.

We fix the finite alphabet $\Gamma$ in this section. As usual, $\Gamma^{-1} = \{ a^{-1} \colon a \in \Gamma \}$ is a set
of formal inverses. Let $\Sigma = \Gamma \cup \Gamma^{-1}$. Recall from Section~\ref{sec-groups} that we identified the free group $F(\Gamma)$
with the set $\IRR(\Sigma)$ of all irreducible words over the alphabet $\Sigma$ and that the reduced normal form $\red(w)$
is the element of $F(\Gamma)$ represented by a word $w \in \Sigma^*$.

In the following we have to deal with a special class of finite automata over the alphabet
$\Sigma$. A \emph{partial DFA} is defined as an ordinary DFA except that
the transition function $\delta : Q \times \Sigma \to Q$ is only partially defined. 
As for (total) DFAs we extend the partial
transition function $\delta : Q \times \Sigma \to Q$ to a partial function
$\delta : Q \times \Sigma^* \to Q$. 
For $q \in Q$ and $w \in \Sigma^*$ we write
$\delta(q,w) = \bot$ if $\delta(q,w)$ is undefined, which means that one cannot read the word $w$ into the automaton
$\mathcal{A}$ starting from state $q$.
A \emph{partial inverse automaton} $\mathcal{A} = (Q, \Sigma, q_0, \delta, q_f)$ over the alphabet $\Sigma = \Gamma \cup \Gamma^{-1}$ is a partial DFA 
with a single final state $q_f$ and such that
for all $p,q \in Q$ and $a \in \Sigma$,  $\delta(p, a) = q$ implies $\delta(q,a^{-1})=p$. 

The main technique to deal with finitely generated subgroups of a free group is Stallings' folding \cite{KaMya02}. We do not need the details
of the technique.  All we need is that for every finitely generated subgroup
$G \leq F(\Gamma)$ there exists a partial inverse automaton $\mathcal{A}_G$ over the alphabet $\Sigma$ 
such that for every word $w \in \IRR(\Sigma)$ we have: $w \in G$ if and only if $w \in L(\mathcal{A}_G)$. We call $\mathcal{A}_G$ the 
\emph{Stallings automaton} for $G$. It
can be constructed quite efficiently from a given set of generators for $G$ \cite{Tou06}, but we do not need this
fact since $G$ will be fixed and not considered to be part of the input in our main result, 
Theorem~\ref{thm-sbmp-free} below.\footnote{This setting is more natural in our context, where we consider
streaming algorithms for languages. In Theorem~\ref{thm-sbmp-free} below, we will consider the language $\MP(F(\Gamma), G)$ of all words representing an element from the subgroup $G$.} The Stallings automaton has the additional property that its final state is also  the initial state.

Let $G$ be a fixed finitely generated subgroup of $F(\Gamma)$ and let
$\mathcal{A}_G = (Q, \Sigma, q_0, \delta, q_0)$ be its Stallings automaton in the following.
An important property of $\mathcal{A}_G$ is the following: If $q, q' \in Q$ and $u \in \Sigma^*$ ($u$ is not necessarily 
irreducible) are such that $\delta(q, u) = q'$ then also $\delta(q, \mathsf{red}(u)) = q'$. This follows from the fact
that $\delta(q, aa^{-1})=q$ for every $q \in Q$ and $a \in \Sigma$.
In particular, if $\delta(q_0, u) \neq \bot$, then $u \in L(\mathcal{A}_G)$ if and only if
$\mathsf{red}(u) \in L(\mathcal{A}_G)$ if and only if $\mathsf{red}(u) \in G$.

\begin{lemma} \label{lemma-stallings1}
Let $u,v \in \Sigma^*$ and $a \in \Sigma$ 
such that $\delta(q_0, u) = q_1 \in Q$, $\delta(q_1, a) = \bot$ and $v$ has no prefix $x$ with 
$\mathsf{red}(ax) = \varepsilon$. Then $\mathsf{red}(uav) \notin G$. 
\end{lemma}

\begin{proof}
 Let $u' = \mathsf{red}(u)$. Then we also have $\delta(q_0, u') = q_1$
and $\mathsf{red}(uav) \notin G$ if and only if $\mathsf{red}(u'av) \notin G$. We can therefore assume
for the rest the proof that $u \in \IRR(\Sigma)$.
Moreover, observe that $u$ cannot end with the symbol $a^{-1}$: if $u = u'a^{-1}$, then 
$\bot = \delta(q_1, a) = \delta(\delta(q_0, u'a^{-1}), a) = \delta(q_0, u') \neq \bot$.

Assume now that $v$ has no prefix $x$ with 
$\mathsf{red}(ax) = \varepsilon$. We claim that $\mathsf{red}(av)$ begins with the symbol $a$.
Assume for a moment that this is already shown. Write $\mathsf{red}(av) = ay$ for some word $y$.
Then $uav$ reduces to $uay$. The latter word is irreducible, since $ua$ is irreducible ($u$ is irreducible and
$u$ does not end with $a^{-1}$) and $ay$ is also irreducible. But $uay \notin L(\mathcal{A})$, because $\delta(q_0, uay) = \bot$.
Hence, we have $uay \notin G$ and thus $\mathsf{red}(uav) \notin G$.

It therefore remains to show that $\mathsf{red}(av)$ begins with the symbol $a$. We prove by induction that
for every prefix $x$ of $v$, $\mathsf{red}(ax)$ begins with the symbol $a$. For $x = \varepsilon$ this is clear.
Now assume that $xb$ is a prefix of $v$ ($b \in \Sigma$) and we have already shown that $\mathsf{red}(ax) = ax'$ for some word $x'$.
We obtain $\mathsf{red}(axb) = \mathsf{red}(ax'b)$. 

If $x' = \varepsilon$ then $\mathsf{red}(axb) = \mathsf{red}(ab)$.
If $b = a^{-1}$ then we obtain $\mathsf{red}(axb) = \varepsilon$. This leads to a contradiction, since $xb$ is a prefix of $v$. 
Hence, we have $b \neq a^{-1}$ and thus $\mathsf{red}(axb) = ab$ starts with $a$.

Let us now assume that $x' \neq \varepsilon$ and write $x' = x''c$ for a symbol $c \in \Sigma$.
Since $a x'' c = a x'$ is irreducible, we obtain
\[ 
\mathsf{red}(axb) = \begin{cases}
a x'' c b & \text{if } c \neq b^{-1} \\
a x'' & \text{if } c = b^{-1} .
\end{cases}
\]
In both cases, $\mathsf{red}(axb)$ starts with the symbol $a$.
This concludes the proof of the lemma.
\end{proof}

\begin{definition} \label{def-AG-factor}
For a word $w \in \Sigma^*$ we define the $\mathcal{A}_G$-factorization of $w$
 uniquely as either
\begin{enumerate}[(i)]
\item $w = w_0 a_1 u_1 \; w_1 a_2 u_2 \cdots w_{k-1} a_k u_k \; w_k$ or
\item $w = w_0 a_1 u_1 \;  w_1 a_2 u_2 \cdots w_{k-1} a_k u_k \;  w_k a_{k+1} \alpha$
\end{enumerate}
such that $k \geq 0$ and the following properties hold, where we set $\ell=k$ in case (i) and $\ell=k+1$ in case (ii):
\begin{itemize}
\item $w_0, \ldots, w_k, u_1, \ldots, u_k,\alpha \in \Sigma^*$, $a_1, \ldots, a_\ell \in \Sigma$,
\item there are states $q_1, \ldots, q_{k+1} \in Q$ such that
$\delta(q_i, w_i) = q_{i+1}$ for all $i \in [0,k]$ (recall that $q_0$ is the initial state of $\mathcal{A}_G)$,
\item $\delta(q_i,a_i) = \bot$ for all $i \in [1,\ell]$,
\item for all $i \in [1,k]$, $\mathsf{red}(a_i u_i) = \varepsilon$ but there is no prefix $u \neq u_i$ of $u_i$ with $\mathsf{red}(a_i u) = \varepsilon$, and
\item in case (ii), $\alpha$ has no prefix $x$ with $\mathsf{red}(a_{k+1} x) = \varepsilon$.
\end{itemize}
\end{definition}
Depending on which of the two cases (i) and (ii)
in Definition~\ref{def-AG-factor} holds, we say that $w$ has an $\mathcal{A}_G$-factorization of type (i) or type (ii).

Let us explain the intuition of the $\mathcal{A}_G$-factorization of $w$; see also Figures~\ref{fig-A-G} and \ref{fig-A-G2}.
We start reading the word $w$ into the automaton $\mathcal{A}_G$, beginning at state $q_0$, as long as possible.
If it turns out that $\delta(q_0,w)$ is defined, then the $\mathcal{A}_G$-factorization of $w$ consists of the single factor $w_0 = w$
and we obtain type (i) with $k=0$.
Otherwise, there is a shortest prefix $w_0$ of $w$ (the first factor of the $\mathcal{A}_G$-factorization) such that
after reading $w_0$ we reach the state $\delta(q_0, w_0) = q_1$ of $\mathcal{A}_G$ and $\delta(q_1,a_1) = \bot$, where $a_1$
is the symbol following $w_0$ in $w$.
In other words, when trying to read $a_1$, we escape the automaton $\mathcal{A}_G$ for the first time. 
At this point let $w = w_0 a_1 x$.
We then take the shortest prefix $u_1$ of $x$ such that $a_1 u_1$
evaluates to the identity in the free group $F(\Gamma)$ (if such a prefix does not exist,
we terminate in case (ii) with $\alpha=x$). This yields
a new factorization $w = w_0 a_1 u_1 y$. We then repeat this process with the word $y$ starting from the state $q_1$
as long as possible. There are two possible terminations of the process: starting from state $q_k$ we can read the whole
remaining suffix into $\mathcal{A}_G$ (and arrive in state $q_{k+1}$). This suffix then yields the last factor $w_{k}$ (it can be the empty word) 
and we are in case (i); see Figure~\ref{fig-A-G}.
In the other case, we leave the automaton $\mathcal{A}_G$ with the symbol $a_{k+1}$ from state $q_{k+1}$ ($\delta(q_{k+1},a_{k+1}) = \bot$)
and the remaining suffix has no prefix $x$ such that $a_{k+1}x$ evaluates to the identity in the free group $F(\Gamma)$. The remaining suffix
then yields the last factor $\alpha$ and we are in case (ii); see Figure~\ref{fig-A-G2}.

Note that every factor $u_i$ in Definition~\ref{def-AG-factor} must end with $a_i^{-1}$, because otherwise we would get a proper factorization $a_i u_i = xy$
with $x\neq \varepsilon \neq y$ and $\red(x) = \red(y) = \varepsilon$.

\colorlet{lightgrey}{white!90!black}
\begin{figure}[t]
\tikzset{dot/.style={draw, circle, inner sep = 1pt, minimum size = 7pt}}
\tikzset{smalldot/.style={draw, circle, inner sep = 1pt, minimum size = 3pt}}
\tikzset{line/.style={thick,->,darkgreen}}
\begin{center}
\begin{tikzpicture}
    \draw[fill=lightgrey] (2,0) ellipse (3 and 1.7) ;
    \node at (2, 0)   (a) {$\mathcal{A}_G$};
    \node[dot,fill=white] at ($(2,0)+(180:3 and 1.7)$) (q0) {$q_0$};
    \node[dot,fill=white] at ($(2,0)+(120:3 and 1.7)$) (q1) {$q_1$};
    \node[dot,fill=white] at ($(2,0)+(60:3 and 1.7)$) (q2) {$q_2$};
    \node[dot,fill=white] at ($(2,0)+(0:3 and 1.7)$) (q3) {$q_3$};
    \node[dot,fill=white] at ($(2,0)+(-60:3 and 1.7)$) (q4) {$q_4$};
    \node[dot, fill=white, above left = .4cm and 1.5cm of q4]  (q5) {$q_5$};
    
    \draw (q0) edge[line, bend right]  node[above=1mm,pos=0.4] {$w_0$} (q1);
    \draw (q1) edge[line, bend right]  node[above=-.5mm] {$w_1$} (q2);
    \draw (q2) edge[line, bend right]  node[above=.5mm,pos=0.6] {$w_2$} (q3);
    \draw (q3) edge[line, bend right]  node[above=1.3mm,pos=0.65] {$w_3$} (q4);
    \draw (q4) edge[line, bend right]  node[above=-.5mm,pos=.5] {$w_4$} (q5);
        
     \node[smalldot, fill, above left = .5cm and .5cm of q1]  (p1) {};
     \node[smalldot, fill, above left = .6cm and .4cm of q2]  (p2) {};
     \node[smalldot, fill, above right = .5cm and .5cm of q3]  (p3) {};
    \node[smalldot, fill, below right = .5cm and .5cm of q4]  (p4) {};

    \draw (q1) edge[line, red, bend left]  node[left=-1mm] {$a_1$} (p1);
    \draw (q2) edge[line, red, bend left]  node[left=-1mm] {$a_2$} (p2);
    \draw (q3) edge[line, red, bend left]  node[above,pos=0.4] {$a_3$} (p3);
    \draw (q4) edge[line, red, bend left]  node[right=-.8mm,pos=0.5] {$a_4$} (p4);

    \draw (p1) edge[line, blue, out = 90, in = 50]  node[right] {$u_1$} (q1);
    \draw (p2) edge[line, blue, out = 80, in = 30]  node[right] {$u_2$} (q2);
    \draw (p3) edge[line, blue, out = 15, in = -30]  node[right] {$u_3$} (q3);
    \draw (p4) edge[line, blue, out = -80, in = -120]  node[below=-.5mm] {$u_4$} (q4);

  \end{tikzpicture}
\end{center}
\caption{\label{fig-A-G} An $\mathcal{A}_G$-factorization of type (i) for $k=4$. The red-blue loops outside of $\mathcal{A}_G$ are loops
in the Cayley graph of $F(\Gamma)$.}
\end{figure}

\colorlet{lightgrey}{white!90!black}
\begin{figure}[t]
\tikzset{dot/.style={draw, circle, inner sep = 1pt, minimum size = 7pt}}
\tikzset{smalldot/.style={draw, circle, inner sep = 1pt, minimum size = 3pt}}
\tikzset{line/.style={thick,->,darkgreen}}
\begin{center}
\begin{tikzpicture}
    \draw[fill=lightgrey] (2,0) ellipse (3 and 1.7) ;
    \node at (2, 0)   (a) {$\mathcal{A}_G$};
    \node[dot,fill=white] at ($(2,0)+(180:3 and 1.7)$) (q0) {$q_0$};
    \node[dot,fill=white] at ($(2,0)+(120:3 and 1.7)$) (q1) {$q_1$};
    \node[dot,fill=white] at ($(2,0)+(60:3 and 1.7)$) (q2) {$q_2$};
    \node[dot,fill=white] at ($(2,0)+(0:3 and 1.7)$) (q3) {$q_3$};
    \node[dot,fill=white] at ($(2,0)+(-60:3 and 1.7)$) (q4) {$q_4$};
     \node[dot,fill=white] at ($(2,0)+(-120:3 and 1.7)$) (q5) {$q_5$};

    \draw (q0) edge[line, bend right]  node[above=1mm,pos=0.4] {$w_0$} (q1);
    \draw (q1) edge[line, bend right]  node[above=-.5mm] {$w_1$} (q2);
    \draw (q2) edge[line, bend right]  node[above=.5mm,pos=0.6] {$w_2$} (q3);
    \draw (q3) edge[line, bend right]  node[above=1.3mm,pos=0.65] {$w_3$} (q4);
    \draw (q4) edge[line, bend right]  node[above=-.5mm,pos=.5] {$w_4$} (q5);
        
     \node[smalldot, fill, above left = .5cm and .5cm of q1]  (p1) {};
     \node[smalldot, fill, above left = .6cm and .4cm of q2]  (p2) {};
     \node[smalldot, fill, above right = .5cm and .5cm of q3]  (p3) {};
    \node[smalldot, fill, below right = .5cm and .5cm of q4]  (p4) {};
      \node[smalldot, fill, below right = .6cm and .4cm of q5]  (p5) {};
      \node[smalldot, fill, left = 1cm  of p5]  (p6) {};
    
    \draw (q1) edge[line, red, bend left]  node[left=-1mm] {$a_1$} (p1);
    \draw (q2) edge[line, red, bend left]  node[left=-1mm] {$a_2$} (p2);
    \draw (q3) edge[line, red, bend left]  node[above,pos=0.4] {$a_3$} (p3);
    \draw (q4) edge[line, red, bend left]  node[right=-.8mm,pos=0.5] {$a_4$} (p4);
    \draw (q5) edge[line, red, bend left]  node[right=-.8mm,pos=0.5] {$a_5$} (p5);

    \draw (p1) edge[line, blue, out = 90, in = 50]  node[right] {$u_1$} (q1);
    \draw (p2) edge[line, blue, out = 80, in = 30]  node[right] {$u_2$} (q2);
    \draw (p3) edge[line, blue, out = 15, in = -30]  node[right] {$u_3$} (q3);
    \draw (p4) edge[line, blue, out = -80, in = -120]  node[below=-.5mm] {$u_4$} (q4);
    \draw (p5) edge[line, blue, out = -90, in = -90]  node[above=-.8mm,pos=0.5] {$\alpha$} (p6);

  \end{tikzpicture}
\end{center}
\caption{\label{fig-A-G2} An $\mathcal{A}_G$-factorization of type (ii) for $k=4$. The red-blue loops outside of $\mathcal{A}_G$ are loops
in the Cayley graph of $F(\Gamma)$.}
\end{figure}

\begin{lemma} \label{lemma-stallings2}
Let $w \in \Sigma^*$ and assume that the $\mathcal{A}_G$-factorization of $w$ and
the states $q_1, \ldots, q_{k+1}$ are as in  Definition~\ref{def-AG-factor}.
\begin{itemize}
\item If the $\mathcal{A}_G$-factorization of
$w$ is of type (i) then $\mathsf{red}(w) \in G$ if and only if $q_{k+1} = q_0$ (the initial and final state of $\mathcal{A}_G$).
\item If the $\mathcal{A}_G$-factorization of
$w$ is of type (ii) then $\mathsf{red}(w) \notin G$.
\end{itemize}
\end{lemma}

\begin{proof}
Let us first assume that the  $\mathcal{A}_G$-factorization of
$w$ is of type (i). Then the word $w$ can be reduced to $w_0 w_1 \cdots w_k$.
Moreover, we have $\delta(q_0, w_0 w_1 \cdots w_k) = q_{k+1}$.
Hence, we have $\mathsf{red}(w) \in G$ if and only if 
$\mathsf{red}(w_0 w_1 \cdots w_k) \in G$ if and only if
$w_0 w_1 \cdots w_k \in L(\mathcal{A}_G)$  if and only if
$q_{k+1} = q_0$.

Now assume that the  $\mathcal{A}_G$-factorization of
$w$ is of type (ii). The word $w$ can be reduced to $w_0 w_1 \cdots w_{k-1} w_k a_{k+1} \alpha$. 
Let $w' = w_0 w_1 \cdots w_k$.
Note that $\delta(q_0, w') = q_{k+1}$, $\delta(q_{k+1}, a_{k+1}) = \bot$ and 
$\alpha$ has no prefix $x$ with $\mathsf{red}(a_{k+1} x) = \varepsilon$.
Hence, Lemma~\ref{lemma-stallings1} yields $\mathsf{red}(w' a_{k+1} \alpha) \notin G$.
Since $\mathsf{red}(w' a_{k+1} \alpha) = \mathsf{red}(w)$, we obtain $\mathsf{red}(w) \notin G$.
This concludes the proof of the lemma.
\end{proof}
We now come to the main result of this section.

\begin{theorem} \label{thm-sbmp-free}
Let $G$ a fixed finitely generated subgroup of $F(\Gamma)$. Then for every $c > 0$ there exists a $1/n^c$-correct randomized streaming algorithm for the language $\MP(F(\Gamma), G)$ with space complexity $\mathcal{O}(\log n)$.
\end{theorem}

\begin{proof}
Take the Stallings automaton $\mathcal{A}_G = (Q, \Sigma, q_0, \delta, q_0)$ and note that $|Q|$ is a constant since $G$ is fixed.
We would like to use $\mathcal{A}_G$ as a streaming algorithm for $\MP(F(\Gamma), G)$. The problem is that we cannot
assume that the input word is irreducible. We solve this problem by using a $(1/n^{c+1},0)$-distinguisher 
$(\mathcal{B}_n)_{n \ge 0}$ for $F(\Gamma)$ with space complexity $\mathcal{O}(\log n)$.\footnote{Note that although 
$(\mathcal{B}_n)_{n \ge 0}$ has a one-sided error, our final randomized streaming algorithm for 
$\MP(F(\Gamma), G)$ will have a two-sided error.}
It exists by Theorem~\ref{thm-lin} since finitely generated
free groups are linear.

Fix an input length $n$ and let $\mathcal{B}_n = (R_n, \Sigma, \lambda_n, \sigma_n)$.
Consider an input word $w \in \Sigma^{\le n}$.
Our randomized streaming algorithm for $\MP(F(\Gamma), G)$ is
Algorithm~\ref{algo-stallings}.

\begin{algorithm}[t]
\SetKwComment{Comment}{(}{)}
\SetKwInput{KwGlobal}{global variables}
\SetKwInput{KwInit}{initialization}
\SetKwInput{KwNext}{next input letter}
\BlankLine
\KwGlobal{$q  \in Q$,  $p,r \in R_n$, $\phi \in \{0,1\}$} 
\BlankLine
\KwInit{}
$q  \coloneqq  q_0$ \label{line-init-q} \\
$\phi  \coloneqq  1$ \label{line-init-phi}\\
guess $r \in R_n$ according to the initial state distribution $\lambda_n$ of $\mathcal{B}_n$ \label{line-guess-r} \\
\BlankLine
\KwNext{$a \in \Sigma$}
\If{$\phi = 1$ and $\delta(q,a)=\bot$}{$\phi  \coloneqq  0$ \label{line-set-phi-0} \\ $p  \coloneqq  r$ \label{line-reset-p}}
\eIf{$\phi = 1$}{$q  \coloneqq  \delta(q,a)$ \label{line-sim-A_G}}{$p  \coloneqq  \sigma_n(p,a)$}
\If{$\phi = 0$ and $p=r$ \label{line-of-phi=0-p=r}}{$\phi  \coloneqq  1$ \label{line-set-phi-1}}
 accept if $\phi=1$ and $q = q_0$ \label{line-accept}
\caption{A randomized streaming algorithm for $\MP(F(\Gamma), G)$.  \label{algo-stallings}}
\end{algorithm}

The space needed by Algorithm~\ref{algo-stallings} is $\mathcal{O}(\log n)$. The variables $q$ and $\phi$ need constant space
and $p$ and $r$ both need $\mathcal{O}(\log n)$ bits. Let us now show that the error probability of Algorithm~\ref{algo-stallings} is bounded
by $1/n^c$. 
For this, assume that the $\mathcal{A}_G$-factorization of $w$
and the states $q_1, \ldots, q_{k+1} \in Q$ are as in Definition~\ref{def-AG-factor}.
Let $S$ be the set of all non-empty prefixes of the factors $a_i u_i$ and  $a_{k+1} \alpha$ (the latter only in case the 
$\mathcal{A}_G$-factorization of $w$ has type (ii)). Note that $|S| \le n$.
For the initially guessed state $r \in R_n$ (line~\ref{line-guess-r})
we have
\[ 
\Prob_{r \sim \lambda_n}\big[ \bigwedge_{s \in S} \red(s)=\varepsilon \Leftrightarrow \sigma_n(r,s)=r \big]
\geq 1- \frac{1}{n^{c+1}} \cdot n = 1 - \frac{1}{n^c}.
\]
Let us assume for the further consideration that the guessed state $r$ is such that
$\red(s)=\varepsilon \Leftrightarrow \sigma_n(r,s)=r$ holds for all $s \in S$.
We claim that under this assumption, Algorithm~\ref{algo-stallings}
accepts in line~\ref{line-accept} after reading $w$ if and only if $\mathsf{red}(w) \in G$. 
By Lemma~\ref{lemma-stallings2} it suffices to show the following:
\begin{enumerate}[(a)]
\item If the $\mathcal{A}_G$-factorization of
$w$ is of type (i) then after reading $w$ we have $\phi=1$ and $q = q_{k+1}$ in Algorithm~\ref{algo-stallings}.
\item If the $\mathcal{A}_G$-factorization of
$w$ is of type (ii) then after reading $w$ we have $\phi=0$ in Algorithm~\ref{algo-stallings}.
\end{enumerate}
To see this, observe that Algorithm~\ref{algo-stallings} starts to simulate
the Stallings automaton $\mathcal{A}_G$ on $w$ as long as possible (lines~\ref{line-init-q} and \ref{line-sim-A_G}).
If this is possible for the whole input $w$ 
(i.e., $\delta(q_0, w) \neq \bot$) then $w$ has an $\mathcal{A}_G$-factorization of type (i) consisting of the single factor $w$ (i.e., $k=0$).
Moreover, after processing $w$ by Algorithm~\ref{algo-stallings}, we have $\phi=1$ and the program variable $q$ holds
$\delta(q_0, w) = q_1 = q_{k+1}$. We obtain the above case (a).

Assume now that $\delta(q_0, w) = \bot$. Then the $\mathcal{A}_G$-factorization of
$w$ starts with $w_0 a_1$, where $\delta(q_0, w_0)=q_1$ and $\delta(q_1, a_1) = \bot$.
After processing $w_0$ by Algorithm~\ref{algo-stallings} we have $q = q_1$.
While processing the next letter $a_1$, Algorithm~\ref{algo-stallings} sets $\phi$ to $0$ (line~\ref{line-set-phi-0}) and $p$ to the initially
guessed state $r$ of $\mathcal{B}_n$ (line~\ref{line-reset-p}). Let us write
$w = w_0 a_1v$. Since the flag $\phi$ was set to $0$, 
Algorithm~\ref{algo-stallings} starts the simulation of $\mathcal{B}_n$
on input $a_1v$ starting from state $r$, whereas $\mathcal{A}_G$ is stalled.
If no non-trivial prefix of $a_1v$ is trivial in the free group then we have $\alpha = v$ and the 
$\mathcal{A}_G$-factorization of $w$ is of type (ii) with $k=0$.
Our assumption that  $\red(a_1 s)=\varepsilon \Leftrightarrow \sigma_n(r,a_1s)=r$ for all prefixes of $\alpha$
ensures that the flag $\phi$ will be never set to $1$ in line~\ref{line-set-phi-1}. We obtain the above case (b).

On the other hand, if a non-trivial prefix of $a_1v$ is trivial in the free group, then 
the $\mathcal{A}_G$-factorization of $w$ starts with $w_0 a_1 u_1$ 
By assumption, we have $\red(a_1 s)=\varepsilon \Leftrightarrow \sigma_n(r,a_1s)=r$ for all prefixes $s$ of $u_1$.
Hence, after reading $a_1$, the semiPFA $\mathcal{B}_n$ returns to the state $r$ for the first time after reading $u_1$.
Then the if-condition in line~\ref{line-of-phi=0-p=r} becomes true and the algorithm switches $\phi$ back to $1$ and resumes the simulation of 
$\mathcal{A}_G$ in state $q_1$ on the factor $w_1$.
This process now repeats and we see that the algorithm correctly locates the factors of the $\mathcal{A}_G$-factorization of $w$.
This shows the above points (a) and (b) and concludes the proof of the theorem.
\end{proof}

\subsection{Distinguishers for Schreier graphs}

For an f.g.~group $G$ with the finite generating set $\Sigma$ and a subgroup $H \le G$, 
the \emph{Schreier coset graph} $\Sch(G,H,\Sigma)$ is the following deterministic automaton:
The set of states is the set of cosets $\{ Hg \colon g \in G \}$, the initial state is $H$ and the transition
function $\delta$ is defined by $\delta(Hg, a) = Hga$ for $a \in \Sigma$.

The following result generalizes Theorem~\ref{thm-sbmp-free}:

\begin{theorem} \label{thm-schreier-dist}
Let $G$ be an f.g.~subgroup of the f.g.~free group $F(\Gamma)$. 
Then, for every $c > 0$ there is a $1/n^c$-distinguisher for  $\Sch(F(\Gamma),G,\Sigma)$ with space complexity $\mathcal{O}(\log n)$.
\end{theorem}

For the proof we need the following notations and a generalization of Lemma~\ref{lemma-stallings2}.
We build on the notations from Section~\ref{sec-MPs}. Note that two words $u,v \in \Sigma^*$ lead to the same state in 
$\Sch(F(\Gamma),G,\Sigma)$ if and only if $\red(uv^{-1}) \in G$.

Consider a word $w \in \Sigma^*$, its $\mathcal{A}_G$-factorization and the states $q_i \in Q$ of $\mathcal{A}_G$
from Definition~\ref{def-AG-factor}. We say that the $\mathcal{A}_G$-factorization of $w$ ends in state $q_{k+1}$. In addition,
 if the $\mathcal{A}_G$-factorization of $w$ is of type (ii), then the suffix $a_{k+1}\alpha$ if $w$ is called the 
\emph{$\mathcal{A}_G$-leaving suffix} of $w$.

\begin{lemma} \label{lemma-Schreier-dist}
Two words $u,v \in \Sigma^*$ satisfy $\red(uv^{-1}) \in G$ if and only if the following hold:
\begin{itemize} 
\item The $\mathcal{A}_G$-factorizations of $u$ and $v$ have the same type ((i) or (ii)) and 
end in the same state of $\mathcal{A}_G$. 
\item If the $\mathcal{A}_G$-factorizations of $u$ and $v$ are both of type (ii) then their  $\mathcal{A}_G$-leaving suffixes 
are equivalent in the free group $F(\Gamma)$.
\end{itemize}
\end{lemma}

\begin{proof}
First assume that $\red(uv^{-1}) \in G$. Hence, by Lemma~\ref{lemma-stallings2} 
the $\mathcal{A}_G$-factorization of $uv^{-1}$ is of type (i) and ends
in $q_0$. Let us write the $\mathcal{A}_G$-factorization of $uv^{-1}$ as 
$w_0 a_1 u_1 w_1 a_2 u_2 \cdots w_{k-1} a_k u_k w_k$ with the same notation as in Definition~\ref{def-AG-factor}.
In particular, the states $q_i$ have the same meaning as in Definition~\ref{def-AG-factor}, where $q_{k+1} = q_0$.
Recall that $a_i u_i$ can be written as $a_i u_i = v_i^{-1} a_i^{-1}$ for some word $v_i$ so that $(a_i u_i)^{-1} = a_i v_i$.
We make a case distinction on the position of the cut between $u$ and $v^{-1}$ in the $\mathcal{A}_G$-factorization of $uv^{-1}$.

If $u$ and $v^{-1}$ can be written as 
\begin{eqnarray}
u & = & w_0 a_1 u_1 \cdots w_{i-1} a_i u_i x \quad \text{ and } \label{eq proof factorization of uv' 1} \\
v^{-1} & = & y a_{i+1} u_{i+1} w_{i+1} \cdots a_k u_k w_k \label{eq proof factorization of uv' 2}
\end{eqnarray}
 with $w_i = x y$ (where $x = \varepsilon$ or $y = \varepsilon$ is possible)
then $w_0 a_1 u_1 \cdots w_{i-1} a_i u_i x$ is the 
$\mathcal{A}_G$-factorization of $u$ and 
$w_k^{-1} a_k v_k \cdots  w_{i+1}^{-1} a_{i+1} v_{i+1} y^{-1}$ is the $\mathcal{A}_G$-factorization of $v$.  
For $u$ this follows immediately from the definition of the $\mathcal{A}_G$-factorization. For $v$ this can be seen as follows:
Assume that the $\mathcal{A}_G$-factorization of $u$ ends in the state $q \in Q$. Hence, the word 
$w_0 w_1 \cdots w_{i-1} x$ spells a path from $q_0$ to $q$ in $\mathcal{A}_G$, whereas $y w_{i+1} \cdots w_k$ spells a path from $q$ back
to $q_0$. Since $\mathcal{A}_G$ is a partial inverse automaton, $w_k^{-1} \cdots w_{i+1}^{-1} y^{-1}$  spells the reversed path in $\mathcal{A}_G$, starting in $q_0$ and ending in $q$.
More precisely, we have $\delta(q_0, w_k^{-1}) = q_k$, $\delta(q_{j+1}, w_j^{-1}) = q_j$ for $i+1 \le j \le k-1$ and $\delta(q_{i+1}, y^{-1}) = q$.
Finally, we have $\delta(q_i, a_i) = \bot$, $\red(a_i v_i) = \red((a_i u_i)^{-1}) = \varepsilon$ and for every 
factorization $v_i = s t$ with $t \neq \varepsilon$ we have $\red(a_i s) \neq \varepsilon$, because otherwise
we would have $\red(t^{-1}) = \varepsilon$ for the proper and non-empty prefix $t^{-1}$ of $a_i u_i$.
We have shown that the $\mathcal{A}_G$-factorizations of $u$ and $v$ are both of type (i) and end in the same state $q$.

The other case happens if $u$ and $v^{-1}$ can be written as 
\begin{eqnarray*}
u & = & w_0 a_1 u_1 \cdots w_{i-2} a_{i-1} u_{i-1} w_{i-1} a_i x \text{ and } \\
v^{-1} & = & y  w_i a_{i+1} u_{i+1}  \cdots w_{k-1} a_k u_k w_k .
\end{eqnarray*}
 with $u_i = x y$ and $y \neq \varepsilon$.
We can argue similarly as for \eqref{eq proof factorization of uv' 1} and \eqref{eq proof factorization of uv' 2}.
This time, the $\mathcal{A}_G$-factorization of $u$ is $w_0 a_1 u_1 \cdots w_{i-2} a_{i-1} u_{i-1} w_{i-1} a_i x$ and the 
$\mathcal{A}_G$-factorization of $v$ is
$w_k^{-1} a_k v_k w_{k-1}^{-1} \cdots a_{i+1} v_{i+1} w_i^{-1} y^{-1}$.
Both of them  are of type (ii), they
both end in the state $q_i$ and their  $\mathcal{A}_G$-leaving suffixes are $a_i x$ and $y^{-1}$, respectively, which
satisfy $\red(a_ix) = \red(y^{-1})$ since $\red(a_i xy) = \varepsilon$.

Let us now assume that the two points in the lemma are satisfied. We have to show that $\red(uv^{-1}) \in G$.
By Lemma~\ref{lemma-stallings2} we have to show that the $\mathcal{A}_G$-factorization of $uv^{-1}$ is of type (i) and ends in $q_0$.

First assume that the $\mathcal{A}_G$-factorizations of $u$ and $v$ are of type (i) and 
end in the same state $q$ of $\mathcal{A}_G$. Hence, we can write the $\mathcal{A}_G$-factorizations of $u$ and $v$
as 
\begin{eqnarray*}
u & = & w_0 \;  a_1 u_1 \cdots w_{k-1} \; a_k u_k \; w_k \text{ and } \\
v & = & x_0 \; b_1 v_1 \cdots x_{l-1} \; b_l v_l \; x_l 
\end{eqnarray*} 
We can write $(b_i v_i)^{-1}$ as $b_i v'_i$.
It then follows that the $\mathcal{A}_G$-factorization of $u v^{-1}$ is
\[
w_0 \; a_1 u_1 \cdots w_{k-1} \;  a_k u_k \;  (w_k x_l^{-1}) \;  b_l v'_l \;  x_{l-1}^{-1} \cdots b_1 v'_1 \;  x_0^{-1} 
\]
and this $\mathcal{A}_G$-factorization is indeed of type (i) and ends in $q_0$. 

Finally, assume that the $\mathcal{A}_G$-factorizations of $u$ and $v$ are of type (ii),
end in the same state $q$ of $\mathcal{A}_G$, and their $\mathcal{A}_G$-leaving suffixes represent the same
element of the free group $F(\Gamma)$. 
Hence, we can write the $\mathcal{A}_G$-factorizations of $u$ and $v$
as 
\begin{eqnarray*}
u & = & w_0 \; a_1 u_1 \cdots w_{k-1} \; a_k u_k \; w_k \; a_{k+1} \alpha \text{ and } \\
v & = & x_0 \; b_1 v_1 \cdots x_{l-1} \; b_l v_l \; x_l \; b_{l+1} \beta ,
\end{eqnarray*} 
where $a_{k+1} \alpha$ and $b_{l+1} \beta$ are the $\mathcal{A}_G$-leaving suffixes.
We claim that the $\mathcal{A}_G$-factorization of $u v^{-1}$ is
\begin{equation} \label{eq-A_G-fact of uv-1}
w_0 \;  a_1 u_1 \cdots w_{k-1} \; a_k u_k \; w_k \; a_{k+1} (\alpha\beta^{-1} b_{l+1}^{-1})  \; x_l^{-1}  \; b_l v'_l \; x_{l-1}^{-1} \cdots b_1 v'_1 \; x_0^{-1} 
\end{equation}
To see this, note that $\red(a_{k+1} \alpha\beta^{-1} b_{l+1}^{-1})=\varepsilon$ since $\red(a_{k+1} \alpha) = \red(b_{l+1} \beta)$.
Furthermore, no proper non-empty prefix of $a_{k+1} \alpha\beta^{-1} b_{l+1}^{-1}$ can be reduced to the empty word.
For a non-empty prefix of $a_{k+1} \alpha$ this follows from the definition of the $\mathcal{A}_G$-factorization of $u$.
Moreover, if we can write $\beta^{-1} b_{l+1}^{-1} = st$ with $s \neq \varepsilon \neq t$ and $\red(a_{k+1} \alpha s) = \varepsilon$
then $\red(t) = \varepsilon$ and hence $\red(t^{-1})=\varepsilon$. But $t^{-1}$ is a proper non-empty prefix of $b_{l+1}\beta$, which
contradicts the definition of the $\mathcal{A}_G$-factorization of $v$.

Finally note that the $\mathcal{A}_G$-factorization \eqref{eq-A_G-fact of uv-1} is of type (i) and ends in $q_0$.
 \end{proof}
We can now prove Theorem~\ref{thm-schreier-dist}.

\medskip
\noindent
\emph{Proof of Theorem~\ref{thm-schreier-dist}.}
Our distinguisher for $\Sch(F(\Gamma),G,\Sigma)$ is
 the randomized streaming algorithm from the proof of Theorem~\ref{thm-sbmp-free} (Algorithm~\ref{algo-stallings})
except that line~\ref{line-accept} is removed. Moreover, for $(\mathcal{B}_n)_{n \ge 0}$ we have to take a $(1/(n^c (2n+1)),0)$-distinguisher 
 for $F(\Gamma)$ with space complexity $\mathcal{O}(\log n)$.

Consider two input words $u, v \in \Sigma^{\le n}$ and let 
\[
w_0 \, a_1 u_1 \, w_1 \, a_2 u_2 \cdots w_{k-1} \, a_k u_k \, w_k \quad\text{or} \quad
w_0 \, a_1 u_1 \,  w_1 \, a_2 u_2 \cdots w_{k-1} \, a_k u_k \,  w_k \, a_{k+1} \alpha
\]
be the $\mathcal{A}_G$-factorization of $u$ and 
\[
x_0 \, b_1 v_1 \, x_1 \, b_2 v_2 \cdots x_{l-1} \, b_l v_l \, x_l \text{ or }
x_0 \, b_1 v_1 \,  x_1 \, b_2 v_2 \cdots x_{l-1} \, b_l v_l \,  x_l \, b_{l+1} \beta
\]
be the $\mathcal{A}_G$-factorization of $v$.

Let $S$ be the set of all non-empty prefixes of the words $a_i u_i$, $b_i v_i$ 
and $a_{k+1} \alpha$, $b_{l+1} \beta$ (if they exist). We have $|S| \le 2n$.
Then with probability at least $1 - 1/n^c$ the following $2n+1$ many equivalences hold for the randomly chosen state $r$ of $\mathcal{B}_n$:
\begin{enumerate}[(a)]
\item $\red(s)=\varepsilon \Leftrightarrow \sigma_n(r,s)=r$ for all $s \in S$,
\item $\red(a_{k+1} \alpha) = \red(b_{k+1} \beta) \Leftrightarrow \sigma_n(r,a_{k+1} \alpha)=\sigma_n(r,b_{k+1} \beta)$
in case the $\mathcal{A}_G$-factorizations of $u$ and $v$ are both of type (ii). 
\end{enumerate}
Let us assume for the further consideration that the guessed state $r$ is such that these conditions 
are satisfied. We claim that $uv^{-1} \in G$ if and only if Algorithm~\ref{algo-stallings} arrives in the same memory
states after reading $u$ and $v$, respectively. In the proof of Theorem~\ref{thm-sbmp-free} we argued that
assumption (a) ensures that Algorithm~\ref{algo-stallings} correctly
locates the factors of the  $\mathcal{A}_G$-factorizations of $u$ and $v$, respectively.

If $uv^{-1} \in G$ then the two points in Lemma~\ref{lemma-Schreier-dist} hold. Conditions (a) and (b) then ensure that
 Algorithm~\ref{algo-stallings} arrives in the same memory states after reading $u$ and $v$, respectively.
 On the other hand, if $uv^{-1} \notin G$ then one of the two points in Lemma~\ref{lemma-Schreier-dist} does not hold.
 If the $\mathcal{A}_G$-factorizations of $u$ and $v$ have different types, then the values of the flag $\phi$ after reading
 $u$ and $v$, respectively, differ. If the $\mathcal{A}_G$-factorizations of $u$ and $v$ are both of type (i) then the values
 of the program variable $q$ (the state in $\mathcal{A}_G$) after reading
 $u$ and $v$, respectively, differ. Finally, if the $\mathcal{A}_G$-factorizations of $u$ and $v$ are both of type (ii) then 
 the $\mathcal{A}_G$-leaving suffixes $a_{k+1} \alpha$ and $b_{k+1} \beta$ must represent different elements in the 
 $F(\Gamma)$. By point (b) this implies  $\sigma_n(r, a_{k+1} \alpha) \neq \sigma_n(r, b_{k+1} \beta)$. Hence,
 the values of the  program variable $p$ (the state in $\mathcal{B}_n$) after reading
 $u$ and $v$, respectively, differ.
 \qed

\subsection{Lower bounds for membership problems}
It is not possible to generalize Theorem~\ref{thm-sbmp-free} to subgroups of $F(\Gamma)$ that are the normal closure $N(R)$
of a finite subset $R \subseteq F(\Gamma)$.
Let $F_2 = F(\{a,b\})$ be the free group generated by two elements. In the following we make use of Thompson's group $F$ (be aware of the fact
that $F$ is Thompson's group, whereas $F_2$ is the free group generated by two elements).
Recall that the randomized streaming space complexity of $\WP(F)$ is $\Theta(n)$;
see Corollary~\ref{coro-thompson}. Moreover, Thompson's group $F$ is finitely presented and generated by two elements $a$ and $b$; 
see \eqref{presentation-F}.\footnote{For our arguments we could take any finitely presented  2-generator group whose randomized streaming space complexity is $\Omega(n)$.}
Let $R = \{ [a b^{-1} \!\!\:,\,  a^{-1} b a],  [a b^{-1} \!\!\:,\, a^{-2} b a^2] \}$ be the set of 
relators from \eqref{presentation-F} and let $N =N(R)$ be the normal closure of $R$ in the free group $F(\{a,b\})$.
Thus, we have $F \cong F(\{a,b\})/N$.

\begin{theorem} \label{thm-thompson2}
For the subgroup $N \leq F_2$ the randomized streaming space
complexity of the language $\MP(F_2, N)$ is $\Theta(n)$. 
\end{theorem}
 
 \begin{proof}
 A word $w \in \{a,b,a^{-1},b^{-1}\}^*$ represents in the free group $F_2$ an element of $N$ if and only if in Thompson's group $F$, $w$
 represents the group identity. The theorem follows directly from  Corollary~\ref{coro-thompson}. 
 \end{proof}
 We now consider the direct product of $F_2 \times F_2$ of two free groups of rank two.
 It is a linear group of exponential growth, hence the randomized streaming space
complexity of $\WP(F_2 \times F_2)$ is in $\Theta(\log n)$. 
For the subgroup membership problem, this fact no longer holds by the following theorem.
We make use of a construction of Miha{\u\i}lova from \cite{Mih66}, where she constructed
a finitely generated subgroup of $F_2 \times F_2$ with an undecidable subgroup membership problem.
 
\begin{theorem} \label{thm-mihailova}
There is a finitely generated subgroup $G$ of $F_2 \times F_2$ such that the randomized streaming space
complexity of  $\MP(F_2 \times F_2, G)$ is in $\Theta(n)$. 
\end{theorem} 

\begin{proof}
Take again Thompson's group $F$ and let $N \leq F(\{a,b\})$ and $R$ be as defined above. 
We make use of Miha{\u\i}lova's construction from \cite{Mih66}. Let
\[
D = \{ (r,1) \colon r \in R \} \cup \{ (a,a), (b,b) \},
\]
which is viewed as a finite subset of $F_2 \times F_2$. Miha{\u\i}lova   \cite{Mih66} showed  that for every element $g \in F_2$:
\[
g \in N \ \Longleftrightarrow \ (g,1) \in \langle D \rangle \leq F_2 \times F_2.
\]
Hence, the theorem follows from Theorem~\ref{thm-thompson2}.
\end{proof}
So far, we only considered the membership problem for subgroups. Membership problems in groups have been also considered
for more general classes of subsets, e.g., finitely generated submonoids or rational subsets that are defined by finite automata, whose
transitions are labelled by group generators. For the free group $F_2 = F(\{a,b\})$ it turns out that already the membership problem in the 
submonoid $\{a,b\}^* \subseteq F(\{a,b\})$ has randomized streaming space complexity $\Theta(n)$:

 \begin{theorem}
The randomized streaming space
complexity of  $\MP(F(\{a,b\}), \{a,b\}^*)$ is in $\Theta(n)$. 
 \end{theorem}
 
 \begin{proof}
 The result is shown by a reduction from the communication complexity of 
 the augmented index problem $\AIDX_n$. 
 
 Assume that $\mathcal{R} = (\mathcal{A}_n)_{n \ge 0}$ with $\mathcal{A}_n = (Q_n, \{a,b,a^{-1}, b^{-1}\}, \iota_n, \delta_n, F_n)$
is a randomized streaming algorithm for $\MP(F(\{a,b\}), \{a,b\}^*)$.
We will construct a  randomized one-way communication protocol for $\AIDX_n$. 

 Let $s=a_1 a_2 \cdots a_n$ with $a_i\in \{a,b\}$ be the input word for Alice, whereas Bob's input is a position $i\in [1,n]$ and the promised suffix $a_{i+1}\cdots a_n$  of $s$. Bob's goal is to find out $a_i$. 
 Let $m=2n$. The protocol works as follows:
 \begin{itemize}
\item Alice guesses an initial state according to the distribution $\iota_m$,
reads his input $s$ into the semiPFA $\mathcal{A}_m$ and sends the resulting state $q_1$ to Bob.
\item Bob continues from $q_1$ and reads $a_n^{-1} a_{n-1}^{-1} \cdots a_{i+1}^{-1} a^{-1}$ into $\mathcal{A}_m$.
Let $q_2$ be the resulting state. Then Bob accepts if and only if $q_2 \in F_m$.
\end{itemize}
Note that $a_i = a$ if and only if the word $s \,a_n^{-1} a_{n-1}^{-1} \cdots a_{i+1}^{-1} a^{-1}$ 
represents an element from the submonoid $\{a,b\}^*$. Hence, if $a_i = a$ then Bob will accept with probability at least $2/3$, whereas
if $a_i = b$ then Bob will accept  with probability at most $1/3$.
Hence, the protocol is correct. 
Since Alice sends at most $s(\mathcal{R},2n)$ bits to Bob, we obtain $s(\mathcal{R}, 2n) \geq \Omega(n)$ by 
 Theorem~\ref{theorem-randCC}.
 \end{proof}

 \section{Open problems}
 
 We have mentioned already several open problems. Here are some further problems that deserve to be investigated.

\subsection*{Hyperbolic groups.}
Hyperbolic groups are one of the most important classes in geometric group theory.
The word problem for a hyperbolic group belongs to the complexity class {\sf LogCFL} (the closure of the context-free languages under logspace
reductions) \cite{Lo05ijfcs}, which is contained in $\mathsf{DSPACE}((\log n)^2)$,
and it is not known whether for every hyperbolic group the word problem belongs to logspace.
What is the space complexity of randomized streaming algorithms for hyperbolic groups? 
In particular, is the randomized streaming space complexity of the word problem for a hyperbolic group 
in $\mathcal{O}(\log n)$? It is known that there exist non-linear hyperbolic groups.

\subsection*{Residually finite groups.}
  All the groups for which we have constructed randomized distinguishers with space complexity $o(n)$ so far
  are residually finite; see the discussion at the end of Section~\ref{sec-fundamental} for logspace distinguishers.
  In addition, an f.g.~group is residually finite if and only if it faithfully acts on a rooted locally finite tree; see e.g.\cite{button2019}.
  This includes for instance the Grigorchuk group, which has a deterministic distinguisher with space complexity $o(n)$
  due to its intermediate growth.
  
Erschler \cite{Erschler04} constructed non-residually finite groups $G$ of intermediate growth. These groups are
quasi-isometric to the Grigorchuk group and therefore have growth functions that are equivalent to 
the growth function of the Grigorchuk group. Therefore, there is a non-residually finite group $G$
having a deterministic distinguisher with space complexity $\mathcal{O}(n^{0.768})$; see Theorem~\ref{thm-grig}.
Is there also a non-residually finite group having a randomized distinguisher with space complexity $\mathcal{O}(\log n)$?
An interesting concrete non-residually finite group is the Baumslag-Solitar group 
$\mathsf{BS}(2,3) = \langle a,t \mid t^{-1} a^2 t = a^3 \rangle$. The word problem for 
every Baumslag-Solitar group $\mathsf{BS}(p,q) = \langle a,t \mid t^{-1} a^p t = a^q \rangle$ can be solved in logspace
\cite{Weiss16BS}.

\subsection*{More subgroup membership problems.}
So far, we only proved the existence of randomized logspace streaming algorithms for subgroup membership problems in free groups.
It would be certainly interesting to find more groups with such algorithms. A promising candidate might be f.g.~nilpotent groups.
Their subgroup membership problems belong to $\TC^0$ \cite{MyasnikovW22}.

Moreover, there are plenty of open problems concerning distinguishers for Schreier coset graphs. Take for instance an f.g.~group $G$
with a logspace $\epsilon$-distinguisher (with $\epsilon$ sufficiently small) and let $H \leq G$ be a finite subgroup. Does 
$\Sch(G,H,\Sigma)$ have a logspace $\epsilon'$-distinguisher for some small $\epsilon'$ depending on $\epsilon$?

\subsection*{Acknowledgements.} This work has been supported by the DFG research project LO748/13-2 (Streaming Automata Theory).

  
 \def\cprime{$'$} \def\cprime{$'$}

\end{document}